\documentclass[12pt,reqno]{amsart}
\usepackage{geometry}                
\usepackage{comment} 
\usepackage{multirow}

\usepackage[flushleft]{threeparttable}

\usepackage{graphicx}
\usepackage{amsmath,amssymb,amsthm}
\usepackage{epstopdf}
\usepackage{setspace}
\usepackage[authoryear]{natbib}
\usepackage{color}

\newtheorem{theorem}{Theorem}[section]	
\newtheorem{cor}{Corollary}[section]	
\newtheorem{prop}{Proposition}[section]	
\newtheorem{assumption}{Assumption}[section]
\newtheorem{lem}{Lemma}[section]
\theoremstyle{definition}
\newtheorem{algo}{Algorithm}[section]

\newtheorem{remark}{Remark}[section]

\newcommand{\E}{\mathsf{E} }
\newcommand{\I}{\mathsf{I} }
\newcommand{\tr}{\mathsf{tr}}
\newcommand{\Var}{\mathsf{Var}}
\newcommand{\Cov}{\mathsf{Cov}}
\newcommand{\rank}{\mathsf{rank}}
\newcommand{\diag}{\mathsf{diag}}
 
\numberwithin{figure}{section}
 \numberwithin{equation}{section}
 \numberwithin{table}{section}
\begin{document}

\title{Inference for Heterogeneous Effects using Low-Rank Estimation of Factor Slopes}\thanks{We are
grateful to seminar participants at Iowa State, University of North Carolina - Chapel Hill, Duke, Connecticut, Syracuse, and IAAE 2019 for helpful discussions.}

\author{Victor Chernozhukov}
\address{Department of Economics, MIT,  Cambridge, MA 02139}
\email{vchern@mit.edu}

\author{Christian Hansen}
\address{Booth School of Business, University of Chicago, Chicago, IL 60637}
\email{Christian.Hansen@chicagobooth.edu}
 
\author{Yuan Liao}
\address{Department of Economics, Rutgers University, New Brunswick, NJ 08901}
\email{yuan.liao@rutgers.edu}

\author{Yinchu Zhu}
\address{Lundquist College of Business, University of Oregon, 1208 University St, Eugene, OR 97403 }
\email{yzhu6@uoregon.edu}

\date{First draft: December 25, 2018. This draft: \today}                                           

\begin{abstract}
We study a panel data model with general heterogeneous effects where slopes are allowed to vary across both individuals and over time. The key dimension reduction assumption we employ is that the heterogeneous slopes can be expressed as  having a factor structure so that the  high-dimensional slope matrix is low-rank and can thus be estimated using low-rank regularized regression. We provide a simple multi-step estimation procedure for the heterogeneous effects. The procedure makes use of sample-splitting and orthogonalization to accommodate inference following the use of penalized low-rank estimation. We formally verify that the resulting estimator is asymptotically normal allowing simple construction of inferential statements for  {the individual-time-specific effects and for cross-sectional averages of these effects}. We illustrate the proposed method in simulation experiments and by estimating the effect of the minimum wage on employment.
\end{abstract}

\maketitle

\textbf{Key words:}  nuclear norm penalization, singular value thresholding, sample splitting, interactive effects, post-selection inference
 
 
\onehalfspacing

 \newpage

\section{Introduction}

This paper studies estimation and inference within the following panel data model:
\begin{eqnarray}\label{baseline model}
y_{it}&=&x_{it}'\theta_{it}+ \alpha_i'g_t+u_{it} ,\quad i=1,..., N,\quad t=1,..., T,
\end{eqnarray}
where $y_{it}$ is the scalar outcome of interest, $x_{it}$ is a $d$-dimensional vector of observed covariates with heterogeneous slopes $\theta_{it}$, $\alpha_i$ and $g_t $ are unobserved fixed effects, and $u_{it}$ is an unobserved error term. The model permits general heterogeneity in the sense that fixed effects appear in the model interactively and the slope $\theta_{it}$ is allowed to vary across both $i$ and $t$. 
{The main goal of this paper is to provide an asymptotically valid procedure for performing statistical inference for averages of the $\theta_{it}$ taken across subgroups in the population at specific time periods. The subgroups may consist of single individuals, in which case inference is for a specific $\theta_{it}$; the entire cross-section; or any pre-specified subset of the cross-sectional units.}
 
The main dimension reduction assumption employed in this paper is that $\theta_{it}$ can be expressed as
$$\theta_{it}= \lambda_i'f_t.$$
That is, we assume the slopes $\theta_{it}$ can be represented by a factor structure where $\lambda_i$ is a matrix of loadings and $f_t$ is a vector of factors. We allow $f_t$ and $g_t$ to have overlapping components and allow $\lambda_i$ and $f_t$ to be constant across $i\leq N$ and $t\leq T$ respectively. Thus, the model accommodates homogeneous slopes as a special case.

It will be useful to represent the model in matrix form. 
Let $\Theta_r$ and $X_r$ be the $N \times T$ matrices with $r^{\text{th}}$ component $\theta_{it,r}$ and $x_{it,r}$ respectively. Let $M, Y,$ and $U$ be the $N\times T$ matrices of $\alpha_i'g_t$, $y_{it}$, and $u_{it}$. Finally, let $\odot$ denote the matrix element-wise product. Using this notation, we may express the model in matrix form as
$$
Y=\sum_{r=1}^d X_r\odot \Theta_r+ M+U.
$$

Under the maintained factor structure, the slope and fixed effect matrices, $\Theta_{r}$ and $M$, have rank at most equal to their associated numbers of factors. That is, these matrices are low-rank assuming the number of factors is small. 
This structure motivates estimating the model parameters via low-rank estimation:
\begin{align}
\begin{split}\label{e1.2}
\min_{\{\Theta_1,...,\Theta_d, M\}} & \|Y- \sum_{r=1}^d X_r\odot \Theta_r-M\|_F^2+P_0(\Theta_1,...,\Theta_d, M) \\
\text{where} & 
\quad P_0(\Theta_1,...,\Theta_d, M)=\sum_{r=1}^d\nu_r\|\Theta_r\|_n+\nu_0\|M\|_n
\end{split}
\end{align}
for some tuning parameters $\nu_0, \nu_1,...,\nu_d>0$ and $\|.\|_F$ and $\|.\|_n$ respectively denote the matrix Frobenius norm and the nuclear norm. In particular, let $\psi_1(\Theta)\geq...\geq\psi_{\min\{T, N\}}(\Theta)$ be the sorted singular values of  an $N\times T$ matrix $\Theta$, then 
$$
\|\Theta\|_n:=\sum_{i=1}^{\min\{T,N\}} \psi_i(\Theta).
$$
The nuclear norm thus provides a convex-relaxation of the rank of $\Theta$, and solution of \eqref{e1.2} can be achieved using efficient algorithms such as the singular value decomposition.  

The penalized low-rank estimators defined in \eqref{e1.2} will be consistent for matrices $\Theta_1$, ..., $\Theta_d$ and $M$ with suitable choice of the tuning parameters $\nu_0, ..., \nu_d$. However, the use of regularized regression also substantially complicates inference. For example, the estimators from \eqref{e1.2} may have large shrinkage bias in finite samples. 
This paper makes a useful contribution to the literature on penalized low-rank estimation by providing a formally valid approach to obtaining inferential statements about elements of the heterogeneous slope matrices $\Theta_1$, ..., $\Theta_d$ after application of singular value thresholding (SVT) type regularization. That is, we provide an approach to obtaining valid ``post-SVT inference.''

The chief difficulty to providing valid inferential statements about elements of $\theta_{it}$ is that estimates obtained by directly solving \eqref{e1.2} may suffer from substantial regularization bias. To overcome this bias, we use the solution of \eqref{e1.2} as initial estimates in an iterative least squares procedure. More specifically, let $\widetilde \Theta_1$, ..., $\widetilde \Theta_d$, and $\widetilde M$ denote the estimates of the $\Theta_r$ and $M$ obtained by solving \eqref{e1.2}.  We use them to estimate the rank of the $\Theta_r$ and $M$ and to obtain preliminary estimates of the $\lambda_{i}$ and $\alpha_i$, $\widetilde \lambda_i$ and $\widetilde\alpha_i$, by extracting the respective singular vectors from the estimated matrices. We then regress $y_{it}$ on the elements of $(x_{it}'\widetilde \lambda_i, \widetilde \alpha_i)$ to obtain $\widehat f_t$ and $\widehat g_t$ - ``de-biased'' estimates of $f_t$ and $g_t$. Finally, we obtain final estimates of $\lambda_i$ and $\alpha_i$ by regressing $y_{it}$ on $(\widehat g_t, \widehat f_tx_{it})$. The final estimator of $\theta_{it}$ is then $\widehat \theta_{it} = \widehat \lambda_i'\widehat f_t$. This procedure is analogous to the ``post-lasso'' procedure of \cite{belloni2013least} which proposes using OLS regression using only the covariates selected in a first stage lasso step to alleviate shrinkage bias in the high-dimensional linear model context. 

In practice and in our theoretical development, we make two modifications to the algorithm sketched in the preceding paragraph. First, we assume that the observed variables in $x_{it}$ also follow factor structures and make use of estimates of this factor structure to appropriately orthogonalize all variables in the model. This orthogonalization removes strong sources of dependence captured by a factor structure and importantly alleviates the impact of estimation of the interactive fixed effects, which are high-dimensional nuisance parameters from the perspective of learning $\theta_{it}$. Without this orthogonalization, the asymptotic normality of $\widehat\theta_{it}$ would require extremely strong restrictions on the observed covariates in $x_{it}$. The second modification is that we make use of sample splitting to avoid technical complications arising from potential correlation between estimation error in nuisance parameters and $x_{it}$. Section \ref{sec:2.2} explains the full algorithm and provides more discussion of the use of orthogonalization and sample-splitting.

 In summary, our method, building on the factor-structure of   heterogeneous slopes, is different from the usual low-rank estimation  in the literature. The factor-structure allows us to iteratively estimate the factors and loadings, producing unbiased estimators for heterogeneous effects that are ready for  inference. 
 
Nuclear norm penalization is now standard technique in the statistical literature for estimating low-rank/factor models; see, e.g., \cite{negahban2011estimation, recht2010guaranteed, sun2012calibrated, candes2010power, koltchinskii2011nuclear}. However, there is relatively little work surrounding element-wise inference for the low-rank matrices. The only work that we are aware of is \cite{cai2016geometric} which proposes a debiased estimator to make inference about the low-rank matrix after applying nuclear-norm penalization. We note that this debiased procedure does not apply in the model we are considering (even for the special case of our model with factor models for intercepts
and homogeneous models for slopes due to presence of additional control variables). The main issue is that, under sensible conditions, low-rank matrices resulting from a high-dimensional factor structure   have ``very spiked" eigenvalues, possessing leading eigenvalues that grow at rate $O(\sqrt{NT})$.  The presence of these fast-growing eigenvalues calls for  new analysis.

Nuclear-norm penalized estimation has also been applied in the recent econometric literature for panel data models. \cite{bai2017principal} studied the properties of nuclear-norm penalized estimation in the pure factor model setting. {\cite{athey2018matrix} consider the use of matrix completion methods to impute missing potential outcomes for estimating causal effects with a binary treatment variable. They show how estimating missing potential outcomes via nuclear-norm penalized methods relates to other imputation methods used in the causal inference literature and provide some results on rates of convergence for the estimated low-rank matrix.} \cite{moon2018nuclear} studied inference for common parameters in panel data models with homogeneous slopes and interactive fixed effects - i.e., \eqref{baseline model} where $\theta_{it}=\theta$ is imposed for all $(i,t)$. They apply nuclear-norm penalization to estimate the low-rank matrix of the interactive fixed effects.

The paper is also closely related to the literature on estimating interactive fixed effects model with homogeneous slopes: 
 \cite{bai09}, \cite{ MW11}, and \cite{ahn2013panel}.
   In addition, \cite{pesaran} 
studies time-invariant homogeneous models in which the slopes are allowed  to vary across $i$ - i.e., models with $\theta_{it} = \theta_i$ for all $t$. {\cite{su2015specification} consider testing against a nonlinear homogeneous effects model  - i.e. a model where $y_{it} = g(x_{it}) + \alpha_i'g_t + u_{it}$ for all $i$ and $t$.} Our approach builds upon this literature by considering a model with slopes that are allowed to be heterogeneous both across individual units $i$ and across time $t$ which allows capturing a richer spectrum of effects than can be accommodated in analogous homogeneous effects models.



The rest of the paper is organized as follows. Section 2 introduces the post-SVT algorithms that define our estimators. It also explains the rationale for using orthogonalization and sample splitting as part of our overall estimation and inference strategy. Section 3 provides asymptotic inferential theory. Section 4 presents simulation results, and Section 5 applies the proposed method to estimate minimum wage effects on employment. All proofs are presented in the appendix.

\section{The Model}

We consider the model
\begin{align}
\begin{split}\label{baseline model 2}
y_{it} &= x_{it}\theta_{it}+ \alpha_i'g_t+u_{it} ,\quad i=1,..., N,\quad t=1,..., T, \ \text{with} \\ 
\theta_{it} &= \lambda_i'f_t.
\end{split}
\end{align}
We observe $(y_{it}, x_{it})$, and the goal is to make inference about $\theta_{it}$, which we assume follows a latent factor structure, {or averages of the $\theta_{it}$ taken across groups formed from cross-sectional units}. The fixed effects structure $\alpha_i'g_t$ in \eqref{baseline model 2} allows fixed effects to be additive, interactive, or both; see \cite{bai09}. We assume that $\dim(\lambda_i)=K_1$ and $\dim(\alpha_i)=K_2$ where both $K_1$ and $K_2$ are fixed. For ease of presentation, we focus on the case where $x_{it}$ is univariate, that is, $\dim(x_{it})=1$. The extension to the multivariate case is straightforward and is provided in Appendix \ref{app: multivariate}.

We assume that $\{\lambda_i, \alpha_i\}$ are deterministic sequences, while $\{f_t, g_t\}$ are random. We allow arbitrary dependence and similarity among $\{f_t: i\leq N, t\leq T\}$, and impose nearly no restrictions on the sequence for $\lambda_i$ and $f_t$. 
 For instance,     it  allows  homogeneous models  $\theta_{it}=\theta$ for all $(i,t)$ and a common parameter $\theta$ as   special cases by setting $\lambda_i=\lambda$ and $f_t=f$ for all $(i,t)$ and dimension $\dim(\lambda_i)=\dim(f_t)=1$.  This then reduces to the homogeneous interactive effect model, studied by \cite{bai09, MW11}.  


\subsection{Nuclear Norm Penalized Estimation}

Let $\Lambda, A$ respectively be $N\times K_1$ and $N\times K_2$ matrices of $\lambda_i$ and $\alpha_i$. Let $F, G$ respectively be $T\times K_1$ and $T \times K_2$ matrices of $f_t$ and $g_t$. Let $(Y, X, U)$ be $N\times T$ matrices of $(y_{it}, x_{it}, u_{it})$. Then \eqref{baseline model 2} may be expressed in matrix form as
\begin{eqnarray*}
Y&=&AG'+ X\odot (\Lambda F') +U
\end{eqnarray*}
where $\odot$ represents the element-wise product.   
Further let  $\Theta:=\Lambda F' $ and $M:=AG'$. Note that both $\Theta$ and $M$  are $N\times T$ matrices whose ranks are respectively $K_1$ and $K_2$. We consider asymptotics where $N,T\to\infty$ but $K_1, K_2$ are fixed constants. Thus, $\Theta$ and $M$ are low-rank matrices. Motivated by this structure, a simple baseline estimator of $(M, \Theta)$ results from the following penalized nuclear-norm optimization problem:
\begin{align}
\begin{split}\label{e1}
 (\widetilde\Theta,\widetilde M)&= \arg\min_{\Theta, M} F(\Theta, M),\\
 F(\Theta, M)&:=\|Y-M- X\odot \Theta\|_F^2+\nu_2\|M\|_n+\nu_1\|\Theta\|_n 
\end{split}
\end{align}
 for some tuning parameters $\nu_2, \nu_1>0$.

\subsection{Computation using Singular Value Thresholding}
 
 For a fixed matrix $Y$, let $UDV'=Y$ be its singular value decomposition. Define the singular value thresholding operator $$S_{\lambda}(Y)= UD_{\lambda}V',$$ where $D_{\lambda}$ is defined by replacing the  diagonal entry $D_{ii}$ of $D$ by  $\max\{D_{ii}-\lambda, 0\}$.  That is, $S_{\lambda}(Y)$ applies ``soft-thresholding" on the singular values of $Y$.

The solution of (\ref{e1}) can be obtained by iteratively using \textit{singular value thresholding estimations}. 
To discuss computation in more detail, we start with two simpler problems.
 
\subsubsection{Pure factor model}
 
We start by discussing estimating a pure factor model:
$$
Y= M+ U
$$
where $M= \Lambda F'$ is low-rank. $M$ can thus be estimated using nuclear-norm optimization:
\begin{equation}\label{soft-2.1}
\widehat M= \arg\min_{ M} \|Y- M\|_F^2+\nu\|M\|_n
\end{equation}
for some tuning parameter $\nu>0$. The solution $\widehat M$ is {well-known to have the closed-form expression} (e.g., \citep{ma2011fixed})
$$
\widehat M= S_{\nu/2}(Y),
$$
which simply applies the singular value thresholding operator to $Y$ with tuning $\nu/2$.

To see the connection with the usual principal components (PC)-estimator (e.g., \cite{SW02}) for $M$, let the singular value decomposition for $Y$ be $Y= UDV'$; and let $\bar D_{\widehat K}$ be the diagonal matrix produced by keeping the $\widehat K$ largest singular values of $D$ and setting the remaining singular values to zero. The PC-estimator for $M$ with $\widehat K$ factors is given by
$$
\widehat M_{PC} = U\bar D_{\widehat K} V'.
$$
Note that  $\bar D_{\widehat K}$ is equivalent to replacing the  diagonal entry $D_{ii}$ of $D$ by  $ D_{ii}1\{D_{ii}>\nu/2\}$ for any  $\nu/2\in[ D_{\widehat K+1, \widehat K+1}, D_{\widehat K, \widehat K} ]$. That is, $\bar D_{\widehat K}$ is produced by applying ``hard-thresholding" on the singular values of $Y$. As such, the nuclear-norm estimator (\ref{soft-2.1}) and the usual PC-estimator are closely related: the former applies soft-thresholding on the singular values of $Y$, while the latter applies hard-thresholding.  

\subsubsection{Low-rank estimation with a regressor}\label{subsubsec: regressor}

Next, we consider estimating a  slightly more complex model:
$$
Y= X\odot\Theta +U,
$$
where $\Theta$ is a low-rank matrix. Again, we can estimate $\Theta$ using nuclear-norm optimization: 
\begin{equation}\label{low-s2.3}
\widehat\Theta:=\arg\min_{\Theta} \|Y- X\odot \Theta\|_F^2+\nu\|\Theta\|_n.
\end{equation}
While the solution to \eqref{low-s2.3} does not have a closed-form expression for a general $X$ matrix, it satisfies the following KKT condition \citep{ma2011fixed}: For any $\tau>0$,  
$$
\widehat \Theta=S_{\tau\nu/2}(\widehat\Theta- \tau X\odot(X\odot \widehat \Theta -Y)).
$$ 
This fact suggests a simple iterative scheme to solve (\ref{low-s2.3}): Let $ \Theta_{k}$ denote the solution at step $k$ in the iteration. We update this solution at step $k+1$ to
\begin{equation}\label{itera-2.4}
\Theta_{k+1}=S_{\tau\nu/2}(\Theta_k- \tau X\odot(X\odot   \Theta_k -Y)).
\end{equation}
  
The iterative scheme (\ref{itera-2.4})  is in fact a  gradient descent procedure (e.g., \cite{beck2009fast}). Fix the solution at the $k^{\textnormal{th}}$ iteration $\Theta_{k}$ and some ``step size"  $\tau>0$. Define
\begin{eqnarray*}
p(\Theta,\Theta_k,\tau)&:=&\tau^{-1}\|\Theta_k-\tau A_k-\Theta\|_F^2+ C_k,\cr
A_k&:=&X\odot (X\odot\Theta_k-Y).
\end{eqnarray*}
where $C_k=L(\Theta_k)-\tau\|  A_k \|_F^2$  does not depend on $\Theta$. Then, $L(\Theta):=\|Y-X\odot\Theta\|_F^2$ is \textit{universally majorized} by $p(\Theta,\Theta_k,\tau)$ for any $\tau\in(0,1/\max_{it}x_{it}^2)$:  
\begin{eqnarray*}
L(\Theta)&\leq& p(\Theta,\Theta_k,\tau),\quad \forall \Theta\cr
L(\Theta_k)&=& p(\Theta_k,\Theta_k,\tau).
\end{eqnarray*}
Therefore, the gradient descent method at step $k+1$ solves the optimization problem
$$
\Theta_{k+1}=\arg\min_{\Theta} p(\Theta,\Theta_k,\tau)+\nu\|\Theta\|_n,
$$
which produces a solution of the form given in (\ref{itera-2.4}).
 
\subsubsection{The joint algorithm and its convergence}

Building from the previous discussion, we are now ready to solve the joint problem (\ref{e1}):
$$
\min_{M,\Theta}\|Y-M- X\odot \Theta\|_F^2+\nu_2\|M\|_n+\nu_1\|\Theta\|_n.
$$
We have that, given $\Theta$, solving for $M$  leads to the closed form solution $S_{\nu_2/2}( Y-X\odot \Theta)$.  We also have that solving for $\Theta$ given $M$ is the same as solving
$$
 \min_{\Theta} \|Z_M- X\odot \Theta\|_F^2+\nu_1\|\Theta\|_n,
$$
where $Z_M=Y-M$, which can be solved following the discussion in Section \ref{subsubsec: regressor}. 
As such, we can employ the following algorithm to iteratively solve for $\widetilde M$ and $\widetilde \Theta$ as the global solution to (\ref{e1}).

\begin{algo}\label{algo 1} Compute the nuclear-norm penalized regression as follows: 
 
\textit{Step 1:} Fix the ``step size" $\tau\in(0,1/\max_{it}x_{it}^2)$. Initialize $\Theta_0, M_0$ and set $k=0.$
 
\textit{Step 2:} Let 
\begin{eqnarray*}
\Theta_{k+1}&=& S_{\tau\nu_1/2}(\Theta_k-\tau X\odot(X\odot\Theta_k -Y+M_k)),
\cr
M_{k+1}&=& S_{\nu_2/2} (Y-X\odot\Theta_{k+1}).
\end{eqnarray*}
 Set $k$ to $k+1.$

\textit{Step 3:} Repeat step 2 until convergence.

\end{algo}
 
The following proposition verifies that the evaluated objective function $F(\Theta_{k+1}, M_{k+1})$ is monotonically decreasing and converges to the global minimum at the rate $O(k^{-1})$. 
 
 \begin{prop}\label{p2.1} Let $(\widetilde \Theta,\widetilde M)$ be a global minimum  for $F(\Theta, M)$. 
Then for any $\tau\in(0,1/\max_{it}x_{it}^2)$, and any initial $\Theta_0, M_0$, we have: 
$$ F(\Theta_{k+1}, M_{k+1})\leq  F(\Theta_{k+1}, M_{k})\leq  F(\Theta_{k}, M_{k}),$$
 for each $k\geq 0$. In addition,  for all $k\geq 1$, 
 \begin{equation}\label{e2.1}
 F(\Theta_{k+1}, M_{k+1}) -F(\widetilde \Theta, \widetilde M)\leq \frac{1}{k\tau} \|\Theta_1-\widetilde \Theta\|_F^2.
 \end{equation}
 \end{prop}
Proposition \ref{p2.1} shows that Algorithm \ref{algo 1} converges to the global optimum from an arbitrary initial value. The upper bound in (\ref{e2.1})  depends on the initial values through the accuracy of the first iteration $\Theta_1-\widetilde \Theta$. Not surprisingly, the upper bound  does not depend on $M_1- \widetilde M$ because, given $\Theta_1$, minimizing with respect to $M$ has a one-step closed form solution whose accuracy is completely determined by $\Theta_1$.  We also see that the largest possible value for $\tau$ to ensure the convergence is $1/\max_{it}x_{it}^2$. In practice, we set $\tau=(1-\epsilon)/\max_{it}x_{it}^2$ for  $\epsilon=0.01$.

\subsection{Post-SVT Inference}\label{sec:2.2}

While estimates of $\Theta$ resulting from solving the nuclear-norm penalized regression \eqref{e1} will be consistent with good rates of convergence under reasonable conditions, using these estimates will generally be inappropriate for inference because they may suffer from substantial shrinkage bias. To address this issue, we consider a ``debiased'' estimator that relies on using additional ordinary least squares steps to obtain updated estimates of $f_t$ and $\lambda_i$, starting from  model structure and initial values learned from a first-step solution of \eqref{e1}, that are less subject to shrinkage bias. 


At a high-level our estimation and inference strategy proceeds in three stages, with \textit{sample splitting} over time, as to be detailed later: 

\vspace{.1in}

\textit{Stage 1:} Obtain initial estimates of $(M, \Theta)$ from \eqref{e1}. Obtain estimates of the rank of $M$ and $\Theta$.

\textit{Stage 2:}   Extract estimated loadings, $(\widetilde\alpha_i, \widetilde \lambda_i)$, from the singular vectors of the subsample estimates of $(M, \Theta)$ corresponding to the rank estimated in Stage 1.
 
\textit{Stage 3:} Appropriately orthogonalize $y_{it}$ and $x_{it}$ based on a model for sources of strong dependence for $x_{it}$. Iteratively estimate $f_t$ and $\lambda_i$ using least squares and the orthogonalized data: 
  
(i) Obtain estimator $\widehat f_t$ for $f_t$ (and estimator $\widehat g_t$ for $g_t$) from  running OLS on the ``orthogonalized data" and  $(\widetilde \lambda_i,\widetilde\alpha)$.
  
(ii) Obtain estimator $\widehat \lambda_i$ for $\lambda_i$ (and estimator $\widehat \alpha_i$ for $\alpha_i$) from using  OLS  on the  ``orthogonalized data" and $(\widehat f_t , \widehat g_t)$.

(iii) Obtain the estimator $\widehat\theta_{it}:=\widehat \lambda_i'\widehat f_t$.

\vspace{.1in}

%
 %
  %
  %
%
  
As Stage 1 employs a singular value thresholding  method, we call our procedure ``post-SVT inference". We show in Section \ref{sec: asymptotics} that $\widehat\theta_{it}$ is ${C_{NT}}$-consistent, where $C_{NT}=\sqrt{\min\{N, T\}}$, and is asymptotically normally distributed, centered around $\theta_{it}$, for each fixed $(i, t)$. We note that, due to the rotation discrepancy, $\lambda_i$ and $f_t$ are not separately identified; they can only be estimated up to a rotation transformation. However, the effect coefficient, $\theta_{it}$, is not subject to this rotation discrepancy and is well-identified and cleanly estimable.
    
Stage 3, which involves unregularized OLS estimation, is the essential stage to alleviating shrinkage bias and obtaining well-centered estimators for $\theta_{it}$.  This stage involves two important ingredients of our estimation algorithm: (1) orthgonalization and (2) sample splitting, as we now explain. 

Consider step (i) in Stage 3 in the simplified case where we have $\dim(f_t) = \dim(g_t) = 1$. 
If we did not use ``orthogonalized data" but use the raw data $(y_{it}, x_{it})$, the estimator $\check f_t$ for $f_t$ would have  the following expansion:
\begin{align}
\begin{split}\label{eq: uncentered OLS}
\sqrt{N}(\check f_t-H_1^{-1}f_t)&= \widehat{Q}\frac{1}{\sqrt{N}}\sum_{i=1}^N\lambda_i x_{it}u_{it}+R_t   \\
&\qquad 
+\widehat{Q}\frac{1}{\sqrt{N}}\sum_{i=1}^N\lambda_i x_{it}(\widetilde\alpha_i-H_2\alpha_i)g_t\\
&\qquad+\widehat{Q}\frac{1}{\sqrt{N}}\sum_{i=1}^N\lambda_i x_{it}^2(\widetilde \lambda_i-H_1\lambda_i)f_t 
\end{split}
\end{align} 
for some $\widehat Q$, 
where $H_1, H_2$ are rotation matrices, and $R_t$ is a higher-order remainder term. 
 The usual asymptotic behavior will be driven by the term $ \widehat{Q}\frac{1}{\sqrt{N}}\sum_{i=1}^N\lambda_i x_{it}u_{it}$, and we wished to be able to show that all other terms are asymptotically negligible, especially the last two terms in (\ref{eq: uncentered OLS}).  
For this heuristic illustration, we focus on the term 
$$
 \frac{1}{\sqrt{N}}\sum_{i=1}^N\lambda_i (\widetilde\alpha_i-H_2\alpha_i) x_{it}:=  \Delta.
$$
It is however, challenging to argue that $\Delta$ is negligible for two reasons.  First,	it can be shown that the nuclear-norm penalized estimator has a rate of convergence
$$
\frac{1}{N}\sum_{i=1}^N\|\widetilde\alpha_i-H_2\alpha_i\|^2=O_P(C_{NT}^{-2}).
$$
By itself, this rate is insufficient to ensure $\Delta=o_P(1)$ in the case $N\geq T$. Second, $x_{it}$ will generally not be mean zero and will be expected to exhibit substantial correlation in  both the time series and cross-section in many applications. 
 
We resolve these difficulties through orthogonalization and sample-splitting.  We will also use a new argument to deal with the last term involving $\widetilde\lambda_i-\lambda_i$ in (\ref{eq: uncentered OLS}), which does not appear previously in the literature, to our best knowledge. 

\textbf{Orthogonalization} 

 To investigate the role of the dependence structure in $x_{it}$, write 
$$
x_{it}=\mu_{it}+e_{it},
$$
where $\mu_{it}$ is the mean process of $x_{it}$ which is assumed to capture both time series dependence and strong sources of cross-sectional correlation. Specifically, we will maintain the assumption that $e_{it}$ is a zero-mean process that is serially independent and cross-sectionally weakly dependent throughout our formal analysis. 
Using this decomposition of $x_{it}$, we can produce an ``orthogonalized'' version of \eqref{baseline model 2}:  
\begin{equation}\label{e3.1}
\dot y_{it}= \alpha_i'g_t + e_{it}\lambda_i'f_t+u_{it},\quad \text{ where } \dot y_{it}= y_{it}- \mu_{it}\theta_{it}.
\end{equation}

Based on \eqref{e3.1}, we can estimate $f_t$ by regressing the estimated $\dot y_{it}$ onto estimated $(\alpha_i, e_{it}\lambda_i)$. This leads to the expansion of estimated $f_t$ as
\begin{align}
\begin{split}\label{eq: centered OLS}
\sqrt{N}(\widehat f_t-H_1^{-1}f_t)&= \widehat{Q}_e\frac{1}{\sqrt{N}}\sum_{i=1}^N\lambda_i e_{it}u_{it}  +o_P(1) \\
&\qquad 
+\widehat{Q}_e\frac{1}{\sqrt{N}}\sum_{i=1}^N\lambda_i e_{it}(\widetilde\alpha_i-H_2\alpha_i)g_t \\
&\qquad+\widehat{Q}_e\frac{1}{\sqrt{N}}\sum_{i=1}^N\lambda_i (\E e_{it}^2)(\widetilde \lambda_i-H_1\lambda_i)f_t
\end{split}
\end{align} 
for some $\widehat Q_e.$
 In this case, the  effect of estimating  $\alpha_i$ becomes 
$$
\Delta_e:= \ \frac{1}{\sqrt{N}}\sum_{i=1}^N\lambda_i(\widetilde\alpha_i-H_2\alpha_i)e_{it}
$$
where $e_{it}$ is a cross-sectionally weakly dependent zero-mean process that is independent across $t$. If $e_{it}$ is independent of $\lambda_i(\widetilde\alpha_i-H_2\alpha_i)$ and sufficient moments exist, it is then immediate that $\Delta_e = o_P(1)$. We note that this orthogonalization  results in partialing-out the strongly dependent components in $x_{it}$ and is tightly tied to the so-called \textit{Neyman's orthogonality} that has been shown to be important in obtaining valid inference after high-dimensional estimation in a variety of contexts; see, e.g., \cite{dml}. 

\textbf{Sample-splitting} 

To formally argue that   $\Delta_e=o_P(1)$, we use sample-splitting.  For the fixed $t$, let  $I \subset \{1,...,T\}\backslash t$ be a set of time indexes, and let $$D_I=\{(y_{is}, x_{is}): i\leq N, s\in I\}.$$ Rather than using the full sample to obtain initial estimates of $\widetilde \lambda_i$ and $\widetilde \alpha_i$, we run the nuclear-norm optimization (\ref{e1}) using only data $D_I$. Because $e_{it}$ is  serially  independent and $t\notin I$, we will have that $\widetilde\lambda_i$ and $\widetilde\alpha_i$ are independent of $e_{it}$ which allows us to easily verify that $\Delta_e$ and similar terms vanish asymptotically. For example, in the case of cross-sectional independence, we have $\E(\Delta_e| D_I)=0$ and 
$$
\Var(\Delta_e|D_I)=\frac{1}{N}\sum_{i=1}^N\Var(e_{it}) (\widetilde\alpha_i-H_2\alpha_i) ^2\lambda_i^2.
$$ 
Therefore $\Delta_e=o_P(1)$ so that the effect of estimating the nuisance parameters on the final estimator $\widehat\theta_{it}$ is negligible.

In practice, operationalizing the orthogonalization in \eqref{e3.1} requires that a sufficiently high-quality estimate of $e_{it}$ can be obtained which will require restricting the process $\mu_{it}$. In this paper, we assume 
$$
\mu_{it}=l_i'w_t.
$$
for some $l_i$ and $w_t$. Hence, $x_{it}$ is allowed to follow a factor structure, where $(l_i, w_t)$ respectively represent loadings and factors, and may thus be strongly intertemporally and cross-sectionally correlated. We note that this structure allows for $w_t$ to overlap with $(f_t, g_t)$. We take the dimension $K_x:=\dim(w_t)$ to be fixed in our analysis. Other structures of $\mu_{it}$ could also be imposed and analyzed with all results going through as long as $\mu_{it}$ is sufficiently well-estimable. We discuss another example for factor models of missing data later, and  leave other extensions to future work.

\textbf{The effect of $\widetilde\lambda_i-H_1\lambda_i$} 

   We now explain  the effect of $\widetilde\lambda_i-H_1\lambda_i$ in (\ref{eq: centered OLS}).   We shall use a new argument to address it. 

When we estimate $f_t$  using the orthogonalized data (i.e., estimated $(\dot y_{it}, e_{it})$,   the effect of estimating $\lambda_i$ becomes 
$$\widetilde \Delta_e:=\widehat{Q}_e\frac{1}{\sqrt{N}}\sum_{i=1}^N\lambda_i (\E e_{it}^2)(\widetilde \lambda_i-H_1\lambda_i).$$
This term, however, is not $o_P(1)$ since $\E e_{it}^2\neq 0$. As such, we conclude that the effect of estimating $\widetilde\lambda_i$ from the nuclear-norm regularized estimation is not asymptotically negligible.
Returning to (\ref{eq: centered OLS}) and noting that $\Delta_e=o_P(1)$, we have  
\begin{align*} 
\sqrt{N}(\widehat f_t-H_1^{-1}f_t)= \widehat{Q}_e\frac{1}{\sqrt{N}}\sum_{i=1}^N\lambda_i e_{it}u_{it}   
+\widetilde\Delta_ef_t +o_P(1).
\end{align*}
Thus, $\widetilde\Delta_e$ results in a non-vanishing asymptotic bias in the estimator $\widehat f_t$. Importantly, this bias manifests as an additional time-invariant rotation of the factors $f_t$. We can thus define $H_f:=H_1^{-1} + \widetilde\Delta_eN^{-1/2}$ and establish that 
$$
\sqrt{N}( \widehat f_t- H_ff_t)=\widehat Q_e \frac{1}{\sqrt{N}}\sum_{i=1}^N\lambda_ie_{it}u_{it}+o_P(1).
$$
Therefore, the effect of first-step estimation error in $\lambda_i$, $\widetilde\lambda_i-H_1'\lambda_i$, is ``absorbed" by the adjusted rotation matrix.  

Here we are saved by the fact that, for estimating  $\theta_{it}=\lambda_i'f_t$, we only require that $f_t$ be well estimated up to a rotation matrix; that is, we only need that the estimated $f_t$ has the same span as the actual latent $f_t$. This phenomenon is new relative to the existing post-regularization inference literature and arises due to the factor structure, though it is analogous to the usual rotational indeterminacy issues in factor models. Once we obtain that $\widehat f_t$ appropriately recovers the span of $f_t$, the final least squares iteration where we use $\widehat f_t$ as an observed input produces an estimator $\widehat\lambda_i$ that suitably recovers the appropriately rotated $\lambda_i$ to correspond to the rotation in $\widehat f_t$. Recovering these two compatibly rotated versions of $f_t$ and $\lambda_i$ is then sufficient for inferential theory for $\widehat\theta_{it}$, which depends only on their product.

\subsection{Formal Estimation Algorithm.}

We now state the full estimation algorithm for $\theta_{it}$ for some fixed $(i,t)$, including explicit steps for orthogonalization and sample splitting. We state the algorithm in the leading case where 
$$
x_{it} = l_i'w_t+e_{it},\quad \E e_{it}=0.
$$
We recall that we use this model for $x_{it}$ to partial out the common component $\mu_{it} = l_i'w_t$ by subtracting the estimated $\mu_{it}$ from $x_{it}$ to obtain $e_{it} = x_{it} - \mu_{it}$ and working with the model 
\begin{equation}\label{e3.5}
\dot y_{it} =\alpha_i'g_t+ e_{it}\theta_{it} + u_{it},\quad \text{where }\dot y_{it}=y_{it}-\mu_{it}\theta_{it}.
\end{equation}


\begin{algo}\label{al2.2} Estimate $\theta_{it}$ as follows. 

\textit{Step 1. Estimate the number of factors.} 
Run nuclear-norm penalized regression:
\begin{equation}\label{e2.5}
(\widetilde M, \widetilde\Theta):=\arg\min_{M,\Theta}\|Y-M- X\odot \Theta\|_F^2+\nu_2\|M\|_n+\nu_1\|\Theta\|_n.
\end{equation}
Estimate $K_1, K_2$ by
   $$
  \widehat K_1=\sum_{i} 1\{\psi_i(\widetilde \Theta)\geq (\nu_2  \|\widetilde \Theta\|)^{1/2}\},\quad  \widehat K_2=\sum_{i} 1\{\psi_i(\widetilde M)\geq  (\nu_1  \|\widetilde M\|)^{1/2}\}
  $$
 where $\psi_i(.)$ denotes the $i$th largest singular-value.

\medskip

\textit{Step 2. Estimate the structure $x_{it}=\mu_{it}+e_{it}$.} In the factor model, use the PC estimator to obtain  $(\widehat{\mu}_{it},\widehat e_{it})$ for all $i=1,..., N, \ \ t=1,..,T.$ 
 
\medskip

\textit{Step 3: Sample splitting.} 
Randomly split the sample into $\{1,..., T\}/\{t\}=I\cup I^c$, so that $|I|_0= [(T-1)/2]$. Denote respectively by $Y_I, X_I$ as the  $N\times |I|_0$ matrices of $(y_{is}, x_{is})$ for  observations $s\in I$. Estimate the low-rank matrices $\Theta$ and $M$ as in (\ref{e2.5}) with $(Y, X)$ replaced with $(Y_I, X_I)$ to obtain $(\widetilde M_I, \widetilde\Theta_I)$.

Let $\widetilde\Lambda_I=(\widetilde\lambda_1,...,\widetilde\lambda_N)'$ be the $N\times \widehat K_1$ matrix whose columns are defined as $\sqrt{N}$ times  the first $\widehat K_1$ eigenvectors of $\widetilde \Theta_I\widetilde \Theta_I'$.  Let $\widetilde A_I=(\widetilde\alpha_1,...,\widetilde\alpha_N)'$ be the $N\times \widehat K_2$ matrix whose columns are defined as $\sqrt{N}$ times the first $\widehat K_2$ eigenvectors of $\widetilde  M_I\widetilde  M_I'$. 

\medskip

\textit{Step 4. Estimate components for ``orthogonalization.''} Using $\widetilde A_I$ and $\widetilde \Lambda_I$, obtain   
$$
(  \widetilde f_s,\widetilde g_s):= \arg\min_{f_s, g_s}\sum_{i=1}^N(y_{is}- \widetilde\alpha_i'g_s- x_{is} \widetilde \lambda_i'f_s)^2,\quad s\in I^c\cup\{t\}.
$$
Update estimates of loadings as
$$
(\dot\lambda_i, \dot\alpha_i)= \arg \min_{\lambda_i, \alpha_i}\sum_{s\in I^c\cup\{t\}} (y_{is} - \alpha_i'  \widetilde g_s- x_{is} \lambda_i'\widetilde f_s)^2,\quad i=1,..., N.
$$
    
\medskip

\textit{Step 5. Estimate $(f_t,\lambda_i)$ for use in inference about $\theta_{it}$.} Motivated by (\ref{e3.5}), define $\widehat y_{is}= y_{is}- \widehat{\mu}_{is}\dot\lambda_i'\widetilde f_s$ and $\widehat e_{is}=x_{is}- \widehat{\mu}_{is}$.
For all $s\in I^c\cup\{t\}$, let 
$$
(\widehat f_{I,s},\widehat g_{I,s}):= \arg\min_{f_s, g_s}\sum_{i=1}^N(\widehat y_{is}- \widetilde\alpha_i'g_s- \widehat e_{is} \widetilde \lambda_i'f_s)^2.
$$
Fix $i\leq N$, let
$$
(\widehat\lambda_{I,i},\widehat\alpha_{I,i})=\arg \min_{\lambda_i, \alpha_i}\sum_{s\in I^c\cup\{t\}} (\widehat y_{is}-   \alpha_i'\widehat g_{I,s}- \widehat e_{is}  \lambda_i'\widehat f_{I,s})^2.
$$
  
\medskip

\textit{Step 6. Exchange $I$ and $I^c$.} Repeat steps  3-5 with $I$ and $I^c$ exchanged to obtain $(\widehat{\lambda}_{I^c,i}, \widehat{f}_{I^c,s}: s\in I\cup \{t\}, i\leq N)$. 

\medskip

\textit{Step 7. Estimate $\theta_{it}$. } Obtain the estimator of $\theta_{it}$:
$$
\widehat\theta_{it}:= \frac{1}{2}[\widehat\lambda_{I,i}'\widehat f_{I,t}+ \widehat{\lambda}_{I^c,i}'\widehat f_{I^c,t}]. 
$$

\end{algo}

\begin{remark}\label{re2.1}
We split the sample to $\{1,..., T\}=I\cup I^c\cup\{t\}$ which ensures that $e_{it}$ is independent of the data in both $I$ and $I^c$ for the given $t$. In steps 3-5, we obtain the quantities we need to estimate $\theta_{is}$ for $s\in \{t\}\cup I^c$, and we switch the roles of $I$ and $I^c$ in step 6 to allow us to estimate $\theta_{is}$ for $s\in \{t\}\cup I$. As such, we can estimate $\theta_{it}$ twice at $s=t$. The final estimator is taken as the average of the two available estimators of $\theta_{it}$: $\widehat\lambda_{I,i}'\widehat f_{I,t}$ and $\widehat{\lambda}_{I^c,i}'\widehat f_{I^c,t}$. This averaging restores the asymptotic efficiency that would otherwise be lost due to the sample splitting. 
\end{remark}

\begin{remark}\label{re2.2} 
Step 4 is needed to obtain a sufficiently high-quality estimate for $\dot y_{it}=y_{it}-\mu_{it}\lambda_i'f_t$ to allow application of the ``partialed out equation'' (\ref{e3.1}). 
$\widetilde\lambda_i$ and $\widetilde\alpha_i$ obtained in step 1 are not sufficiently well-behaved to produce the desired result which motivates the need for the estimators $(\dot\lambda_i, \dot\alpha_i)$. The estimators obtained in step 4 are still unsuitable for inference due to the use of the raw $x_{it}$ which will generally not be zero-mean or serially independent. Accounting for these issues gives rise to the need for step 5.
\end{remark}

\subsection{Choosing the tuning parameters}\label{s:choo}

In practice, we also need to choose the tuning parameters, $\nu_1$ and $\nu_2$, used in solving \eqref{e1}. We adopt a simple plug-in approach that formally guarantees asymptotic score domination and is sufficient to provide good rates of convergence of the estimators $\widetilde M$ and $\widetilde \Theta$ and guarantee consistency of their estimated ranks.

The ``scores" of the nuclear-norm penalized regression are given by $2\|U\|$ and $2\|X\odot U\|$, where $\|.\|$ denotes the matrix operator norm and $U$ and $X \odot U$ are $N \times T$ matrices with typical elements $U_{it} = u_{it}$ and  $X\odot U_{it} = x_{it}u_{it}$. 
We then choose tuning parameters $(\nu_2, \nu_1)$ so that
$$
2\|U\|< (1-c)\nu_2,\quad 2\|X\odot U\|< (1-c)\nu_1
$$
for some $c>0$ with high probability.   

To quantify the operator norm of $U$ and $X \odot U$, we assume that the columns of $U$ and $X\odot U$, respectively $\{u_t\} $ and $\{x_t\odot u_t\}$,  are sub-Gaussian vectors. In the absence of serial correlation, the eigenvalue-concentration inequality for independent sub-Gaussian random vectors (Theorem 5.39 of \cite{vershynin2010introduction}) implies
$$
\|(X\odot U)(X\odot U)' -  \E (X\odot U)(X\odot U)'\|=O_P(\sqrt{  NT} +N),
$$
which provides a sharp bound for $\|X\odot U\|$; and a similar upper bound holds for $\|UU'-\E UU'\|$. Hence, $\nu_2$ and $\nu_1$ can be chosen to satisfy $\nu_2\asymp \nu_1\asymp \max\{\sqrt{N}, \sqrt{T}\}$.

In the presence of serial correlation, we assume the following representation:
$$
X\odot U= \Omega_{NT}\Sigma_T^{1/2}
$$
where $\Omega_{NT}$ is an $N\times T$   matrix with independent, zero-mean,  sub-Gaussian columns. Then, by the  eigenvalue-concentration inequality for sub-Gaussian random vectors, we continue to have that  
$$
\|\Omega_{NT}\Omega_{NT}' -  \E\Omega_{NT}\Omega_{NT}'\|=O_P(\sqrt{  NT} +N).
$$ 
In addition, we maintain that $\Sigma_T$ is a $T\times T$, possibly non-diagonal, deterministic matrix whose eigenvalues are bounded from below and above by  constants. Allowing $\Sigma_T$ to be non-diagonal captures serial-correlation in $\{x_t\odot u_t\}.$ Putting all of these conditions together will also imply $\|X\odot U\|\leq O_P(\max\{\sqrt{N},\sqrt{T}\})$. Therefore, the tuning parameters   can be chosen to satisfy 
$$\nu_2\asymp   \nu_1\asymp \max\{\sqrt{N},\sqrt{T}\}$$
quite generally.  

In the Gaussian case, we can  compute appropriate tuning parameters via simulation. 
Suppose $u_{it}$ is independent across both $(i,t)$ and $u_{it}\sim \mathcal N(0,\sigma_{ui}^2)$. Let $Z$ be an $N\times T$ matrix whose elements $z_{it}$ are generated as $\mathcal N(0,\sigma_{ui}^2)$ independent across  $(i,t)$. Then $\|X\odot U\|=^d\|X\odot Z\| $ and $\|U\|=^d \|Z\|$
where $=^d$ means ``is identically distributed to''. Let $\bar Q(W ; m)$ denote the $m^{\textnormal{th}}$ quantile of a random variable $W$. For  $\delta_{NT}=o(1)$, we can take  
$$
\nu_2=2(1+c_1)\bar Q(\|Z\| ; 1-\delta_{NT}),\quad \nu_1=2(1+c_1)\bar Q(\|X\odot Z\| ; 1-\delta_{NT})
$$
which respectively denote $2(1+c_1)$ multiplied by the $1-\delta_{NT}$ quantile of $\|Z\|$ and $\|X\odot Z\|$. We will then have that 
$$
2\|U\|< (1-\frac{c_1}{1+c_1})\nu_2,\quad2 \|X\odot U\|< (1-\frac{c_1}{1+c_1})\nu_1
$$
holds with probability $1-\delta_{NT}$. In practice, we compute the quantiles by simulation replacing $\sigma_{ui}^2$ with an initial consistent estimator. In our simulation and empirical examples, we set $c_1=0.1$ and $\delta_{NT}=0.05$.

\section{ Asymptotic Results}\label{sec: asymptotics}

\subsection{Estimating Low Rank Matrices}
   
We first introduce a key  assumption  about the nuclear-norm SVT procedure. We require some ``invertibility" condition for the operator:
$$
(\Delta_1, \Delta_2):\to \Delta_1+\Delta_2\odot X
$$
when $   (\Delta_1, \Delta_2)$ is restricted to a ``cone,'' consisting  of  low-rank matrices, which we call the  \textit{restricted low-rank set}. These invertibility conditions were previously introduced and studied by \cite{negahban2011estimation}. To describe this cone,  we first introduce some notation. Define
$$
\Theta_I^0=(\lambda_i'f_t: i\leq N, t\in I),\quad  M_I^0=(\alpha_i'g_t: i\leq N, t\in I).
$$ 
Define  $U_1 D_1 V_1'=\Theta_I^0$ and  $U_2 D_2 V_2'=M_I^0$  as the singular value decomposition of $\Theta_I^0$ and $M_I^0$,  where the superscript $0$ represents the ``true" parameter values. Further decompose, for $j=1,2,$
$$
U_j= (U_{j, r}, U_{j,c}),\quad V_j= (V_{j, r}, V_{j,c})
$$
Here $(U_{j,r}, V_{j,r})$ are the singular vectors corresponding to nonzero singular values, while  $(U_{j,c}, V_{j,c})$ are singular vectors corresponding to the zero singular values. In addition, for any $N\times T/2$ matrix $\Delta$, let 
$$
\mathcal P_{j} (\Delta )= U_{j,c} U_{j, c}' \Delta V_{j,c} V_{j,c}',\quad \mathcal M_{j}(\Delta)=\Delta -\mathcal P_{j}(\Delta). 
$$
Here $U_{j,c} U_{j, c}' $ and $V_{j,c} V_{j, c}' $ respectively are the projection matrices onto the columns of $U_{j,c}$ and $V_{j,c}$.  Therefore, $\mathcal M_1(\cdot)$ and $\mathcal M_2(\cdot)$ can be considered as the projection matrices onto the ``low-rank spaces" of  $\Theta_I^0$  and  $M_I^0$  respectively, and $\mathcal P_1(\cdot)$ and $\mathcal P_2(\cdot)$ are projections onto their orthogonal spaces.

\begin{assumption}[Restricted strong convexity]\label{a3.1}  Define the \textit{restricted low-rank set} as, for some $c>0,$
$$
\mathcal C(c)=\{(\Delta_1,\Delta_2):   
\|\mathcal P_1(\Delta_1)\|_n+ \|\mathcal P_2(\Delta_2)\|_n
\leq c\| \mathcal M_1(\Delta_1)\|_n+c\| \mathcal M_2(\Delta_2)\|_n\}.
$$	
For any $c>0$,  there is a constant $\kappa_c>0$ so that If $(\Delta_1, \Delta_2)\in\mathcal C(c) $ then 
$$
\|\Delta_1+\Delta_2\odot X\|_F^2\geq \kappa_c \|\Delta_1\|_F^2+\kappa_c\|\Delta_2\|_F^2.
$$
The same condition holds when $(M_I^0, \Theta_I^0)$ are replaced with  $(M_{I^c}^0, \Theta_{I^c}^0)$, or the full-sample $(M^0, \Theta^0)$. 
\end{assumption}

Our next assumption allows non-stationary and arbitrarily serially dependent $(f_t, g_t)$. These conditions not only hold for serially weakly dependent sequences but also allow for perfectly dependent sequences: $(f_t, g_t)=(f, g)$ for some time-invariant $(f,g)$ by setting $\dim(f_t)=\dim(g_t)=1$.
          
\begin{assumption}\label{a3.2} As $T\to\infty$, the sub-samples $(I,I^c)$ satisfy: 
\begin{eqnarray*}&&\frac{1}{|I|_0}\sum_{t\in I} f_tf_t' =\frac{1}{T}\sum_{t=1}^T f_tf_t' +o_P(1)=
      \frac{1}{|I^c|_0}\sum_{t\in I^c} f_tf_t' ,\cr
      &&\frac{1}{|I|_0}\sum_{t\in I} g_tg_t' =\frac{1}{T}\sum_{t=1}^T g_tg_t' +o_P(1)=
      \frac{1}{|I^c|_0}\sum_{t\in I^c} g_tg_t'.
\end{eqnarray*}
In addition, there is a $c>0$ such that all the eigenvalues of  $\frac{1}{T}\sum_{t=1}^T f_tf_t'  $ and $\frac{1}{T}\sum_{t=1}^T g_tg_t' $ are bounded from below by $c$ almost surely.
\end{assumption}

The next assumption requires that the factors be strong. In addition, we require distinct eigenvalues in order to identify their corresponding eigenvectors, and therefore, $(\lambda_i,\alpha_i)$. {These conditions are standard in the factor modeling literature. We do wish to note though that, in the present context, these conditions also serve a role analogous to ``$\beta$-min'' conditions in the sparse high-dimensional linear model framework in that they provide strong separation between relevant and irrelevant factors and effectively allow oracle selection of the rank of the low-rank matrices. It may be interesting, but is beyond the scope of the present work, to consider estimation and inference in the presence of weak factors.} 

\begin{assumption}[Valid factor structures with strong factors]\label{a3.2}   
There are constants  $c_1>...>c_{K_1}>0$, and $c_1'>...>c_{K_2}'>0$, so that, up to a term $o_P(1)$,

(i) $c_j'$ equals the $j$ th largest eigenvalue of $(\frac{1}{T}\sum_tg_tg_t')^{1/2} 	\frac{1}{N}\sum_{i=1}^N\alpha_i\alpha_i'(\frac{1}{T}\sum_tg_tg_t')^{1/2} $ for all $j=1,..., K_1$, and 
	
	(ii) $c_j$ equals the $j$ th largest eigenvalue of $(\frac{1}{T}\sum_tf_tf_t')^{1/2} 	\frac{1}{N}\sum_{i=1}^N\lambda_i\lambda_i'(\frac{1}{T}\sum_tf_tf_t')^{1/2} $ for all $j=1,..., K_2$.

\end{assumption}

Recall that $\widetilde M_S$ and $\widetilde \Theta_S$ respectively are the estimated low-rank matrices obtained by the nuclear-norm penalized estimations on sample $S\in\{I, I^c,\{1,...,T\}\}$. Given the above assumptions, we have consistency in the Frobenius norm for the estimated low-rank matrices.
 
\begin{prop}\label{p3.1} Suppose $2\|X\odot U\|<(1-c)\nu_1$, $2\|U\|<(1-c)\nu_2$ and $\nu_2\asymp \nu_1$. Then  under Assumption \ref{a3.1},
for $S\in \{I, I^c,  \{1,...,T\}\}$  
(i)
$$
\frac{1}{NT}\|\widetilde M_S-M_S\|_F^2=O_P(\frac{\nu_2^2+\nu_1^2}{NT})=\frac{1}{NT}\|\widetilde \Theta_S-\Theta_S\|_F^2.
$$
(ii) Additionally with Assumption \ref{a3.2},   there are  square matrices $ H_{S1}, H_{S2}$, so that 
$$
\frac{1}{N} \|\widetilde A_S- AH_{S1}\|_F^2=O_P(\frac{\nu_2^2+\nu_1^2}{NT}),\quad \frac{1}{N} \|\widetilde\Lambda_S-\Lambda H_{S2}\|_F^2=O_P(\frac{\nu_2^2+\nu_1^2}{NT}).
$$
(iii) Furthermore, 
$$
P(\widehat K_1=K_1,\quad \widehat K_2=K_2)\to 1.
$$
\end{prop}

\begin{remark}
The above results provide convergence of the nuclear-norm regularized estimators and the associated singular vectors based upon the entire sample $\{1,...,T\}$ or upon a subsample $S=I$ or $I^c.$ Convergence using the entire sample is sufficient for consistently selecting the rank, as shown in result (iii), while convergence using the  subsamples serves for post-SVT inference after sample-splitting. 
\end{remark}

\subsection{Asymptotic Inferential Theory}
   
Our main inferential theory is for a \textit{group average effect}. Fix a cross-sectional subgroup 
$$
\mathcal G\subseteq\{1,2,...,N\}.
$$
We are interested in inference of the group average effect at a fixed $t\leq T$:
$$
\bar\theta_{\mathcal G,t} := \frac{1}{|\mathcal G|_0}\sum_{i\in\mathcal G} \theta_{it},
$$
where $|\mathcal G|_0 $ denotes the group size, that is, the number of elements in $\mathcal G$.
The group size can be either fixed or grow with $N$.  This structure admits two important special cases as extremes: (i) $\mathcal G=\{i\}$ for any fixed individual $i$, which allows   inference for any fixed individual; and (ii) $\mathcal G=\{1,2,...,N\}$, which allows inference for the cross-sectional average effect $\bar\theta_t:=\frac{1}{N}\sum_{i=1}^N\theta_{it}$.

We will show that the group average effect estimator has rate of convergence 
$$
\frac{1}{|\mathcal G|_0} \sum_{i \in \mathcal G} \widehat\theta_{it} - \bar\theta_{\mathcal G,t} = O_P\left(\frac{1}{\sqrt{T|\mathcal G|_0}} +  \frac{1}{\sqrt{N}}\right).
$$
It is useful to compare this result to the rate that would be obtained in estimating a ``group-homogeneous effect" by imposing homogeneous effects at the group level. Suppose we are interested in the effect on a fixed group $\mathcal G$. One could estimate a simpler model which assumes homogeneity within the group:
\begin{equation}\label{eq3.2}
y_{it}= \theta_{\mathcal G,t} x_{it} +\alpha_i'g_t+ u_{it},\quad i\in\mathcal G.
\end{equation}
The ``group-homogeneous effect estimator" $\widetilde\theta_{G,t}$  would satisfy
$$
\widetilde\theta_{G,t}- \theta_{\mathcal G,t}=O_P\left(\frac{1}{\sqrt{|\mathcal G|_0}}\right).
$$
In effect, the homogeneous effect estimator uses only cross-sectional information within $\mathcal G$. This estimator is thus inconsistent if $\mathcal G$ is finite. In contrast, by leveraging the factor structure for the coefficients, we estimate a heterogeneous average effect $\bar\theta_{\mathcal G,t}=\frac{1}{|\mathcal G|_0}\sum_{i\in\mathcal G}\theta_{it}$ making use of all cross-sectional information.  The estimator $ \frac{1}{|\mathcal G|_0}\sum_{i\in\mathcal G}\widehat\theta_{it}$ has a faster rate of convergence than the more traditional homogeneous effects estimator when  $|\mathcal G|_0$ is small and is consistent even if the group size is finite.    We also note that the homogeneous effects estimator will generally be inconsistent for group average effects when effects are actually heterogeneous. Note that we estimate the average effect for a given \textit{known} group. See, e.g., \cite{bonhomme2015grouped} and \cite{su2016identifying} for  estimating group effects with estimated group memberships.

To obtain the asymptotic distribution, we make a number of additional assumptions. Throughout, let $(F, G, W, E, U)$ be the  $T\times K$ matrices of   $(f_t, g_t, w_t, e_{it}, u_{it})$, where $K$ differs for different quantities. 
   
\begin{assumption}[Dependence]\label{a4.1}    (i) $\{e_{it} , u_{it} \}$ are independent across $t$; $\{e_{it}\}$ are also  conditionally independent across $t$, given $\{F, G, W,U \}$;   $\{u_{it}\}$ are also conditionally independent across $t$, given $\{F,G,W, E\}$;

(ii) $\E(e_{it}|u_{it}, w_{t}, g_t, f_t)= 0,\quad 	\E(u_{it}|e_{it}, w_{t}, g_t, f_t)=0.$

(iii) The $N\times T$ matrix $X\odot U$
has the following decomposition:
$$
X\odot U= \Omega_{NT}\Sigma_T^{1/2}
$$
where 
\begin{enumerate}
\item  $\Omega_{NT}:=(\omega_1,...,\omega_T)$ is an $N\times T$   matrix whose columns $\{\omega_t\}_{t\leq T}$ are independent sub-gaussian random vectors with $\E\omega_t=0$; more specifically, there is $C>0$ such that 
$$
\max_{t\leq T} \sup_{\|x\|=1} \E\exp(s \omega_t'x)\leq \exp(s^2C),\quad \forall s\in\mathbb R.
$$
\item   
$\Sigma_T$ is a $T\times T$ deterministic matrix whose eigenvalues are bounded from both below and above by  constants.   
\end{enumerate}

(iv) There is weak conditional cross-sectional dependence. Specifically, let $\mathcal W=(F, G, W)  $. Let $\omega_{it}=u_{it}e_{it}$, and let $c_i$ be a bounded nonrandom sequence.   Almost surely,   \begin{eqnarray*}
&&\max_{t\leq T}\frac{1}{N^3}\sum_{i,j,k,l\leq N}|\Cov(e_{it} e_{jt}, e_{kt} e_{lt} |\mathcal W, U )|<C ,\quad\max_{t\leq T}\|\E(u_tu_t'|\mathcal W, E)\|<C\cr
&& \max_{t\leq T}\E(|\frac{1}{\sqrt{N}}\sum_{i=1}^Nc_i\omega_{it}|^4|\mathcal W)<C,\quad \max_{t\leq T}\E(| \frac{1}{\sqrt{N}}\sum_{i=1}^Nc_ie_{it}|^4|\mathcal W, U)<C\cr
&& \max_{t\leq T}\E(|\frac{1}{\sqrt{N}}\sum_{i=1}^Nc_iu_{it}|^4|\mathcal W, E)<C,\quad  \max_{t\leq T}\max_{i\leq N} \frac{1}{N} \sum_{k,j\leq N}  |\E(e_{kt}e_{it}e_{jt}|\mathcal W, U)|<C\cr 
&&\max_{t\leq T}\max_{i\leq N}\sum_{j=1}^N|\Cov(e_{it}^m, e_{jt}^r|   \mathcal W, U  )|<C,\quad m,r\in\{1,2\}\cr
&&\max_{t\leq T}\max_{i\leq N} \frac{1}{N}\sum_{k,j\leq N}|\Cov(  \omega_{it}\omega_{jt},  \omega_{it}\omega_{kt}|\mathcal W)|<C \cr
&&\max_{t\leq T}\max_{i\leq N}\sum_{j=1}^N|\Cov(\omega_{it},\omega_{jt})|\mathcal W)|<C.
  \end{eqnarray*}

\end{assumption}

The sub-Gaussian condition allows us to apply the   eigenvalue-concentration inequality for independent  random vectors  (Theorem 5.39 of \cite{vershynin2010introduction}) to bound the scores $\|X\odot U\|$ and $\|U\|$. 
Under Assumption \ref{a4.1},  we can take  $\nu_2,\nu_1=O_P(\sqrt{N+T})$. Therefore, we will have
$$
\frac{1}{N} \|\widetilde A_S- AH_{S1}\|_F^2=O_P\left( \frac{1}{N}+\frac{1}{T}\right)= \frac{1}{N} \|\widetilde\Lambda_S-\Lambda H_{S2}\|_F^2 .
$$

Condition (iii) allows $x_{it}$  to be serially weakly dependent with serial correlation captured by $\Sigma_T$. The imposed conditions allow us to apply the eigenvalue-concentration inequality for independent sub-Gaussian  random vectors on $\Omega_{NT}$. In addition, we allow arbitrary dependence among rows of $\Omega_{NT}$, and thus strong cross-sectional dependence in $x_{it}$ is allowed. Allowing strong cross-sectional dependence is desirable given the factor structure. 

Let $$
 V_{\lambda2}= \Var\left(\frac{1}{\sqrt{N}}\sum_{i=1}^N\lambda_i e_{it}u_{it}\bigg{|}F\right). 
$$
\begin{assumption}
[Cross-sectional CLT] As $N\to\infty$,
$$
V_{\lambda2}^{-1/2} \frac{1}{\sqrt{N}}\sum_{i=1}^N\lambda_i e_{it}u_{it}\to^d\mathcal N(0,I).
$$
\end{assumption}

Before stating our final assumption, it is useful to define a number of objects. First, define
  \begin{eqnarray*}
b_{NT,1}&=&\max_{t\leq T}\|  \frac{1}{NT}\sum_{i=1}^N\sum_{s=1}^Tw_s(e_{is}e_{it}-\E e_{is}e_{it})\|\cr
b_{NT,2}&=&(\max_{t\leq T}\frac{1}{T}\sum_{s=1}^T ( \frac{1}{N}\sum_{i=1}^Ne_{is}e_{it}-\E e_{is}e_{it})^2   )^{1/2}\cr
b_{NT,3}&=&\max_{t\leq T}\| \frac{1}{N}\sum_{i=1}^Nl_ie_{it}\|   \cr
b_{NT,4}&=&\max_{i\leq N}\| \frac{1}{T}\sum_{s=1}^T e_{is}w_s\|  \cr
b_{NT,5}&=&\max_{i\leq N}\|    \frac{1}{NT}\sum_{j=1}^N\sum_{s=1}^T l_j(e_{js}  e_{is}-\E e_{js}  e_{is}   )  \|  \cr
  \end{eqnarray*}

In addition, we introduce Hessian matrices that are involved when iteratively estimating $\lambda_i$ and $f_t$. 
    \begin{eqnarray*}
D_{ft}&=&\frac{1}{N}\Lambda '(\diag(X_t) M_{\alpha} \diag(X_t)\Lambda\cr
\bar D_{ft}&=&\frac{1}{N}\Lambda '\E((\diag(e_t) M_{\alpha} \diag(e_t) )\Lambda
+\frac{1}{N}\Lambda '(\diag(Lw_t) M_{\alpha} \diag(Lw_t) \Lambda
\cr
 D_{\lambda i}&=&\frac{1}{T}  F'  (\diag(\underline{X}_i) M_{ g} \diag(\underline{X}_i)  ) F\cr
  \bar D_{\lambda i}&=&\frac{1}{T}  F' \E (\diag(E_i) M_{ g} \diag(E_i)  ) F
  +\frac{1}{T}  F'  (\diag(Wl_i) M_{ g} \diag(Wl_i)  ) F.
\end{eqnarray*}
The above matrices  involve the following notation. Let $X_t, e_t$ denote the $N\times 1$ vector of $x_{it}$ and $e_t$; $\underline{X_i}$ and $E_i$ denote the $T\times 1$ vectors of $x_{it}$ and $e_{it}$.
Let $L$ denote the $N\times \dim(w_t)$ matrix of $l_i$; $W$ denote the $T\times \dim(w_t)$ matrix of $w_t$.  Let $M_g=I-G(G'G)^{-1}G'$ be a $T\times T$ projection matrix, and let $M_\alpha$ be an the $N\times N$ matrix defined similarly. Finally, let $\diag(e_t)$ denote the  diagonal matrix whose entries are elements of $e_t$; all other $\diag(.)$ matrices are defined similarly. 

     \begin{assumption}[Moment bounds]\label{a3.8}
   (i) $\max_i(\|\lambda_i\|+\|\alpha_i\|+ \|l_i\|)<C$.

  (ii)  
  $\max_{t\leq T}\|\frac{1}{N}\sum_ie_{it}\alpha_i\lambda_i'\|_F=o_P(1)$ and  $ \delta_{NT}\max_{it}|e_{it}|=o_P(1)$,    where
  $$   \delta_{NT}:=   ( C_{NT}^{-1}+b_{NT,4}  +b_{NT,5})\max_{t\leq T}\|w_t\|+ b_{NT,1}+b_{NT,3}+C_{NT}^{-1}b_{NT,2}+ C_{NT}^{-1/2} .  $$

   (iii) Let $\psi_j(H)$ denote the $j^{th}$ largest singular value of matrix $H$.
   Suppose 
   there is $c>0$, so that almost surely,   for all $t\leq T$ and $i\leq N$, $\min_{j\leq K_2} \psi_j(D_{\lambda i})>c$, $\min_{j\leq K_2} \psi_j(D_{ft})>c$, $\min_{j\leq K_2} \psi_j(\bar D_{\lambda i})>c$ and $\min_{j\leq K_2} \psi_j(\bar D_{ft})>c$.
      In addition, 
      $$c< \min_j\psi_j(\frac{1}{N}\sum_il_il_i')\leq \max_j\psi_j(\frac{1}{N}\sum_il_il_i')<C$$
          $$c< \min_j\psi_j(\frac{1}{T}\sum_tw_tw_t')\leq \max_j\psi_j(\frac{1}{T}\sum_tw_tw_t')<C.$$
      (iv) $\max_{it}\E( e_{it}^8|U,F)<C$, and $\E \|w_t\|^4+ \E\|g_t\|^4+\E \|f_t\|^4 <C$, and

        $ \E \|g_t\|^4\|f_t\|^4+\E u_{it}^4\|f_t\|^4+\E e_{jt}^4 \|f_t\|^8+\E e_{jt}^4\|f_t\|^4\|g_t\|^4+ 
  \E \|w_t\|^4\|g_t\|^4<C$.

   \end{assumption}

We now present the main inferential result.  Recall that 
   $$
 \bar\theta_{\mathcal G,t} := \frac{1}{|\mathcal G|_0}\sum_{i\in\mathcal G} \theta_{it},
 $$
 with $\bar\theta_t=  \bar\theta_{\mathcal G,t} $ when $\mathcal G=\{1,...,N\}$, where $\mathcal G$ is a known cross-sectional subset of particular interest. 
 
   \begin{theorem}\label{t3.2} Suppose  Assumptions \ref{a3.1}-  \ref{a3.8} hold. 
   Fix any $t\leq T$ and $\mathcal G\subseteq\{1,...,N\}$. Suppose 
 $N,T\to\infty$, and either (i) $|\mathcal G|_0=o(N)$, or (ii) $N=o(T^2)$ holds.  In addition, suppose $f_t' V_ff_t $  and $\bar\lambda_{\mathcal G}' V_\lambda \bar\lambda_{\mathcal G}$ are both bounded away from zero.
  Then 
   $$
\Sigma_{\mathcal G}^{-1/2}\left( \frac{1}{|\mathcal G|_0}\sum_{i\in\mathcal G}\widehat\theta_{it}-\bar\theta_{\mathcal G,t}\right)\to^d\mathcal N(0,1)
   $$
   where, 
   $$
 \Sigma_{\mathcal G}:=\frac{1}{T|\mathcal G|_0}f_t' V_ff_t+ \frac{1}{N}\bar\lambda_{\mathcal G}' V_\lambda \bar\lambda_{\mathcal G},
   $$
 with
 \begin{eqnarray*}
 V_f&=&\frac{1}{T}\sum_{s=1}^T \Var\left(\frac{1}{\sqrt{|\mathcal G|_0}}\sum_{i\in\mathcal G}\Omega_i f_se_{is}u_{is}\bigg{|}F\right),\quad \Omega_i=(\frac{1}{T}\sum_{s=1}^Tf_sf_s' \E e_{is}^2)^{-1}\cr
 V_{\lambda}&=& V_{\lambda1}^{-1}V_{\lambda2}V_{\lambda1}^{-1},\quad V_{\lambda1}= \frac{1}{N}\sum_{i=1}^N\lambda_i\lambda_i'\E e_{it}^2, \cr 
 \bar\lambda_{\mathcal G}&=&  \frac{1}{|\mathcal G|_0}\sum_{i\in\mathcal G}\lambda_i.
 \end{eqnarray*}
 
   \end{theorem}

Theorem \ref{t3.2}  immediately leads to two special cases.
\begin{cor}\label{cor3.1}
Suppose   Assumptions \ref{a3.1}-  \ref{a3.8} hold.  Fix $t\leq T$.

(i) Individual effect: Fix  any $i\leq N$, then 
  $$
\Sigma_{i}^{-1/2}\left( \widehat\theta_{it}- \theta_{it}\right)\to^d\mathcal N(0,1)
   $$
   where 
   $$
 \Sigma_{i}:=\frac{1}{T }f_t' V_{f,i}f_t+ \frac{1}{N} \lambda_{i}' V_\lambda \lambda_{i}
   $$
 with $V_{f,i}=\frac{1}{T}\sum_{s=1}^T \Var\left( \Omega_i f_se_{is}u_{is}\bigg{|}F\right)$.

 (ii) Cross-sectional average effect:    Suppose  $N=o(T^2)$ and $\liminf_{N}\sigma_\lambda^2>0$, then
 $$
 \sqrt{N}\sigma_{\lambda}^{-1}\left( \frac{1}{N}\sum_{i=1}^N\widehat\theta_{it}-\bar \theta_{t}\right)\to^d\mathcal N(0,1)
   $$
   where $\bar \theta_{t}=\frac{1}{N}\sum_{i=1}^N\theta_{it}$,  $\bar\lambda=\frac{1}{N}\sum_{i=1}^N\lambda_i
   $ and $\sigma_\lambda^2:=\bar\lambda ' V_\lambda \bar\lambda.$

\end{cor}

We now discuss estimating the asymptotic variance.  While all quantities can be estimated using their sample counterparts, complications arise due to the rotation discrepancy coupled with sample splitting. We note that the fact that factors and loadings are estimated up to rotation  matrices does not usually result in extra complications because the asymptotic variance of $\widehat\theta_{it}$ is rotation-invariant. In the current context, however, the estimated factors in the different subsamples may correspond to different rotation matrices. 
 
To preserve the rotation invariance property of the asymptotic variance, we therefore estimate  quantities that go into the asymptotic variance separately   within the subsamples, and produce the final asymptotic variance estimator by averaging the results across subsamples. Specifically, assuming $u_{jt}$ to be cross-sectionally independent given $\{e_{jt}\}$ for simplicity, we respectively estimate the above quantities on $S$, for $S\in \{I,I^c\}$ denoting one of the two subsamples, by    
 \begin{eqnarray*}
 \widehat v_{\lambda} &=&\frac{1}{2N}( \widehat \lambda_{I,\mathcal G}' \widehat V_{\lambda,1}^{I-1} \widehat V_{\lambda,2}^I  \widehat V_{\lambda,1}^{I-1} \widehat \lambda_{I,\mathcal G} +   \widehat \lambda_{I^c,\mathcal G}' \widehat V_{\lambda,1}^{I^c-1} \widehat V_{\lambda,2}^{I^c}  \widehat V_{\lambda,1}^{I^c-1} \widehat \lambda_{I^c,\mathcal G} )  \cr
  \widehat v_{f} &=&\frac{1}{2T|\mathcal G|_0}( \widehat f_{I,t}' \widehat V_{I, f}  \widehat f_{I,t} +   \widehat f_{I^c,t}' \widehat V_{I^c, f}   \widehat f_{I^c,t} ) \cr
   \widehat V_{S, f} &=& \frac{1}{|\mathcal G|_0|S|_0}\sum_{s\notin S}\sum_{i\in\mathcal G} \widehat\Omega_{S,i}\widehat f_{S,s}\widehat f_{S,s}' \widehat\Omega_{S,i}  \widehat e_{is}^2\widehat u_{is}^2    \cr
 \widehat V_{\lambda,1}^S&=&\frac{1}{N}\sum_j\widehat  \lambda_j\widehat \lambda_j'\widehat e_{jt}^2 \cr
 \widehat V_{\lambda,2}^S&=&\frac{1}{N}\sum_j\widehat \lambda_j\widehat \lambda_j'\widehat e_{jt}^2 {\widehat u_{jt}^2} \cr
    \widehat \lambda_{S,\mathcal G} &=&  \frac{1}{|\mathcal G|_0}\sum_{i\in\mathcal G}\widehat \lambda_{S,i},\quad \widehat\Omega_{S,i}= (\frac{1}{|S|_0} \sum_{s\in S}\widehat f_{S,s}\widehat f_{S,s}')^{-1} (\frac{1}{T}\sum_{s=1}^T\widehat e_{is}^2)^{-1}.
 \end{eqnarray*}
The estimated  asymptotic variance of $\widehat\theta_{it}$ is then $  \widehat v_{\lambda} + \widehat v_{f} .$
       
\begin{cor}\label{t3.3} In addition to the assumptions  of Theorem \ref{t3.2}, assume $u_{it}$
to be cross-sectionally independent given $\{e_{it}\}$.   
  Then for any fixed $t\leq T$ and $\mathcal G\subseteq\{1,...,N\}$, 
   $$
 (\widehat v_{\lambda} + \widehat v_{f} )^{-1/2}\left( \frac{1}{|\mathcal G|_0}\sum_{i\in\mathcal G}\widehat\theta_{it}-\bar\theta_{\mathcal G,t}\right)\to^d\mathcal N(0,1).
   $$
         
\end{cor}

\section{Application to Inference for Low-rank Matrix Completion}
 
Our inference method is fairly general and is applicable to many low-rank based inference problems. Here, we briefly discuss an application to large-scale factor models where observations are missing at random.  Existing methods for inference in this setting often use the EM algorithm to estimate $\theta_{it}$, as discussed in \cite{SW02} and theoretically justified in \cite{su2019factor}. In the following, we outline how our proposed approach can be used to perform estimation and inference in this setting. 
 
Consider a factor model
$$
y^*_{it}= \theta_{it}+ v_{it},\quad \theta_{it}=\lambda_i'f_t
$$
where $\lambda_i, f_t$ are respectively the factor loadings and factors and we assume that the $y^*_{it}$ are missing at random.  Let $x_{it}=1\{y_{it}^* \text{ is not missing}\}$ indicate that $y^*_{it}$ is observed. We then observe $y_{it}:=y^*_{it}x_{it}$ and $x_{it}$. Write $u_{it}:= v_{it}x_{it}$. Then multiplying $x_{it} $ on both sides of the model we have
 \begin{equation}\label{eq4.1}
 y_{it} = x_{it}\theta_{it} + u_{it},
 \end{equation}
which can be written in matrix form as $Y= X\odot\Theta+U$ using the notation from previous sections.

Recovering $\Theta$ can thus be viewed as a \textit{noisy low-rank matrix completion} problem. This problem has been studied, for example, in \cite{koltchinskii2011nuclear} where $\Theta$ is estimated via
\begin{equation}\label{e4.2}
\widetilde\Theta=\arg\min_\Theta \|Y-X\odot\Theta\|_F^2+\nu_1\|\Theta\|_{n}.
\end{equation}
This estimator is consistent with good rates of convergence, but it is unsuitable for inference in general as $\widetilde\Theta$ tends to have large shrinkage bias. 

To address the shrinkage bias from directly solving \eqref{e4.2}, we can apply our proposed algorithm. While one could apply our full algorithm from the previous sections, we note that there is no need to orthogonalize $(x_{it}, y_{it})$ in this case as the model does not have an interactive fixed effects structure. We can thus state a simplified algorithm as follows:

\begin{algo}\label{al4.1} Estimation of $\theta_{it}$ in factor models with data missing-at-random:

\textit{Step 1. Estimate the number of factors.} 
Run  penalized regression (\ref{e4.2}). Estimate $K =\dim(f_t)$ by
   $
  \widehat K=\sum_{i} 1\{\psi_i(\widetilde \Theta)\geq (\nu_1  \|\widetilde \Theta\|)^{1/2}\},
  $
 where $\psi_i(.)$ denotes the $i^{\text{th}}$ largest singular-value.

\medskip

\textit{Step 2: Sample splitting.} 
Randomly split the sample into $\{1,..., T\}/\{t\}=I\cup I^c$, so that $|I|_0= [(T-1)/2]$. Estimate the low-rank matrices $\Theta$  as in (\ref{e4.2}) with $(Y, X)$ replaced with $(Y_I, X_I)$ to obtain $ \widetilde\Theta_I$.
Let $\widetilde\Lambda_I=(\widetilde\lambda_1,...,\widetilde\lambda_N)'$ be the $N\times \widehat K$ matrix whose columns are defined as $\sqrt{N}$ times  the first $\widehat K$ eigenvectors of $\widetilde \Theta_I\widetilde \Theta_I'$.  

\medskip

\textit{Step 3. Estimate $(f_t,\lambda_i)$ for use in inference about $\theta_{it}$.} Using  $\widetilde \Lambda_I$,
$$
\widehat f_{I,s}:= \arg\min_{f_s}\sum_{i=1}^N(y_{is}- x_{is} \widetilde \lambda_i'f_s)^2,\quad s\in I^c\cup\{t\}.
$$
Update estimates of loadings as
$$
\widehat\lambda_{I,i}= \arg \min_{\lambda_i}\sum_{s\in I^c\cup\{t\}} (y_{is} - x_{is} \lambda_i'\widehat f_{I,s})^2,\quad i=1,..., N.
$$
    
\medskip
 
\textit{Step 4. Exchange $I$ and $I^c$.} Repeat steps  2-3 with $I$ and $I^c$ exchanged to obtain $(\widehat{\lambda}_{I^c,i}, \widehat{f}_{I^c,s}: s\in I\cup \{t\}, i\leq N)$. 

\medskip

\textit{Step 5. Estimate $\theta_{it}$. } Obtain the estimator of $\theta_{it}$:
$$
\widehat\theta_{it}:= \frac{1}{2}[\widehat\lambda_{I,i}'\widehat f_{I,t}+ \widehat{\lambda}_{I^c,i}'\widehat f_{I^c,t}]. 
$$

\end{algo}

It is worth noting that the restricted strong convexity condition (Assumption \ref{a3.1}) has a clear, intuitive interpretation in this case. Specifically, the condition is that there exists a $\kappa_c>0$ such that
$$
\|\Delta\odot X\|_F^2\geq \kappa_c \|\Delta\|_F^2 \quad (\text{equivalently } \sum_{it}(\Delta_{it}x_{it})^2\geq \kappa_c\sum_{it}\Delta_{it}^2)
$$
for any $\Delta$ inside the ``low-rank cone.'' Intuitively, this condition means that $y_{it}^*$ should be observed sufficiently often, ruling out cases where the fraction of observations is vanishingly small. 

Maintaining the analogs of the assumptions of Corollary \ref{cor3.1} and the additional condition that $\mu_{i}:=P(x_{it}=1)$ is time-invariant, 
it can be shown that 
$$
\frac{\widehat\theta_{it}-\theta_{it}}{\left(\frac{1}{T}f_t'V_{f,i}f_t+\frac{1}{N}\lambda_i'V_\lambda \lambda_i\right)^{1/2}}\to^d\mathcal N(0,1) 
$$
where $(V_{f,i}, V_\lambda)$ are as given in Corollary \ref{cor3.1} with $e_{it}$ replaced by  $x_{it}$ in all relevant quantities to define the covariance.
Note that our method allows the observation probabilities $\mu_i$ to be individual specific. In the case  that $\mu_i$ does not  vary over $i$, 
 the asymptotic variance is the same as that of  \cite{su2019factor} obtained from the EM algorithm. 


\section{Monte Carlo Simulations}

In this section, we provide some simulation evidence about the finite sample performance of our proposed procedure in two settings. The first setting is a static model with heterogeneous effects where all of our assumptions are satisfied. The second case is a dynamic model which is not formally covered by our theoretical analysis. 

\subsection{Static Model}
For our first set of simulations, we generate outcomes as 
$$
y_{it}= \alpha_i'g_t+x_{it,1}\theta_{it} +  x_{it,2}\beta_{it} + u_{it}
$$
where $\theta_{it}=\lambda_{i,1}'f_{t,1}$ and $\beta_{it}=\lambda_{i,2}'f_{t,2}$. Generate the observed   $x_{it}$ according to
$$
x_{it,r}=  l_{i,r}'w_{t,r}+\mu_x+e_{it,r},\quad r=1,2,
$$
where $(e_{it,r}, u_{it})$ are generated independently from the standard normal distribution across $(i,t,r)$. We set $\mu_x=2$ so $x_{it,r}$ follows a factor model with an intercept, where $(l_{i,r}, w_{t,r}, e_{it,r})$ can be estimated via the PC estimator on the de-meaned sample covariance matrix $(s_{ij,r})_{N\times N}$:   $$s_{ij,r}=\frac{1}{T}\sum_{t=}^T(x_{it,r}-\bar x_{i,r})(x_{jt,r}-\bar x_{j,r}),\quad \bar x_{i,r}=\frac{1}{T}\sum_tx_{it,r}.$$
We  set $\dim(g_t) = \dim(f_{t,1}) = \dim(f_{t,2}) = \dim(w_{t,1}) = \dim(w_{t,2}) = 1$ and generate all factors and factor loadings as independent draws from $\mathcal N(2,1)$. The loadings $(\alpha_i, \lambda_{i,1},\lambda_{i,2}, l_{i,r})$ are 
treated fixed  in the simulation replications while the factors $(g_t, f_{t,1}, f_{t,2}, w_{t,r})$ are random and repeatedly sampled in replications.
 We report results for $N = T = 100$, $N = T = 200$, and $N = 500$ and $T = 100$. Results are based on 1000 simulation replications.

We compare  three  estimation methods: 

(I) (``Partial-out") the proposed estimator that uses orthogonalization (partialing out $\mu_{it}$ from $x_{it}$) and sample-splitting.

(II) (``No Partial-out") the estimator that uses  sample splitting but  not partial-out $\mu_{it}$ from $x_{it}$. That is, we run Steps 1-4 of the algorithm using data for $s\in I$, and obtain 
\begin{eqnarray*}
(  \widetilde f_{I,s},\widetilde g_s)&=& \arg\min_{f_s, g_s}\sum_{i=1}^N(y_{is}- \widetilde\alpha_i'g_s- x_{is} \widetilde \lambda_i'f_s)^2,\quad s\in I^c\cup\{t\},\cr
(\dot\lambda_{I,i}, \dot\alpha_i)&=& \arg \min_{\lambda_i, \alpha_i}\sum_{s\in I^c\cup\{t\}} (y_{is} - \alpha_i'  \widetilde g_s- x_{is} \lambda_i'\widetilde f_s)^2,\quad i=1,..., N.
\end{eqnarray*}
We then exchange $I$ and $I^c$ to obtain
$(\widetilde f_{I^c, t}, \dot\lambda_{I,i})$. The estimator is then defined  as $\frac{1}{2}(\dot\lambda_{I,i}'\widetilde f_{I,t}+\dot\lambda_{I^c, i}'\widetilde f_{I^c, t})$.

(III) (``Regularized") the estimator using the nuclear-norm regularizion only. That is, $\theta_{it}$ is directly  estimated as the $(i,t)$ th element of 
$\widetilde\Theta$ in (\ref{e1}) with no further debiasing.

We report results for standardized estimates defined as
$$
\frac{\widehat\theta_{it}-\theta_{it}}{se(\widehat\theta_{it})}.
$$
For estimator (I) (``Partial-out''), we use the feasible estimator of the theoretical standard error
$$
se(\widehat\theta_{it})=( \widehat v_{\lambda} +  \widehat v_{f} )^{1/2}.
$$
given in 
Corollary \ref{t3.3}. If the asymptotic theory is adequate, the standardized estimates should be approximately $\mathcal N(0, 1)$ distributed. We standardize estimators (II) and (III) (``No Partial-out'' and ``Regularized'') using the sample standard deviation obtained from the Monte Carlo repetitions as $se(\widehat\theta_{it})$ as the theoretical variances of these estimators is unknown. This standardization is of course infeasible in practice but is done for comparison purposes. 


We report results only for  $\theta_{it}$ with $t=i=1$;  results for other values of $(t,i)$ and   for  $\beta_{it}$  are similar.   In Table \ref{tab1}, we report the fraction of simulation draws where the true value for $\theta_{it}$ was contained in the 95\% confidence interval:
$$
[\widehat\theta_{it}-1.96 se(\widehat\theta_{it}), \widehat\theta_{it}-1.96 se(\widehat\theta_{it})].
$$
Again, $se(\widehat\theta_{it})$  is the feasible  standard error   in Corollary \ref{t3.3} for estimator (I), and is the infeasible standard error obtained from Monte Carlo repetitions for estimators (II) and (III).

Figure \ref{sec: fig1} plots the histogram  of the standardized estimates, superimposed with the standard normal density. The histogram is scaled to be a density function. The top panels of  Figure \ref{sec: fig1} are for the proposed estimation method; the middle panels are for the estimation without partialing-out the mean structure of $x_{it}$, and the bottom panels are for the nuclear-norm regularized estimators directly.   It appears that asymptotic theory provides a good approximation to the finite sample distributions for the proposed post-SVT method. In contrast, both the estimated $\theta_{it}$ without partialing-out and the regularization-only method noticeably deviate from the standard normal distribution.

\begin{table}[htp]
\begin{center}
\label{tab1}
\caption{Coverage Probability in Static Simulation Design}
\begin{tabular}{cc|ccc}
\hline
\hline
$N$& $T$ &   Partial-out & No partial-out & Regularized\\
\hline
100 & 100 & 0.943 & 0.794&0.778\\
200 & 200 & 0.956 & 0.835&0.911\\
500 & 100 & 0.969 & 0.804&0.883\\
\hline
\multicolumn{5}{p{10.5cm}}{\footnotesize Note: This table reports the   coverage probability of 95\% confidence intervals in the static simulation design. Column ``Partial-out'' provides results from our proposed procedure; column ``No partial-out'' reports results from a procedure that   directly uses the observed covariates without partialing-out its mean structure. ``Regularized" directly uses the nuclear-norm regularized estimator, without further debiasing. The confidence interval for the ``Partial-out" estimator is constructed using feasible standard errors, while confidence intervals for the other two estimators use infeasible standard errors obtained from the Monte Carlo repetitions. Results are based on 1000 simulation replications.}
\end{tabular}
\end{center}
\end{table}%

  \begin{figure}[htbp!]
\begin{center}
 \includegraphics[width=6cm]{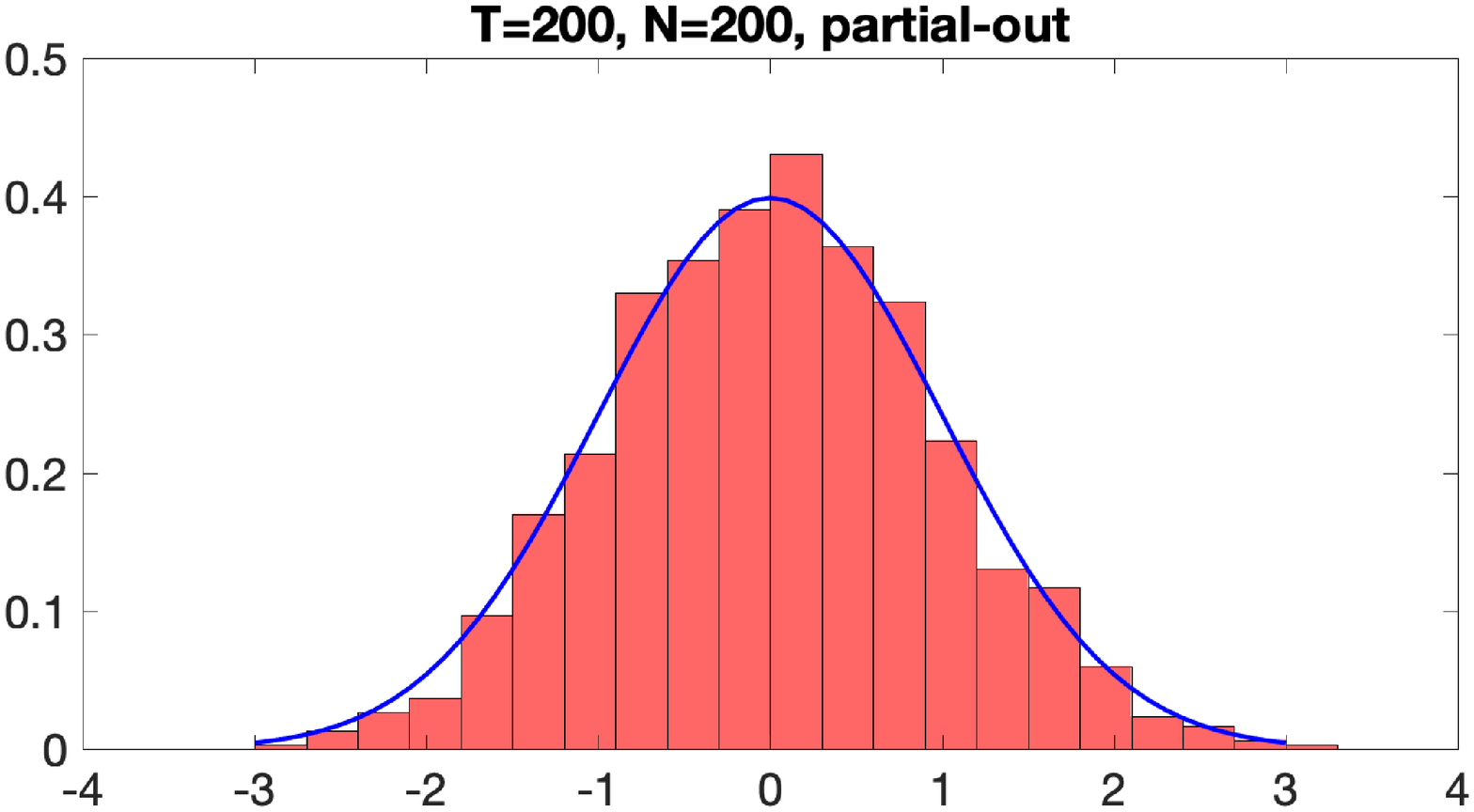}
  \includegraphics[width=6cm]{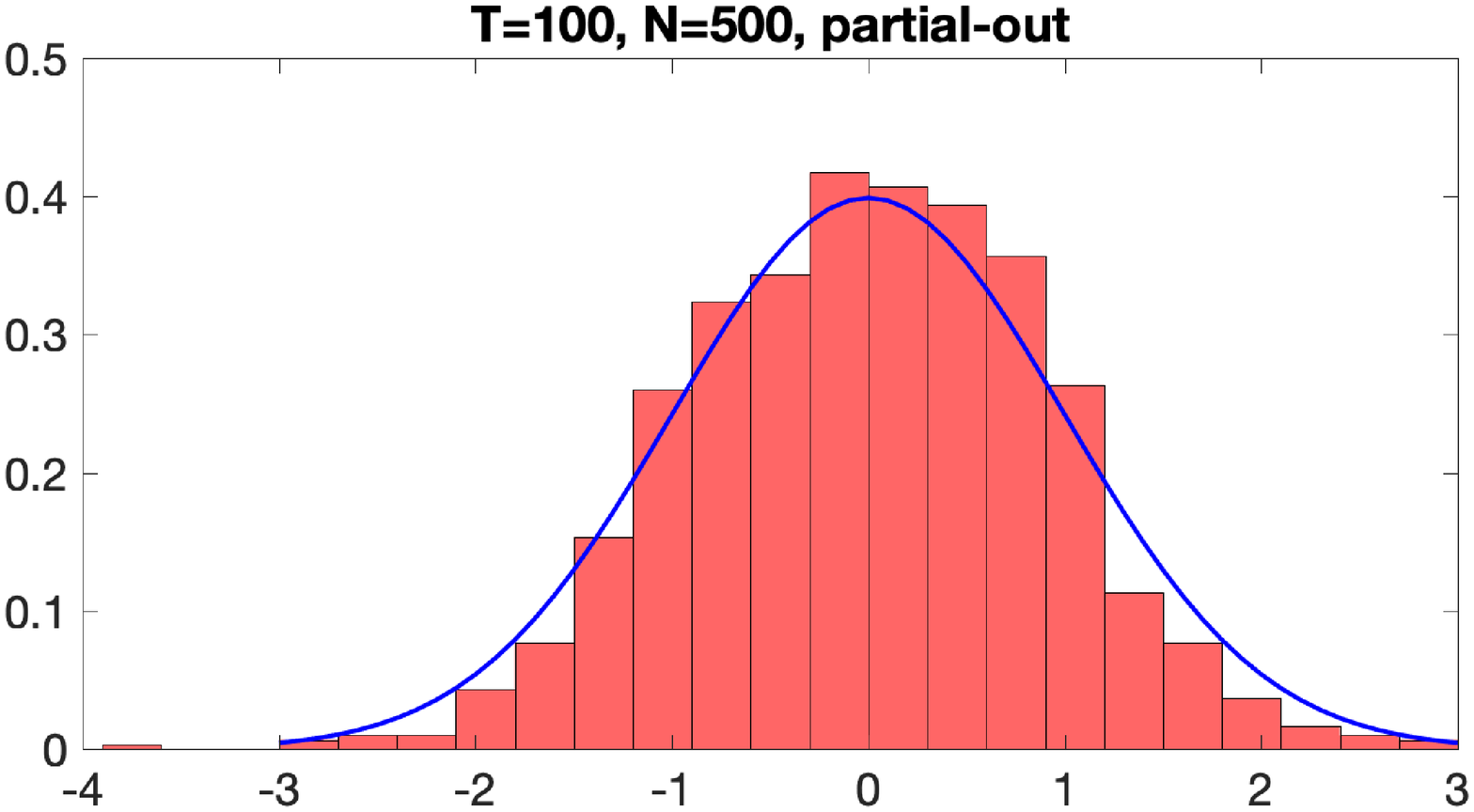}
\includegraphics[width=6cm]{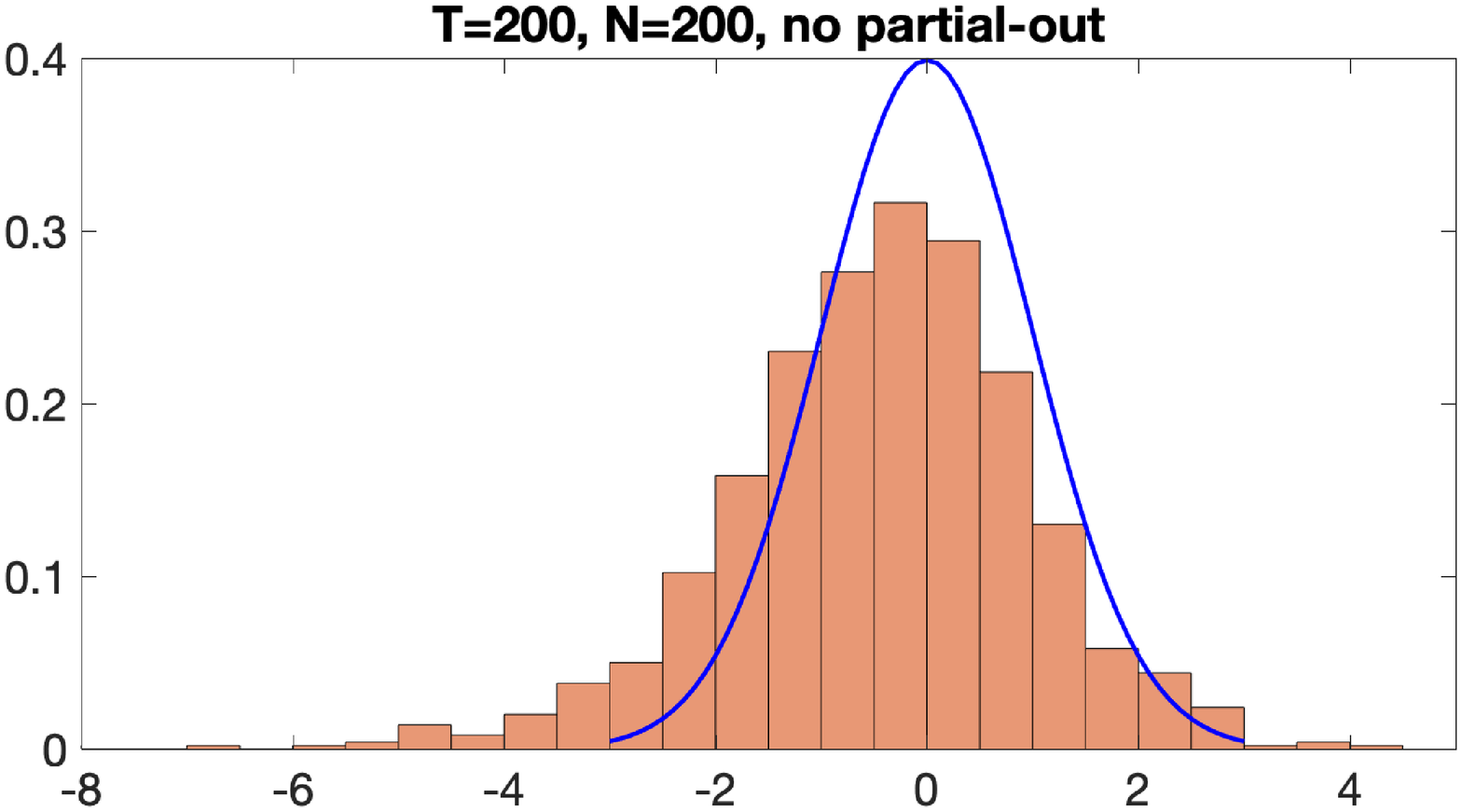}
\includegraphics[width=6cm]{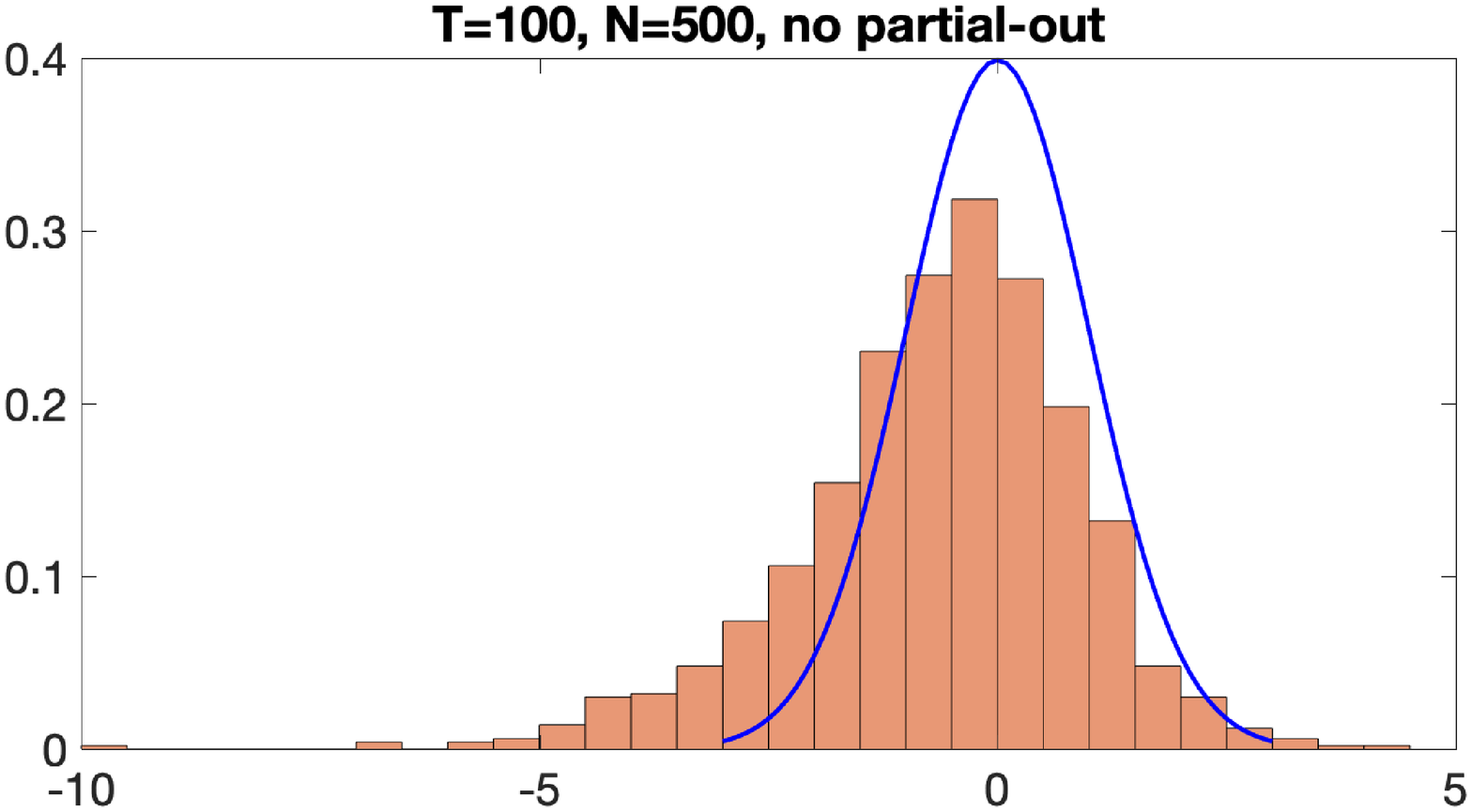}
 \includegraphics[width=6cm]{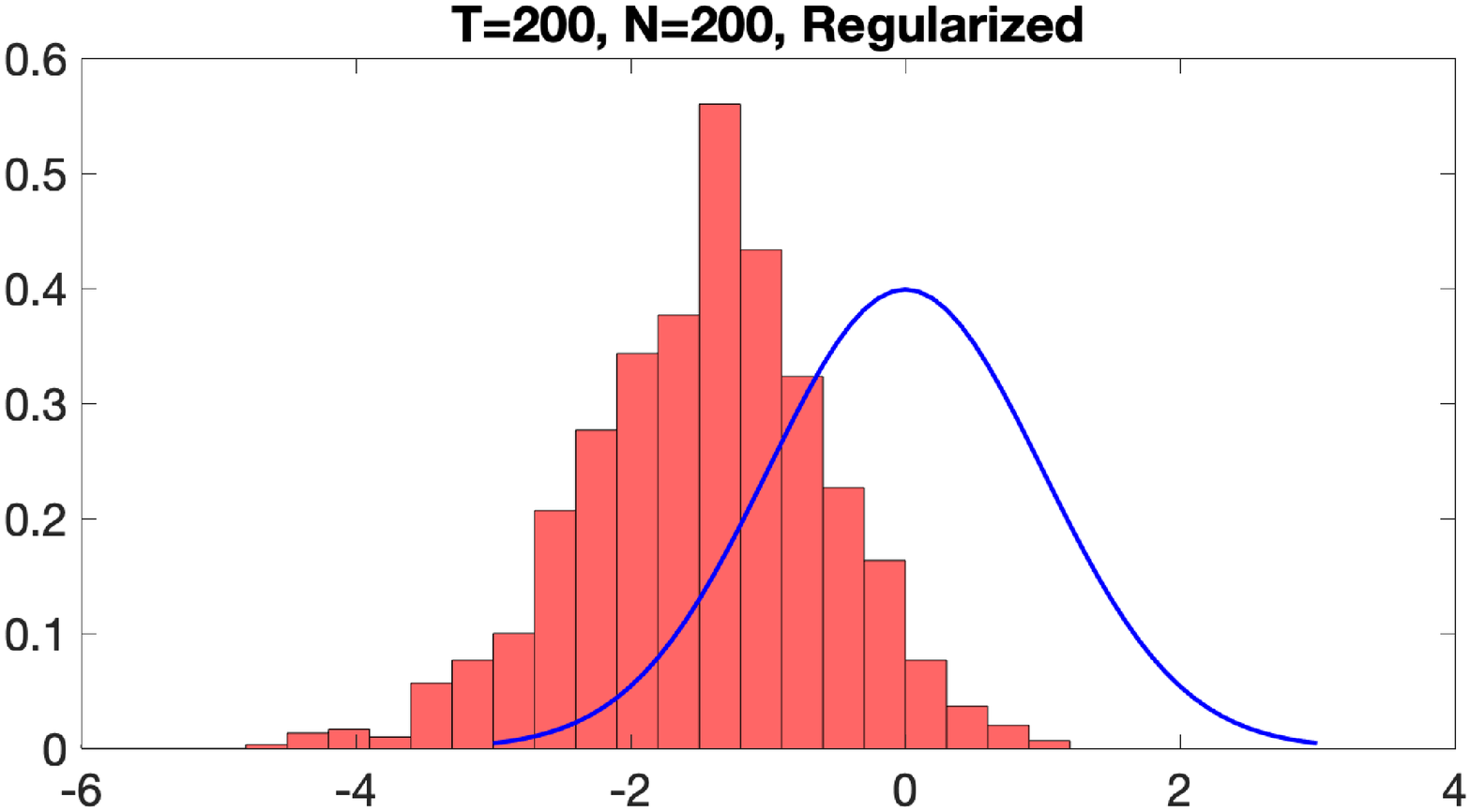}
 \includegraphics[width=6cm]{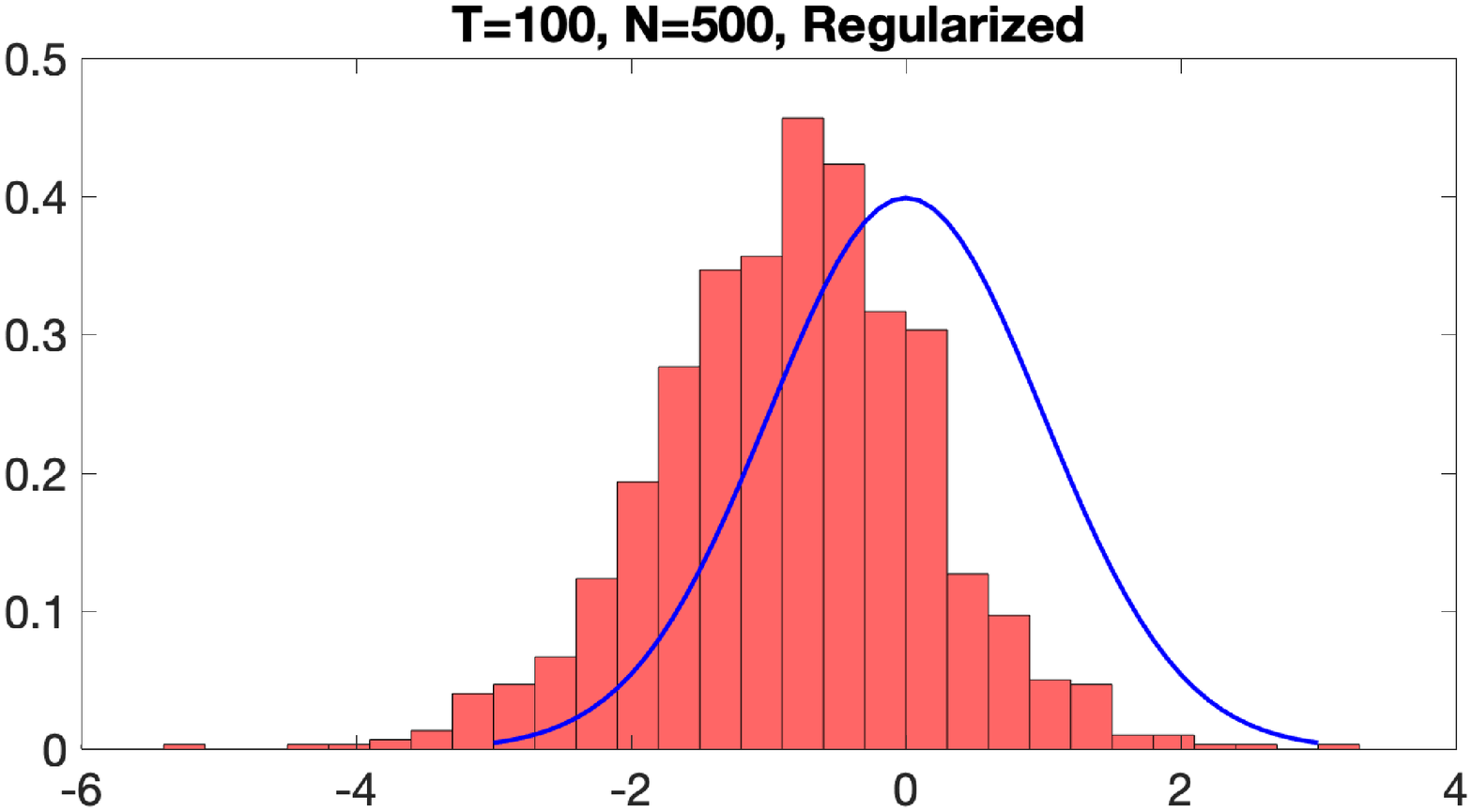}
\caption{\footnotesize  Histograms of standardized estimates in static models   ($(\widehat\theta_{11}-\theta_{11})/se(\widehat\theta_{11})$). The ``partial-out" plots are based on the full method including partialing-out the mean structure of $x_{it}$ and use a feasible estimate of the theoretical standard error. The ``no partial-out" plots apply our proposed procedure ignoring the step that partials out the mean structure of $x_{it}$ and rely on infeasible standard errors. The ``Regularized" plots are based on directly using the nuclear-norm regularized estimator and rely on infeasible standard errors. The standard normal density function is superimposed on each histogram. 
}
\end{center}
\label{sec: fig1}
\end{figure}

%
%
%
%
%

\subsection{Dynamic Model}

We now consider an example with a lagged dependent variable to allow for dynamics.  Specifically, we consider the model
\begin{align}\label{eq: dynamic}
y_{it}= \alpha_i'g_t+x_{it}\lambda_{i,1}'f_{t,1} 
+y_{i,t-1}\lambda_{i,2}'f_{t,2}+ u_{it},
\end{align}
where $x_{it}=l_i'w_t+ e_{it}$. Our asymptotic theory does not allow the lagged variable    $y_{i,t-1}$ because the presence of this variable is incompatible with an exact representation within a static factor model structure. 
Nevertheless, as this dynamic model is interesting, we investigate the finite sample performance of our method in this case with the understanding that it is not covered by our formal theory.

We calibrate the parameters of the data generating process for this simulation by estimating \eqref{eq: dynamic} using the data from \cite{acemoglu2008income}. Here $y_{it}$ is the democracy score of country $i$ in period $t$, and $x_{it}$ is the GDP per capita over the same period. From the data, we estimate that all low-rank matrices have two factors with the exception of $\alpha_i'g_t$, which is estimated to have one factor. 
We then generate all factors and loadings from normal distributions with means and covariances matching with the sample means and covariances of the estimated factors and loadings obtained from the actual democracy-income data. We generate initial conditions $y_{i1}$  independently from $0.3\mathcal N(0,1) + 0.497$, whose parameters are calibrated from the real data at time $t=1$. $u_{it}$ is generated independently from a $\mathcal N(0,\sigma^2)$ with $\sigma= 0.1287$ calibrated from the real data. Finally, $y_{it}$ is generated iteratively according to the model using the draws of initial conditions. 

Let $z_{it}:=y_{i, t-1}$. Figure  \ref{fig2} plots  the first twenty eigenvalues of the $N\times N$ sample covariance of  $ z_{it}$ when $N=T=100$, averaged over 100 repetitions. The plot reveals one very spiked eigenvalue. Therefore, we estimate a one-factor model for $z_{it}$ in our estimation procedure.

\begin{figure}[htbp]
\label{fig2}
\begin{center}
\includegraphics[width=6cm]{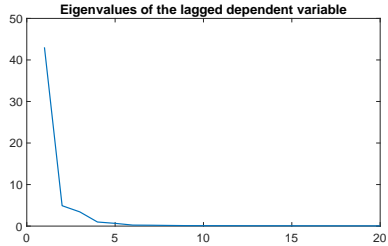}
\caption{\footnotesize Sorted eigenvalues of the sample covariance of the lagged dependent variable, averaged over 100 repetitions. }
 
\end{center}

\end{figure}

In this design, we only report results from applying the full multi-step procedure including sample-splitting and partialing-out. We report results for standardized estimates defined as
$$
\frac{\widehat\theta_{r,it}-\theta_{r,it}}{( \widehat v_{r,\lambda} +  \widehat v_{r,f} )^{1/2}}.\
$$
for $t=i=1$ and $r \in \{1,2\}$ for $\theta_{r,it} = \lambda_{i,r}'f_{t,r}$ and $\widehat v_{r,\lambda} +  \widehat v_{r,f}$ a corresponding estimate of the asymptotic variance. 
All results are based on 1000 simulation replications.

Table \ref{t4} reports the sample means and standard deviations of the t-statistics as well as coverage probabilities of feasible 95\% confidence intervals, and we report the histograms of the standardized estimates in Figure \ref{fig3}. Looking at the histograms, we can see that the bias in the estimated effects of $x_{it}$ is relatively small while there is a noticeable downward bias in the estimated effect of $y_{i,t-1}$. We also see from the numeric results that the standard deviation of the standardized effect of $x_{it}$ appears to be systematically too large, i.e. the estimated standard errors used for standardization are too small, which translates into some distortions in coverage probability. Coverage of the effect of lagged $y$ seems reasonably accurate despite the large estimated bias. Importantly, we see that performance improves on all dimensions as $N = T$ becomes larger. 
 
\begin{table}[htp]
\caption{Estimation and inference results in dynamic model simulation}
\begin{center}
\begin{tabular}{c|cc|cc|cc}
\hline
\hline
$N=T$ &  \multicolumn{2}{c|}{mean } &   \multicolumn{2}{c|}{standard deviation } &  \multicolumn{2}{c}{coverage probability } \\
  &  $x_{it}$ &   $y_{i,t-1} $   &  $x_{it}$ &$y_{i,t-1}$ &$x_{it}$ &$y_{i,t-1}$ \\
\hline
50 &  0.019& -0.345&1.342 & 1.140& 0.856&  0.919\\
100 & -0.096&  -0.267&  1.178&   1.015 &0.920&0.932\\
150 & -0.017 & -0.276 &1.117&0.942 &0.926& 0.944\\
\hline
\multicolumn{7}{p{12cm}}{\footnotesize Note: This table reports the simulation mean, standard deviation, and coverage probability of 95\% confidence intervals for the estimated effect of $x_{it}$ and $y_{i,t-1}$ for $i = t = 1$. Simulation mean and standard deviation are for estimates centered around the true parameter values and normalized in each replication by the estimated standard error. Thus, ideally, the entries in the mean panel would be 0 and the entries in the standard deviation panel would be 1. Results are based on 1000 simulation replications.}
\end{tabular}
\end{center}
\label{t4}
\end{table}%

\begin{figure}[htbp]
\label{fig3}
\begin{center}

\includegraphics[width=6cm]{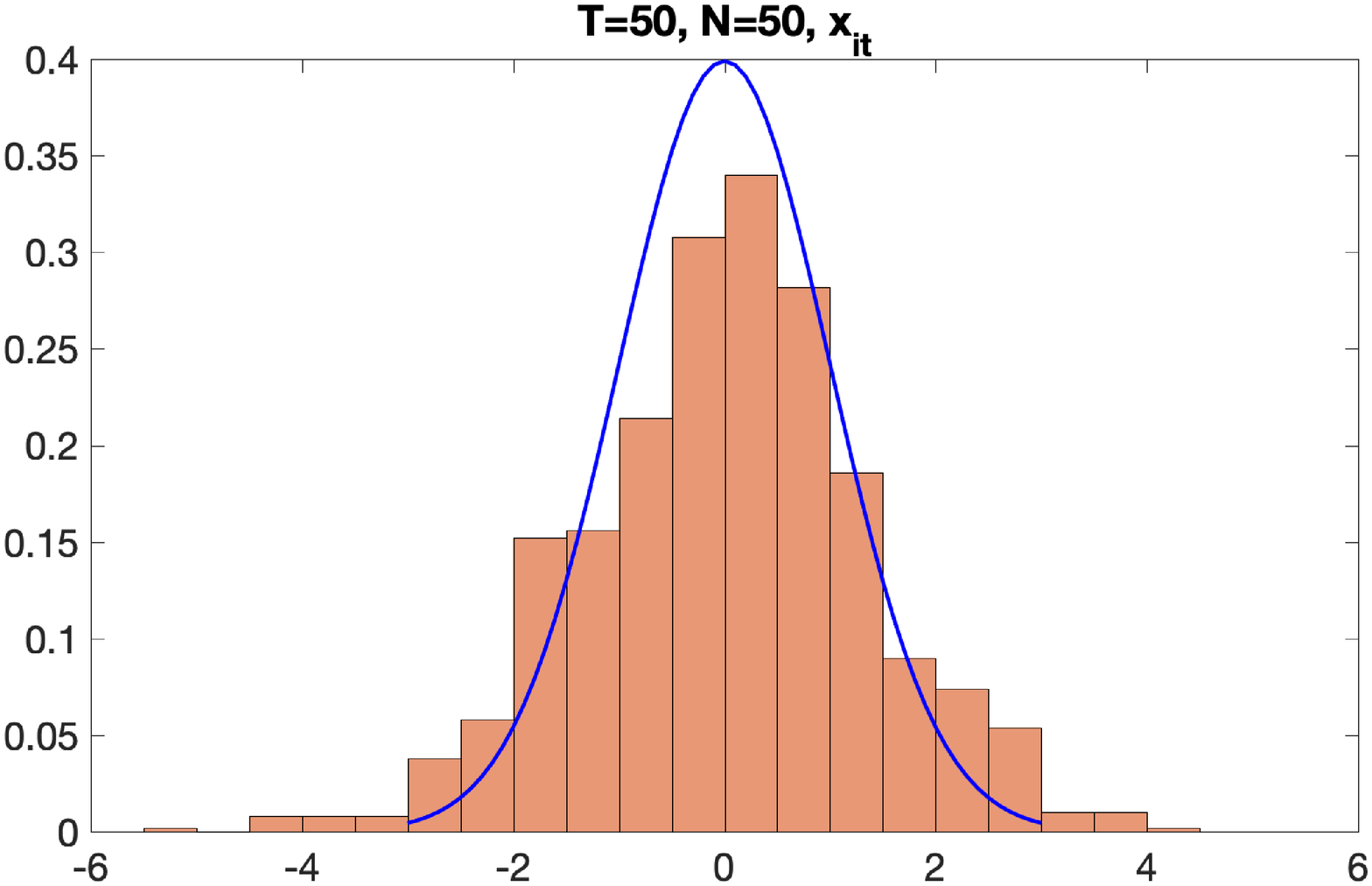}
\includegraphics[width=6cm]{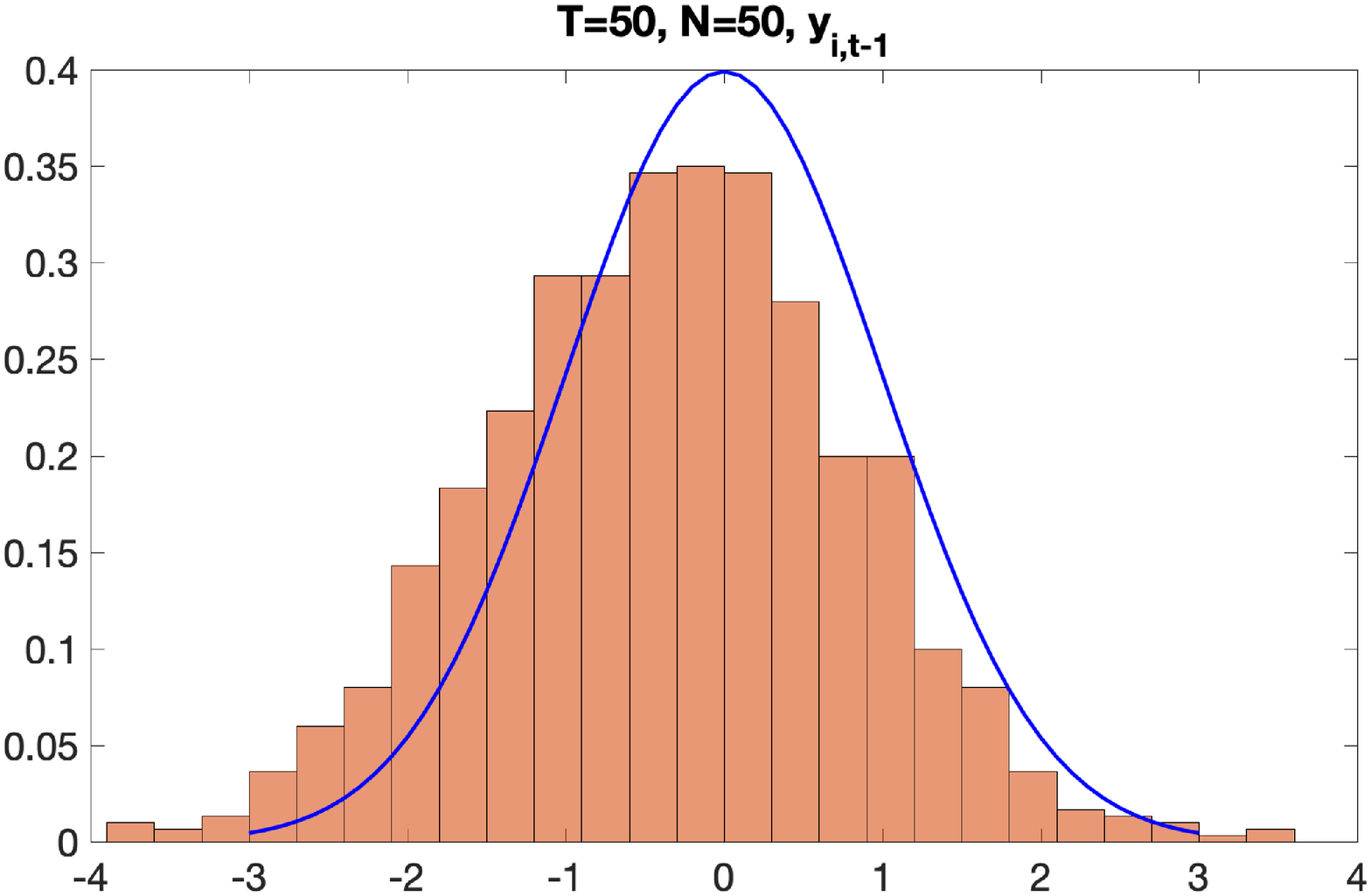}
\includegraphics[width=6cm]{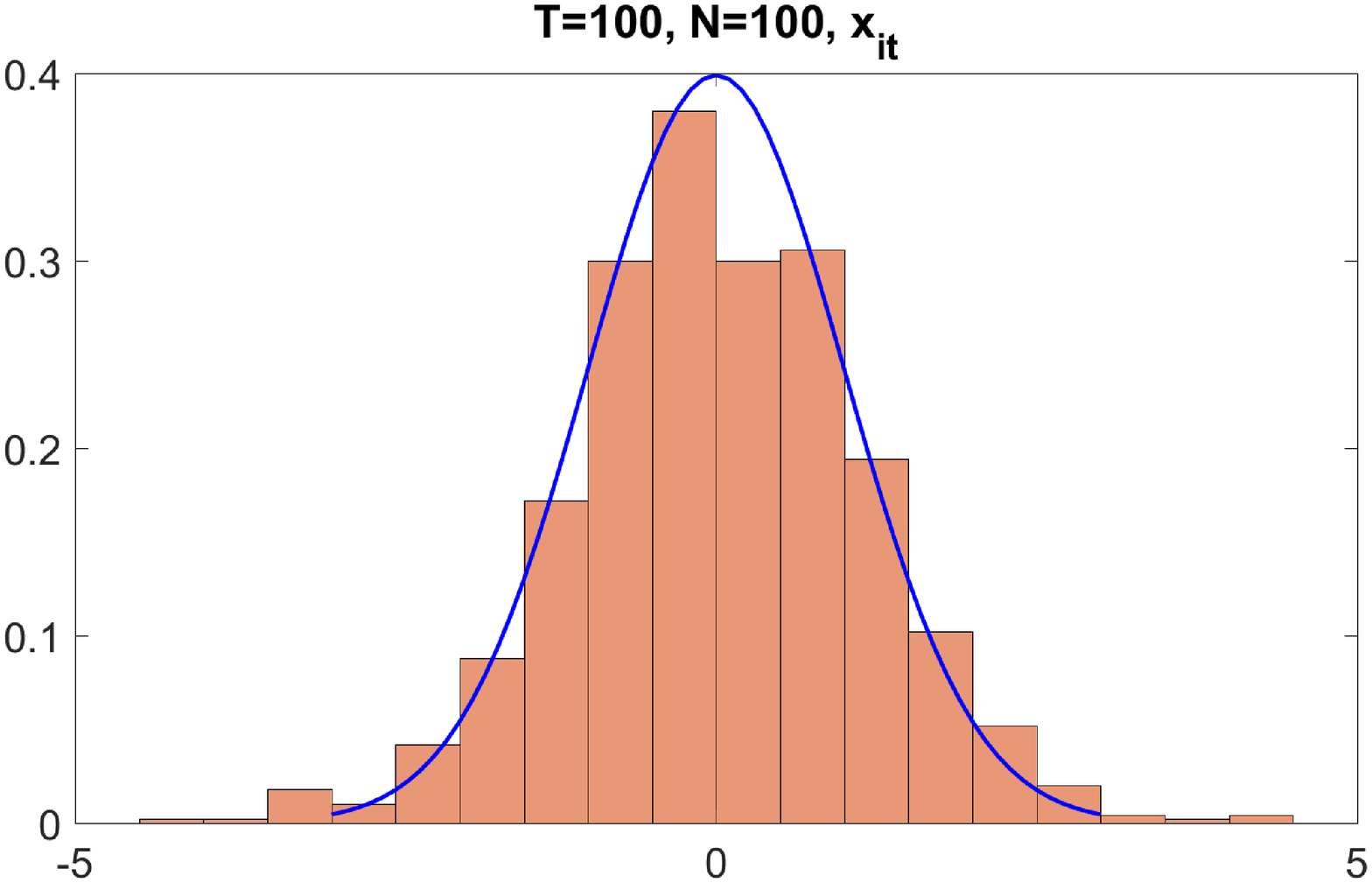}
\includegraphics[width=6cm]{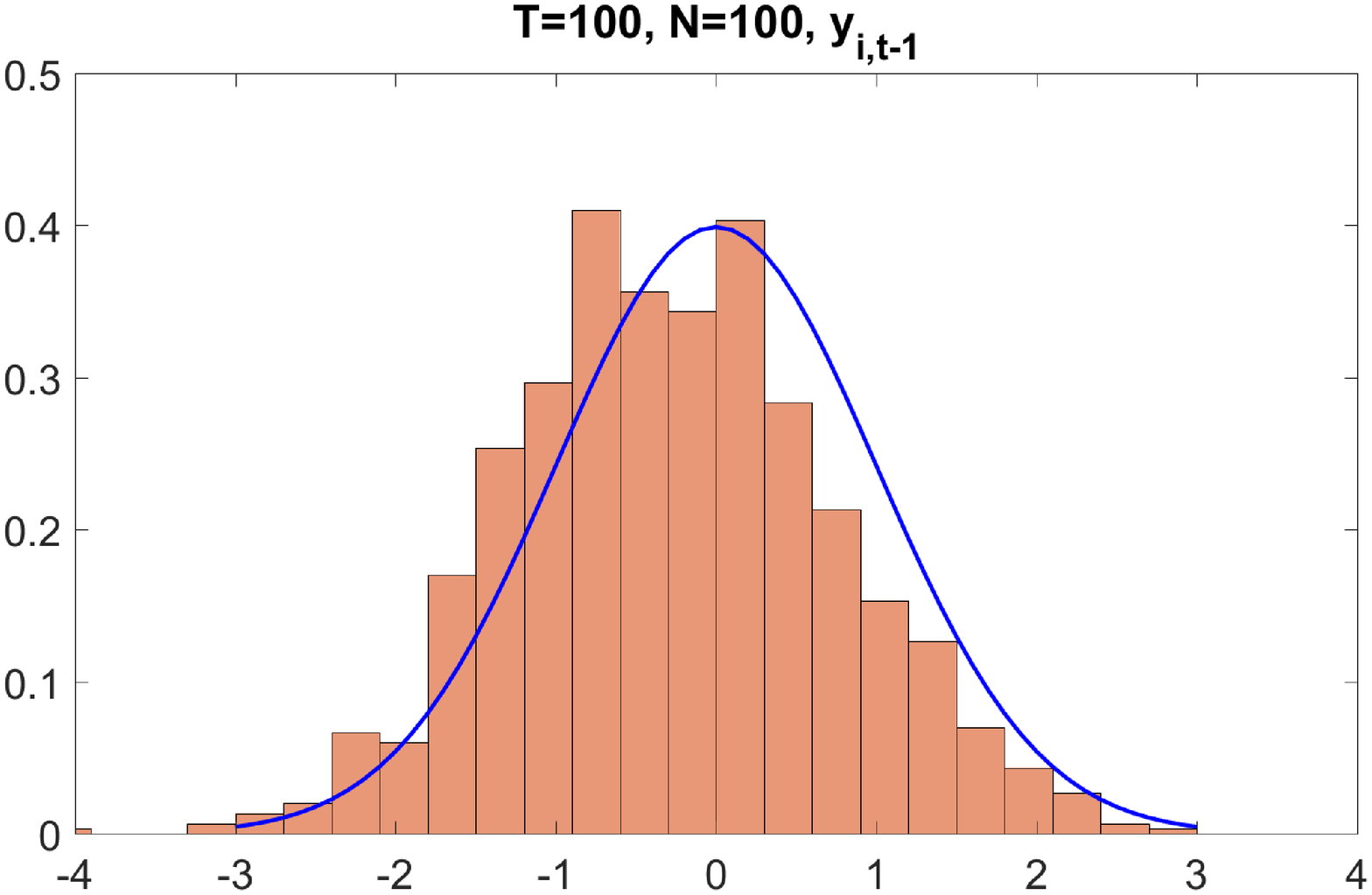}
\includegraphics[width=6cm]{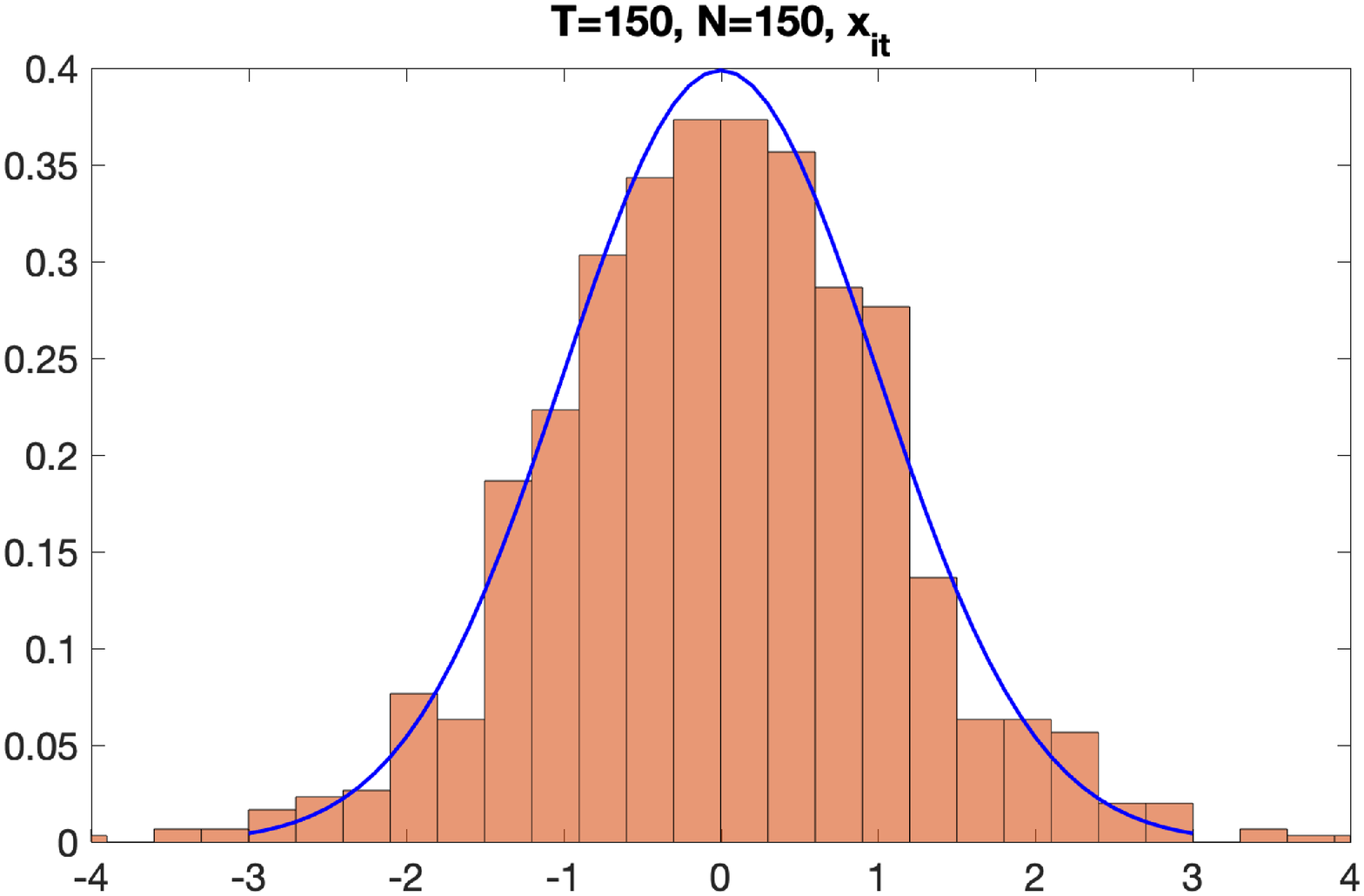}
\includegraphics[width=6cm]{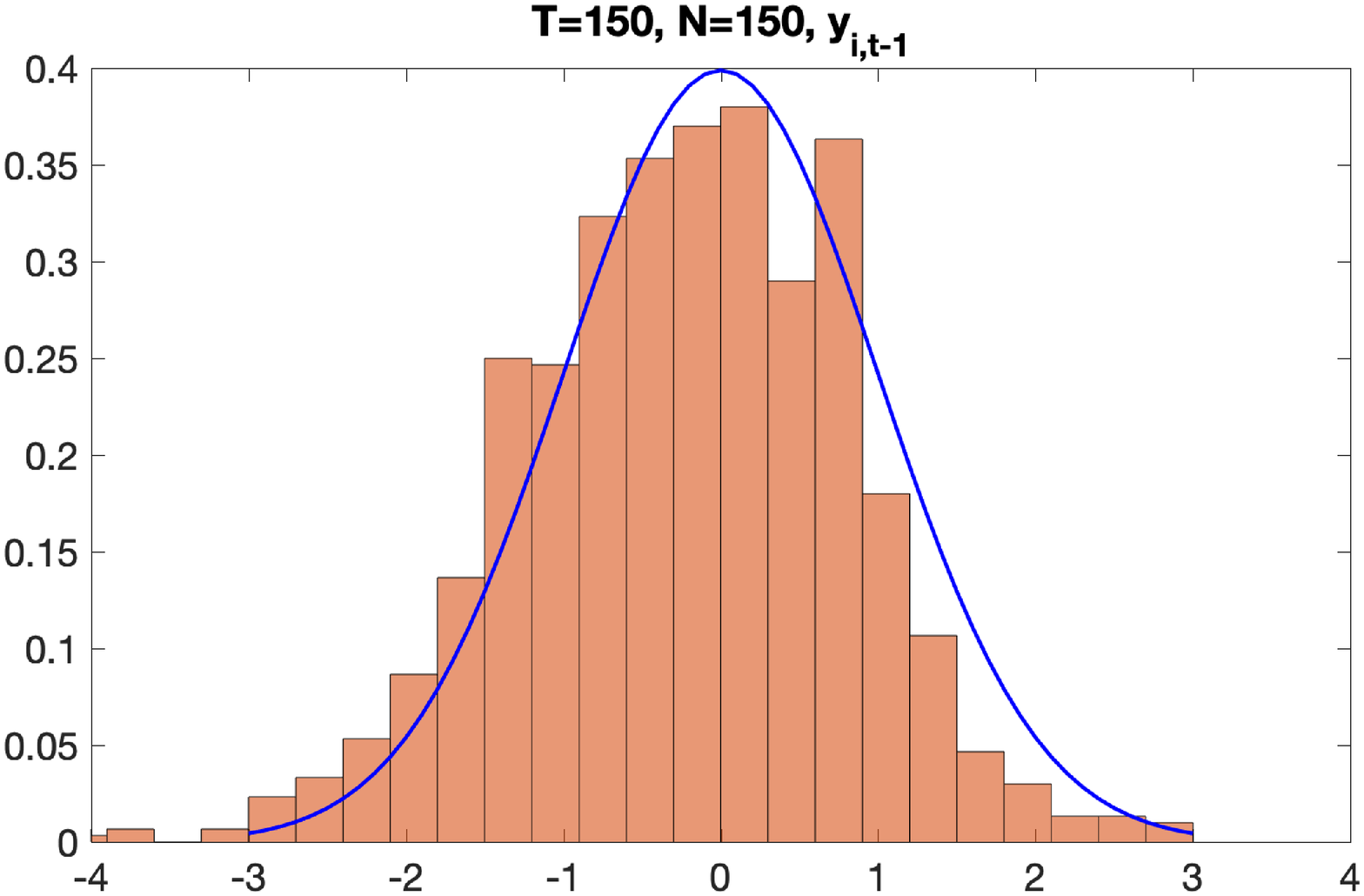}
\caption{\footnotesize Histograms of standardized estimates in  dynamic models   ($\widehat\theta_{11}-\theta_{11}$ divided by the estimated asymptotic standard deviation). The left three plots are the effect of  $x_{i,t}$; the right three plots are for the effect of $y_{i,t-1}$. The standard normal density function is superimposed on the histograms. }
  \end{center}

\end{figure}

\section{Empirical Illustration: Estimating the Effect of the Minimum Wage on Employment}\label{sec:demo}

 
\subsection{Background}
To illustrate the use of our approach, we consider estimation of the effect of the minimum wage on youth employment. Previous studies in the literature reach mixed conclusions. For instance, \cite{card1994minimum} find that the minimum wage has a positive effect on employment  using data of restaurants in New Jersey and Eastern Pennsylvania.  \cite{dube2010minimum} conclude the minimum wage has negative employment effects when county-level population is controlled for but also found no adverse employment effects from minimum wage increases when only variation among contiguous county-pairs is used for identifications. 
 \cite{phillips2018heterogeneous} consider slope heterogeneity at the county level by estimating a model of the form
\begin{equation}\label{e6.13}
y_{it}=x_{it,1}\theta_{i}+x_{it,2}\beta_{i}+ \alpha_i+g_t+u_{it}
\end{equation}
where $i$ denotes county, $t$ indexes quarter, $y_{it}$ is log employment, $x_{it,1}$ is the log minimum wage, and $x_{it,2}$ is log population, and $\alpha_i$ and $g_t$ are respectively additive county and time fixed effects.  They estimate county-level heterogeneous effects by imposing a grouping structure where $\theta_i$ and $\beta_i$ take on a small number of distinct values that are homogeneous within groups. The values of the group effects and group membership of each county is then learned as part of their estimation algorithm. They find evidence of the minimum wage having mixed effects across different groups. 
Also see \cite{neumark2008minimum} for an overview and  more recent works by \cite{cengiz2019effect} and  \cite{callaway2019difference}.

We contribute to this literature by estimating our model which allows for a rich intertemporal and cross-sectional heterogeneity structure in both the estimated effect of the minimum wage and in the fixed effects structure. It seems interesting to allow for intertemporal and cross-sectional heterogeneity of effects as there seems to be no \textit{ex ante} reason to believe that these effects should be constant. For example, these effects are presumably related to pressures from both the supply and demand side and other market conditions, all of which are plausibly time varying and related to local labor market conditions. Allowing for the interactive fixed effects structure may also be useful for capturing potentially complex unobserved heterogeneity that is correlated to the minimum wage. 

\subsection{Implementation}
We base our analysis on state-level data analogous to the county-level data in \cite{dube2010minimum} and \cite{phillips2018heterogeneous}.\footnote{We are grateful to Wuyi Wang for sharing the data with us. } The data are a balanced panel with $T = 66$ and $N = 51$ observations on the 50 states in the United States plus Washington D.C. running from the first quarter of 1990 to the second quarter of 2006.
Using these data, we estimate a model for log youth employment ($y_{it}$):
\begin{align*} 
y_{it} = x_{it,1}\theta_{it} + x_{it,2}\beta_{it} + x_{it,3}\gamma_{it} + \alpha_i'g_t + u_{it}
\end{align*}
with
\begin{align*}
\theta_{it}=\lambda_{i,1}'f_{t,1},\quad 	\beta_{it}=\lambda_{i,2}'f_{t,2},\quad \gamma_{it}=\lambda_{i,3}'f_{t,3},
\end{align*} 
where $x_{it,1 }$, $x_{it,2}$, and $x_{it,3}$ respectively denote the log minimum wage, log state population, and log total employment. 

For the orthogonalization step of our estimation procedure, we model each of the right-hand-side variables with a factor structure as 
$$
x_{it,k} = \mu_{it,k} + e_{it,k},\quad k=1,2,3.  
$$
where $\mu_{it,k} = \alpha_{ik} + l_{i,k}' w_{t,k}$. We treat the
$\alpha_{ik}$ as additive fixed effect and allow for the presence of additional factors by including $l_{i,k}'  w_{t,k}$. That is, we effectively enforce that there is at least one factor in every variable - a factor that is always one. The remaining factor structure can then pick up any other sources of strong correlation and is expected to pick up the general trend that all of the variables we consider tend to increase over the sample period.    
Using the eigenvalue-ratio method of \cite{ahn2013eigenvalue}, we estimate $\dim( w_{t,k})=3$ for each regressor $x_{it,k}$. 
We then partial out the estimated $\mu_{it,k}$ to estimate the heterogeneous effects $\theta_{it}$. To estimate  $   \mu_{it,k}$, we first subtract the cross-sectional mean from $x_{it,k}$ and then extract the first three principal components of the demeaned data matrix $x_{it,k} - \frac{1}{T} \sum_{t=1}^T x_{it,k}$.  

We choose the tuning parameters for the nuclear-norm regularized estimation using an iterative method based on Section \ref{s:choo}. Specifically, we first estimate $\Var(u_{it})$ using a homogeneous model 
$y_{it} = \sum_{k=1}^3 x_{it,k} \theta_k +\Var(u_{it})^{-1/2}v_{it}$ where $\Var(v_{it})=1$. We then set the tuning parameter for estimating the matrix of the  $\theta_{it}$ as
$$
\nu_1=2.2\bar Q(\|X_1\odot Z\|; 0.95),
$$
where $Z$ is an $N\times T$ matrix whose elements are generated from a $\mathcal N(0, \Var(u_{it}))$ and $X_1$ is the observed $N\times T$ matrix of $x_{it,1}$. The tuning parameters for the matrix of the $\beta_{it}$ and the matrix of the $\gamma_{it}$ are set similarly.  Let $(\widetilde\theta_{it},  \widetilde\beta_{it}, \widetilde\gamma_{it} ,\widetilde\alpha,\widetilde g_t)$ denote the nuclear-norm regularized estimators obtained with this initial set of tuning parameters. We re-estimate $\Var(u_{it})$ by  $\frac{1}{NT}\sum_{i,t} \widetilde u_{it}^2,$ where
$$
\widetilde u_{it}=y_{it}-x_{it,1}\widetilde \theta_{it}-x_{it,2}\widetilde\beta_{it}-x_{it,3}\widetilde\gamma_{it}- \widetilde\alpha_i'\widetilde g_t.
$$
We then update the values of the tuning parameters by re-simulating using the updated estimate of $\Var(u_{it})$. We iterate this procedure until it converges. We note that upon convergence, the estimated number of factor for each low-rank matrix in this example is two. I.e. we estimate $\dim(f_{t,1}) = \dim(f_{t,2}) = \dim(f_{t,3}) = \dim(g_t) = 2$.

\subsection{Results}

For each state, we estimate the minimum wage effect over all quarters. To summarize the results, we group the states into two equal-size groups according to their average minimum wage over the sample period: high minimum wage states (High) and low minimum wage states (Low). 

Figure \ref{fig5} plots the averaged effects $\bar\theta_{\mathcal G,t}$ for $\mathcal G \in\{\text{Low, High}\}$ across time. The most striking feature of the graph is that the estimated average minimum wage effect within the low minimum wage group is much more volatile than the estimated average effect within the high minimum wage states. This difference is especially pronounced in the early part of the sample period where we see extreme swings between estimated effects. In the latter part of the sample period, the estimated average effects in both the low and high minimum wage groups appear relatively stable and concentrated on relatively small values. We also see that estimated effects in the low-minimum wage group are associated with much larger standard errors than effects estimated in the high-minimum wage group. It is interesting, that broadly speaking, the estimated effects seem to follow the same pattern in both groups with estimated positive and negative effects in both groups tending to occur in the same time periods.

\begin{figure}[htbp]
	
	\begin{center}

		\includegraphics[width=10cm]{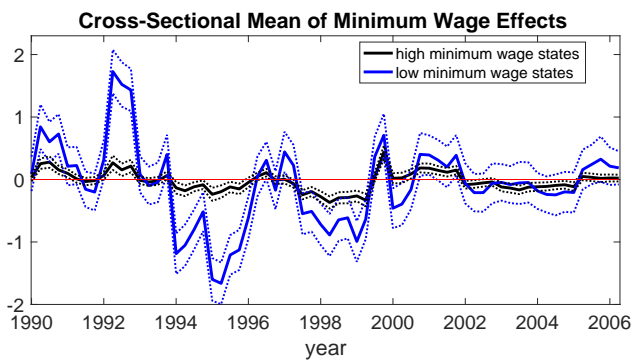}
		\caption{\footnotesize This figure plots the time series of cross-sectional averages of estimated state-level minimum wage effects with groups. The solid black line presents the estimated average effect within the high-minimum wage group (``high MW states''), and the solid blue curve presents the estimated average effect with the low-minimum wage group (``low MW states''). Dashed black and blue provide pointwise 95\% confidence intervals for the high- and low-minimum wage groups respectively.}
	\end{center}
	\label{fig5}
	
\end{figure}

To further aid interpretability, we present numeric summaries of results obtained by breaking the sample period into three time spans in Table \ref{t6.1}. Specifically, we consider the three periods defined as the first quarter of 1990 through the last quarter of 1993, the first quarter of 1994 through the last quarter of 1999, and the first quarter of 2000 through the second quarter of 2006.  Within each window for each of the high- and low-minimum group, we report several quantities. As a simple summary of the estimated effect sizes, we report the time series average of the time specific within group means in column $\widehat{\theta}_{\mathcal{G}}$  - i.e. we report $\frac{1}{|S|_0}\sum_{t \in S} \widehat{\theta}_{\mathcal{G},t}$ where $\widehat{\theta}_{\mathcal{G},t}=\frac{1}{|\mathcal G|_0}\sum_{i\in\mathcal G}\widehat\theta_{it}$
and   $S$ denotes one of the three time windows.
 To summarize statistical significance of the estimated average effects, we compute the fraction of time periods within each time window for which the estimate of the group-level average effect is negative and the corresponding 95\% pointwise confidence interval excludes zero and the fraction of observations for which the estimate of the group level average effect is positive and the corresponding 95\% pointwise confidence interval excludes zero broken out by group. Specifically, for each $\mathcal G \in\{\text{Low, High}\}$ and for $S$ denoting each of the time windows, we compute
\begin{eqnarray*}
	P_{\mathcal G}^+&=&\frac{1}{ |S|_0}\sum_{t\in S}1\{\widehat\theta_{\mathcal G,t} > 0\} 1\{|\widehat\theta_{\mathcal G,t}/s.e.(\widehat\theta_{\mathcal G,t})| > 1.96 \} \cr
	P_{\mathcal G}^-&=&\frac{1}{ |S|_0}\sum_{t\in S}1\{\widehat\theta_{\mathcal G,t} < 0\} 1\{|\widehat\theta_{\mathcal G,t}/s.e.(\widehat\theta_{\mathcal G,t})| > 1.96 \} 
\end{eqnarray*}   
where $s.e.(\widehat\theta_{\mathcal G,t})$ denotes the estimated standard error of $\widehat\theta_{\mathcal G,t}$.
Similarly, we summarize significance of the effects at the individual level by reporting the fraction of $(i,t)$ combinations within each window where the estimated $\theta_{it}$ is positive and significant and the fraction of $(i,t)$ combinations where the estimated $\theta_{it}$ is negative and significant:
 \begin{eqnarray*}
 	Q_{\mathcal G}^+&=&\frac{1}{|\mathcal G|_0}\sum_{i\in \mathcal G}\frac{1}{ |S|_0}\sum_{t\in S}1\{\widehat\theta_{it} > 0\} 1\{|\widehat\theta_{it}/s.e.(\widehat\theta_{it})| > 1.96 \} \cr
 	Q_{\mathcal G}^-&=&\frac{1}{|\mathcal G|_0}\sum_{i\in\mathcal G}\frac{1}{ |S|_0}\sum_{t\in S}1\{\widehat\theta_{it} < 0\} 1\{|\widehat\theta_{it}/s.e.(\widehat\theta_{it})| > 1.96 \} .
 \end{eqnarray*}

\begin{table}[htp]
	
	\begin{center}
		\caption{\footnotesize Summary of Estimated Minimum Wage Effects within Groups and Time Windows}
		\begin{tabular}{c|ccc|ccc|ccc}
			\hline
			\hline
			&  \multicolumn{3}{c|}{1990 through 1993 } &  \multicolumn{3}{c|}{1994 through 1999 }&\multicolumn{3}{|c}{2000 through 2006 } \\
				&&&&&&\\
			&   $P_{\mathcal G}^+$ & $P_{\mathcal G}^-$ & $\widehat{\theta}_{\mathcal{G}}$ &    $P_{\mathcal G}^+$ & $P_{\mathcal G}^-$  & $\widehat{\theta}_{\mathcal{G}}$ & $P_{\mathcal G}^+$ & $P_{\mathcal G}^-$  & $\widehat{\theta}_{\mathcal{G}}$  \\
			High min-wage  & 0.500& 0& 0.091&  0.083 &  0.708 & -0.123& 0.231&   0.423   &   -0.010\\
			Low min-wage  &  0.500& 0&0.479& 0.166&  0.708& -0.558   &0.115& 0.077&-0.002\\
			&&&&&&\\
	&   $Q_{\mathcal G}^+$ & $Q_{\mathcal G}^-$ &   &    $Q_{\mathcal G}^+$ & $Q_{\mathcal G}^-$  &  & $Q_{\mathcal G}^+$ & $Q_{\mathcal G}^-$  &    \\
High min-wage  & 0.118& 0.031& &  0.075&  0.202 &  & 0.075&   0.088&  \\
Low min-wage  &  0.327& 0.031& & 0.112&  0.453&    &0.123& 0.126& \\
			\hline
\multicolumn{10}{p{14.5cm}}{\footnotesize Note: This table reports summaries of group level average effects within time windows. $P_{\mathcal G}^+$ and $P_{\mathcal G}^-$ respectively denote the fraction of time periods within the specified window where the estimated group-level average effect is positive and statistically significant at the 5\% level and the fraction of time periods within the specified window where the estimated group-level average effect is negative and statistically significant at the 5\% level. $\widehat{\theta}_{\mathcal{G}}$ denotes the within time window average of the group-level average effects. $Q_{\mathcal G}^+$ and $Q_{\mathcal G}^-$ respectively denote the fraction of (state, time) combinations within the specified window where the estimated $\theta_{it}$ is positive and statistically significant at the 5\% level and where the estimated $\theta_{it}$ is negative and statistically significant at the 5\% level.}
		\end{tabular}
		
	\end{center}
	\label{t6.1}
\end{table}%

Looking at Table \ref{t6.1}, we see that the estimated average effect in both groups is positive and statistically significant in  50\% of the time periods and never negative and statistically significant. The overall averages are also positive over this window, with the average in the high-minimum wage group suggesting a relatively small positive effect of minimum wage increases on employment and the average in the low-minimum wage group suggesting a large positive effect of minimum wages on employment. These patterns essentially reverse in the second period where both groups have estimated average effects that are negative and statistically significant in the majority of time periods, though the estimated effects are positive and significant in some time periods. The overall averages are also negative and of nearly the same magnitude as the estimates in the first period, suggesting large disemployment effects associated to minimum wage increases in the low-minimum wage group and much more modest disemployment effects in the high-minimum wage group. 

Relative to the first two time blocks, the results in the last time block are much more stable and overall suggest modest effects of minimum wage changes on employment. The point estimates of average effects over this time period tend to be negative, while the average effect of the low minimum wage group has a noticeably larger magnitude than that of the high minimum wage group. 
 Over this window, the estimated effects in the low-minimum wage group tend to be small relative to standard errors, leading to a relatively small fraction of statistically significant estimated effects.  This pattern differs in the high-minimum wage group where the relatively more precise estimates lead to approximately sixty percent of the effects being statistically significant with a somewhat higher fraction being negative and significant than positive and significant. {Overall, these patterns are consistent with there being substantial heterogeneity in the coefficient on the minimum wage, with the minimum wage having both statistically significant large positive and large negative effects on youth employment depending on the time period.}

 {As a point of reference, we also obtained an estimate of the minimum wage effect using the usual homogeneous two-way fixed effects model:
$$
y_{it} = x_{it,1}\theta + x_{it,2}\beta + x_{it,3}\gamma + \alpha_i + g_t + u_{it}.
$$
Within our data, the estimate of $\theta$ is -0.1128 with estimated standard error of 0.0583, where the standard error is obtained using two-way clustering by state and time period. These results are negative and insignificant and thus the conclusions one would draw differ sharply from those one would draw using the richer heterogeneous effects model where we find substantial evidence for significant effects of the minimum wage on employment.} 

 {We wish to reiterate that the heterogeneous model we consider is substantially more flexible than the homogeneous two-way fixed effects in that it accommodates both a richer structure over unobserved heterogeneity through the interactive fixed effects and a richer structure over slopes which are allowed to vary over time and in the cross-section.  Within this richer structure, there are then two leading reasons for the apparent improvement in precision, in the sense of finding many statistically significant effects, relative to the homogeneous model. First, even if the homogeneous model were consistent for estimating the average effect of the minimum wage, it is of course completely possible that heterogeneity is such that the average effect is near zero while effects in different states and time periods are far from zero. The small and insignificant homogeneous effect could then be masking what are important heterogeneous effects at the state or time period level. Second, under the heterogeneous effect model, the homogeneous effects model is misspecified and the heterogeneity in slopes and the difference between the interactive and additive fixed effects structure will not be captured by the model. This misspecification will then manifest as unexplained variation and potentially large standard errors for the estimated homogeneous coefficients. Overall, we believe that the estimates obtained from the two models provide substantial evidence that the homogeneous effects model is indeed inadequate for modeling the impacts of the minimum wage.}

 {As a final caution, we note that interpreting our reported estimates as causal still requires the assumption that there are no sources of confounding leading to association between the minimum wage and the unobservable $u_{it}$. The potentially implausibly large estimated minimum wage effects and the large swings in these estimated effects may indeed be evidence for the presence of potential sources of endogeneity. Importantly, any such sources of confounding would also contaminate homogeneous effects models using the same set of variables that we use, as our estimated model nests these models.}

 {Finally, it should be noted that the flexible approach we take imposes insufficient structure for extrapolating estimated effects to learn objects such as the effect of the minimum wage next period on employment next period. Adding structure to allow answering such questions may be practically useful and would also potentially add further regularizing restrictions that could improve estimation performance. Exploring the impact of adding such additional structure may be interesting for future research.}

\bibliographystyle{ims}
\bibliography{liaoBib_newest}

\pagebreak

\appendix
   
\section{Estimation Algorithm for the Multivariate Case}\label{app: multivariate}
   
Consider  \begin{eqnarray*}
y_{it}&=&\sum_{r=1}^Rx_{it,r}\theta_{it,r}+ \alpha_i'g_t+u_{it} ,\quad i=1,..., N,\quad t=1,..., T. \cr 
\theta_{it,r}&=& \lambda_{i,r}'f_{t,r}.
\end{eqnarray*}
   where $(x_{it,1},...,x_{it,R})'$ is an $R$-dimensional vector of covariate. Each coefficient $\theta_{it,r}$ admits a factor structure with $\lambda_{i,r}$ and $f_{t,r}$ as the ``loadings" and ``factors"; the factors and loadings are $\theta_{it,r}$ specific, but are allowed to have overlap. Here $(R,\dim(\lambda_{i,1}),...,\dim(\lambda_{i,R}))$ are all assumed fixed.

Suppose $x_{it,r}= \mu_{it,r}+e_{it,r}$. For instance, $\mu_{it,r}=l_{i,r}'w_{t,r}$ also admits a factor structure.  Then after partialing out $\mu_{it,r}$,  the model can also written as:
$$
\dot y_{it}= \sum_{r=1}^Re_{it,r}\lambda_{i,r}'f_{t,r}+ \alpha_i'g_t+u_{it},\quad \dot y_{it}=y_{it}- \sum_{r=1}^R\mu_{it,r}\theta_{it,r}
$$
As such, it is  straightforward to extend   the estimation algorithm   to the multivariate case.  Let $X_r$ be the $N\times T$ matrix of $x_{it,r}$,  
Let $\Theta_r$ be the $N\times T$ matrix of $\theta_{it,r}$. We first estimate these low rank matrices using penalized nuclear-norm regression. We then  apply sample splitting, and employ steps 2-4 to iteratively estimate $(f_{t,r}, \lambda_{i,r})$.  The formal algorithm is stated as follows.

 \begin{algo}\label{ala.1} Estimate $\theta_{it,r}$ as follows. 

\textit{Step 1. Estimate the number of factors.} 
Run nuclear-norm penalized regression:
$$
(\widetilde M, \widetilde\Theta_r):=\arg\min_{M,\Theta_r}\|Y-M-\sum_{r=1}^R X_r\odot \Theta_r\|_F^2+\nu_0\|M\|_n+\sum_{r=1}^R\nu_r\|\Theta\|_n.
$$

Estimate $K_r=\dim(\lambda_{i,r}), K_0=\dim(\alpha_i)$ by
   $$
  \widehat K_r=\sum_{i} 1\{\psi_i(\widetilde \Theta_r)\geq (\nu_r  \|\widetilde \Theta_r\|)^{1/2}\},\quad  \widehat K_0=\sum_{i} 1\{\psi_i(\widetilde M)\geq  (\nu_0  \|\widetilde M\|)^{1/2}\}.
  $$

    \textit{Step 2. Estimate the structure $x_{it,r}=\mu_{it,r}+e_{it,r}$.} 
  
  In the many mean model, let $\widehat e_{it,r}=  x_{it,r}-\frac{1}{T}\sum_{t=1}^Tx_{it,r}$.  In the factor model, use the PC estimator to obtain  $(\widehat{l_{i,r}'w_{t,r}},\widehat e_{it,r})$ for all $i=1,..., N, t=1,.., T$ and $r=1,..., R.$

\textit{Step 3: Sample splitting.} 
Randomly split the sample into $\{1,..., T\}/\{t\}=I\cup I^c$, so that $|I|_0= [(T-1)/2]$.    Denote by    $Y_I$, $X_{I,r}$ as the  $N\times |I|_0$ matrices of $(y_{is}, x_{is,r})$ for  observations  at  $s\in I$.  
Estimate the low-rank matrices  $\Theta $ and $M $ as in step 1,  with $(Y, X_{r})$ replaced with  $(Y_I, X_{I,r})$, and obtain $(\widetilde M_I, \widetilde\Theta_{I,r})$.

Let $\widetilde\Lambda_{I,r}=(\widetilde\lambda_{1,r},...,\widetilde\lambda_{N,r})'$ be the $N\times \widehat K_r$ matrix, whose columns are defined as $\sqrt{N}$ times  the first $\widehat K_r$ eigenvectors of $\widetilde \Theta_{I,r}\widetilde \Theta_{I,r}'$.  Let $\widetilde A_{I}=(\widetilde\alpha_{1},...,\widetilde\alpha_{N})'$ be the $N\times \widehat K_0$ matrix, whose columns are defined as $\sqrt{N}$ times  the first $\widehat K_0$ eigenvectors of $\widetilde  M_I\widetilde  M_I'$.

    \textit{Step 4. Estimate the ``partial-out" components.}

Substitute in    $(\widetilde\alpha_i, \widetilde\lambda_{i,r})$,  and define   $$
(  \widetilde f_{s,r},\widetilde g_s):= \arg\min_{f_{s,r}, g_s}\sum_{i=1}^N(y_{is}- \widetilde\alpha_i'g_s- \sum_{r=1}^Rx_{is,r} \widetilde \lambda_{i,r}'f_{s,r})^2,\quad s\in I^c\cup\{t\}.
  $$
  and
       $$
(\dot\lambda_{i,r}, \dot\alpha_i)= \arg \min_{\lambda_{i,r}, \alpha_i}\sum_{s\in I^c\cup\{t\}} (y_{is} - \alpha_i'  \widetilde g_s- \sum_{r=1}^Rx_{is,r} \lambda_{i,r}'\widetilde f_{s,r})^2,\quad i=1,..., N.
  $$

    \textit{Step 5. Estimate $(f_{t,r},\lambda_{i,r})$ for inferences.} 
   
  Motivated by (\ref{e3.5}), 
for all $s\in I^c\cup\{t\}$,  let 
  $$
(\widehat f_{I,s,r},\widehat g_{I,s}):= \arg\min_{f_{s,r}, g_s}\sum_{i=1}^N(\widehat y_{is}- \widetilde\alpha_i'g_s- \sum_{r=1}^R\widehat e_{is,r} \widetilde \lambda_{i,r}'f_{s,r})^2.
  $$
 Fix $i\leq N$, 
  $$
(\widehat\lambda_{I,i,r},\widehat\alpha_{I,i})=\arg \min_{\lambda_{i,r}, \alpha_i}\sum_{s\in I^c\cup\{t\}} (\widehat y_{is}-   \alpha_i'\widehat g_{I,s}-\sum_{r=1}^R \widehat e_{is,r}  \lambda_{i,r}'\widehat f_{I,s,r})^2.
  $$
  where $\widehat y_{is}=y_{is}-\bar x_{i,r} \dot\lambda_i'\widetilde f_s$ and $\widehat e_{is,r}=x_{is,r}-\bar x_{i,r}$ in the many mean model, and $\widehat y_{is}= y_{is}- \widehat{l_{i,r}'w_{s,r}}\dot\lambda_{i,r}'\widetilde f_{s,r}$, $\widehat e_{is,r}=x_{is,r}- \widehat{l_{i,r}'w_{s,r}}$ in  the factor model.
  
      \textit{Step 6. Estimate   $\theta_{it,r}$. } 
  Repeat steps  3-5 with $I$ and $I^c$ exchanged, and obtain  $(\widehat{\lambda}_{I^c,i,r}, \widehat{f}_{I^c,s,r}: s\in I\cup \{t\}, i\leq N, r\leq R)$. 
Define
   $$
\widehat\theta_{it,r}:= \frac{1}{2}[\widehat\lambda_{I,i,r}'\widehat f_{I,t,r}+ \widehat{\lambda}_{I^c,i,r}'\widehat f_{I^c,t,r}]. $$

\end{algo}

The asymptotic variance can be estimated by $ \widehat v_{\lambda,r}+  \widehat v_{f,r} $, where 
\begin{eqnarray*}
 \widehat v_{\lambda,r} &=&\frac{1}{2N}( \widehat \lambda_{I,r,\mathcal G}' \widehat V_{\lambda,1}^{I-1} \widehat V_{\lambda,2}^I  \widehat V_{\lambda,1}^{I-1} \widehat \lambda_{I,r,\mathcal G} +   \widehat \lambda_{I^c,r,\mathcal G}' \widehat V_{\lambda,1}^{I^c-1} \widehat V_{\lambda,2}^{I^c}  \widehat V_{\lambda,1}^{I^c-1} \widehat \lambda_{I^c,r,\mathcal G} )  \cr
  \widehat v_{f} &=&\frac{1}{2T|\mathcal G|_0}( \widehat f_{I,t,r}' \widehat V_{I, f}  \widehat f_{I,t,r} +   \widehat f_{I^c,t,r}' \widehat V_{I^c, f}   \widehat f_{I^c,t,r} ) \cr
   \widehat V_{S, f} &=& \frac{1}{|\mathcal G|_0|S|_0}\sum_{s\notin S}\sum_{i\in\mathcal G} \widehat\Omega_{S,i,r}\widehat f_{S,s,r}\widehat f_{S,s,r}' \widehat\Omega_{S,i,r}  \widehat e_{is,r}^2\widehat u_{is}^2    \cr
 \widehat V_{\lambda,1}^S&=&\frac{1}{N}\sum_j\widehat  \lambda_{S,r,j}\widehat \lambda_{S,r, j}'\widehat e_{jt,r}^2 \cr
 \widehat V_{\lambda,2}^S&=&\frac{1}{N}\sum_j \widehat  \lambda_{S,r,j}\widehat \lambda_{S,r, j}'\widehat e_{jt,r}^2  {\widehat u_{jt}^2} \cr
    \widehat \lambda_{S,r,\mathcal G} &=&  \frac{1}{|\mathcal G|_0}\sum_{i\in\mathcal G}\widehat \lambda_{S,r,i},\quad \widehat\Omega_{S,i,r}= (\frac{1}{|S|_0} \sum_{s\in S}\widehat f_{S,s,r}\widehat f_{S,s,r}')^{-1} (\frac{1}{T}\sum_{s=1}^T\widehat e_{is,r}^2)^{-1}
 \end{eqnarray*}

It is also straightforward to extend the univariate asymptotic analysis to the multivariate case, and establish the asymptotic normality for $\widehat\theta_{it,r}$.   The proof techniques are the same, subjected to more complicated notation. Therefore our proofs below focus on the univariate case.

\section{Proof of Proposition \ref{p2.1}}

\begin{proof}

Recall that $\Theta_{k+1}= S_{\tau\nu_1/2}(\Theta_k-\tau A_k)$, where   $$
A_k=   X\odot(X\odot\Theta_k -Y+M_k).
$$
By Lemma \ref{la.2},  set $\Theta= \widetilde \Theta$ and $M=\widetilde M$,  and replace $k$ with subscript $m$, 
$$
F( \widehat \Theta, \widehat  M) - F(\Theta_{m+1}, M_{m+1})\geq \frac{1}{\tau } \left(\|\Theta_{m+1}-  \widehat  \Theta\|_F^2 -   \|\Theta_{m}- \widehat   \Theta\|_F^2 \right).
$$
Let $m=1,..., k$, and sum these inequalities up,    since $F(\Theta_{m+1}, M_{m+1})\geq F(\Theta_{k+1}, M_{k+1})$ by Lemma \ref{la.1}, 
\begin{eqnarray*}
&&k F( \widehat \Theta, \widehat  M) -  kF(\Theta_{k+1}, M_{k+1})\geq 
k F( \widehat \Theta, \widehat  M) - \sum_{m=1}^{k}F(\Theta_{m+1}, M_{m+1})\cr
&\geq&  \frac{1}{\tau } \left(\|\Theta_{k+1}-  \widehat  \Theta\|_F^2 -   \|\Theta_{1}- \widehat   \Theta\|_F^2 \right)\geq -\frac{1}{\tau } \|\Theta_{1}- \widehat   \Theta\|_F^2. 
\end{eqnarray*} 

\end{proof}

The above proof depends on  the following lemmas.
\begin{lem}\label{la.1} We have: (i)
 \begin{eqnarray*}
\Theta_{k+1}&=&\arg\min_{\Theta}  p(\Theta, \Theta_k,  M_k) +  \nu_1\|\Theta\|_n,  \cr
p(\Theta, \Theta_k,  M_k)&:=&   \tau^{-1}\left\|
\Theta_k-\Theta
\right\| ^2_F -2 \tr\left((\Theta_k-\Theta)'  A_k \right).
\end{eqnarray*}
(ii) For any  $\tau\in(0, 1/\max x_{it}^2)$, 
$$F(\Theta, M_k)\leq p(\Theta, \Theta_k, M_k)   +\nu_1\|\Theta\|_n+\nu_2\| M_k\|_n +  
 \|Y-M_k-X\odot\Theta_k\|_F^2  .
$$
(iii) 
$ F(\Theta_{k+1}, M_{k+1})\leq  F(\Theta_{k+1}, M_{k})\leq  F(\Theta_{k}, M_{k}).$

\end{lem}
\begin{proof}
(i)  We have
$   \|
\Theta_k-\tau A_k -\Theta\|_F^2
=\| \Theta_k -\Theta\|_F^2+\tau^2 \|A_k\|_F^2-2\tr[(\Theta_k-\Theta)' A_k]\tau.
$
So \begin{eqnarray*}
&&\arg\min_{\Theta} \|\Theta_k-\tau A_k -\Theta\|_F^2+\tau  \nu_1\|\Theta\|_n\cr
&=&\arg\min_{\Theta} \| \Theta_k -\Theta\|_F^2-2\tr[(\Theta_k-\Theta)' A_k]\tau+\tau  \nu_1\|\Theta\|_n\cr
&=&\arg\min_{\Theta}  \tau \cdot p(\Theta, \Theta_k, M_k)+\tau  \nu_1\|\Theta\|_n.
\end{eqnarray*}
On the other hand,  it is well known that $\Theta_{k+1}=S_{\tau\nu_1/2}(\Theta_k-\tau A_k)$ is the solution to the first problem in the above equalities \citep{ma2011fixed}.  This proves (i).



 (ii)
Note that   for $\Theta_k=(\theta_{k,it})$ and $\Theta=(\theta_{it})$,  and any $\tau ^{-1}>\max_{it}x_{it}^2$, we have
$$
 \|   X\odot(\Theta_k-\Theta)\|_F^2=\sum_{it} x_{it}^2(\theta_{k,it}-\theta_{it})^2
 <\tau^{-1} \|\Theta_k-\Theta\|_F^2.
$$
So 
\begin{eqnarray*}
 F(\Theta, M_k)&=&  \|Y-M_k- X\odot \Theta\|_F^2+\nu_2\|M_k\|_n+\nu_1\|\Theta\|_n
 \cr
&=&\|Y-M_k- X\odot \Theta_k\|_F^2 + \|   X\odot(\Theta_k-\Theta)\|_F^2
-2\tr[A_k' (\Theta_k-\Theta)]\cr
&&+\nu_2\|M_k\|_n+\nu_1\|\Theta\|_n\cr
&\leq &\|Y-M_k- X\odot \Theta_k\|_F^2 + \tau^{-1}\|   \Theta_k-\Theta\|_F^2
-2\tr[A_k' (\Theta_k-\Theta)]\cr
&&+\nu_2\|M_k\|_n+\nu_1\|\Theta\|_n\cr
&=&p(\Theta, \Theta_k,  M_k) + \|Y-M_k- X\odot \Theta_k\|_F^2 +\nu_2\|M_k\|_n+\nu_1\|\Theta\|_n.
\end{eqnarray*}

 (iii) By definition, $p(\Theta_{k}, \Theta_k, M_k) =0$. So
\begin{eqnarray*}
 F(\Theta_{k+1}, M_{k+1})&=&\|Y-X\odot \Theta_{k+1}- M_{k+1}\|_F^2 +\nu_1\|\Theta_{k+1}\|_n +\nu_2\|M_{k+1}\|_n\cr
 &\leq ^{\text{(a)}}&\|Y-X\odot \Theta_{k+1}- M_{k}\|_F^2 +\nu_1\|\Theta_{k+1}\|_n +\nu_2\|M_{k}\|_n\cr
 &=& F(\Theta_{k+1}, M_{k})\cr
 &\leq ^{\text{(b)}}&
 p(\Theta_{k+1}, \Theta_k, M_k)   +\nu_1\|\Theta_{k+1}\|_n+\nu_2\| M_k\|_n +  
 \|Y-M_k-X\odot\Theta_k\|_F^2\cr
 &\leq ^{\text{(c)}}& p(\Theta_{k}, \Theta_k, M_k)   +\nu_1\|\Theta_{k}\|_n+\nu_2\| M_k\|_n +  
 \|Y-M_k-X\odot\Theta_k\|_F^2\cr
 &=& F(\Theta_k, M_k).
 \end{eqnarray*}
(a) is  due to the definition of $M_{k+1}$; (b) is due to (ii); (c) is due to (i). 

\end{proof}

\begin{lem}\label{la.2} For any  $\tau\in(0, 1/\max x_{it}^2)$,  any $(\Theta, M)$ and any $k\geq 1$, 
$$
F( \Theta, M) - F(\Theta_{k+1}, M_{k+1})\geq \frac{1}{\tau } \left(\|\Theta_{k+1}- \Theta\|_F^2 -   \|\Theta_{k}- \Theta\|_F^2 \right).
$$
\end{lem}

\begin{proof} The proof is similar to that of Lemma 2.3 of \cite{beck2009fast}, with the extension that $M_{k+1}$ is updated after $\Theta_{k+1}$. The key difference  here  is that, while  an update to $M_{k+1}$   is added to the iteration, we show that the lower bound does not depend on $M_{k+1}$ or $ M_k$.  Therefore, the convergence property of the algorithm depends mainly on the step of updating $\Theta$.

Let $\partial \|A\|_n$ be an element that belongs to the subgradient of $\|A\|_n$. Note that $\partial \|A\|_n$ is convex in $A$. Also, $\|Y-X\odot\Theta- M\|_F^2$ is convex in $(\Theta, M)$, so   for any $\Theta, M$,  we have the  following three inequalities:
\begin{eqnarray*}
\|Y-X\odot\Theta- M\|_F^2&\geq&  \| Y- X\odot\Theta_k-M_k\|_F^2 \cr
&&-2\tr[(\Theta-\Theta_k)'    (X\odot(Y-X
\odot\Theta_k-M_k))]\cr
&&-2\tr[(M-M_k)'   (Y-X
\odot\Theta_k-M_k)]\cr
 \nu_1   \|\Theta\|_n   &\geq&   \nu_1 \|\Theta_{k+1}\|_n + \nu_1\tr[(\Theta-\Theta_{k+1})'  \partial \|\Theta_{k+1}\|_n ] \cr
  \nu_2    \| M \|_n  &\geq&    \nu_2 \| M_{k}\|_n +  \nu_2\tr[( M-M_{k})'  \partial \| M_{k}\|_n  ] .
\end{eqnarray*}
In addition, 
\begin{eqnarray*}
-F(\Theta_{k+1}, M_{k+1}) &\geq  &-F(\Theta_{k+1}, M_{k}) \cr
&\geq  & -p(\Theta_{k+1}, \Theta_k, M_k)   -\nu_1\|\Theta_{k+1}\|_n - \nu_2\| M_k\|_n -  
 \|Y-M_k-X\odot\Theta_k\|_F^2.
\end{eqnarray*}
where the  two inequalities are due to Lemma \ref{la.1}. Sum up  the above inequalities, 
\begin{eqnarray*}
&&F(\Theta, M)-F(\Theta_{k+1}, M_{k+1}) \geq (A)\cr
(A)&:=&  -2\tr[(\Theta-\Theta_k)'  (X\odot  (Y-X
\odot\Theta_k-M_k))]-2\tr[(M-M_k)'   (Y-X
\odot\Theta_k-M_k)]\cr
&&+  \nu_1\tr[(\Theta-\Theta_{k+1})'  \partial \|\Theta_{k+1}\|_n ] +  \nu_2\tr[( M-M_{k})'  \partial \| M_{k}\|_n  ] -p(\Theta_{k+1}, \Theta_k, M_k)   . 
\end{eqnarray*}
 We now simplify $(A)$.  Since $k\geq 1$, both $M_k$ and $\Theta_{k+1}$ should satisfy  the KKT condition. By Lemma  \ref{la.1}, they are:
\begin{eqnarray*}
0&=&\nu_1\partial \|\Theta_{k+1}\|_n -\tau^{-1}2(\Theta_k-\Theta_{k+1})+2A_k\cr
0&=&\nu_2\partial \|M_{k}\|_n - 2(Y-X\odot \Theta_{k}-M_k). 
\end{eqnarray*}
Plug in, we have
\begin{eqnarray*}
(A)&=& 
 \tau^{-1} 2 \tr[(\Theta-\Theta_{k+1})' (\Theta_k-\Theta_{k+1}) ]  - \tau^{-1}\left\|
\Theta_k-\Theta_{k+1}
\right\| ^2_F  \cr
&=&\frac{1}{\tau } \left(\|\Theta_{k+1}- \Theta\|_F^2 -   \|\Theta_{k}- \Theta\|_F^2 \right).
\end{eqnarray*} 

\end{proof}

\section{Proof of Proposition \ref{p3.1}}

\subsection{Level of the Score}

   \begin{lem}
    In the presence of serial correlations in $x_{it}$ (Assumption \ref{a4.1}), $\|X\odot U\|=O_P(\sqrt{N+T})=\|U\|$.  In addition, the chosen $\nu_2, \nu_1$ in Section \ref{s:choo}  are $ O_P(\sqrt{N+T})$.
    \end{lem}
    
 \begin{proof}

  The assumption that $\Omega_{NT}$ contains independent sub-Gaussian columns ensures that, by the  eigenvalue-concentration inequality for sub-Gaussian random vectors  (Theorem 5.39 of \cite{vershynin2010introduction}):
  $$
  \|\Omega_{NT}\Omega_{NT}' -  \E\Omega_{NT}\Omega_{NT}'\|=O_P(\sqrt{  NT} +N).
  $$ In addition, let $w_i$ be the $T\times 1$ vector of $\{x_{it}u_{it}: t\leq T\}$. We have, for each $(i,j,t,s)$,
  $$
  \E(w_iw_j')_{s,t}=\E x_{it}x_{js}u_{it}u_{js} = \begin{cases}
  \E x_{it}^2u_{it}^2, & i=j, t=s\cr
0, & \text{otherwise}
\end{cases}
  $$
  due to the conditional cross-sectional and serial independence  in $u_{it}$.
  Then 
  for the $(i,j)$'th entry of $\E\Omega_{NT}\Omega_{NT}'$,  
\begin{eqnarray*}
(\E\Omega_{NT}\Omega_{NT}' )_{i,j}&=&(\E(X\odot U)\Sigma_T^{-1}(X\odot U)' )_{i,j}\cr
&=&
\E w_i'\Sigma_T^{-1}w_j =tr(\Sigma_T^{-1} \E w_jw_i')
= \begin{cases}
\sum_{t=1}^T(\Sigma_T^{-1})_{tt} \E x_{it}^2u_{it}^2, & i=j\cr
0, & i\neq j.
\end{cases}
\end{eqnarray*} Hence $\|\E\Omega_{NT}\Omega_{NT}' \|\leq O(T)$. This implies $\| \Omega_{NT}\Omega_{NT}' \|\leq O(T+N)$.
Hence 
$\|X\odot U\|\leq \|\Omega_{NT}\|\|\Sigma_T^{1/2}\|\leq O_P(\max\{\sqrt{N},\sqrt{T}\})$.  The rate for $\|U\|$ follows from the same argument. 
 The second claim that $\nu_2, \nu_1$ satisfy the same rate constraint follows from the same argument, by replacing $U$ with $Z$, and Assumption  \ref{a4.1} is still satisfied  by $Z$ and $X\odot Z$.  
 
 \end{proof}
 
 \subsection{Useful Claims }

The proof of Proposition \ref{p3.1} uses some claims that are proved in the following lemma. 
Let us first recall the notations. Define $U_2D_2 V_2'=\Theta_I^0$ and $U_1 D_1 V_1'=M_I^0$
  as the singular value decompositions  of the true values $\Theta_I^0$ and $M_I^0$.    Further decompose, for $j=1,2,$
   $$
   U_j= (U_{j, r}, U_{j,c}),\quad V_j= (V_{j, r}, V_{j,c})
   $$
   Here $(U_{j,r}, V_{j,r})$
   corresponds to the nonzero singular values, while  $(U_{j,c}, V_{j,c})$
   corresponds to the zero singular values. In addition, for any $N\times T/2$ matrix $\Delta$, let 
   $$
   \mathcal P_{j} (\Delta )= U_{j,c} U_{j, c}' \Delta V_{j,c} V_{j,c}',\quad \mathcal M_{j}(\Delta)=\Delta -\mathcal P_{j}(\Delta). 
   $$
   Here $U_{j,c} U_{j, c}' $ and $V_{j,c} V_{j, c}' $ respectively are the projection matrices onto the columns of $U_{j,c}$ and $V_{j,c}$.  Therefore, $\mathcal M_1(\cdot)$ and $\mathcal M_2(\cdot)$ can be considered as the projection matrices onto the ``low-rank" spaces of $\Theta_I^0$ and $M_I^0$ respectively, and $\mathcal P_1(\cdot)$ and $\mathcal P_2(\cdot)$ ar projections onto  their orthogonal spaces.

\begin{lem}
[claims]  Same results below also apply  to $\mathcal P_2$   $\Theta_I^0$, and samples on $I^c$.  For any matrix $\Delta$, 

(i) $ \|\mathcal P_{1} (\Delta )+M_I^0\|_n=\|\mathcal P_{1} (\Delta )\|_n+\|M_I^0\|_n$.

(ii) $\| \Delta\|_F^2= \|\mathcal M_1(\Delta)\|_F^2+\|\mathcal P_1(\Delta)\|_F^2$

(iii)  $\rank(\mathcal M_1(\Delta))\leq 2K_1$, where $K_1=\rank(M_I^0)$.

(iv) $\|\Delta\|_F^2=\sum_j\sigma_j^2$ and $\|\Delta\|_n^2\leq\|\Delta\|_F^2\rank(\Delta)$, with $\sigma_j$ as the singular values of $\Delta$. 

(v) For any $\Delta_1,\Delta_2$,   $|\tr(\Delta_1\Delta_2)|\leq \|\Delta_1\|_n\|\Delta_2\|$,  Here $\|.\|$ denotes the operator norm.  
\end{lem}

\proof (i) Note that $M_I^0=U_{1,r}D_{1,r}V_{1,r}'$ where $D_{1,r}$ are the subdiagonal matrix of nonzero singular values.  The claim follows from Lemma 2.3 of \cite{recht2010guaranteed}. 

(ii)  Write 
$$
U_1'\Delta V_1=\begin{pmatrix}
A& B\\
C& U_{1,c}' \Delta V_{1,c}
\end{pmatrix}
=\begin{pmatrix}
A& B\\
C& 0
\end{pmatrix} +\begin{pmatrix}
0& 0\\
0& U_{1,c}' \Delta V_{1,c}
\end{pmatrix}:= H_2+H_1.
$$
Then $\mathcal P_1(\Delta)= U_1H_1V_1'$ and $\mathcal M_1(\Delta)= U_1H_2V_1'$. So
$$
\|\mathcal P_1(\Delta)\|_F^2=\tr( U_1H_1V_1'V_1H_1'U_1')= \tr(H_1H_1')=\|H_1\|_F^2.
$$
Similarly, $\|\mathcal M_1(\Delta)\|_F^2= \|H_2\|_F^2$. So 
$$
\|H_2\|_F^2+\|H_1\|_F^2=\|U_1'\Delta V_1\|_F^2=\|\Delta\|_F^2. 
$$

(iii)  This is  Lemma 1 of \cite{negahban2011estimation}. 

(iv) The first is a basic equality, and the second follows from the Cauchy-Schwarz inequality. 

(v)  Let $UDV'=\Delta_1$ be the singuar value decomposition of $\Delta_1$, then 
$$
|\tr(\Delta_1\Delta_2)| =|\sum_iD_{ ii} (V'\Delta_2 U)_{ii}|
\leq \max_i|(V'\Delta_2 U)_{ii}|\sum_iD_{ii} \leq \|\Delta_1\|_n\|\Delta_2\|.
$$
 
\subsection{Proof of Proposition \ref{p3.1}: convergence of $\widetilde \Theta_S, \widetilde M_S$}

In the proof below, we set $S= I$. That is, we consider estimation using only the data with $t\in I$. We set $T_0=|I|_0$. The proof carries over to $S= I^c$ or $S=I\cup I^c$. We suppress the subscript $S$ for notational simplicity.
 Let $\Delta_1=\widetilde  M-M$ and $\Delta_2= \widetilde \Theta-\Theta$. Then
$$
\|Y-\widetilde M-X\odot\widetilde \Theta\|_F^2= \|\Delta_1+X\odot\Delta_2\|_F^2+\|U\|_F^2- 2 \tr[U'(\Delta_1+X\odot\Delta_2)].
$$
Note that $\tr(U'(X\odot\Delta_2))=\tr(\Delta_2'(X\odot U))$. Thus by claim (v),
\begin{eqnarray*}
 |2 \tr[U'(\Delta_1+X\odot\Delta_2)]|&\leq& 2\|U\|\|\Delta_1\|_n +2\|X\odot U\|\|\Delta_2\|_n\cr
 &\leq& (1-c)\nu_2\|\Delta_1\|_n+ (1-c)\nu_1\|\Delta_2\|_n.
\end{eqnarray*}
Thus $\|Y-\widetilde M-X\odot\widetilde \Theta\|_F^2\leq \|Y-  M-X\odot\Theta\|_F^2$ (evalued at the true parameters) implies 
\begin{eqnarray*}
\|\Delta_1+X\odot\Delta_2\|_F^2  +\nu_2\|\widetilde M\|_n+\nu_1\|\widetilde\Theta\|_n&\leq& (1-c)\nu_2\|\Delta_1\|_n+ (1-c)\nu_1\|\Delta_2\|_n\cr
&&+ \nu_2\|M\|_n+ \nu_1\|\Theta\|_n.
\end{eqnarray*}
Now \begin{eqnarray*}
\|\widetilde M\|_n&=&\|\Delta_1+ M\|_n=\|M+\mathcal P_1(\Delta_1)+\mathcal M_1(\Delta_1)\|_n\cr
&\geq& \|M+\mathcal P_1(\Delta_1)\|_n-\|\mathcal M_1(\Delta_1)\|_n\cr
&=& \|M\|_n+\|\mathcal P_1(\Delta_1)\|_n-\|\mathcal M_1(\Delta_1)\|_n,
\end{eqnarray*}
where the last equality follows from claim (i). Similar lower bound applies to $\|\widetilde\Theta\|_n$. Therefore, 
\begin{eqnarray}\label{eb.7}
&&\|\Delta_1+X\odot\Delta_2\|_F^2  +c\nu_2  \|\mathcal P_1(\Delta_1)\|_n +c\nu_1  \|\mathcal P_2(\Delta_2)\|_n  \cr
&\leq&
(2-c)\nu_2\|\mathcal M_1(\Delta_1)\|_n
+(2-c)\nu_1\|\mathcal M_2(\Delta_2)\|_n.
\end{eqnarray}
In the case $U$ is Gaussian,  $\|U\|$ and $\|X\odot U\|\asymp \max\{\sqrt{N}, \sqrt{T}\}$, while in the more general case, set $\nu_2\asymp \nu_1\asymp\max\{\sqrt{N}, \sqrt{T}\}$. Thus the above inequality implies $(\Delta_1,\Delta_2)\in \mathcal C(a)$ for some $a>0$. Thus apply Assumption \ref{a3.1} and claims to (\ref{eb.7})   , for a generic $C>0$,
\begin{eqnarray*}
\|\Delta_1\|_F^2+\|\Delta_2\|_F^2
&\leq& C\nu_2\|\mathcal M_1(\Delta_1)\|_n
+C\nu_1\|\mathcal M_2(\Delta_2)\|_n\cr
&\leq^{\text{claim (iv)}}&  C\nu_2\|\mathcal M_1(\Delta_1)\|_F\sqrt{\rank(\mathcal{M}_2(\Delta_1))}
\cr
&&+C\nu_1\|\mathcal M_2(\Delta_2)\|_F\sqrt{\rank(\mathcal{M}_2(\Delta_2))}\cr
&\leq^{\text{claim (iii)}}& C\nu_2\|\mathcal M_1(\Delta_1)\|_F\sqrt{ 2K_1}
+C\nu_1\|\mathcal M_2(\Delta_2)\|_F\sqrt{2K_2}\cr
&\leq^{\text{claim (ii)}}& C\nu_2\|\Delta_1\|_F +C\nu_1\|\Delta_2\|_F\cr
&\leq& C\max\{\nu_2, \nu_1\} \sqrt{\|\Delta_1\|_F^2+\|\Delta_2\|_F^2}.
\end{eqnarray*}
Thus 
$
\|\Delta_1\|_F^2+\|\Delta_2\|_F^2\leq C(\nu_2^2+\nu_1^2).
$

 \subsection{Proof of Proposition \ref{p3.1}: convergence of $\widetilde\Lambda_S, \widetilde A_S$}

 We proceed the proof in  the following steps.

 \textbf{step 1: bound the eigenvalues}
 
Replace  $\nu_2^2+\nu_1^2$ with $O_P(N+T)$,  then 
$$
 \|\widetilde\Theta_S-\Theta_S\|_F^2 =O_P(N+T).
$$
Let $S_f=\frac{1}{T_0} \sum_{t\in I}f_tf_t' $, $\Sigma_f=  \frac{1}{T} \sum_{t=1}^T f_tf_t' $ and $S_{\Lambda} = \frac{1}{N}\Lambda'\Lambda$. 
Let $\psi_{I,1}^2\geq ...\geq \psi_{I, K_1}^2$ be the  $K_1$ nonzero eigenvalues of $\frac{1}{NT_0}\Theta_I\Theta_I'
=\frac{1}{N}\Lambda S_f\Lambda'
$.  Let $\widetilde  \psi^2_1\geq ...\geq \widetilde  \psi^2_{K_2}$ be the  first  $K_2$ nonzero singular values of $\frac{1}{NT_0}\widetilde \Theta_I\widetilde \Theta_I'$.  
Also, let $\psi_j^2 $ be the $j$ th largest eigenvalue of $\frac{1}{N}\Lambda \Sigma_f\Lambda'$. 
  Note that  $\psi_1^2...\psi_{K_1}^2$ are the same as the eigenvalues of $\Sigma_f^{1/2} S_\Lambda   \Sigma_f^{1/2}$.  Hence by Assumption \ref{a3.2}, there are constants $c_1,...,c_{K_1}>0$, so that 
$$
\psi_j^2=c_j,\quad j=1,..., K_1.
$$
Then by Weyl's  theorem, for $j=1,..., \min\{T_0,N\}$, with the assumption that $\|S_f-\Sigma_f\|=O_P(\frac{1}{\sqrt{T}})$,
$
| \psi_{I,j}^2-\psi_j^2|\leq   \frac{1}{N} \|\Lambda (S_f- \Sigma_f)\Lambda' \|
\leq O(1) \|S_f-\Sigma_f\| =O_P(\frac{1}{\sqrt{T}})
$. This also implies $\|\Theta_I\|=\psi_{I,1}\sqrt{NT_0}=\sqrt{(c_1+o_P(1))T_0N}$.

Still by Weyl's  theorem, for $j=1,..., \min\{T_0,N\}$,
\begin{eqnarray*}
|\widetilde\psi_j^2-\psi_{I,j}^2|&\leq& \frac{1}{NT_0} \|\widetilde\Theta_I\widetilde\Theta_I'-\Theta_I \Theta_I' \|
\cr
&\leq& \frac{2}{NT_0}\| \Theta_I\| \|\widetilde\Theta_I-\Theta_I\| + \frac{1}{NT_0} \|\widetilde\Theta_I-\Theta_I\|^2 
=O_P(\frac{1}{\sqrt{N}}+\frac{1}{\sqrt{T}}).\cr
\text{implying}\cr
 |\widetilde\psi_j^2-\psi_{j}^2|&=&O_P(\frac{1}{\sqrt{N}}+\frac{1}{\sqrt{T}}).
\end{eqnarray*}

Then for all $j\leq K_1$,  with probability approaching one, 
\begin{eqnarray}\label{ea.8}
| \psi_{j-1}^2    - \widetilde \psi_{j}^2|
&\geq&  | \psi_{j-1}^2    - \psi_{j}^2|
- | \psi_{j}^2-\widetilde \psi_{j}^2| \geq (c_{j-1}-c_j)  /2\cr
| \widetilde  \psi_{j}^2    -   \psi_{j+1}^2| 
& \geq& 
|  \psi_{j}^2    -   \psi_{j+1}^2| - | \widetilde  \psi_{j}^2    -   \psi_{j}^2|\geq (c_j-c_{j+1}) /2
\end{eqnarray}
 with $\psi_{K_1+1}^2=c_{K_1+1}=0$ because $\Theta_I\Theta_I'$ has at most $K_1$ nonzero eigenvalues.

 \textbf{step 2: characterize the  eigenvectors}
 
Next, we show that there is a $K_1\times K_1$ matrix $H_1$, so that the columns of $\frac{1}{\sqrt{N}}\Lambda H_1$ are the first  $K_1$ eigenvectors  of $\Lambda \Sigma_f\Lambda'$. 
Let  $L= S_\Lambda^{1/2} \Sigma_f S_\Lambda^{1/2}$. Let $R$ be  a $K_1\times K_1$ matrix whose columns are the eigenvectors of $L$. Then $D= R'LR$ is a diagonal  matrix of the eigenvalues of $L$ that are distinct nonzeros according to Assumption  \ref{a3.2}.  Let $H_1 = S_\Lambda^{-1/2} R$.     Then 
\begin{eqnarray*}
\frac{1}{N} \Lambda \Sigma_f\Lambda'\Lambda H_1 &=&     \Lambda S_\Lambda^{-1/2}   S_\Lambda^{1/2}  \Sigma_f S_\Lambda ^{1/2}S_\Lambda ^{1/2}H_1
= \Lambda S_\Lambda^{-1/2}  RR' L R\cr
&=&\Lambda H_1 D.
\end{eqnarray*}
Now $\frac{1}{N}(\Lambda H_1)' \Lambda H_1= H_1'S_\Lambda H_1= R'R=I$.
So the columns of $\Lambda H_1/\sqrt{N}$ are the eigenvectors of 
$\Lambda \Sigma_f\Lambda'$, corresponding to the eigenvalues in $D$.

  Importantly, the rotation matrix $H_1$, by definition, depends only on $S_\Lambda,  \Sigma_f$, which is time-invariant, and does not depend on the splitted sample.

 \textbf{step 3: prove the convergence}
 
 We first assume  $\widehat K_1=K_1$. The proof of the consistency is given in step 4 below. Once this is true, then the following argument can be carried out conditional on the event $\widehat K_1=K_1$. 
 Apply   Davis-Kahan sin-theta inequality, and by (\ref{ea.8}), 
\begin{eqnarray*}
&&\| \frac{1}{\sqrt{N}}\widetilde\Lambda_I- \frac{1}{\sqrt{N}}\Lambda H_1\|_F\leq\frac{\frac{1}{N} \|\Lambda \Sigma_f\Lambda'-\frac{1}{T_0}\widetilde\Theta_I\widetilde\Theta_I'\|}{\min_{j\leq K_2}\min\{  
| \psi_{j-1}^2    - \widetilde \psi_{j}^2|,   | \widetilde  \psi_{j}^2    -   \psi_{j+1}^2| \}}\cr
&\leq& O_P(1)\frac{1}{N} \|\Lambda \Sigma_f\Lambda'-\frac{1}{T_0}\widetilde\Theta_I\widetilde\Theta_I'\|\cr
&\leq& O_P(1)\frac{1}{N} \|\Lambda (\Sigma_f-S_f)\Lambda'\|
 +\frac{1}{NT_0} \|  \Theta_I\Theta_I'-\widetilde\Theta_I\widetilde\Theta_I'\| =O_P(\frac{1}{\sqrt{T}}+\frac{1}{\sqrt{N}}). 
\end{eqnarray*}

 \textbf{step 4: prove $P(\widehat K_1=K_1)=1$.}
 
 Note that $\psi_j(\widetilde\Theta)=\widetilde \psi_j\sqrt{NT}$. 
 By step 1,  for all $j\leq K_1$,  $\widetilde\psi_j^2\geq c_j- o_P(1)\geq c_j/2$ with probability approaching one. Also,
$\widetilde \psi_{K_1+1}^2\leq O_P(T^{-1/2}+N^{-1/2})$, implying that
$$
\min_{j\leq K_1}\psi_j(\widetilde\Theta)\geq c_{K_1}\sqrt{NT}/2,\quad \max_{j>K_1}\psi_j(\widetilde\Theta)\leq O_P(T^{-1/4}+N^{-1/4})\sqrt{NT}.
$$
In addition, $\nu_2^{1/2-\epsilon}\|\widetilde\Theta\|^{1/2+\epsilon}\asymp   (\sqrt{N+T})^{1/2-\epsilon}(\sqrt{NT})^{1/2+\epsilon}   $ for some small $\epsilon\in(0,1)$, 
$$
\min_{j\leq K_1}\psi_j(\widetilde\Theta)\geq \nu_2^{1/2-\epsilon}\|\widetilde\Theta\|^{1/2+\epsilon},\quad \max_{j>K_1}\psi_j(\widetilde\Theta)\leq o_P(1)\nu_2^{1/2-\epsilon}\|\widetilde\Theta\|^{1/2+\epsilon}.
$$
This proves the consistency of $\widehat K_1$.

 Finally, the proof of the convergence for  $\widetilde A_I$ and the consistency of $\widehat K_2$ follows from the same argument.  Q.E.D.

       \section{Proof of Theorem   \ref{t3.2}}

        In the factor model
 $$
 x_{it}= l_i'w_t+e_{it},
 $$
we write  $
\widehat e_{it}=x_{it}-\widehat{l_i'w_t}$, 
  $\widehat\mu_{it}= \widehat{l_i'w_t} $ and $\mu_{it}=l_i'w_t $.
   The proof in this section works for both models.

       Let
 $$
      C_{NT}= \min\{\sqrt{N}, \sqrt{T}\}.
      $$
      
    First  recall that   $   (  \widetilde f_s, \dot\lambda_i)$   are computed as the preliminary estimators in step 3.  The main technical requirement of these estimators  is that their estimation effects are negligible, specifically,  there is a rotation matrix $H_1$ that is independent of the sample splitting,  for each fixed $t\notin I$,   
    \begin{eqnarray*}
\|\frac{1}{N}\sum_{j}  (\dot\lambda_j-H_1' \lambda_j)e_{jt}\|^2&=&O_P( C_{NT}^{-4}) , \cr
  \frac{1}{|\mathcal G|_0}\sum_{i\in\mathcal G} \|\frac{1}{T}\sum_{s\notin I}  f_s  (\widetilde f_s-H_1^{-1}f_s )'\mu_{is}e_{is}\|^2  &=& O_P(C_{NT}^{-4}).
    \end{eqnarray*}
   These are given in Lemmas  \ref{ld.5} and \ref{lc.2}  for the factor model.
    
      \subsection{Behavior of $\widehat f_t$  }

Recall that  for each $t\notin I$,   $$
(\widehat f_{I,t},\widehat g_{I,t}):= \arg\min_{f_t, g_t}\sum_{i=1}^N(\widehat y_{it}- \widetilde\alpha_i'g_t- \widehat e_{it} \widetilde \lambda_i'f_t)^2.
  $$
 
  For notational simplicity, we simply write $\widehat f_t=\widehat f_{I,t}$ and $\widehat g_t=\widehat g_{I,t}$, but keep in mind that $\widetilde\alpha$ and $\widetilde\lambda$ are estimated through the low rank estimations on data $I$.  Note that $\widetilde \lambda_i$  consistently estimates $\lambda_i$ up to a rotation matrix $H_1'$, so $\widehat f_t$ is consistent for $H_1^{-1}f_t$. However, as we shall explain below, it is  difficult to establish the asymptotic normality for $\widehat f_t$ centered at $H_1^{-1}f_t$. Instead, we obtain a new centering quantity, and obtain an expansion for 
  $$
\sqrt{N}(\widehat f_t-  H_ff_t)
  $$
  with a new rotation matrix $H_f$ that is also independent of $t$. For the purpose of inference for $\theta_{it}$, this is sufficient.

Let $\widehat w_{it}=(\widetilde \lambda_i' \widehat e_{it}, \widetilde \alpha_i')'$, and $\widehat B_{t}= \frac{1}{N}\sum_i\widehat w_{it}\widehat w_{it}'$. Define $w_{it}=(\lambda_i'e_{it},\alpha_i')'$, and 
$$\widehat Q_t=  \frac{1}{N}\sum_i\widehat w_{it}( \mu_{it}\lambda_i'f_t - \widehat{\mu}_{it}\dot\lambda_i'\widetilde f_t +u_{it}).$$
 
 We have
  \begin{eqnarray}\label{ed.14}
\begin{pmatrix}
\widehat f_t\\
\widehat g_t
\end{pmatrix}&=&\widehat B_t^{-1} \frac{1}{N}\sum_i\widehat w_{it}(y_{it}- \widehat{\mu}_{it} \dot\lambda_i'\widetilde f_t)\cr
&=&\begin{pmatrix}
 H_1^{-1}f_t\\
 H_2^{-1}g_t
\end{pmatrix}+\widehat B_t^{-1}\widehat S_t\begin{pmatrix}
 H_1^{-1}f_t\\
 H_2^{-1}g_t
\end{pmatrix}+\widehat B_t^{-1}\widehat Q_t
  \end{eqnarray}
  where 
  $$
  \widehat S_t= \frac{1}{N}\sum_i
\begin{pmatrix}
 \widetilde \lambda_i  \widehat e_{it}( \lambda_i'H_1e_{it}-\widetilde\lambda_i'\widehat e_{it})& 
  \widetilde \lambda_i  \widehat e_{it} ( \alpha_i'H_2-\widetilde\alpha_i')\\
  \widetilde\alpha_i( \lambda_i'H_1e_{it}-\widetilde\lambda_i'\widehat e_{it}) & 
    \widetilde\alpha_i( \alpha_i'H_2-\widetilde\alpha_i')
\end{pmatrix}.
  $$

 Note that the ``upper block" of $ \widehat B_t^{-1} \widehat S_t$  is not first-order negligible. Essentially this is due to the fact that the moment condition 
 $$
\frac{\partial }{\partial\lambda_i} \E   \lambda_ie_{it}(\dot y_{it}- \alpha_i'g_t-e_{it}\lambda_i'f_t)\neq 0,
 $$
 so is not ``Neyman orthogonal" with respect to $\lambda_i$. 
 On the other hand,  we can get around such    difficulty.  
In Lemma \ref{ld.6} below, we show that   $\widehat B_t$ and $\widehat S_t$ both converge in probability  to  block diagonal matrices that are independent of $t$. So $g_t$ and $f_t$ are ``orthogonal", and 
 $$
 \widehat B_t^{-1} \widehat S_t= \begin{pmatrix}
 \bar H_3&0\\
 0& \bar H_4
 \end{pmatrix}+ o_P(N^{-1/2}).
 $$
  Define $H_f:= H_1^{-1}+\bar H_3H_1^{-1}.$ Then (\ref{ed.14}) implies that 
$$
\widehat f_t =H_ff_t+ \text{upper block of } \widehat B_t^{-1}\widehat Q_t.
$$
 Therefore $\widehat f_t$ converges to $f_t$ up to a  new  rotation matrix $H_f$, which equals $H_1^{-1}$ up to an $o_P(1)$ term $\bar H_3H_1^{-1}$.  While  the effect of $\bar H_3$  is not negligible, it is  ``absorbed" into the new rotation matrix.  As such, we are able to establish the asymptotic normality for $
\sqrt{N}( \widehat f_t -H_ff_t).
 $
 
 \begin{prop}\label{pc.1} For each fixed $t\notin I$, for both the (i) many mean model and  (ii) factor model, 
 we have 
 $$
   \widehat f_t-H_ff_t= H_f (\frac{1}{N}\sum_i  \lambda_i\lambda_i' \E e_{it}^2)^{-1}   \frac{1}{N}\sum_i \lambda_ie_{it}u_{it}+O_P(C_{NT}^{-2}).
 $$
 \end{prop}
 
\begin{proof} Define 
$$
B=\frac{1}{N}\sum_i\begin{pmatrix}
H_1'\lambda_i\lambda_i'H_1\E e_{it}^2 &0 \\
0& H_2'\alpha_i\alpha_i'H_2
\end{pmatrix},\quad S=\frac{1}{N}\sum_i
\begin{pmatrix}
 H_1'\lambda_i ( \lambda_i'H_1-  \widetilde\lambda_i'  )\E e_{it}^2&0\\
 0&    \widetilde\alpha_i( \alpha_i'H_2-\widetilde\alpha_i')
\end{pmatrix}.
$$
 Both
$B, S$ are independent of $t$ due to the stationarity of $e_{it}^2$. But $S$ depends on the sample splitting through $(\widetilde\lambda_i, \widetilde\alpha_i)$.

From (\ref{ed.14}),  
   \begin{eqnarray}\label{ed.24}
\begin{pmatrix}
\widehat f_t\\
\widehat g_t
\end{pmatrix}
&=& (B^{-1} S+\I)\begin{pmatrix}
 H_1^{-1}f_t\\
 H_2^{-1}g_t
\end{pmatrix}+ B^{-1}    \frac{1}{N}\sum_i 
\begin{pmatrix}
 H_1'\lambda_i e_{it}   \\
 H_2'\alpha_i
\end{pmatrix}
 u_{it} 
\cr
  &&+\sum_{d=1}^6A_{dt}, \text{ where} \cr
A_{1t}&=&   (\widehat B_t^{-1}\widehat S_t-  B^{-1} S )\begin{pmatrix}
 H_1^{-1}f_t\\
 H_2^{-1}g_t
\end{pmatrix}\cr
A_{2t}&=&
(\widehat B_t^{-1}- B^{-1})  \frac{1}{N}\sum_i 
\begin{pmatrix}
 H_1'\lambda_i e_{it}   \\
 H_2'\alpha_i
\end{pmatrix}
 u_{it} 
  \cr
A_{3t}&=& \widehat B_t^{-1}   \frac{1}{N}\sum_i 
\begin{pmatrix}
\widetilde\lambda_i\widehat e_{it}  - H_1'\lambda_i e_{it}   \\
\widetilde\alpha_i -H_2'\alpha_i
\end{pmatrix}
 u_{it} \cr
 A_{4t}&=& \widehat B_t^{-1}    \frac{1}{N}\sum_i 
\begin{pmatrix}
\widetilde\lambda_i\widehat e_{it}  - H_1'\lambda_i e_{it}   \\
\widetilde\alpha_i -H_2'\alpha_i
\end{pmatrix}( \mu_{it}\lambda_i'f_t - \widehat \mu_{it} \dot\lambda_i'\widetilde f_t  ) \cr
A_{5t}&=& (\widehat B_t^{-1} -B^{-1}) \frac{1}{N}\sum_i 
\begin{pmatrix}
 H_1'\lambda_i e_{it}   \\
 H_2'\alpha_i
\end{pmatrix}( \mu_{it}\lambda_i'f_t - \widehat \mu_{it} \dot\lambda_i'\widetilde f_t  ) \cr
A_{6t}&=& B ^{-1}  \frac{1}{N}\sum_i 
\begin{pmatrix}
 H_1'\lambda_i e_{it}   \\
 H_2'\alpha_i
\end{pmatrix}( \mu_{it}\lambda_i'f_t - \widehat \mu_{it} \dot\lambda_i'\widetilde f_t  ) .
  \end{eqnarray}
  Note that $B^{-1}S$ is a block-diagonal matrix, with the upper block being
  $$
\bar H_3:=H_1^{-1}(\frac{1}{N}\sum_i  \lambda_i\lambda_i' \E e_{it}^2)^{-1}  \frac{1}{N}\sum_i  \lambda_i ( \lambda_i'H_1-  \widetilde\lambda_i'  )\E e_{it}^2. 
  $$
  Define 
  $$
  H_f:=(\bar H_3+\I)H_1^{-1}.
  $$
Fixed $t\notin I$, in the factor model, in Lemma \ref{ld.9} we show that 
 $ \sum_{d=2}^5A_{dt}=O_P(C_{NT}^{-2})$, the ``upper block" of $A_{6t}$, $\frac{1}{N}\sum_i 
 \lambda_i e_{it} ( l_i'w_t\lambda_i'f_t - \widehat{l_i'w_t}\dot\lambda_i'\widetilde f_t  )=O_P(C_{NT}^{-2})$ and the upper block of $A_{1t}$ is $O_P(C_{NT}^{-2})$.
 Therefore,  
 $$
 \widehat f_t=  H_ff_t+ H_1^{-1} (\frac{1}{N}\sum_i  \lambda_i\lambda_i' \E e_{it}^2)^{-1}   \frac{1}{N}\sum_i \lambda_ie_{it}u_{it}+O_P(C_{NT}^{-2}).
$$
  Given that $\bar H_3=O_P(C_{NT}^{-1})$ we have $H_1^{-1}=H_f+O_P(C_{NT}^{-1})$. 
  By
   $ \frac{1}{N}\sum_i \lambda_ie_{it}u_{it}=O_P(N^{-1/2})$, 
  \begin{equation}
 \widehat f_t=  H_ff_t+ H_f (\frac{1}{N}\sum_i  \lambda_i\lambda_i' \E e_{it}^2)^{-1}   \frac{1}{N}\sum_i \lambda_ie_{it}u_{it}+O_P(C_{NT}^{-2}).
  \end{equation}
  
\end{proof}

 \subsection{Behavior of $\widehat \lambda_i$}
Recall that  fix $i\leq N$, 
  $$
(\widehat\lambda_{I,i},\widehat\alpha_{I,i})=\arg \min_{\lambda_i, \alpha_i}\sum_{s\notin I} (\widehat y_{is}-   \alpha_i'\widehat g_{I,s}- \widehat e_{is} \lambda_i'\widehat f_{I,s})^2.
  $$
  For notational simplicity, we simply write $\widehat \lambda_i=\widehat \lambda_{I,i}$ and $\widehat \alpha_i=\widehat \alpha_{I,i}$, but keep in mind that $\widetilde\alpha$ and $\widetilde\lambda$ are estimated through the low rank estimations on data $I$. Write
  $$
  T_0=|I^c|_0.
  $$
    \begin{eqnarray}\label{ec.33}
&&(    B^{-1}S+\I)
  \begin{pmatrix}
H_1^{-1}&0\\
0& H_2^{-1}
  \end{pmatrix}
:=\begin{pmatrix}
H_f&0\\
0& H_g 
  \end{pmatrix}\cr
&&  \widehat D_i= \frac{1}{T_0}\sum_{s\notin I}  \begin{pmatrix}
\widehat f_s\widehat f_s'\widehat e_{is}^2&\widehat f_s\widehat g_s'\widehat e_{is}\\
\widehat g_s\widehat f_s'\widehat e_{is}& \widehat g_s\widehat g_s'
  \end{pmatrix},\quad 
D_i=   \frac{1}{T_0}\sum_{s\notin I}  \begin{pmatrix}
H_ff_s  f_s'H_f' \E e_{is}^2&  0\\
0& H_gg_s g_s'H_g'
  \end{pmatrix}\cr
  \end{eqnarray}

  \begin{prop}\label{pc.2}  Fixed a group $\mathcal G\subset\{1,..., N\}$. Write
  	$
  	\Omega_{i}:= (\frac{1}{T}\sum_{s=1 }^T f_sf_s'\E e_{is}^2)^{-1}.
  	$
$$\frac{1}{|\mathcal G|_0}\sum_{i\in\mathcal G}  (\widehat\lambda_{I,i}-H_f^{'-1}\lambda_{i})=  H_f^{'-1} \frac{1}{|\mathcal G|_0T_0}\sum_{i\in\mathcal G} \sum_{s\notin I} \Omega_{i} f_se_{is}u_{is}+
  O_P(C_{NT}^{-2})+o_P(\frac{1}{\sqrt{T|\mathcal G|_0}})$$  
  \end{prop}
  
\begin{proof} By definition and (\ref{ed.24}),    \begin{eqnarray}\label{ec.32}
  \begin{pmatrix}
  \widehat\lambda_i\\
  \widehat\alpha_i
  \end{pmatrix}
  &=& \widehat D_i^{-1}\frac{1}{T_0}\sum_{s\notin I}
   \begin{pmatrix}
\widehat f_s \widehat e_{is} \\
\widehat g_s 
  \end{pmatrix}(y_{is}-\widehat \mu_{is} \dot\lambda_i'\widetilde f_s)\cr
  &=&    \begin{pmatrix}
H_f^{'-1}\lambda_i\\
H_g^{'-1}\alpha_i
  \end{pmatrix}+  D_i^{-1}\frac{1}{T_0}\sum_{s\notin I}
   \begin{pmatrix}
H_ff_se_{is}  \\
H_gg_s
  \end{pmatrix}u_{is} 
  + \sum_{d=1}^6R_{di},\quad \text{ where},
  \cr
  \cr
  R_{1i}&=&  (\widehat D_i^{-1}-D_i^{-1})\frac{1}{T_0}\sum_{s\notin I}
   \begin{pmatrix}
H_ff_se_{is}  \\
H_gg_s
  \end{pmatrix}u_{is} \cr
    R_{2i}&=& 
  \widehat D_i^{-1}\frac{1}{T_0}\sum_{s\notin I}
   \begin{pmatrix}
 H_ff_s(\widehat e_{is} -e_{is})u_{is} \\
0
  \end{pmatrix}\cr
  R_{3i}&=& (\widehat D_i^{-1}-D_i^{-1})\frac{1}{T_0}\sum_{s\notin I}
   \begin{pmatrix}
\widehat f_s \widehat e_{is} \\
\widehat g_s 
  \end{pmatrix}(\mu_{is}\lambda_i'f_s- \widehat \mu_{is}\dot\lambda_i'\widetilde f_s  )\cr
R_{4i}&=&  D_i^{-1}\frac{1}{T_0}\sum_{s\notin I}
\begin{pmatrix}
	\widehat f_s \widehat e_{is} \\
	\widehat g_s 
\end{pmatrix}(\mu_{is}\lambda_i'f_s- \widehat \mu_{is}\dot\lambda_i'\widetilde f_s  )\cr
  R_{5i}&=&\widehat D_i^{-1}\frac{1}{T_0}\sum_{s\notin I}
\begin{pmatrix}
\widehat e_{is}\I&0  \\
0&\I
  \end{pmatrix}     \begin{pmatrix}
 \widehat f_s  -H_ff_s  \\
\widehat g_s -H_gg_s
  \end{pmatrix}    u_{is} \cr
  R_{6i}&=&\widehat D_i^{-1}\frac{1}{T_0}\sum_{s\notin I}
   \begin{pmatrix}
\widehat f_s \widehat e_{is} \\
\widehat g_s 
  \end{pmatrix}( \lambda_i'H_f^{-1}, \alpha_i'H_g^{-1})\begin{pmatrix}
\widehat e_{is}\I&0  \\
0&\I
  \end{pmatrix}     \begin{pmatrix}
 \widehat f_s  -H_ff_s  \\
\widehat g_s -H_gg_s
  \end{pmatrix}   
  \end{eqnarray}
  
   Let  $r_{di}$ be the upper blocks of $R_{di}$.     
 Lemma \ref{lc.8add} shows
$\frac{1}{|\mathcal G|_0}\sum_{i\in\mathcal G}\|\widehat D_i^{-1}-D_i^{-1}\|^2=O_P(C_{NT}^{-2})$. So  
 $\frac{1}{|\mathcal G|_0}\sum_{i\in\mathcal G}\|R_{1i}\|=O_P(C_{NT}^{-2})$.
 Lemma  \ref{ld.2a} shows \\
 $\frac{1}{|\mathcal G|_0}\sum_{i\in\mathcal G}\| \frac{1}{T}\sum_{s\notin I} f_s(\widehat e_{is}-e_{is})u_{is}  \|^2=O_P(C_{NT}^{-4})$, so  $\frac{1}{|\mathcal G|_0}\sum_{i\in\mathcal G}\|R_{2i}\|=O_P(C_{NT}^{-2})$.
 By Lemma \ref{ld.13},
 	$$
   	\frac{1}{|\mathcal G|_0}\sum_{i\in\mathcal G}\|\frac{1}{T}\sum_{s\notin I}\widehat f_s\widehat e_{is} ( \mu_{it}\lambda_i'f_s-   \widehat\mu_{it}\dot\lambda_i'\widetilde f_s  )\|^2 =O_P(C_{NT}^{-4})$$
	$$ 	\frac{1}{|\mathcal G|_0}\sum_{i\in\mathcal G}\|\frac{1}{T}\sum_{s\notin I}\widehat g_s (\mu_{it}   \lambda_i'f_s- \widehat{\mu}_{it}\dot\lambda_i'\widetilde f_s  )\|^2 =O_P(C_{NT}^{-2}).
   	$$
	So $\frac{1}{|\mathcal G|_0}\sum_{i\in\mathcal G}\|R_{3i}\|=O_P(C_{NT}^{-2})= \frac{1}{|\mathcal G|_0}\sum_{i\in\mathcal G}\|r_{4i}\|$.  Next,  by Lemma \ref{ld.14},   $\frac{1}{|\mathcal G|_0}\sum_{i\in\mathcal G}\|r_{5i}\|=O_P(C_{NT}^{-2}).$ By Lemma \ref{ld.15}, $\frac{1}{|\mathcal G|_0}\sum_{i\in\mathcal G}\|r_{6i}\|=O_P(C_{NT}^{-2}).$  Therefore,
	$$
\sum_{d=1}^6	\frac{1}{|\mathcal G|_0}\sum_{i\in\mathcal G}\|r_{di}\|=O_P(C_{NT}^{-2}).
	$$

Now write  $S_{f, I}=\frac{1}{T_0}\sum_{s\notin I} f_sf_s' $ and $S_{f}=\frac{1}{T}\sum_{s=1}^T f_sf_s' $, 
\begin{eqnarray*}
 \Omega_{I,i}&:=& (\frac{1}{T_0}\sum_{s\notin I} f_sf_s'\E e_{is}^2)^{-1} =S_{f,I}^{-1} (\E e_{is}^2)^{-1} \cr
 \Omega_{i}&:=& (\frac{1}{T_0}\sum_{s=1}^T f_sf_s'\E e_{is}^2)^{-1}=S_{f}^{-1} (\E e_{is}^2)^{-1}.
\end{eqnarray*}
Then by the assumption  $ S_{f,I}=S_f+o_P(1)$.
  	It remains to bound, for $\omega_{is}= e_{is}u_{is}$,
\begin{eqnarray*}
 &&\|\frac{1}{|\mathcal G|_0}\sum_{i\in\mathcal G} [ \Omega_{I,i}- \Omega_i]\frac{1}{T_0} \sum_{s\notin I} f_s\omega_{is}\|^2
 = \|\frac{1}{|\mathcal G|_0}\sum_{i\in\mathcal G}   \Omega_{I,i}( S_{f,I}-S_f ) \Omega_i\frac{1}{T_0} \sum_{s\notin I} f_s\omega_{is} \E e_{is}^2\|^2\cr
 &\leq&\|S_{f,I}-S_f \|^2 \|S_f^{-1}\|\|S_{f,I}^{-1}\| \|\frac{1}{|\mathcal G|_0}\sum_{i\in\mathcal G}   \frac{1}{T_0} \sum_{s\notin I} f_s\omega_{is} (\E e_{is}^2)^{-1}\|^2_F \leq o_P(T^{-1}|\mathcal G|_0^{-1}). 
\end{eqnarray*}
	
	This finishes the proof. 
      	
\end{proof}
 
  \subsection{Proof of normality of $\widehat\theta_{it}$}
Fix $t\notin I$.   Suppose $T$ is odd and $|I|_0=|I^c|_0=(T-1)/2$.  Let 
  \begin{eqnarray*}
  Q_I&:=& \frac{1}{|\mathcal G|_0|I|_0}\sum_{i\in\mathcal G} \sum_{s\notin I} \Omega_{i} f_se_{is}u_{is}=O_P( C_{NT}^{-1} |\mathcal G|_0^{-1/2}),\cr
   Q_{I^c}&:=& \frac{1}{|\mathcal G|_0|I|_0}\sum_{i\in\mathcal G} \sum_{s\notin I^c} \Omega_{i} f_se_{is}u_{is}=O_P( C_{NT}^{-1} |\mathcal G|_0^{-1/2}),\cr
    Q&:=& \frac{1}{|\mathcal G|_0T}\sum_{i\in\mathcal G} \sum_{s=1}^T \Omega_{i} f_se_{is}u_{is},\cr
  J&:=& (\frac{1}{N}\sum_i  \lambda_i\lambda_i' \E e_{it}^2)^{-1}   \frac{1}{N}\sum_i \lambda_ie_{it}u_{it}=O_P(C_{NT}^{-1})\cr
  \bar\lambda_{\mathcal G}&:=& \frac{1}{|\mathcal G|_0}\sum_{i\in\mathcal G} \lambda_i.
  \end{eqnarray*}
    Also $Q_I=O_P(\frac{1}{\sqrt{T|\mathcal G|_0}})=Q_{I^c}$  follows from $\max_i\sum_j|\Cov(u_{is}e_{is},  u_{js}e_{js}|F)|<C$, and 
    $\max_i\|\Omega_i\|
    \leq \lambda^{-1}_{\min}( \frac{1}{T}\sum_{s } f_sf_s') (\min_i\E e_{is}^2)^{-1}<C.
    $
 By Propositions \ref{pc.1} \ref{pc.2},  $\widehat f_t-H_ff_t= H_f J+O_P(C_{NT}^{-2})$ and  $
  \frac{1}{|\mathcal G|_0}\sum_{i\in\mathcal G}  (\widehat\lambda_{I,i}-H_f^{'-1}\lambda_i)= H_f^{'-1} Q_I+
  O_P(C_{NT}^{-2}). $
Therefore, 
 $$
    \frac{1}{|\mathcal G|_0}\sum_{i\in\mathcal G} (\widehat\lambda_{I,i}'\widehat f_{I,t}-\lambda_i'f_t)=f_t' Q_I   +  \bar\lambda_{\mathcal G}' J+O_P(C_{NT}^{-2}).
 $$
Exchanging $I$ with $I^c$, we have, for $t\notin I^c$,  
 \begin{eqnarray*}
  \frac{1}{|\mathcal G|_0}\sum_{i\in\mathcal G} ( \widehat\lambda_{I^c,i}'\widehat f_{I^c,t}-\lambda_i'f_t)&=&f_t' Q_{I^c}  +  \bar\lambda_{\mathcal G}' J+O_P(C_{NT}^{-2}).
        \end{eqnarray*}
      Note that the fixed $t\notin I\cup I^c$, so take the average:
   \begin{eqnarray}
 \frac{1}{|\mathcal G|_0}\sum_{i\in\mathcal G}( \widehat\theta_{it}-\theta_{it})&=&f_t' \frac{ Q_I+Q_{I^c}}{2}
  +  \bar\lambda_{\mathcal G}' J+O_P(C_{NT}^{-2})\cr
  &=& f_t' \frac{1}{|\mathcal G|_0T} \sum_{i\in\mathcal G} \bigg{[} \sum_{s \in I^c\cup\{t\}} \Omega_{i} f_se_{is}u_{is}
  + \sum_{s\in I\cup\{t\}} \Omega_{ i} f_se_{is}u_{is}\bigg{]} +  \bar\lambda_{\mathcal G}' J\cr
  &&+O_P(C_{NT}^{-2})+o_P(\frac{1}{\sqrt{T|\mathcal G|_0}})\cr
  &=&f_t' Q+\bar\lambda_{\mathcal G}' J+O_P(C_{NT}^{-2})+o_P(\frac{1}{\sqrt{T|\mathcal G|_0}})\cr
  &=& \frac{1}{\sqrt{T|\mathcal G|_0}}\zeta_{NT} +\frac{1}{\sqrt{N}}       \xi_{NT}  
  +O_P(C_{NT}^{-2})+o_P(\frac{1}{\sqrt{T|\mathcal G|_0}}) .
  \end{eqnarray}
      where 
\begin{eqnarray}
      \zeta_{NT}&=&f_t' \frac{1}{\sqrt{|\mathcal G|_0T}}\sum_{i\in\mathcal G} \sum_{s=1}^T \Omega_{i} f_se_{is}u_{is},\cr
       \xi_{NT}&=&\bar\lambda_{\mathcal G}'V_{\lambda1}^{-1}   \frac{1}{\sqrt{N}}\sum_i \lambda_ie_{it}u_{it}\cr
       V_{\lambda1}&=&\frac{1}{{N}}\sum_i  \lambda_i\lambda_i' \E e_{it}^2,\quad V_{\lambda2}=\Var(  \frac{1}{\sqrt{N}}\sum_i \lambda_ie_{it}u_{it}|F)\cr
     V_{\lambda}&=& V_{\lambda1}^{-1}   V_{\lambda2}   V_{\lambda1}^{-1},\quad 
     V_f=\frac{1}{T}\sum_{s=1}^T \Var(\frac{1}{\sqrt{|\mathcal G|_0}}\sum_{i\in\mathcal G} \Omega_{i} f_se_{is}u_{is}|F),\cr
 \Sigma_{NT}&=&\frac{1}{T|\mathcal G|_0}f_t'V_ff_t+\frac{1}{N}\bar\lambda_{\mathcal G}'V_\lambda  \bar\lambda_{\mathcal G} .
\end{eqnarray}
Next,  
$
 \Cov(\xi_{NT}, \zeta_{NT}|F)=\frac{1}{\sqrt{|\mathcal G|_0TN}}
\sum_{i\in\mathcal G}  \sum_j \bar\lambda_{\mathcal G}'V_{\lambda1}^{-1} \lambda_j   f_t' \Omega_{i} f_t\E (\omega_{it}  \omega_{jt} |F)
 =o_P(1).
$
We can then  use the same argument as  the proof of Theorem 3 in \cite{bai03} to claim that 
$$
\frac{ \frac{1}{\sqrt{T|\mathcal G|_0}}\zeta_{NT} +\frac{1}{\sqrt{N}}       \xi_{NT}  
}{\Sigma_{NT}^{1/2}}\to^d \mathcal N(0,1).
$$ 
Also, when either $|\mathcal G|_0=o(N)$ or $N=o(T^2)$ holds,  
 $$
   \frac{O_P(C_{NT}^{-2})+ o_P(\frac{1}{\sqrt{T|\mathcal G|_0}}) }{\Sigma_{NT}^{1/2}} =o_P(1).
 $$ 
Therefore,  $\Sigma_{NT}^{-1/2} \frac{1}{|\mathcal G|_0}\sum_{i\in\mathcal G}( \widehat\theta_{it}-\theta_{it})\to^d\mathcal N(0,1).$

   \subsection{Two special cases }
  
  We now consider two special cases. 
  
  \textit{Case I.} fix an  $i\leq N$ and set $\mathcal G=\{i\}$. In this case we are interested in inference about individual $\theta_{it}$. We have 
  $$
  V_{f,i}=  \Omega_i \E(\frac{1}{T} \sum_{s=1}^T f_s f_s' e_{is}^2u_{is}^2|F) \Omega_i,\quad  \Sigma_{NT,i}:= \frac{1}{T }f_t'V_{f,i}f_t+\frac{1}{N} \lambda_{i}'V_\lambda   \lambda_{i}.
  $$
  So $\Sigma_{NT,i}^{-1/2} ( \widehat\theta_{it}-\theta_{it})\to^d\mathcal N(0,1).$
 
    \textit{Case II.}  $\mathcal G=\{1,..., N\}$.  Then $|\mathcal G|_0=N.$ Set
    $
    \bar\lambda=\frac{1}{N}\sum_i\lambda_i.
    $ We have
    $$
\sqrt{N} ( \bar\lambda 'V_\lambda  \bar\lambda)^{-1/2} \frac{1}{N}\sum_{i=1}^N( \widehat\theta_{it}-\theta_{it})\to^d\mathcal N(0,1).
    $$
    
       \subsection{Covariance estimations} Suppose $C_{NT}^{-1}\max_{is}e_{is}^2u_{is}^2=o_P(1)$. 
    We now consider estimating $\Sigma_{NT}$ in the general case, with the assumption that $e_{it}u_{it}$ are cross sectionally independent given $F$. Then
    $$
    V_{\lambda2}= \frac{1}{N}\sum_i \lambda_i\lambda_i'\E( e_{it}^2u_{it}^2|F) ,\quad
  V_f=\frac{1}{T|\mathcal G|_0}\sum_{s=1}^T \sum_{i\in\mathcal G}  \Omega_{i}f_sf_s'\Omega_{i}\E( e_{is}^2u_{is}^2|F)
 $$
 The main goal is to consistently estimate the above two quantities.  For notational simplicity, we assume $\dim(f_t)= \dim(\alpha_i)=1$.
 
\subsubsection{Consistency for $ V_{\lambda }$}
 
 We have $\widehat V_{\lambda2, I}:=  \frac{1}{N}\sum_i \widehat \lambda_{I, i}\widehat \lambda_{I, i}' \widehat e_{it}^2\widehat u_{it, I}^2$, where $\widehat u_{it, I}=y_{it}- x_{it}\widehat\lambda_{I,i}'\widehat f_{I,t}-\widehat\alpha_{I,i}'\widehat g_{I,t}$.
 
 We remove the subscript ``$I$" for notational simplicity. 
 
 Step 1. Show $\|\widehat V_{\lambda2, I}- \widetilde V_{\lambda2}\|=o_P(1)$
 where $ \widetilde V_{\lambda2}=H_f^{'-1} \frac{1}{N}\sum_i   \lambda_{ i} \lambda_{i}'  e_{it}^2 u_{it}^2 H_f^{-1}$.
 
By Lemma \ref{ld.15}  $ \max_i[\|\widehat\lambda_i-H_f^{'-1}\lambda_i\|+ \|\widehat\alpha_i-H_g^{'-1}\alpha_i\|]=o_P(1)$.   By Lemma \ref{ld.17add}, 
 \begin{eqnarray*}
\frac{1}{N}\sum_i|\widehat u_{it} -u_{it}|^2(1+e_{it}^2) &\leq& \max_{it}|\widehat u_{it}-u_{it}|^2O_P(1)=o_P(1). 
  \end{eqnarray*}
  So we have
 \begin{eqnarray*}
\| \widehat V_{\lambda2 }- \widetilde V_{\lambda2}\|&\leq &
 \frac{1}{N}\sum_i (\widehat \lambda_{  i} -\lambda_{i}) ^2\widehat e_{it}^2\widehat u_{it }^2
 + \frac{4}{N}\sum_i |(\widehat \lambda_{  i} -\lambda_{i})  \lambda_{i} |  (\widehat u_{it }-u_{it})^2\max_i\widehat e_{it}^2\cr
 &&+ \frac{4}{N}\sum_i (\widehat \lambda_{  i} -\lambda_{i})  \lambda_{i}\widehat e_{it}^2 u_{it }^2
 + \frac{2}{N}\sum_i \lambda_{i}^2(\widehat e_{it}^2-e_{it}^2)(\widehat u_{it }-u_{it})^2
 \cr
 &&+\frac{2}{N}\sum_i \lambda_{i}^2(\widehat e_{it}^2-e_{it}^2)  u_{it }^2+ \frac{2}{N}\sum_i \lambda_{i}^2 e_{it}^2(\widehat u_{it }- u_{it})u_{it}
 + \frac{1}{N}\sum_i \lambda_{i}^2 e_{it}^2(\widehat u_{it }- u_{it})^2\cr
 &\leq &
 o_P(1).
 \end{eqnarray*}
 
  Step 2. Show $\| \widetilde V_{\lambda2}-H_f^{'-1}V_{\lambda2}H_f^{-1}\|=o_P(1)$. 
  
  It suffices to show 
   $\Var( \frac{1}{N}\sum_i \lambda_i^2e_{it}^2u_{it}^2|F)\to 0$ almost surely.   Almost surely in $F$,
  $$
  \Var( \frac{1}{N}\sum_i \lambda_i^2e_{it}^2u_{it}^2|F) 
\leq    \frac{\max_i\|\lambda_i\|^2}{N}\max_{i\leq N}\E(e_{it}^4u_{it}^4|F) \to0
  $$
 given that $\max_{i\leq N}\E(e_{it}^4u_{it}^4|F) <C$ almost surely.
 
 The consistency for $ V_{\lambda1}$ follows from the same argument. 
 
\subsubsection{Consistency for $ V_{f}$} Write $\widehat\Omega_{I,i}= (\frac{1}{T_0}\sum_{s\notin I}\widehat f_{I,s}\widehat f_{I,s}')^{-1}\widehat\sigma_i^{-2}$. \\
$\widehat V_{f, I}=\frac{1}{T_0|\mathcal G|_0}\sum_{s\notin I} \sum_{i\in\mathcal G}  \widehat \Omega_{I, i}\widehat f_{I,s} \widehat f_{I,s}'\widehat\Omega_{I,i} \widehat  e_{is}^2\widehat  u_{is, I}^2 $. We aim to show the consistency of $\widehat V_{f, I}.$
   We remove the subscript ``$I$" for notational simplicity. Then simply write 
\begin{eqnarray*}
\widehat V_{f}&=&\widehat S_f^{-1}\frac{1}{T|\mathcal G|_0}\sum_{s} \sum_{i\in\mathcal G}  \widehat f_{s} \widehat f_{s}' \widehat  e_{is}^2\widehat  u_{is}^2 \widehat\sigma_i^{-4} \widehat S_f^{-1}\cr
\widetilde V_{f}&=&H_f^{'-1}  S_f^{-1}\frac{1}{T|\mathcal G|_0}\sum_{s} \sum_{i\in\mathcal G}    f_{s} f_{s}' e_{is}^2 u_{is}^2  \sigma_i^{-4} S_f^{-1}H_f^{-1}  
\end{eqnarray*} where 
 $\sigma_i^2=\E e_{it}^2$. 
 
  Step 1. Show $\|\widehat V_{f}- \widetilde V_{f}\|=o_P(1)$. First note that 
  $$
  \widehat S_f- H_fS_fH_f'=\frac{1}{T}\sum_t\widehat f_t\widehat f_t' -H_ff_tf_t'H_f'=o_P(1).
  $$
We show $\|\frac{1}{T|\mathcal G|_0}\sum_{s} \sum_{i\in\mathcal G}  \widehat f_{s} \widehat f_{s}' \widehat  e_{is}^2\widehat  u_{is}^2 \widehat\sigma_i^{-4} -H_f\frac{1}{T|\mathcal G|_0}\sum_{s} \sum_{i\in\mathcal G}    f_{s} f_{s}' e_{is}^2 u_{is}^2  \sigma_i^{-4} H_f'\|^2=o_P(1) $. Ignoring $H_f$, it is bounded by 
\begin{eqnarray*}
	&&O_P(C_{NT}^{-2})  \max_s\|\frac{1}{|\mathcal G|_0} \sum_{i\in\mathcal G}   \widehat  e_{is}^2\widehat  u_{is}^2 \widehat\sigma_i^{-4} \|^2
	+O_P(C_{NT}^{-2})  \frac{1}{T|\mathcal G|_0}\sum_{s} \sum_{i\in\mathcal G}  \|  f_{s}   f_{s}' \widehat  u_{is}^2 \widehat\sigma_i^{-4}\|^2\cr
	&&+ \frac{1}{T|\mathcal G|_0}\sum_{s} \sum_{i\in\mathcal G}  (\widehat  u_{is} -u_{is} )^2   \frac{1}{T|\mathcal G|_0}\sum_{s} \sum_{i\in\mathcal G}  \|  f_{s}  f_{s}'   e_{is}^2(\widehat  u_{is} +u_{is} ) \widehat\sigma_i^{-4}\|^2 + O_P(1) \max_i|\widehat\sigma_i^4-\sigma_i^4|^2
\end{eqnarray*}
By Lemma \ref{ld.15add}, 
$\max_{is}|\widehat u_{is}-u_{is}|=o_P(1)$, $\max_i|\widehat \sigma_i^2-\sigma_i^2|=o_P(1)$, $C_{NT}^{-1}\max_{is}e_{is}^2u_{is}^2=o_P(1)$. So the above is $o_P(1)$. 

Step 2. Show $\|\widetilde V_f-H_f^{'-1}V_f H_f^{-1}\|=o_P(1)$.

It suffices to prove 
$
\Var(\frac{1}{T|\mathcal G|_0}\sum_{s} \sum_{i\in\mathcal G}    f_{s}  ^2 e_{is}^2 u_{is}^2  \sigma_i^{-4} |F)\to 0
$ almost surely.  It is in fact bounded by 
  $$
 \frac{1}{T^2|\mathcal G|_0^2}\sum_{s} \sum_{i\in\mathcal G}  f_{s}  ^4 \E(   e_{is}^4 u_{is}^4    |F)\max_i\sigma_i^{-8}\to 0
  $$
  given that $\frac{1}{T}\sum_t\|f_t\|^4<C$ almost surely.

\section{Technical lemmas  }

 \subsection{The effect of $\widehat e_{it}-e_{it}$ in the factor model}
 
Here we present the intermediate results when $x_{it}$ admits a factor structure:
  $$
  x_{it}=l_i'w_t+e_{it}.
  $$

   Let $\widehat w_t$ be the PC estimator of $w_t$.   Then $\widehat l_i'=\frac{1}{T}\sum_s x_{is}\widehat w_s'$ and $ \widehat e_{it}-e_{it}=  l_i'H_x(\widehat w_t-H_x^{-1}w_t)+(\widehat l_i'-l_i'H_x) \widehat w_t.$
   
      \begin{lem}\label{ld.1a} (i)
   $\max_{it}|\widehat e_{it}-e_{it}|=O_P(\phi_{NT})$, where
$$   \phi_{NT}:=   (C_{NT}^{-2}(\max_i\frac{1}{T}\sum_se_{is}^2)^{1/2}+b_{NT,4}  +b_{NT,5})(1+ \max_{t\leq T}\|w_t\|)+ (b_{NT,1}+C_{NT}^{-1}b_{NT,2}+b_{NT,3}).$$
So $\max_{it} |\widehat e_{it}-e_{it}|\max_{it}|e_{it}|=O_P(1)$.

(ii) All terms below are $O_P(C_{NT}^{-2})$, for a fixed $t$:
 $\frac{1}{N}\sum_i (\widehat e_{it}-e_{it}) ^4   $,
$\frac{1}{N}\sum_i (\widehat e_{it}-e_{it}) ^2    $, and 
 $\frac{1}{N}\sum_i (\widehat e_{it}-e_{it}) ^2e_{it}^2 
 $ ,  $\frac{1}{N}\sum_i    \lambda_i  \lambda_i'  (\widehat e_{it}-e_{it})e_{it} $, $\frac{1}{N}\sum_i \lambda_i(\widehat e_{it}  -e_{it})u_{it}$. 
 
 (iii)   $\frac{1}{N}\sum_i\lambda_i \alpha_i' (\widehat e_{it} -e_{it})=O_P(C_{NT}^{-1})$ for a fixed $t$.
 
    \end{lem}

\begin{proof}
   Below we first simplify the expansion of $\widehat e_{it}-e_{it}$.
Let $K_3=\dim(l_i)$.
Let
 $\mathcal Q$ be a diagonal matrix consisting of the reciprocal of the  first $K_3$ eigenvalues of $XX'/(NT)$. Let 
\begin{eqnarray*}
 \zeta_{st}&=& \frac{1}{N}\sum_i(e_{is}e_{it}-\E e_{is}e_{it}),\quad \eta_{t}= \frac{1}{N}\sum_il_ie_{it}  ,\cr
 \sigma^2&=& \frac{1}{N}\sum_{i}\E  e_{it}^2.
  \end{eqnarray*}
For the PC estimator,  there is a rotation matrix  $\bar H_x$,  by (A.1) of \cite{bai03}, (which can be simplified due to the serial independence in $e_{it}$)
\begin{eqnarray*}
\widehat w_t-\bar H_xw_t&=&\mathcal Q\frac{ \sigma^2}{T}(\widehat w_t-\bar H_xw_t)
+\mathcal Q[\frac{ \sigma^2}{T} \bar H_x+ \frac{1}{TN}\sum_{is}\widehat w_s l_i'e_{is}]w_t
\cr
&&+\mathcal Q \frac{1}{T}\sum_{s}\widehat w_s \zeta_{st}+ \mathcal Q\frac{1}{T}\sum_s\widehat w_sw_s'\eta_t   .
  \end{eqnarray*}
Move $\mathcal Q\frac{ \sigma^2}{T}(\widehat w_t-\bar H_xw_t)$ to the left hand side (LHS); then LHS  becomes \\$(\I-\mathcal Q\frac{\sigma^2}{T})  (\widehat w_t-\bar H_xw_t)$. Note that $\|\mathcal Q\|=O_P(1)$ so $\mathcal Q_1:=(\I-\mathcal Q\frac{\sigma^2}{T})^{-1}$ exists whose eigenvalues all converge to one. Then multiply $\mathcal Q_1$ on both sides, we reach 
\begin{eqnarray*}
\widehat w_t-\bar H_xw_t&=& 
\mathcal Q_1   \mathcal Q[\frac{ \sigma^2}{T} \bar H_x+ \frac{1}{TN}\sum_{is}\widehat w_s l_i'e_{is}]w_t
\cr
&&+\mathcal Q_1 \mathcal Q \frac{1}{T}\sum_{s}\widehat w_s \zeta_{st}+\mathcal Q_1  \mathcal Q\frac{1}{T}\sum_s\widehat w_sw_s'\eta_t   .
  \end{eqnarray*}
  Next, move $\mathcal Q_1   \mathcal Q[\frac{ \sigma^2}{T} \bar H_x+ \frac{1}{TN}\sum_{is}\widehat w_s l_i'e_{is}]w_t$ to LHS, combined with $-\bar H_xw_t$, then LHS becomes $\widehat w_t- H_x^{-1}w_t$, where
   $H_x^{-1}=   (\I+ \mathcal Q_1\mathcal Q\frac{ \sigma^2}{T}) \bar H_x
  +\mathcal Q_1\mathcal Q \frac{1}{TN}\sum_{is}\widehat w_s l_i'e_{is}
  $, and $\mathcal Q_1\mathcal Q \frac{1}{TN}\sum_{is}\widehat w_s l_i'e_{is}=o_P(1)$. So the eigenvalues of $H_x^{-1}$ converge to those of $\bar H_x$, which are well known to be bounded away from both zero and infinity \citep{bai03}.  Finally, let $R_1=\mathcal Q_1 \mathcal Q $ and $R_2= \mathcal Q_1\mathcal Q \frac{1}{T}\sum_{s}\widehat w_s w_s'$, we reach
 \begin{eqnarray}\label{ed.43}
\widehat w_t- H_x^{-1}w_t&=&R_1  \frac{1}{T}\sum_{s}\widehat w_s \zeta_{st}+  R_2\eta_t .
  \end{eqnarray}
  with $\|R_1\|+\|R_2\|=O_P(1)$.

  Also, $\widehat l_i'=\frac{1}{T}\sum_s x_{is}\widehat w_s'$,  
 for $\mathcal Q_3=  -H_x(R_1  \frac{1}{T^2}   \sum_{m,s\leq T}\zeta_{ms}  \widehat w_m  \widehat w_s'  +R_2 \frac{1}{T}\sum_s  \eta_s  \widehat w_s' )=O_P(C_{NT}^{-2})$,
          \begin{eqnarray}
   \widehat e_{it}-e_{it}& =&  l_i'H_x(\widehat w_t-H_x^{-1}w_t)+(\widehat l_i'-l_i'H_x) \widehat w_t\cr
   &=& 
 l_i'H_x(\widehat w_t-H_x^{-1}w_t) +   \frac{1}{T}\sum_s e_{is}(\widehat w_s'-w_s'H_x^{-1'})   \widehat w_t\cr
                  &&+    \frac{1}{T}\sum_s e_{is}w_s'H_x^{-1'}  \widehat w_t +   l_i'H_x   \frac{1}{T}\sum_s(H_x^{-1}w_s-\widehat w_s)\widehat w_s'   \widehat w_t\cr
   &=&   \frac{1}{T}\sum_s e_{is}w_s'H_x^{-1'}  \widehat w_t +  \frac{1}{T^2}    \sum_{s,m\leq T}e_{is} \zeta_{ms} \widehat w_m 'R_1 '  \widehat w_t 
 +  \frac{1}{T}\sum_s e_{is}  \eta_s'R_2'    \widehat w_t   \cr
      &&+  l_i'  \mathcal Q_3 \widehat w_t+      l_i'H_x  R_1  \frac{1}{T}\sum_{s}\widehat w_s \zeta_{st}   +      l_i'H_x    R_2\eta_t  .
  \end{eqnarray}

  (i)  We first show that $\max_{t\leq T}\|\widehat w_t- H_x^{-1}w_t\|=O_P(1)$.
Define
  \begin{eqnarray*}
b_{NT,1}&=&\max_{t\leq T}\|  \frac{1}{NT}\sum_{is}w_s(e_{is}e_{it}-\E e_{is}e_{it})\|\cr
b_{NT,2}&=&(\max_{t\leq T}\frac{1}{T}\sum_s ( \frac{1}{N}\sum_ie_{is}e_{it}-\E e_{is}e_{it})^2   )^{1/2}\cr
b_{NT,3}&=&\max_{t\leq T}\| \frac{1}{N}\sum_il_ie_{it}\|   \cr
b_{NT,4}&=&\max_i\| \frac{1}{T}\sum_s e_{is}w_s\|  \cr
b_{NT,5}&=&\max_i\|    \frac{1}{NT}\sum_{js} l_j (e_{js}  e_{is} -\E e_{js}  e_{is}  ) \|   
  \end{eqnarray*}
Then  \begin{eqnarray*}
\max_{t\leq T}\|\frac{1}{T}\sum_{s}\widehat w_s \zeta_{st}\|&\leq&
O_P(1)\max_{t\leq T}\|\frac{1}{T}\sum_{s} w_s \zeta_{st}\|+ O_P(C_{NT}^{-1})(\max_{t\leq T}\frac{1}{T}\sum_s\zeta_{st}^2)^{1/2}\cr
&=& O_P(b_{NT,1}+C_{NT}^{-1}b_{NT,2}).
  \end{eqnarray*}
  Then by assumption, 
  $\max_{t\leq T}\|\widehat w_t- H_x^{-1}w_t\|=O_P(b_{NT,1}+C_{NT}^{-1}b_{NT,2}+b_{NT,3})=O_P(1)$.  As such $\max_{t\leq T}\|\widehat w_t\|=O_P(1)+ \max_{t\leq T}\|w_t\|$.
  
  In addition, 
   \begin{eqnarray*}
\max_i\| \frac{1}{T^2}    \sum_{m\leq T}\sum_{s\leq T}e_{is} \zeta_{ms} \widehat w_m '\|&\leq&
O_P(C_{NT}^{-1}) (\max_i \frac{1}{T} \sum_{s\leq T}e_{is} ^2 )^{1/2}(  \frac{1}{T^2} \sum_{s,t\leq T}\zeta_{st} ^2)^{1/2}
\cr
&&+O_P(1) (\max_i\frac{1}{T}\sum_se_{is}^2)^{1/2} (\frac{1}{T}\sum_s\| \frac{1}{TN}\sum_{jt}(e_{jt}e_{js}-\E e_{jt}e_{js})   w_t\|^2   )^{1/2}\cr
&=&O_P(C_{NT}^{-2})(\max_i\frac{1}{T}\sum_se_{is}^2)^{1/2}.
  \end{eqnarray*}
So 
 $\max_{it}|   \widehat e_{it}-e_{it}|=O_P(\phi_{NT})$, where 
$$
\phi_{NT}:=   (C_{NT}^{-2}(\max_i\frac{1}{T}\sum_se_{is}^2)^{1/2}+b_{NT,4}  +b_{NT,5})(1+ \max_{t\leq T}\|w_t\|)+ (b_{NT,1}+C_{NT}^{-1}b_{NT,2}+b_{NT,3}).
$$

(ii)   
  Let $a\in\{1, 2,4\}$, and $b\in\{0,1,2\}$, and a bounded constant sequence $c_i$,  consider, up to a $O_P(1)$ multiplier that is independent of $(t,i)$,
        \begin{eqnarray}
\frac{1}{N}\sum_ic_i e_{it}^b (    \widehat e_{it}-e_{it}   )   ^a
&=&  \frac{1}{N}\sum_ic_i e_{it}^b (   \frac{1}{T}\sum_s e_{is}w_s)^a     \widehat w_t     ^a\cr
&&+\frac{1}{N}\sum_ic_i e_{it}^b (     \frac{1}{T^2}    \sum_{s,m\leq T}e_{is} \zeta_{ms} \widehat w_m )^a \widehat w_t     ^a   \cr
&&+\frac{1}{N}\sum_ic_i e_{it}^b (    \frac{1}{T}\sum_s e_{is}  \eta_s')^a  \widehat w_t    ^a   +\frac{1}{N}\sum_ic_i e_{it}^b      l_i^a \mathcal Q_3^a  \widehat w_t     ^a   \cr
&&+\frac{1}{N}\sum_ic_i e_{it}^b     l_i^a ( \frac{1}{T}\sum_{s}\widehat w_s \zeta_{st}   )   ^a   +\frac{1}{N}\sum_ic_i e_{it}^b     l_i^a   \eta_t    ^a   .
  \end{eqnarray}
Note that 
        \begin{eqnarray*}
\frac{1}{T}\sum_{s}\widehat w_s \zeta_{st}  &\leq&O_P(1)\frac{1}{TN}\sum_{is} w_s  (e_{is}e_{it}-\E e_{is}e_{it})\cr
&&+(\frac{1}{T}\sum_{s}(\frac{1}{N}  \sum_i (e_{is}e_{it}-\E e_{is}e_{it}))^2)^{1/2}O_P(C_{NT}^{-1})=O_P(C_{NT}^{-2})\cr
 \frac{1}{N}\sum_ie_{it}^b(   \frac{1}{T}\sum_s e_{is}w_s)^a  &=& O(T^{-a/2}),\quad a=2,4,\quad b=0,2\cr
 \frac{1}{N}\sum_i (     \frac{1}{T^2}    \sum_{s,m\leq T}e_{is} \zeta_{ms} \widehat w_m )^4&\leq&
   (     \frac{1}{T^2}\sum_{m,s\leq T}\zeta_{ms} ^2)^2(\frac{1}{T}\sum_m w_m ^2 )^2\frac{1}{N}\sum_i (  \frac{1}{T}\sum_s e_{is}^2 )^2\cr
&&+   \frac{1}{N}\sum_i   (     \frac{1}{T}    \sum_{m\leq T} ( \frac{1}{T}    \sum_{s\leq T}e_{is} \zeta_{ms} )^2  )^2 
O_P(C_{NT}^{-4}) =O_P(C_{NT}^{-4}) \cr
 \frac{1}{N}\sum_i  (     \frac{1}{T^2}    \sum_{s,m\leq T}e_{is} \zeta_{ms} \widehat w_m )^2&\leq&
 \frac{1}{N}\sum_i        \frac{1}{T}   \sum_m   (\frac{1}{T}\sum_se_{is} \zeta_{ms} )^2   \frac{1}{T}   \sum_m  w_m ^2
 \cr
 &&+   \frac{1}{N}\sum_i       \frac{1}{T}    \sum_{m\leq T} ( \frac{1}{T}    \sum_{s\leq T}e_{is} \zeta_{ms} )^2   
O_P(C_{NT}^{-2})=O_P(C_{NT}^{-2}) \cr
\frac{1}{N}\sum_i     (    \frac{1}{T}\sum_s e_{is}  \eta_s')^4&\leq& 
\frac{1}{N}\sum_i     (    \frac{1}{T}\sum_s e_{is}^2)^2(\frac{1}{T}\sum_t  \eta_t^2)^2=O_P(C_{NT}^{-4})\cr
\frac{1}{N}\sum_i  e_{it}^2(    \frac{1}{T}\sum_s e_{is}  \eta_s')^2 &=&O_P(C_{NT}^{-2})\cr
\frac{1}{N}\sum_i  e_{it}^2     l_i^2 ( \frac{1}{T}\sum_{s}\widehat w_s \zeta_{st}   )   ^2 &\leq&
\frac{1}{N}\sum_i  e_{it}^2     l_i^2( \frac{1}{T}\sum_{s} w_s \zeta_{st}   )   ^2 +
O_P(C_{NT}^{-2})=O_P(C_{NT}^{-2})\cr
  \frac{1}{T}\sum_t(    \frac{1}{T}\sum_{s}\widehat w_s \zeta_{st}   )^2 &\leq&
  O_P(C_{NT}^{-4})
  \cr
  \frac{1}{N}\sum_ic_i e_{it} (   \frac{1}{T}\sum_s e_{is}w_s) &=&O_P(1)(\frac{1}{N^2T^2}
  \sum_{ij\leq N}\sum_{sl\leq T} c_ic_j \E  w_sw_l \E (e_{it}e_{jt}e_{is}e_{jl}
|W)  )^{1/2}\cr
&=&O_P(C_{NT}^{-2})\cr
\frac{1}{N}\sum_ic_i e_{it} (     \frac{1}{T^2}    \sum_{s,m\leq T}e_{is} \zeta_{ms} \widehat w_m )  &\leq&O_P(C_{NT}^{-2})+ \frac{1}{N}\sum_ic_i e_{it} (     \frac{1}{T}    \sum_{s\leq T}  \frac{1}{T}    \sum_{m\leq T}e_{is} \zeta_{ms}   w_m )  \cr
&\leq&  O_P(C_{NT}^{-2})+ O_P(1) (  \frac{1}{T}    \sum_{s\leq T}  \E(  \frac{1}{T}    \sum_{m\leq T} \zeta_{ms}   w_m )  ^2)^{1/2}\cr
&=&O_P(C_{NT}^{-2}).
  \end{eqnarray*}
  
  where the last equality follows from the following: 
         \begin{eqnarray*}
 &&\frac{1}{T}    \sum_{s\leq T}  \E(  \frac{1}{T}    \sum_{m\leq T} \zeta_{ms}   w_m )  ^2
  = O(T^{-2})+  \frac{1}{T}    \sum_{s\leq T}   \frac{1}{T^2}       \sum_{t\neq s}  \E  w_s  w_t \Cov(\zeta_{ss},  \zeta_{ts}  |W) 
\cr
&& + \frac{1}{T}    \sum_{s\leq T}   \frac{1}{T}    \sum_{m\neq s} \frac{1}{T}    \sum_{t\leq T}   \E w_m  w_t \Cov(\zeta_{ms},  \zeta_{ts}  |W  )\cr
&=&O(T^{-2}) + \frac{1}{T}    \sum_{s\leq T}   \frac{1}{T}    \sum_{t\neq s} \frac{1}{T}   \E  w_t^2\frac{1}{N^2}\sum_{ij} \E(  e_{is}   e_{js} |W)\E( e_{it}e_{jt}    |W  )=O(C_{NT}^{-4}).
  \end{eqnarray*}
 
 With the above results ready, we can proceed proving (ii)(iii) as follows.

  Now for $a=4, b=0$, $c_i=1,$    up to a $O_P(1)$ multiplier 
         \begin{eqnarray*}
\frac{1}{N}\sum_i    (    \widehat e_{it}-e_{it}   )   ^4
&=&  \frac{1}{N}\sum_i   (   \frac{1}{T}\sum_s e_{is}w_s)^4     +\frac{1}{N}\sum_i    (     \frac{1}{T^2}    \sum_{s,m\leq T}e_{is} \zeta_{ms} \widehat w_m )^4   \cr
&&+\frac{1}{N}\sum_i     (    \frac{1}{T}\sum_s e_{is}  \eta_s')^4    + \mathcal Q_3^4     +  ( \frac{1}{T}\sum_{s}\widehat w_s \zeta_{st}   )   ^4   + \eta_t    ^4   \cr
&\leq&O_P(C_{NT}^{-4}) .
  \end{eqnarray*}
  For $a=2$, $b=0$,  $c_i=1,$ 
          \begin{eqnarray*} 
  \frac{1}{N}\sum_i  (\widehat e_{it}-e_{it}) ^2 
 & \leq&   \frac{1}{N}\sum_i    (   \frac{1}{T}\sum_s e_{is}w_s)^2     +\frac{1}{N}\sum_i    (     \frac{1}{T^2}    \sum_{s,m\leq T}e_{is} \zeta_{ms} \widehat w_m )^2     \cr
&&+\frac{1}{N}\sum_i   (    \frac{1}{T}\sum_s e_{is}  \eta_s')^2    + \mathcal Q_3^2     +  ( \frac{1}{T}\sum_{s}\widehat w_s \zeta_{st}   )   ^2   +   \eta_t    ^2\cr
&\leq&O_P(C_{NT}^{-2}) .
   \end{eqnarray*}
     For $a=2$, $b=2$,  $c_i=1,$ 
    \begin{eqnarray*}
\frac{1}{N}\sum_i  e_{it}^2 (    \widehat e_{it}-e_{it}   )   ^2
&=&  \frac{1}{N}\sum_i  e_{it}^2 (   \frac{1}{T}\sum_s e_{is}w_s)^2    +\frac{1}{N}\sum_i  e_{it}^2 (     \frac{1}{T^2}    \sum_{s,m\leq T}e_{is} \zeta_{ms} \widehat w_m )^2      \cr
&&+\frac{1}{N}\sum_i  e_{it}^2(    \frac{1}{T}\sum_s e_{is}  \eta_s')^2     + \mathcal Q_3^2   +  ( \frac{1}{T}\sum_{s}\widehat w_s \zeta_{st}   )   ^2   +   \eta_t    ^2   \cr
&\leq&O_P(C_{NT}^{-2}) .
  \end{eqnarray*}
   
Next, let $a=b=1$ and  $c_i$ be any element of $\lambda_i\lambda_i'$,
     \begin{eqnarray*}
\frac{1}{N}\sum_ic_i e_{it} (    \widehat e_{it}-e_{it}   )    
&=&  \frac{1}{N}\sum_ic_i e_{it} (   \frac{1}{T}\sum_s e_{is}w_s)          +\frac{1}{N}\sum_ic_i e_{it} (     \frac{1}{T^2}    \sum_{s,m\leq T}e_{is} \zeta_{ms} \widehat w_m )        \cr
&&+\frac{1}{N}\sum_ic_i e_{it} (    \frac{1}{T}\sum_s e_{is}  \eta_s')        +\frac{1}{N}\sum_ic_i e_{it}      l_i  \mathcal Q_3         \cr
&&+\frac{1}{N}\sum_ic_i e_{it}     l_i  ( \frac{1}{T}\sum_{s}\widehat w_s \zeta_{st}   )     
+\frac{1}{N}\sum_ic_i e_{it}    l_i   \eta_t   \cr
&\leq&O_P(C_{NT}^{-2}) .
  \end{eqnarray*}
  Next, ignoring an $O_P(1)$ multiplier,
       \begin{eqnarray*}
  \frac{1}{N}\sum_i \lambda_i(\widehat e_{it}  -e_{it})u_{it}&\leq &
    \frac{1}{N}\sum_i \lambda_i   \frac{1}{T}\sum_s e_{is}w_su_{it}\cr
    &&+    \frac{1}{N}\sum_i \lambda_i   \frac{1}{T^2}    \sum_{s,m\leq T}e_{is} \zeta_{ms} \widehat w_m u_{it}  +    \frac{1}{N}\sum_i \lambda_i   \frac{1}{T}\sum_s e_{is}  \eta_su_{it}  \cr
            &&+    \frac{1}{N}\sum_i \lambda_i   l_i'  \mathcal Q_3  u_{it}  +    \frac{1}{N}\sum_i \lambda_i    l_i u_{it}  \frac{1}{T}\sum_{s}\widehat w_s \zeta_{st}   
+    \frac{1}{N}\sum_i \lambda_i   l_i\eta_t    u_{it}  \cr      
&\leq& O_P(C_{NT}^{-2}).         
    \end{eqnarray*}

     (iii)  Let $a=1, b=0$ and $c_i$ be any element of $\lambda_i\alpha_i'$,
    \begin{eqnarray*}
\frac{1}{N}\sum_ic_i   (    \widehat e_{it}-e_{it}   )    
&=&  \frac{1}{N}\sum_ic_i   (   \frac{1}{T}\sum_s e_{is}w_s)    +\frac{1}{N}\sum_ic_i   (     \frac{1}{T^2}    \sum_{s,m\leq T}e_{is} \zeta_{ms} \widehat w_m )  \cr
&&+\frac{1}{N}\sum_ic_i   (    \frac{1}{T}\sum_s e_{is}  \eta_s')      +     \mathcal Q_3     + ( \frac{1}{T}\sum_{s}\widehat w_s \zeta_{st}   )      +  \eta_t       \cr
&\leq &O_P(C_{NT}^{-1}),
  \end{eqnarray*}
  where the dominating term is $\eta_t= O_P(C_{NT}^{-1}).$

\end{proof}

  \begin{lem}\label{ld.2a} Assume $\max_{it}|e_{it}|C_{NT}^{-1}=O_P(1)$ and $\E e_{it}^8<C$.

  Let $c_i$ be a non-random bounded sequence.

  (i)
$\max_{t\leq T} \frac{1}{N}\sum_i    (\widehat e_{it}-e_{it} )^4=O_P(1+\max_{t\leq T}\| w_t\|^4+b_{NT,2}^4)C_{NT}^{-4}  + O_P(b_{NT,1}^4+b_{NT,3}^4). $

$\max_{t\leq T} \frac{1}{N}\sum_i  (\widehat e_{it}-e_{it} )^2e_{it}^2 
\leq  O_P(1+\max_{t\leq T}\| w_t\|^2+  b_{NT,2}^2  )\max_{it}|e_{it}|^2C_{NT}^{-2}
$

$+O_P(b_{NT,1}^2+b_{NT,3}^2) \max_{t\leq T}\frac{1}{N}\sum_ie_{it}^2 .
$

$\max_{t}|\frac{1}{N}\sum_i c_i (\widehat e_{it}-e_{it} )e_{it} |
\leq O_P(1+\max_{t\leq T}\| w_t\|+b_{NT,2})\max_{it}|e_{it}|C_{NT}^{-1} $

$+  \max_{t\leq T}\|\frac{1}{N}\sum_ic_i e_{it}     \|_FO_P(b_{NT,1}+b_{NT,3}).
 $

$\max_{t}|\frac{1}{N}\sum_i c_i(\widehat e_{it}-e_{it} )^2|\leq O_P(1+\max_{t\leq T}\| w_t\|^2+ b_{NT,2}^2) C_{NT}^{-2}+O_P(b_{NT,1}^2+b_{NT,3}^2 ).$

$\max_{t\leq T} \frac{1}{N}\sum_i  (\widehat e_{it}-e_{it})^2\leq O_P(1+\max_{t\leq T}\| w_t\|^2+ b_{NT,2}^2) C_{NT}^{-2}+O_P(b_{NT,1}^2+b_{NT,3}^2 )$

$\max_{t}| \frac{1}{N}\sum_i  c_i (\widehat e_{it}-e_{it})|\leq  O_P(1+\max_{t\leq T}\| w_t\|)C_{NT}^{-2} + O_P(b_{NT,1}+b_{NT,3}+C_{NT}^{-1}b_{NT,2}).$

$ \max_i \frac{1}{T}\sum_t(\widehat e_{it}-e_{it})^2\leq O_P(b_{NT, 4}^2+b_{NT,5}^2 +C_{NT}^{-2})     .$

(ii)  All terms below are $O_P(C_{NT}^{-2})$:  $\frac{1}{NT}\sum_{it}(\widehat e_{it}-e_{it})^4$

$\frac{1}{NT}\sum_{it}(\widehat e_{it}-e_{it})^2$,   $ \frac{1}{NT}\sum_{it}e_{it}^2(\widehat e_{it}-e_{it})^2$,
$\frac{1}{T}\sum_{t } |\frac{1}{N}\sum_i c_i(e_{it}-\widehat e_{it})
  |^2m_t^2 $,  where\\
 $m_t^2= 1+\frac{1}{|\mathcal G|_0}\sum_{j\in\mathcal G}  \|e_{jt}    f_t\|^2 (\|f_t\|+\|g_t\|)^2 .$
  

 
(iii) All terms below are $O_P(C_{NT}^{-4})$:  
 $ \frac{1}{T}\sum_{t } |\frac{1}{N}\sum_i  c_i (\widehat e_{it}-e_{it})e_{it}|^2 $,

$\frac{1}{T}\sum_{t } |\frac{1}{N}\sum_i c_i (\widehat e_{it}-e_{it}) ^2  |^2$, $\frac{1}{T}\sum_t|  \frac{1}{N}\sum_i
c_i (\widehat e_{it} -e_{it})u_{it} |^2$, 
 
 

$  \frac{1}{|\mathcal G|_0}\sum_{i\in\mathcal G}\| \frac{1}{T}\sum_{s\notin I} f_s(\widehat e_{is}-e_{is})u_{is}  \|^2, $ 
$ \frac{1}{|\mathcal G|_0}\sum_{i\in\mathcal G}\|   \frac{1}{T}\sum_{s\notin I}  f_se_{is}(\widehat e_{is}-e_{is}) \lambda_i'  f_s  \|^2, $

$ \frac{1}{|\mathcal G|_0}\sum_{j\in\mathcal G}\|   \frac{1}{T_0}\sum_{t\in I^c}
    \frac{1}{N}\sum_i  c_i f_t ^2  e_{jt}  (e_{it}-\widehat e_{it})    \|^2$

  \end{lem}

 \begin{proof} (i) 
  First,   in the proof of Lemma \ref{ld.1a}(i), we showed $\max_{t\leq T}\|\widehat w_t\|\leq O_P(1+\max_{t\leq T}\| w_t\|)$.  By the proof of Lemma \ref{ld.1a}(ii),
    \begin{eqnarray*}
  &&  \max_{t\leq T} \frac{1}{N}\sum_i    (\widehat e_{it}-e_{it} )^4
  \leq  \frac{1}{N}\sum_i    (   \frac{1}{T}\sum_s e_{is}w_s)^4    \max_{t\leq T}  \widehat w_t     ^4+\frac{1}{N}\sum_i    (     \frac{1}{T^2}    \sum_{s,m\leq T}e_{is} \zeta_{ms} \widehat w_m )^4 \max_{t\leq T} \widehat w_t     ^4  \cr
&&+\frac{1}{N}\sum_i    (    \frac{1}{T}\sum_s e_{is}  \eta_s')^4 \max_{t\leq T}  \widehat w_t    ^4   +  \mathcal Q_3^4 \max_{t\leq T} \widehat w_t     ^4  +\max_{t\leq T}  ( \frac{1}{T}\sum_{s}\widehat w_s \zeta_{st}   )   ^4   +  \max_{t\leq T}   \eta_t    ^4 \cr
&\leq& O_P(1+\max_{t\leq T}\| w_t\|^4+b_{NT,2}^4)(C_{NT}^{-4}  )+ O_P(b_{NT,1}^4+b_{NT,3}^4). 
      \end{eqnarray*}
      Next, 
  \begin{eqnarray*}
 \max_{t\leq T} \frac{1}{N}\sum_i  e_{it}^2 (    \widehat e_{it}-e_{it}   )   ^2
&=&  \max_{t\leq T}  \frac{1}{N}\sum_i  e_{it}^2 (   \frac{1}{T}\sum_s e_{is}w_s)^2  \max_{t\leq T}  \widehat w_t     ^2 +\max_{t\leq T} \frac{1}{N}\sum_i e_{it}^2      l_i^4    \mathcal Q_3^2  \max_{t\leq T}  \widehat w_t     ^2    \cr
&&   + \max_{t\leq T} \frac{1}{N}\sum_i  e_{it}^2 (     \frac{1}{T^2}    \sum_{s,m\leq T}e_{is} \zeta_{ms} \widehat w_m )^2   \max_{t\leq T}  \widehat w_t     ^2   \cr
&&+ \max_{t\leq T} \frac{1}{N}\sum_i  e_{it}^2(    \frac{1}{T}\sum_s e_{is}  \eta_s')^2  \max_{t\leq T}  \widehat w_t     ^2  \cr
&&+\max_{t\leq T}\frac{1}{N}\sum_ie_{it}^2     l_i^2 ( \frac{1}{T}\sum_{s}\widehat w_s \zeta_{st}   )   ^2   +\max_{t\leq T}\frac{1}{N}\sum_ie_{it}^2    l_i^2  \eta_t    ^2  \cr
&\leq& O_P(1+\max_{t\leq T}\| w_t\|^2+b_{NT,2}^2)\max_{it}|e_{it}|^2C_{NT}^{-2}\cr
&&
+O_P(b_{NT,1}^2+b_{NT,3}^2) \max_{t\leq T}\frac{1}{N}\sum_ie_{it}^2 .
\cr
\frac{1}{N}\sum_ic_i e_{it}  (    \widehat e_{it}-e_{it}   )    
&\leq&  \max_{t\leq T} \frac{1}{N}\sum_ic_i e_{it}  (   \frac{1}{T}\sum_s e_{is}w_s)      \widehat w_t      \cr
&&+ \max_{t\leq T}\frac{1}{N}\sum_ic_i e_{it}  (     \frac{1}{T^2}    \sum_{s,m\leq T}e_{is} \zeta_{ms} \widehat w_m )  \widehat w_t         \cr
&&+ \max_{t\leq T}\frac{1}{N}\sum_ic_i e_{it}  (    \frac{1}{T}\sum_s e_{is}  \eta_s')   \widehat w_t        + \max_{t\leq T}\frac{1}{N}\sum_ic_i e_{it}       l_i  \mathcal Q_3   \max_{t\leq T} \widehat w_t         \cr
&&+ \max_{t\leq T}\frac{1}{N}\sum_ic_i e_{it}      l_i  ( \frac{1}{T}\sum_{s}\widehat w_s \zeta_{st}   )       + \max_{t\leq T}\frac{1}{N}\sum_ic_i e_{it}      l_i    \max_{t\leq T} \eta_t    \cr
&\leq& O_P(1+\max_{t\leq T}\| w_t\|+b_{NT,2})\max_{it}|e_{it}|C_{NT}^{-1} \cr
&&+  \max_{t\leq T}\|\frac{1}{N}\sum_ic_i e_{it}      l_i \|_FO_P(b_{NT,1}+b_{NT,3}).
\cr
\max_{t\leq T}\frac{1}{N}\sum_ic_i   (    \widehat e_{it}-e_{it}   )   ^2
&\leq&  \frac{1}{N}\sum_ic_i   (   \frac{1}{T}\sum_s e_{is}w_s)^2   \max_{t\leq T}  \widehat w_t     ^2
 +\max_{t\leq T} ( \frac{1}{T}\sum_{s}\widehat w_s \zeta_{st}   )   ^2   +\max_{t\leq T}    \eta_t    ^2  
\cr
&&+\frac{1}{N}\sum_ic_i   (     \frac{1}{T^2}    \sum_{s,m\leq T}e_{is} \zeta_{ms} \widehat w_m )^2 \max_{t\leq T}\widehat w_t     ^2   \cr
&&+\frac{1}{N}\sum_ic_i   (    \frac{1}{T}\sum_s e_{is}  \eta_s')^2\max_{t\leq T}  \widehat w_t    ^2   +\max_{t\leq T}  \mathcal Q_3^2  \widehat w_t     ^2   \cr
&\leq& O_P(1+\max_{t\leq T}\| w_t\|^2+ b_{NT,2}^2) C_{NT}^{-2}+O_P(b_{NT,1}^2+b_{NT,3}^2 ), 
\cr
\max_{t\leq T}\frac{1}{N}\sum_ic_i  (    \widehat e_{it}-e_{it}   )   
&\leq &  \frac{1}{N}\sum_ic_i   (   \frac{1}{T}\sum_s e_{is}w_s)     \max_{t\leq T} \widehat w_t      \cr
&&+\frac{1}{N}\sum_ic_i   (     \frac{1}{T^2}    \sum_{s,m\leq T}e_{is} \zeta_{ms} \widehat w_m )  \max_{t\leq T}\widehat w_t         \cr
&&+\frac{1}{N}\sum_ic_i   (    \frac{1}{T}\sum_s e_{is}  \eta_s')  \max_{t\leq T} \widehat w_t        +  \mathcal Q_3 \max_{t\leq T}  \widehat w_t         \cr
&&+ \max_{t\leq T} ( \frac{1}{T}\sum_{s}\widehat w_s \zeta_{st}   )       + \max_{t\leq T}  \eta_t        \cr
&\leq&  O_P(1+\max_{t\leq T}\| w_t\|)C_{NT}^{-2} + O_P(b_{NT,1}+b_{NT,3}+C_{NT}^{-1}b_{NT,2}).
  \end{eqnarray*}
  Finally,
          \begin{eqnarray*}
          \max_i \frac{1}{T}\sum_t(\widehat e_{it}-e_{it})^2&\leq&
               \max_i \frac{1}{T}\sum_t \widehat w_t  ^2 (      \frac{1}{T}\sum_s e_{is}w_s'H_x^{-1'}   )^2 +     \max_i \frac{1}{T}\sum_t\widehat w_t ^2(    \frac{1}{T^2}    \sum_{s,m\leq T}e_{is} \zeta_{ms} \widehat w_m 'R_1 '     )^2 \cr
                      && +     \max_i \frac{1}{T}\sum_t \widehat w_t^2 (  \frac{1}{T}\sum_s e_{is}  \eta_s'R_2'      )^2  +     \max_i \frac{1}{T}\sum_t\widehat w_t  ^2(  l_i'  \mathcal Q_3  )^2 \cr
                       && +     \max_i \frac{1}{T}\sum_t(     l_i'H_x  R_1  \frac{1}{T}\sum_{s}\widehat w_s \zeta_{st}   )^2  +     \max_i \frac{1}{T}\sum_t\eta_t  ^2(    l_i'H_x    R_2 )^2 \cr    
                       &\leq& O_P(b_{NT, 4}^2+b_{NT,5}^2 +C_{NT}^{-2})     .
 \end{eqnarray*}
   
  (ii) Note that $\max_{t\leq T}\|\widehat w_t\|^2=O_P(1)+O_P(1)\max_{t\leq T}\|w_t\|^2
  \leq O_P(1)+o_P(C_{NT}),
  $ where the last inequality follows from the assumption that $\max_{t\leq T}\|w_t\|^2=o_P(C_{NT}).$
    \begin{eqnarray*}
    \frac{1}{NT}\sum_{it}(\widehat e_{it}-e_{it})^4&\leq& \frac{1}{T}\sum_t\|\widehat w_t \|^2\max_{t\leq T}\|\widehat w_t\|^2[
        \frac{1}{N}\sum_{i}(    \frac{1}{T}\sum_s e_{is}w_s    )^4+   O_P(1)(  \frac{1}{T^2}      \sum_{s,m\leq T}    \zeta_{ms} ^2     )^2   ] \cr
                &&+    \frac{1}{T}\sum_t\|\widehat w_t \|^2\max_{t\leq T}\|\widehat w_t\|^2[         \frac{1}{N}\sum_{i}(   \frac{1}{T}\sum_s e_{is}  \eta_s   )^4  +          \mathcal Q_3    ^4 ]\cr
                                &&+         \frac{1}{T}\sum_{t}(      \frac{1}{T}\sum_s\zeta_{st}   ^2   )^2  O_P(C_{NT}^{-4})+         \frac{1}{T}\sum_{t}   \eta_t   ^4
                                +  \frac{1}{T}\sum_{t}(      \frac{1}{T}\sum_{s} w_s \zeta_{st}      )^4
                                 \cr       
                                 &\leq&       ( O_P(1)+o_P(C_{NT}))C_{NT}^{-4}+O_P(C_{NT}^{-2}) \cr
                                 && +         O_P(1)\frac{1}{T}\sum_{t}        \frac{1}{T^4}\sum_{s, k,l,m\leq T}   \E w_l     w_m w_kw_s \E (\zeta_{st}  
     \zeta_{kt} 
        \zeta_{lt} 
         \zeta_{mt}      |W)                 \cr
         &=&O_P(C_{NT}^{-2}) .
\end{eqnarray*}
Similarly,  $ \frac{1}{NT}\sum_{it}(\widehat e_{it}-e_{it})^2= O_P(C_{NT}^{-2}) .  $
\begin{eqnarray*}
\frac{1}{NT}\sum_{it}e_{it}^2(\widehat e_{it}-e_{it})^2&\leq&
\frac{1}{NT}\sum_{it}e_{it}^2    \widehat w_t  ^2  (      \frac{1}{T}\sum_s e_{is}  \eta_s   )^2+\frac{1}{NT}\sum_{it}e_{it}^2 \widehat w_t  ^2 (     \frac{1}{T}\sum_s e_{is}w_s   )^2  \cr
&&+\frac{1}{NT}\sum_{it}e_{it}^2 \widehat w_t ^2 (   \frac{1}{T^2}    \sum_{s,m\leq T}e_{is} \zeta_{ms} \widehat w_m   )^2  +\frac{1}{NT}\sum_{it}e_{it}^2 \widehat w_t ^2 (  l_i'  \mathcal Q_3   )^2  \cr
&&+\frac{1}{NT}\sum_{it}e_{it}^2(     l_i'  \frac{1}{T}\sum_{s}\widehat w_s \zeta_{st}      )^2  +\frac{1}{NT}\sum_{it}e_{it}^2(     l_i' \eta_t   )^2  =O_P(C_{NT}^{-2}) .
\end{eqnarray*}
Next,        \begin{eqnarray*}
 \frac{1}{T}\sum_{t } (\frac{1}{N}\sum_i   c_i(e_{it}-\widehat e_{it})
)^2 m_t ^2&\leq&
\frac{1}{T}\sum_{t } m_t ^2 \widehat w_t ^2(\frac{1}{N}\sum_i   c_i    \frac{1}{T}\sum_s e_{is}w_s
)^2 \cr
&&+ \frac{1}{T}\sum_{t } m_t ^2 \widehat w_t^2(\frac{1}{N}\sum_i   c_i  \frac{1}{T^2}    \sum_{s,m\leq T}e_{is} \zeta_{ms} \widehat w_m       )^2 \cr
&&+ \frac{1}{T}\sum_{t }m_t ^2\widehat w_t ^2 (\frac{1}{N}\sum_i   c_i   \frac{1}{T}\sum_s e_{is}  \eta_s   )^2 \cr
&&+ \frac{1}{T}\sum_{t } m_t ^2 \widehat w_t    ^2 \mathcal Q_3^2 + \frac{1}{T}\sum_{t } m_t ^2( \frac{1}{T}\sum_{s}\widehat w_s \zeta_{st}    )^2 + \frac{1}{T}\sum_{t }   m_t ^2\eta_t   ^2 \cr
&=&O_P(C_{NT}^{-2}).
   \end{eqnarray*}

(iii)  
\begin{eqnarray*}
&& \frac{1}{T}\sum_{t }(\frac{1}{N}\sum_i   c_i (\widehat e_{it}-e_{it})e_{it})^2
\cr
&\leq& \max_{t\leq T}\|\widehat w_t-H_xw_t\|^2 \frac{1}{T}\sum_{t } [(\frac{1}{N}\sum_i   c_i    \frac{1}{T}\sum_s e_{is}w_s e_{it})^2 +\frac{1}{T}\sum_m(\frac{1}{N}\sum_i   c_i     e_{it}   \frac{1}{T}    \sum_{s\leq T}e_{is} \zeta_{ms} )^2 ]\cr
&&+   \frac{1}{T}\sum_{t } w_t^2[  (\frac{1}{N}\sum_i   c_i    \frac{1}{T}\sum_s e_{is}w_s e_{it})^2 
+\frac{1}{T}\sum_m(\frac{1}{N}\sum_i   c_i     e_{it}   \frac{1}{T}    \sum_{s\leq T}e_{is} \zeta_{ms} )^2 
]\cr
&&+  \frac{1}{T}\sum_{t } \widehat w_t    ^2 
 (  \frac{1}{N}\sum_i   c_i     \frac{1}{T}\sum_s e_{is}  \eta_se_{it}    )^2
+  \frac{1}{T}\sum_{t }\widehat w_t   ^2(\frac{1}{N}\sum_i   c_i   l_i'  e_{it})^2\mathcal Q_3 ^2 \cr
&&+  \frac{1}{T}\sum_{t }(\frac{1}{N}\sum_i   c_i    l_ie_{it})^2 \frac{1}{T}\sum_{s} \zeta_{st}     ^2 +  \frac{1}{T}\sum_{t }\eta_t  ^2(\frac{1}{N}\sum_i   c_i    l_i   e_{it})^2 =O_P(C_{NT}^{-4}) .
 \end{eqnarray*}

 Next, to bound $\frac{1}{T}\sum_{t } ( \frac{1}{N}\sum_i c_i (\widehat e_{it}-e_{it}) ^2      )^2$, 
 we first bound $\frac{1}{T}\sum_{t }\widehat w_t ^4$. By (\ref{ed.43}), 
 \begin{eqnarray*}
 \frac{1}{T}\sum_{t}\widehat w_t ^4
 &\leq&O_P(1)+ \frac{1}{T}\sum_{t}\|\widehat w_t-H_x^{-1}w_t\| ^4\cr
 &\leq&O_P(1)+ 
 \frac{1}{T}\sum_{t }\|\eta_t \| ^4
 +\frac{1}{T}\sum_{t }\|  \frac{1}{T}\sum_{s} w_s \zeta_{st}  \| ^4+\frac{1}{T}\sum_{t } (   \frac{1}{T}\sum_{s}\zeta_{st} ^2 )^2 O_P(C_{NT}^{-4})\cr
 &=&O_P(1).
  \end{eqnarray*}
 Thus  
\begin{eqnarray*}
&&\frac{1}{T}\sum_{t\notin I} ( \frac{1}{N}\sum_i c_i (\widehat e_{it}-e_{it}) ^2      )^2\cr
&\leq&\frac{1}{T}\sum_{t\notin I}\widehat w_t ^4 ( \frac{1}{N}\sum_i c_i (   \frac{1}{T}\sum_s e_{is}w_s  ) ^2      )^2 +   \frac{1}{T}\sum_{t\notin I} \widehat w_t^4( \frac{1}{N}\sum_i c_i (    \frac{1}{T}\sum_s e_{is}  \eta_s      ) ^2      )^2 \cr
&&+  \frac{1}{T}\sum_{t\notin I} \widehat w_t^4  \mathcal Q_3   ^4 +  \frac{1}{T}\sum_{t\notin I}  \eta_t  ^4    
+ \frac{1}{T}\sum_{t\notin I} (  \frac{1}{T}\sum_{s} w_s \zeta_{st}    ) ^4 + \frac{1}{T}\sum_{t\notin I} ( \frac{1}{T}\sum_s \zeta_{st} ^2   ) ^2  O_P(C_{NT}^{-4})\cr
&&+  \frac{1}{T}\sum_{t\notin I} \widehat w_t ^4 ( \frac{1}{N}\sum_i c_i   \frac{1}{T}  \sum_{m\leq T}( \frac{1}{T}  \sum_{s\leq T}e_{is} \zeta_{ms}  ) ^2      )^2  =O_P(C_{NT}^{-4}) .
 \end{eqnarray*}
And up to an $O_P(1)$ as a product term, 
\begin{eqnarray*}
&&\frac{1}{|\mathcal G|_0}\sum_{i\in\mathcal G}\| \frac{1}{T}\sum_{t\notin I} f_t(\widehat e_{it}-e_{it})u_{it}  \|^2\cr
&\leq&       \frac{1}{|\mathcal G|_0}\sum_{i\in\mathcal G}\| \frac{1}{T}\sum_{t\notin I} f_t u_{it} \widehat w_t  \|^2  (  \frac{1}{T^2}    \sum_{s,m\leq T}e_{is} \zeta_{ms} \widehat w_m     \|^2  +\| \frac{1}{T}\sum_s e_{is}w_s     \|^2+\|   \frac{1}{T}\sum_s e_{is}  \eta_s  \|^2+ \|\mathcal Q_3\|^2)\cr
&& +    \frac{1}{|\mathcal G|_0}\sum_{i\in\mathcal G}\| \frac{1}{T}\sum_{t\notin I} f_t u_{it}       l_i\eta_t     \|^2 +    \frac{1}{|\mathcal G|_0}\sum_{i\in\mathcal G}\| \frac{1}{T}\sum_{t\notin I} f_t u_{it}    \frac{1}{T}\sum_{s}\widehat w_s \zeta_{st}     \|^2  .
 \end{eqnarray*}
It is easy to see all terms except for the last one is $O_P(C_{NT}^{-4}) .$ To see the last one, note  that is is bounded by ,  
\begin{eqnarray*}
 &&  O_P(1) \frac{1}{T}\sum_{s}(\widehat w_s-H_xw_s) ^2\frac{1}{T}\sum_{s}\frac{1}{T}\sum_{t\notin I}  \zeta_{st}   ^2
 + O_P(1)\frac{1}{|\mathcal G|_0}\sum_{i\in\mathcal G}\| \frac{1}{T}\sum_{t\notin I} f_t u_{it}      \frac{1}{T}\sum_{s} w_s \zeta_{st}     \|^2   \cr
 &\leq& O_P(C_{NT}^{-4})+
  O_P(1)\frac{1}{|\mathcal G|_0}\sum_{j\in\mathcal G}\| \frac{1}{T}\sum_{t\notin I} f_t u_{jt}      \frac{1}{T}\sum_{s} w_s \frac{1}{N}\sum_i(e_{is}e_{it}-\E e_{is}e_{it})   \|^2\cr
  &=& O_P(C_{NT}^{-4})
 \end{eqnarray*}
 since $e_{it}$ is conditionally serially independent given $(U, W, F)$ and $u_{it}$ is conditionally serially independent given $(E, W,F)$.

      The conclusion that $\frac{1}{|\mathcal G|_0}\sum_{i\in\mathcal G}\|   \frac{1}{T}\sum_{s\notin I}  f_se_{is}(\widehat e_{is}-e_{is}) \lambda_i'  f_s  \|^2= O_P(C_{NT}^{-4})  $ follows similarly, due to $\max_j\sum_{i\leq N}|(\E e_{it}e_{jt}|F, W)|<C$. 
      
      Next, for $a:= \frac{1}{|\mathcal G|_0}\sum_{j\in\mathcal G}      \frac{1}{N}\sum_i \|\frac{1}{T_0}\sum_{t\in I^c}
 c_i f_t ^2  e_{jt} \widehat w_t \|^2=O_P(C_{NT}^{-2}) $, 
\begin{eqnarray*}
  &&     \frac{1}{|\mathcal G|_0}\sum_{j\in\mathcal G}\|   \frac{1}{T_0}\sum_{t\in I^c}
    \frac{1}{N}\sum_i  c_i f_t ^2  e_{jt}  (e_{it}-\widehat e_{it})    \|^2\cr
    &\leq&  a    \frac{1}{N}\sum_i[\|  \frac{1}{T}\sum_s e_{is}w_s     \|^2+\|  \frac{1}{T^2}    \sum_{s,m\leq T}e_{is} \zeta_{ms} \widehat w_m '\|^2+\| \frac{1}{T}\sum_s e_{is}  \eta_s\|^2 +\| \mathcal Q_3\|^2 ]\cr
       &&+   \frac{1}{|\mathcal G|_0}\sum_{j\in\mathcal G}\|   \frac{1}{T_0}\sum_{t\in I^c}
 f_t ^2  e_{jt}     \frac{1}{T}\sum_{s}\widehat w_s   \frac{1}{N}\sum_i(e_{is}e_{it}-\E e_{is}e_{it})  \|^2 \cr
 &&+  \frac{1}{|\mathcal G|_0}\sum_{j\in\mathcal G}\|  \frac{1}{N}\sum_il_i  \frac{1}{T_0}\sum_{t\in I^c}
  f_t ^2 ( e_{jt}  e_{it}-\E e_{jt}  e_{it} )  \|^2+   \frac{1}{|\mathcal G|_0}\sum_{j\in\mathcal G}\|  \frac{1}{N}\sum_il_i  \frac{1}{T_0}\sum_{t\in I^c}
  f_t ^2 \E e_{jt}  e_{it}  \|^2 \cr
  &=&O_P(C_{NT}^{-4}).
       \end{eqnarray*}

\end{proof}
 
        \subsection{Behavior of the preliminary  }\label{pre:many}
       Recall  that 
       $$
       (  \widetilde f_s,\widetilde g_s):= \arg\min_{f_s, g_s}\sum_{i=1}^N(y_{is}- \widetilde\alpha_i'g_s- x_{is} \widetilde \lambda_i'f_s)^2,\quad s\notin I.
  $$
  and
       $$
(\dot\lambda_i, \dot\alpha_i)= \arg \min_{\lambda_i, \alpha_i}\sum_{s\notin I} (y_{is} - \alpha_i'  \widetilde g_s- x_{is} \lambda_i'\widetilde f_s)^2,\quad i=1,..., N.
  $$

The goal of this section is to show that the effect of the preliminary estimation is negligible. Specifically, we aim to show, for each fixed $t\notin I$,         \begin{eqnarray*}
\|\frac{1}{N}\sum_{j}  (\dot\lambda_j-H_1' \lambda_j)e_{jt}\|^2&=&O_P( C_{NT}^{-4}) , \cr
  \frac{1}{|\mathcal G|_0}\sum_{i\in\mathcal G} \|\frac{1}{T}\sum_{s\notin I}  f_s  (\widetilde f_s-H_1^{-1}f_s )'\mu_{is}e_{is}\|^2  &=& O_P(C_{NT}^{-4}).
    \end{eqnarray*}

      Throughout the proof below, we treat $|I^c|=T$ instead of $T/2$ to avoid keeping the constant ``$2$". 
 In addition, for notational simplicity, we write $\widetilde\Lambda= \widetilde\Lambda_I$ and $\widetilde A=\widetilde A_I$ by suppressing the subscripts, but we should keep in mind that $\widetilde\Lambda$ and $\widetilde A$ are estimated on data $D_I$ as  defined in step 2. 
    In addition, let $\E_I$  and $\Var_I$ be the conditional expectation and variance, given $D_I$. 
     Recall that  $X_s$ be the vector of $x_{is}$ fixing $s\leq T$, and $M_{\widetilde\alpha}= I_N- \widetilde A ( \widetilde A' \widetilde A)^{-1} \widetilde A'$; $X_i$ be the vector of $x_{is}$ fixing $i\leq N$, and $M_{\widetilde g} = I - \widetilde G ( \widetilde G' \widetilde G)^{-1} \widetilde G'$, for $\widetilde G$ as the $|I^c|_0\times K_1$ matrix of $\widetilde g_s$. Define $\widetilde F$ similarly. 
  Let $L$ denote $N\times K_3$ matrix of $l_i$, so $X_s=Lw_s+e_s$.     Also let $W$ be $T\times K_3$ matrix of $w_t$.
    
    Define
    \begin{eqnarray*}
\widetilde D_{fs}&=&\frac{1}{N}\widetilde\Lambda'\diag(X_s) M_{\widetilde\alpha} \diag(X_s)\widetilde\Lambda\cr
D_{fs}&=&\frac{1}{N}\Lambda '(\diag(X_s) M_{\alpha} \diag(X_s)\Lambda\cr
\bar D_{fs}&=&\frac{1}{N}\Lambda '\E((\diag(e_s) M_{\alpha} \diag(e_s) )\Lambda
+\frac{1}{N}\Lambda '(\diag(Lw_s) M_{\alpha} \diag(Lw_s) \Lambda
\cr
\widetilde D_{\lambda i}&=&\frac{1}{T}\widetilde F'\diag(X_i) M_{\widetilde g} \diag(X_i)\widetilde F\cr
 D_{\lambda i}&=&\frac{1}{T}  F'  (\diag(X_i) M_{ g} \diag(X_i)  ) F\cr
  \bar D_{\lambda i}&=&\frac{1}{T}  F' \E (\diag(E_i) M_{ g} \diag(E_i)  ) F
  +\frac{1}{T}  F'  (\diag(Wl_i) M_{ g} \diag(Wl_i)  ) F
  \cr
\end{eqnarray*}

 By the stationarity, $D_f$ does not depend on $s$.

    \begin{lem}\label{lc.1} Suppose $\max_{it}e_{it}^2+\max_{t\leq T}\|w_t\|^2=o_P(C_{NT})$, $ \max_{t\leq T}\E \|w_t\|^4=O (1)$.  
    $\frac{1}{N}\sum_{ij}|\Cov(e_{is}, e_{js}|w_s)|<\infty$ and $\|\E e_se_s'\|<\infty$. 
$\frac{1}{N^3}\sum_{ijkl}\Cov(e_{is} e_{js}, e_{ks} e_{ls}  )<C.$
    Also, there is $c>0$, so that  
    $\min_{s} \min_j\psi_j( D_{fs})>c$. 
    Then 
    
    (i) $\max_s\|\widetilde D_{fs}^{-1}\|=O_P(1)$. 
    
    (ii) $\frac{1}{T}\sum_{s\notin I} \|\widetilde D_{fs}^{-1}-(H_1'\bar D_{fs}H_1)^{-1}\|^2= O_P(C^{-2}_{NT})$.    
 
    \end{lem}
    
    \begin{proof}
    (i) The eigenvalues of $D_{fs}$ are bounded from zero uniformly in $s\leq T$. Also, 
\begin{eqnarray}\label{ec.9}
\widetilde D_{fs} -H_1'D_{fs}H_1&=&
\sum_l \delta_l, \quad \text{ where }\cr
\delta_1&=&\frac{1}{N}(\widetilde\Lambda-\Lambda H_1)'\diag(X_s) M_{\widetilde\alpha} \diag(X_s)\widetilde\Lambda,
\cr
\delta_2&=&\frac{1}{N}H_1' \Lambda '\diag(X_s) M_{\widetilde\alpha} \diag(X_s)(\widetilde\Lambda-\Lambda H_1)\cr
\delta_3&=&\frac{1}{N}H_1' \Lambda '\diag(X_s)( M_{\widetilde\alpha} -M_\alpha)\diag(X_s) \Lambda H_1.
\end{eqnarray}
We now bound each term uniformly in $s\leq T$.  The first term is 
\begin{eqnarray*}
&&\frac{1}{N}(\widetilde\Lambda-\Lambda H_1)'\diag(X_s) M_{\widetilde\alpha} \diag(X_s)\widetilde\Lambda
\leq O_P(1)\frac{1}{\sqrt{N}}\|\widetilde \Lambda -\Lambda H_1\|_F \max_{is}x_{is}^2=o_P(1)
\end{eqnarray*}
provided that $\max_{is}x_{is}^2=o_P(C_{NT}).$ The second term is bounded similarly.  The third term is bounded by 
$$
O_P(1)\max_{is}x_{is}^2 \|M_{\widetilde\alpha} -M_\alpha\|=O_P(1)\frac{1}{\sqrt{N}}\|\widetilde A -A H_2\|_F \max_{is}x_{is}^2=o_P(1).
$$
This implies $\max_s \|\widetilde D_{fs} -H_1'D_{fs}H_1\|=o_P(1)$. 
In addition,  because of the convergence of $\|\frac{1}{\sqrt{N}}(\widetilde\Lambda-\Lambda H_1)\|_F$, we have $\min_j\psi_{j}(H_1'H_1)\geq C \min_j\psi_{j}(\frac{1}{N}H_1'\Lambda'\Lambda H_1)$, bounded away from zero. Thus $\min_s\min_j\psi_j(H_1'D_{fs} H_1)
\geq \min_s\min_j\psi_j(D_{fs} )C,
$ bounded away from zero. This together with   $\max_s \|\widetilde D_{fs} -H_1'D_{fs}H_1\|=o_P(1)$ imply    $\min_s\|\widetilde D_{fs}^{-1}\| =O_P(1)$.

(ii) By (\ref{ec.9}), 
\begin{eqnarray*}
&&\frac{1}{T}\sum_{s\notin I} \|\widetilde D_{fs}^{-1}-(H_1'D_{fs}H_1)^{-1}\|^2
\cr
&\leq& \frac{1}{T}\sum_s \|\widetilde D_{fs}-(H_1'D_{fs}H_1)\|^2 \|(H_1'D_{fs}H_1)^{-2}\|\max_s\|\widetilde D_{fs}^{-2}\|\cr
&\leq & O_P(1)\frac{1}{T}\sum_s \|\widetilde D_{fs}-(H_1'D_{fs}H_1)\|^2 \cr
&=& 
O_P(1)\sum_{l=1}^3\frac{1}{T}\sum_s \|\delta_l\|^2.
\end{eqnarray*}We bound each term below. Up to a $O_P(1)$ product, 
\begin{eqnarray*}
 \frac{1}{T}\sum_{s\notin I} \|\delta_1\|^2 &\leq& \frac{1}{N^2}\sum_{ij} \|\widetilde\lambda_i-H_1'\lambda_i\|^2   
  \|\widetilde\lambda_j\|^2 \frac{1}{T}\sum_{s\notin I} x_{is}^2x_{js}^2 \cr
  &\leq& \frac{1}{N^2}\sum_{ij} \|\widetilde\lambda_i-H_1'\lambda_i\|^2   
  \|\widetilde\lambda_j\|^2 \frac{1}{T}\sum_{s\notin I}( 1+ e_{is}^2e_{js}^2) \cr
  &&+ \frac{1}{N^2}\sum_{ij} \|\widetilde\lambda_i-H_1'\lambda_i\|^2   
  \|\widetilde\lambda_j\|^2( \frac{1}{T}\sum_{s\notin I}    e_{is}^4)^{1/2}
  \end{eqnarray*}
 $\{e_{is}\}$ is serially independent, so $ e_{is}, e_{js}$ are independent of $ \|\widetilde\lambda_i-H_1'\lambda_i\|^2   
  \|\widetilde\lambda_j-H_1'\lambda_j\|^2$   for $s\notin I$. 
Take the conditional expectation $\E_I$.  
\begin{eqnarray*}
 \frac{1}{T}\sum_{s\notin I} \|\delta_1\|^2 &\leq& \frac{1}{N^2}\sum_{ij} \|\widetilde\lambda_i-H_1'\lambda_i\|^2   
  \|\widetilde\lambda_j\|^2 \frac{1}{T}\sum_{s\notin I}\E_I( 1+ e_{is}^2e_{js}^2) \cr
  &&+ \frac{1}{N^2}\sum_{ij} \|\widetilde\lambda_i-H_1'\lambda_i\|^2   
  \|\widetilde\lambda_j\|^2\E_I( \frac{1}{T}\sum_{s\notin I}    e_{is}^4)^{1/2}=O_P(C_{NT}^{-2}).
\end{eqnarray*}
Term of $ \frac{1}{T}\sum_{s\notin I} \|\delta_2\|^2 $ is bounded similarly. 
\begin{eqnarray*}
 \frac{1}{T}\sum_{s\notin I} \|\delta_3\|^2 &\leq&  O_P(1) \|M_{\widetilde\alpha} -M_{\alpha}\|=O_P(C_{NT}^{-2}).
\end{eqnarray*}

Next, \begin{eqnarray*}
&& \frac{1}{T}\sum_s \| D_{fs}- \bar D_{fs}\|^2\leq \frac{1}{T}\sum_s \|\frac{1}{N}\sum_{ij}\lambda_i \lambda_j'M_{\alpha, ij}(x_{is}x_{js}-\E e_{is }e_{js} -   l_i'w_sl_j'w_s )\|_F^2     \cr
&\leq& \frac{1}{T}\sum_s \|\frac{1}{N}\sum_{ij}\lambda_i \lambda_j'M_{\alpha, ij}(e_{is}e_{js}-\E e_{is }e_{js}    )\|_F^2 \cr
&&+\frac{2}{T}\sum_s \|\frac{1}{N}\sum_{ij}\lambda_i \lambda_j'M_{\alpha, ij} l_i'w_se_{js}\|_F^2 
\end{eqnarray*}
We assume $\dim(\lambda_i)=\dim(p_i)=1$.
As for the first term, it is less than 
\begin{eqnarray*}
 &&\frac{1}{T}\sum_s \Var(\frac{1}{N}\sum_{ij}\lambda_i \lambda_j'M_{\alpha, ij}e_{is}e_{js})\cr
 &\leq&\frac{1}{T}\sum_s \frac{1}{N^4}\sum_{ijkl}|  \Cov(e_{is}e_{js}, 
e_{ks}e_{ls})|=O(N^{-1})
\end{eqnarray*}
provided that $\frac{1}{N^3}\sum_{ijkl}\Cov(e_{is} e_{js}, e_{ks} e_{ls}  )<C.$
As for the second term, it is less than
\begin{eqnarray*}
O_P(1)\frac{1}{T}\sum_s  \E  w_s^2  \frac{1}{N^2}\sum_{ij}  |\Cov(e_{js},e_{ls}|w_s)|=O(N^{-1})
\end{eqnarray*}
provided that   $\frac{1}{N}\sum_{ij}|\Cov(e_{is}, e_{js}|w_s)|<\infty$ and $\|\E e_se_s'\|<\infty$. 
 So $\frac{1}{T}\sum_s \| D_{fs}- \bar D_{fs}\|^2=O_P(N^{-1})$.

    \end{proof}

    \begin{lem}\label{lc.3}       Suppose $\max_{it}e_{it}^4=O_P(\min\{N,T\})$.
    (i) $\frac{1}{T}\|\widetilde F- F H_1^{-1'}\|_F^2= O_P(C_{NT}^{-2})=\frac{1}{T}\|\widetilde G- G H_2^{-1'}\|_F^2$, and 
    $\frac{1}{T}\sum_{t\notin I} \|\widetilde f_t- H_1^{-1}f_t\|^2 e_{it}^2u_{it}^2=O_P(C_{NT}^{-2})  $.
    \\
    (ii)      $\max_i\|\widetilde D_{\lambda i}^{-1}\|=O_P(1)$. \\    
    (iii) $\frac{1}{N}\sum_{i} \|\widetilde D_{\lambda i}^{-1}-(H_1^{-1}\bar D_{\lambda i }H_1^{-1'})^{-1}\|^2=  O_P(C_{NT}^{-2})$.
    \end{lem}
    \proof (i) The  proof is straightforward given the expansion of (\ref{ec.10}) and $\max_s\|\widetilde D_{fs}^{-1}\|=O_P(1)$. 
So we omit the details for brevity. 
(ii)
 Note that 
\begin{eqnarray}\label{ec.11}
\widetilde D_{\lambda i} -H_1^{-1}D_{\lambda i}H_1^{-1'}&=&
\sum_l \delta_l, \quad \text{ where }\cr
\delta_1&=&\frac{1}{T}(\widetilde F- F H_1^{-1'})'\diag(X_i) M_{\widetilde g} \diag(X_i)\widetilde F,
\cr
\delta_2&=&\frac{1}{T}H_1^{-1} F '\diag(X_i) M_{\widetilde g } \diag(X_i)(\widetilde F- F H_1^{-1'})\cr
\delta_3&=&\frac{1}{T}H_1^{-1} F '\diag(X_i)( M_{\widetilde g} -M_g)\diag(X_i) F H_1^{-1'}.
\end{eqnarray}
The proof is very similar to that of Lemma \ref{lc.1}.

(iii) 
\begin{eqnarray*}
&&\frac{1}{N}\sum_{i} \|\widetilde D_{\lambda i}^{-1}-(H_1^{-1}\bar D_{\lambda i }H_1^{-1'})^{-1}\|^2
\cr
&\leq& \frac{1}{N}\sum_i \|\widetilde D_{\lambda i}-(H_1^{-1}\bar D_{\lambda i}H_1^{-1'})\|^2\max_i \|(H_1^{-1}\bar D_{\lambda i}H_1^{-1'})^{-2}\|\|\widetilde D_{\lambda i}^{-2}\|\cr
&\leq & O_P(1)\frac{1}{N}\sum_i\|\widetilde D_{\lambda i}-(H_1^{-1}\bar D_{\lambda i}H_1^{-1'})\|^2 \cr
&=& 
O_P(1)\sum_{l=1}^3\frac{1}{N}\sum_i \|\delta_l\|^2
+O_P(1) \frac{1}{N}\sum_i \| D_{\lambda i}-\bar  D_{\lambda i}\|^2.
\end{eqnarray*}
We now bound each term.   With the assumption that $\max_{it}x_{it}^4=O_P(\min\{N,T\}),$ 
\begin{eqnarray*}
\frac{1}{N}\sum_i \|\delta_1\|^2&\leq &  
 \frac{1}{T}\|\widetilde F- F H_1^{-1'}\|_F^2 \| M_{\widetilde g} -  M_{g} \|^2\frac{1}{T}\|\widetilde F\|_F^2\frac{1}{N}\sum_i \|\diag(X_i)\|^4\cr
 &&
+\frac{1}{T^2}\| \widetilde F- F H_1^{-1'}\|^4_F\frac{1}{N}\sum_i\|\diag(X_i) \|^4  \cr
&&+\frac{1}{N}\sum_i\|\frac{1}{T}(\widetilde F- F H_1^{-1'})'\diag(X_i) M_{  g} \diag(X_i) FH_1^{-1'}\|^2\cr
&\leq&O_P(C_{NT}^{-4}) \max_{it}x_{it}^4+
\frac{1}{NT}\sum_i\sum_t x_{it}^2    (\sum_s M_{ g, ts} x_{is} f_s)^2O_P(C_{NT}^{-2})\cr
&\leq& O_P(C_{NT}^{-2})  .
\end{eqnarray*}
$\frac{1}{N}\sum_i \|\delta_2\|^2$ is bounded similarly. 
$
\frac{1}{N}\sum_i \|\delta_3\|^2\leq    O_P(1)\|M_{\widetilde g} -M_g\|^2=O_P(C_{NT}^{-2})  .
$
  Finally,
  \begin{eqnarray*} \frac{1}{N}\sum_i \| D_{\lambda i}-\bar  D_{\lambda i}\|^2
  &\leq&  \frac{1}{N}\sum_i \|\frac{1}{T}\sum_{st} f_sf_t'M_{g, st}(x_{is}x_{it}-\E  e_{is}e_{it}-l_i'w_sl_i'w_t )\|_F^2\cr
&\leq&  O_P(1) \frac{1}{N}\sum_i \|\frac{1}{T}\sum_{st} f_sf_t'M_{g, st}(e_{is}e_{it}-\E  e_{is}e_{it}  )\|_F^2\cr
&&+O_P(1) \frac{1}{N}\sum_i \|\frac{1}{T}\sum_{st} f_sf_t'M_{g, st} w_se_{it}\|_F^2\cr
&\leq& O_P(1) \frac{1}{N}\sum_i \|\frac{1}{T}\sum_{t}  f_t^2 w_te_{it}\|_F^2+O_P(1) \frac{1}{N}\sum_i \| \frac{1}{T}\sum_{t} f_tg_te_{it} \|_F^2\cr
&=&O_P(C_{NT}^{-2}).
\end{eqnarray*}

    \begin{lem}\label{ld.5} Suppose $ \Var ( u_s |e_t, e_s, w_s) <C$ and $C_{NT}^{-1} \max_{is}|x_{is}|^2=O_P(1)$.
    
    (i)    For each fixed $t\notin I$, 
  $\frac{1}{ {N}}\sum_{i}  (H_1' \lambda_i-\dot\lambda_i)e_{it}=O_P( C_{NT}^{-2}) .$
  
  (ii) $ \frac{1}{T}\sum_{t\notin I}   \|       \frac{1}{N}\sum_i c_ie_{it}  (\dot \lambda_i-H_1'\lambda_i)  \|^2 =O_P( C_{NT}^{-4})$ for any deterministic  and bounded sequence $c_i$.
  
    \end{lem}
    
\begin{proof} Given that the $|I^c|\times 1$ vector $y_i= G\alpha_i + \diag(X_i)F\lambda_i + u_i$
    where $u_i$ is a $|I^c|\times 1$ vector and $G$ is an $|I^c|\times K_1$ matrix, we have
        \begin{eqnarray*}
\dot\lambda_i&=& \widetilde D_{\lambda i}^{-1} \frac{1}{T}\widetilde F'\diag(X_i) M_{\widetilde g}y_i\cr
&=&  H_1' \lambda_i+ \widetilde D_{\lambda i}^{-1} \frac{1}{T}\widetilde F'\diag(X_i) M_{\widetilde g}(GH_2^{-1'}-\widetilde G)H_2'\alpha_i
\cr
&& 
+ \widetilde D_{\lambda i}^{-1} \frac{1}{T}\widetilde F'\diag(X_i) M_{\widetilde g}\diag(X_i)(FH_1^{-1'}-\widetilde  F)H_1 '\lambda_i + \widetilde D_{\lambda i}^{-1} \frac{1}{T}\widetilde F'\diag(X_i) M_{\widetilde g} u_i.
\end{eqnarray*}

(i)

Step 1:    given $\E e_{it}^4f_t^2+e_{it}^2f_t^2w_t^2<C$,
        \begin{eqnarray*} 
 &&\frac{1}{\sqrt{N}}\sum_{i}   e_{it} (\widetilde D_{\lambda i}^{-1}-(H_1^{-1}\bar D_{\lambda i }H_1^{-1'})^{-1}) \frac{1}{T}\widetilde F'\diag(X_i) M_{\widetilde g}(GH_2^{-1'}-\widetilde G)H_2'\alpha_i\cr
  &\leq&O_P(\sqrt{N}C_{NT}^{-2})+O_P(\sqrt{N}C_{NT}^{-2})(  \frac{1}{T}\sum_s (\widetilde f_s-H_1^{-1}f_s)^2     )^{1/2} \max_{is}|x_{is}|\cr
    &\leq&O_P(\sqrt{N}C_{NT}^{-2})+O_P(\sqrt{N}C_{NT}^{-3}) \max_{is}|x_{is}|=O_P(\sqrt{N}C_{NT}^{-2}).
  \end{eqnarray*}
Step 2:  $\bar D_{\lambda i}$ is nonrandom given $W, G, F$, 
        \begin{eqnarray*} 
   &&\frac{1}{\sqrt{N}}\sum_{i}   e_{it} (H_1^{-1}\bar D_{\lambda i }H_1^{-1'})^{-1} \frac{1}{T}\widetilde F'\diag(X_i) M_{\widetilde g}(GH_2^{-1'}-\widetilde G)H_2'\alpha_i\cr
 &\leq& O_P(\sqrt{N}C_{NT}^{-1})(a^{1/2}+b^{1/2}) \quad \text{ where } \cr
 a&=&\frac{1}{T}\sum_s(\widetilde f_s-H_1^{-1}f_s)^2(
 \frac{1}{N}\sum_{i}  \alpha_i e_{it}  \bar D_{\lambda i } ^{-1} x_{is})^2 \cr
 b&=&   \frac{1}{T}\sum_s  f_s^2(
 \frac{1}{N}\sum_{i}  \alpha_i e_{it}  \bar D_{\lambda i } ^{-1} x_{is})^2.
  \end{eqnarray*}
We now bound each term.  As for $b$, note that  for each fixed $t$, 
  \begin{eqnarray}\label{ec.45}
 &&\E((
 \frac{1}{N}\sum_{i}  \alpha_i e_{it}  \bar D_{\lambda i } ^{-1} x_{is})^2|F, G, W, u_s ) \cr
 &\leq&  \E(
 \frac{1}{N}\sum_{i}  \alpha_i e_{it}  \bar D_{\lambda i } ^{-1} e_{is}|F, G, W, u_s  )^2
 +\E(
 \frac{1}{N}\sum_{i}  \alpha_i e_{it}  \bar D_{\lambda i } ^{-1} l_{i}'w_s|F,  G, W , u_s )^2
 \cr
 &\leq& \frac{C}{N}+ \frac{C\|w_s\|^2}{N}+(\E
 \frac{1}{N}\sum_{i}  \alpha_i e_{it}  e_{is} \bar D_{\lambda i } ^{-1}|F, G, W, u_s )^2
   \end{eqnarray}
with the assumption  $\frac{1}{N}  \sum_{ij}|\Cov(e_{jt}  e_{js}, 
e_{it}  e_{is} |F, G,W, u_s )| <C$.  So
        \begin{eqnarray*}
\E b&\leq &\frac{1}{T}\sum_s\E f_s^2(
 \frac{1}{N}\sum_{i}  \alpha_i e_{it}  \bar D_{\lambda i } ^{-1} x_{is})^2\leq  O(C_{NT}^{-2}).
  \cr
 a&=& \frac{1}{T}\sum_sf_s^2 ( \frac{1}{N}\widetilde\Lambda'\diag(X_s) M_{\widetilde\alpha}\diag(X_s)(\Lambda H_1-\widetilde \Lambda)    )^2(
 \frac{1}{N}\sum_{i}  \alpha_i e_{it}  \bar D_{\lambda i } ^{-1} x_{is})^2\cr
 &&+\frac{1}{T}\sum_s  (\widetilde D_{fs}^{-1} \frac{1}{N}\widetilde\Lambda'\diag(X_s) M_{\widetilde\alpha}\diag(X_s)(\Lambda H_1-\widetilde \Lambda)H_1^{-1} f_s)^2(
 \frac{1}{N}\sum_{i}  \alpha_i e_{it}  \bar D_{\lambda i } ^{-1} x_{is})^2\cr
 &&+\frac{1}{T}\sum_s (\widetilde D_{fs}^{-1} \frac{1}{N}\widetilde\Lambda'\diag(X_s) M_{\widetilde\alpha} u_s )^2(
 \frac{1}{N}\sum_{i}  \alpha_i e_{it}  \bar D_{\lambda i } ^{-1} x_{is})^2\cr
 &\leq&O_P(C_{NT}^{-2})\max_{is}x_{is}^4
   \frac{1}{T}\sum_s \E f_s^2    (
 \frac{1}{N}\sum_{i}  \alpha_i e_{it}  \bar D_{\lambda i } ^{-1} x_{is})^2\cr
 &&+O_P(1)     \frac{1}{T}\sum_s (  \frac{1}{N}\widetilde\Lambda'\diag(X_s) M_{\widetilde\alpha} u_s )^2(
 \frac{1}{N}\sum_{i}  \alpha_i e_{it}  \bar D_{\lambda i } ^{-1} x_{is})^2\cr
 &\leq& O_P(C_{NT}^{-2})
 + O_P(1)     \frac{1}{T}\sum_s \E(  \frac{1}{N} \Lambda'\diag(X_s) M_{ \alpha} u_s )^2(
 \frac{1}{N}\sum_{i}  \alpha_i e_{it}  \bar D_{\lambda i } ^{-1} x_{is})^2\cr
 &&+ O_P(C_{NT}^{-2})   \max_{is}x_{is}^2  \frac{1}{TN}\sum_s  \E \|    u_s \|^2\E((
 \frac{1}{N}\sum_{i}  \alpha_i e_{it}  \bar D_{\lambda i } ^{-1} x_{is})^2|u_s)\cr
 &:=&O_P(C_{NT}^{-2})+ O_P(1)(a.1+a.2).
  \end{eqnarray*}
 We now respectively bound $a.1$ and $a.2$.  As for $a.1$, note that $ \Var ( u_s |e_t, e_s) <C$ almost surely, thus
        \begin{eqnarray*} 
 &&\E (  \frac{1}{N} \Lambda'\diag(X_s) M_{ \alpha} u_s |e_t, e_s)^2
 = \frac{1}{N} \frac{1}{N}\Lambda'\diag(X_s) M_{ \alpha}  \Var ( u_s |e_t, e_s) M_\alpha \diag(X_s)\Lambda\cr
 &\leq&  C\frac{1}{N} \frac{1}{N}\Lambda'\diag(X_s)   ^2\Lambda.
   \end{eqnarray*} 
As for $a.2$, we use   (\ref{ec.45}). Thus, 
      \begin{eqnarray*} 
        a.1  &\leq &       \frac{1}{T}\sum_s \E(
 \frac{1}{N}\sum_{i}  \alpha_i e_{it}  \bar D_{\lambda i } ^{-1} x_{is})^2\E (  \frac{1}{N} \Lambda'\diag(X_s) M_{ \alpha} u_s |e_t, e_s)^2\cr
 &\leq &    \frac{C}{N} \frac{1}{T}\sum_s \E(
 \frac{1}{N}\sum_{i}  \alpha_i e_{it}  \bar D_{\lambda i } ^{-1} x_{is})^2 \ \frac{1}{N}\Lambda'\diag(X_s)   ^2\Lambda=O(N^{-1}).\cr
 a.2&\leq&C_{NT}^{-2}   \max_{is}x_{is}^2  \frac{1}{TN}\sum_s  \E \|    u_s \|^2\E((
 \frac{1}{N}\sum_{i}  \alpha_i e_{it}  \bar D_{\lambda i } ^{-1} x_{is})^2|u_s)\cr
 &\leq& C_{NT}^{-2}   \max_{is}x_{is}^2  \frac{1}{TN}\sum_s  \E \|    u_s \|^2\frac{C}{N}
 \cr
 &&+ C_{NT}^{-2}   \max_{is}x_{is}^2  \frac{1}{TN}\sum_s  \E \|    u_s \|^2(\E
 \frac{1}{N}\sum_{i}  \alpha_i e_{it}  e_{is} \bar D_{\lambda i } ^{-1}|F, u_s)^2\cr
 &\leq& O(C_{NT}^{-2})  .
   \end{eqnarray*}

   Put together,  $a^{1/2}+b^{1/2}=O(C_{NT}^{-1})  .$ So  the first term in the expansion of $\frac{1}{\sqrt{N}}\sum_{i}  (H_1' \lambda_i-\dot\lambda_i)e_{it}$ is 
   $$
   \frac{1}{\sqrt{N}}\sum_{i}   e_{it}  \widetilde D_{\lambda i}^{-1}\frac{1}{T}\widetilde F'\diag(X_i) M_{\widetilde g}(GH_2^{-1'}-\widetilde G)H_2'\alpha_i=O_P(\sqrt{N}C_{NT}^{-2}). 
   $$
   
Step 3:  
\begin{eqnarray*} 
&&   \frac{1}{\sqrt{N}}\sum_{i}   e_{it} (\widetilde D_{\lambda i}^{-1}-(H_1^{-1}\bar D_{\lambda i }H_1^{-1'})^{-1}) \frac{1}{T}\widetilde F'\diag(X_i) M_{\widetilde g}\diag(X_i)(FH_1^{-1'}-\widetilde  F)H_1 '\lambda_i \cr
&\leq&O_P(\sqrt{N}C_{NT}^{-2}) (\frac{1}{NT}\sum_i e_{it}^2\| F'\diag(X_i) M_{  g}\diag(X_i)\|^2  )^{1/2}  \cr 
&&+O_P(\sqrt{N}C_{NT}^{-2}) (\frac{1}{N}\sum_i e_{it}^2\|\diag(X_i) \| ^4  )^{1/2} [\|M_g-M_{\widetilde g}\|\frac{1}{\sqrt{T}} \| \widetilde F\|   +   \frac{1}{\sqrt{T}} \|\widetilde F-FH_1^{-1}\|] \cr
&\leq&O_P(\sqrt{N}C_{NT}^{-2}) (\frac{1}{NT}\sum_i\sum_s e_{it}^2 f_s ^2  x_{is}^4  )^{1/2}  \cr
&&+O_P(\sqrt{N}C_{NT}^{-2}) (\frac{1}{NT}\sum_i\sum_s x_{is}^2e_{it}^2g_s^2(\frac{1}{T}\sum_k f_k g_kx_{ik})^2  )^{1/2}  \cr
&&+O_P(\sqrt{N}C_{NT}^{-3})\max_{it}x_{it}^2=O_P(\sqrt{N}C_{NT}^{-2}).
   \end{eqnarray*}

Step 4:      note $\frac{1}{T}\sum_s(\frac{1}{N}\sum_{i}     \lambda_i  e_{it}  \bar D_{\lambda i } ^{-1}e_{is}^2)^2=O_P(C_{NT}^{-2}),$
   \begin{eqnarray*} 
&&   \frac{1}{\sqrt{N}}\sum_{i}   e_{it} (H_1^{-1}\bar D_{\lambda i }H_1^{-1'})^{-1}\frac{1}{T}(\widetilde F-FH_1^{-1'})'\diag(X_i) M_{\widetilde g}\diag(X_i)(FH_1^{-1'}-\widetilde  F)H_1 '\lambda_i \cr
&\leq& O_P(\sqrt{N}C_{NT}^{-3})\max_{it}x_{it}^2\cr
&&+ \frac{1}{\sqrt{N}}\sum_{i}   e_{it} (H_1^{-1}\bar D_{\lambda i }H_1^{-1'})^{-1}\frac{1}{T}(\widetilde F-FH_1^{-1'})'\diag(X_i) M_{ g}\diag(X_i)(FH_1^{-1'}-\widetilde  F)H_1 '\lambda_i \cr
&\leq&O_P(\sqrt{N}C_{NT}^{-2})+ \sqrt{N } \frac{1}{T}\sum_{s}( \widetilde f_s-H_1^{-1}f_s) ^2    \frac{1}{N}\sum_{i}     \lambda_i  e_{it}  \bar D_{\lambda i } ^{-1}x_{is}^2       \cr
&\leq&O_P(\sqrt{N}C_{NT}^{-2})+\max_{is}x_{is}^2  O_P( \sqrt{N }C_{NT}^{-2})  \frac{1}{NT}\sum_{s}  g_s ^2     w_s^2      \sum_{i}     \lambda_i  e_{it}  \bar D_{\lambda i } ^{-1}  l_i^2     \cr
&& +\max_{is}x_{is}^2 O_P( \sqrt{N }C_{NT}^{-2})  \frac{1}{T}\sum_{s}w_s^2 \frac{1}{N}\sum_j  \lambda_j ^2x_{js}^2   f_s^2 \frac{1}{N}\sum_{i}     \lambda_i  e_{it}  \bar D_{\lambda i } ^{-1}  l_i^2       \cr
& & + \sqrt{N } \frac{1}{T}\sum_{s}w_s^2( \frac{1}{N}\Lambda'\diag(X_s) M_{\alpha} u_s) ^2    \frac{1}{N}\sum_{i}     \lambda_i  e_{it}  \bar D_{\lambda i } ^{-1} l_i^2    \cr
&\leq&O_P(\sqrt{N}C_{NT}^{-2}).
   \end{eqnarray*} 
Step 5:      First we bound $ \frac{1}{NT}  \sum_i\sum_s(\widetilde f_s-H_1^{-1}f_s)^2e_{is}^2$.  We have
\begin{eqnarray*} 
\widetilde f_s&=& \widetilde D_{fs}^{-1} \frac{1}{N}\widetilde\Lambda'\diag(X_s) M_{\widetilde\alpha}y_s\cr
&=&  H_1^{-1} f_s+ \widetilde D_{fs}^{-1} \frac{1}{N}\widetilde\Lambda'\diag(X_s) M_{\widetilde\alpha}(AH_2-\widetilde A)H_2^{-1}g_s
\cr
&& 
+ \widetilde D_{fs}^{-1} \frac{1}{N}\widetilde\Lambda'\diag(X_s) M_{\widetilde\alpha}\diag(X_s)(\Lambda H_1-\widetilde \Lambda)H_1^{-1} f_s \cr
&&+ \widetilde D_{fs}^{-1} \frac{1}{N}\widetilde\Lambda'\diag(X_s) M_{\widetilde\alpha} u_s  .
\end{eqnarray*}
So
  \begin{eqnarray*} 
&& \frac{1}{NT}  \sum_i\sum_s(\widetilde f_s-H_1^{-1}f_s)^2e_{is}^2\cr
&=&    \frac{1}{NT}  \sum_i\sum_{s\notin I}e_{is}^2  \frac{1}{N} \| \Lambda'\diag(X_s)  \|^2[ g_s^2+  \frac{1}{N}\|  u_s   \|^2 ] O_P(C_{NT}^{-2})\cr
&&+O_P(1)  \frac{1}{N}\sum_j(H_1\lambda_j-\widetilde \lambda_j)^2\max_j\frac{1}{NT}  \sum_i\sum_{s\notin I} \E_Ie_{is}^2   f_s ^2  \frac{1}{N}\| \Lambda'\diag(X_s) \|^2 x_{js}^2  \cr
&&+ \frac{1}{NT}  \sum_i\sum_se_{is}^2     ( \frac{1}{N}\Lambda'\diag(X_s) M_{ \alpha} u_s     )^2\cr
&&+ \frac{1}{NT}  \sum_i\sum_se_{is}^2     \frac{1}{N}\|\diag(X_s) M_{ \alpha} u_s     \|^2 
 O_P(C_{NT}^{-2})
 \cr
 &=& O_P(C_{NT}^{-2}).
    \end{eqnarray*}  
    Therefore    
  \begin{eqnarray*} 
&&   \frac{1}{\sqrt{N}}\sum_{i}   e_{it} (H_1^{-1}\bar D_{\lambda i }H_1^{-1'})^{-1}\frac{1}{T} F '\diag(X_i) (M_{\widetilde g}-M_{ g})\diag(X_i)(FH_1^{-1'}-\widetilde  F)H_1 '\lambda_i \cr
&\leq&O_P(\sqrt{N}C_{NT}^{-1})
 (\frac{1}{N}\sum_{i}  \lambda_i ^2e_{it}^2  \frac{1}{T} \sum_kf_k^2x_{ik}^2  )^{1/2}  (\frac{1}{NT}  \sum_i\sum_s(\widetilde f_s-H_1^{-1}f_s)^2e_{is}^2)^{1/2}
 \cr
&\leq&  O_P(\sqrt{N}C_{NT}^{-2}),
    \end{eqnarray*} 
where the last equality is due to $\frac{1}{NT}  \sum_i\sum_{s\notin I}(\widetilde f_s-H_1^{-1}f_s)^2e_{is}^2=O_P(C_{NT}^{-2})$.

Step 6:
 \begin{eqnarray*} 
&&   \frac{1}{\sqrt{N}}\sum_{i}   e_{it} (H_1^{-1}\bar D_{\lambda i }H_1^{-1'})^{-1}\frac{1}{T} F '\diag(X_i) M_g\diag(X_i)(FH_1^{-1'}-\widetilde  F)H_1 '\lambda_i \cr
&\leq& O_P(\sqrt{N}C_{NT}^{-1}) (\frac{1}{T} \sum_sf_s   ^2( \frac{1}{N}\sum_{i}  \lambda_i  x_{is}   ^2e_{it}  \bar D_{\lambda i } ^{-1}  )^2)^{1/2}\cr
&&+  O_P(\sqrt{N}C_{NT}^{-1}) (\frac{1}{T} \sum_sg_ s^2( \frac{1}{N}\sum_{i}   x_{is} \lambda_i  e_{it}  \bar D_{\lambda i } ^{-1}\frac{1}{T} \sum_k  f_k  x_{ik} g_k )^2)^{1/2}\cr
&=&O_P(\sqrt{N}C_{NT}^{-1}) (a^{1/2}+b^{1/2}).
    \end{eqnarray*} 
 We aim to show $a=O_P(C_{NT}^{-2})=b$.  
  \begin{eqnarray*} 
  \E a&:=&\frac{1}{T} \sum_s \E f_s   ^2( \frac{1}{N}\sum_{i}  \lambda_i  x_{is}   ^2e_{it}  \bar D_{\lambda i } ^{-1}  )^2\cr
  &\leq &\frac{1}{T} \sum_s\frac{1}{N^2}\sum_{ij}  \E f_s   ^2   \lambda_i  \bar D_{\lambda i } ^{-1}       \lambda_j \bar D_{\lambda j} ^{-1}  \E (e_{is}   ^2 e_{js}   ^2 |F)\Cov (e_{it} ,  e_{jt}  |F, G, W)
  \cr
&&  
  +\frac{1}{T} \sum_s \E f_s   ^2 \frac{1}{N^2}\sum_{ij}    \lambda_i  l_i^2w_s^2\bar D_{\lambda i } ^{-1} \lambda_j  \mu_{j}   ^2\bar D_{\lambda j} ^{-1}    \Cov (e_{it} ,   e_{jt}   |F, G, W     )\cr
  &&+\frac{2}{T} \sum_s \E f_s   ^2  \frac{1}{N^2}\sum_{ij}  \lambda_i  l_iw_s\bar D_{\lambda i } ^{-1}          \bar D_{\lambda j } ^{-1} \lambda_i  l_jw_s \E (e_{js}  e_{is}|F)    \Cov( e_{jt}, e_{it}  |F, G, W)\cr
  &=&O_P(N^{-1}).\cr
    \E b&:=&   \frac{1}{T} \sum_s\E g_ s^2( \frac{1}{N}\sum_{i}   x_{is} \lambda_i  e_{it}  \bar D_{\lambda i } ^{-1}\frac{1}{T} \sum_k  f_k  x_{ik} g_k )^2 \cr
   &=&O_P(C_{NT}^{-2}).
    \end{eqnarray*} 
 Therefore the second term is 
 $$
  \frac{1}{\sqrt{N}}\sum_{i}   e_{it} \widetilde D_{\lambda i}^{-1} \frac{1}{T}\widetilde F'\diag(X_i) M_{\widetilde g}\diag(X_i)(FH_1^{-1'}-\widetilde  F)H_1 '\lambda_i =O_P(\sqrt{N}C_{NT}^{-2}).
 $$

Step 7:   $\frac{1}{\sqrt{N}}\sum_{i}   e_{it}\widetilde D_{\lambda i}^{-1} \frac{1}{T}\widetilde F'\diag(X_i) M_{\widetilde g} u_i$.   
\begin{eqnarray*} 
&&\frac{1}{\sqrt{N}}\sum_{i}   e_{it} (\widetilde D_{\lambda i}^{-1}-(H_1^{-1}\bar D_{\lambda i }H_1^{-1'})^{-1})\frac{1}{T}\widetilde F'\diag(X_i) M_{\widetilde g} u_i\cr
&\leq&O_P(\sqrt{N}C_{NT}^{-2})+O_P(1) \frac{1}{\sqrt{N}}\sum_{i}   e_{it} (\widetilde D_{\lambda i}^{-1}-(H_1^{-1}\bar D_{\lambda i }H_1^{-1'})^{-1})\frac{1}{T}F' \diag(X_i) M_{ g} u_i\cr 
&\leq&O_P(\sqrt{N}C_{NT}^{-2})+O_P(\sqrt{N}C_{NT}^{-1}) \frac{1}{N}\sum_{i}  |  e_{it}| |    \frac{1}{T}
\sum_{s}f_sx_{is}u_{is}| \cr 
&&+O_P(\sqrt{N}C_{NT}^{-1}) \frac{1}{N}\sum_{i}  |  e_{it}| |    \frac{1}{T}
\sum_{s}f_sx_{is} g_s||\frac{1}{T}\sum_kg_ku_{ik}| \cr
&=&O_P(\sqrt{N}C_{NT}^{-2}).
    \end{eqnarray*} 
Also,
\begin{eqnarray*} 
&&\frac{1}{\sqrt{N}}\sum_{i}   e_{it}   \bar D_{\lambda i } ^{-1}\frac{1}{T}(\widetilde F-FH_1^{-1'})'\diag(X_i) M_{\widetilde g} u_i\cr 
&\leq& \frac{1}{\sqrt{N}}\sum_{i}   e_{it}   \bar D_{\lambda i } ^{-1}\frac{1}{T}(\widetilde F-FH_1^{-1'})'\diag(X_i) M_{  g} u_i\cr
&&
+O_P(\sqrt{N}C_{NT}^{-1}) (\frac{1}{NT}\sum_{i}  e_{it}^2   \|u_i\| ^2  )^{1/2} (\frac{1}{NT}\sum_i\|  (\widetilde F-FH_1^{-1'})'\diag(X_i)\| ^2)^{1/2}\cr
&\leq& O_P(\sqrt{N}C_{NT}^{-2})+O_P( \sqrt{N}C_{NT}^{-1})    \|   \frac{1}{N} \sum_{i} \frac{1}{\sqrt{T}} \diag(X_i) M_{  g} u_i   e_{it}   \bar D_{\lambda i } ^{-1}\|\cr
&\leq& O_P(\sqrt{N}C_{NT}^{-2})+O_P( \sqrt{N}C_{NT}^{-1})   (\frac{1}{T} \sum_s (   \frac{1}{N} \sum_{i} x_{is} M_{  g,s}' u_i   e_{it}   \bar D_{\lambda i } ^{-1})^2)^{1/2}
\cr
&=&O_P(\sqrt{N}C_{NT}^{-2}),
    \end{eqnarray*} 
    where the last equality is due to $\frac{1}{T} \sum_s (   \frac{1}{N} \sum_{i} x_{is} M_{  g,s}' u_i   e_{it}   \bar D_{\lambda i } ^{-1})^2=O_P(C_{NT}^{-2})$, proved as follows:
    \begin{eqnarray*} 
    &&\frac{1}{T} \sum_s \E(   \frac{1}{N} \sum_{i} x_{is} M_{  g,s}' u_i   e_{it}   \bar D_{\lambda i } ^{-1})^2\cr
&\leq& \frac{1}{T} \sum_s \E(   \frac{1}{N} \sum_{i} x_{is}   u_{is}   e_{it}   \bar D_{\lambda i } ^{-1})^2+ \frac{1}{T} \sum_s \E(   \frac{1}{NT} \sum_{i} \sum_kx_{is} g_kg_su_{ik}   e_{it}   \bar D_{\lambda i } ^{-1})^2\cr
&\leq&O(T^{-1})+ \frac{1}{T} \sum_{s\neq t}   \frac{1}{N^2} 
  \sum_{ij} \E|\E (e_{jt}  e_{it} |F) \E (u_{is}    u_{js}   x_{js}  x_{is} |F) |
\cr
&&+ \frac{1}{T} \sum_s \E  \frac{1}{NT} \sum_{i} \sum_k \bar D_{\lambda i } ^{-1} 
\frac{1}{NT} \sum_{j} \sum_l  \bar D_{\lambda j } ^{-1}|x_{js} g_lg_su_{jl}   x_{is} g_kg_su_{ik} | |\Cov(  e_{it} , e_{jt}  |F )|\cr
&\leq& O(T^{-1})+O(N^{-1})
        \end{eqnarray*} 
    where the last equality is due to  $\max_j\sum_i|\Cov(  e_{it} , e_{jt}  |F )|<C$.

    Next,
     \begin{eqnarray*} 
&&\frac{1}{\sqrt{N}}\sum_{i}   e_{it}   \bar D_{\lambda i } ^{-1}\frac{1}{T} F'\diag(X_i) (M_{\widetilde g} -M_g)u_i \cr
&=&\tr\frac{1}{\sqrt{N}}\sum_{i}   u_ie_{it}   \bar D_{\lambda i } ^{-1}\frac{1}{T} F'\diag(X_i) (M_{\widetilde g} -M_g)\cr
&\leq& O_P(\sqrt{N}C_{NT}^{-1})\frac{1}{T}
\|\frac{1}{N}\sum_{i}   u_ie_{it}   \bar D_{\lambda i } ^{-1} F'\diag(X_i) \|_F\cr
&\leq& O_P(\sqrt{N}C_{NT}^{-1})
(\frac{1}{T^2}\sum_{sk} \E(\frac{1}{N}\sum_{i}   u_{is}e_{it}   \bar D_{\lambda i } ^{-1} f_kx_{ik} )^2)^{1/2}=O_P(\sqrt{N}C_{NT}^{-2})
    \end{eqnarray*} 
    where the last equality is due to $\frac{1}{T^2}\sum_{sk} \E(\frac{1}{N}\sum_{i}   u_{is}e_{it}   \bar D_{\lambda i } ^{-1} f_kx_{ik} )^2=O_P(C_{NT}^{-2})$.
      \begin{eqnarray*} 
&&\frac{1}{T^2}\sum_{sk} \E(\frac{1}{N}\sum_{i}   u_{is}e_{it}   \bar D_{\lambda i } ^{-1} f_kx_{ik} )^2\cr
&=& 
\frac{1}{T^2}\sum_{sk} \E \frac{1}{N}\sum_{i}   \bar D_{\lambda i } ^{-1}
\frac{1}{N}\sum_{j}      \bar D_{\lambda j} ^{-1} f_k^2x_{jk}  u_{is}  x_{ik}u_{js}e_{jt}e_{it} \cr
&\leq &  O(T^{-1}) +
\frac{1}{T^2}\sum_{sk}  \frac{1}{N^2}\sum_{ij}   \E\bar D_{\lambda i } ^{-1}
    \bar D_{\lambda j} ^{-1} f_k^2\E( x_{jk}  u_{is}  x_{ik}u_{js}|F)\Cov (e_{jt}, e_{it} |F)\cr
    &\leq & O_P(C_{NT}^{-2}).
    \end{eqnarray*}

Step 8:  since $u_{it}$ is conditionally serially independent given $E, F $, 
   \begin{eqnarray*} 
&&\E(\frac{1}{\sqrt{N}}\sum_{i}   e_{it}   \bar D_{\lambda i } ^{-1}\frac{1}{T} F'\diag(X_i) M_gu_i )^2\cr
&=&\E \frac{1}{ N}\sum_{ij} \sum_{s}   \bar D_{\lambda j} ^{-1} \bar D_{\lambda i } ^{-1}\frac{1}{T} (F'\diag(X_i) M_{g,s})^2
    e_{jt}  e_{it}  \frac{1}{T}  u_{js}u_{is} \cr
    &\leq&\frac{C}{T^2 N}\sum_{ij} \sum_{s}       \E u_{js}u_{is}  \E ((F'\diag(X_i) M_{g,s})^2    e_{jt}  e_{it} |U)\cr
     &\leq&\frac{C}{T^2 N}\sum_{ij} \sum_{s}    |   \Cov( u_{js}, u_{is} ) |  =O(T^{-1})=O_P(C_{NT}^{-2}).
       \end{eqnarray*}

Together, $\frac{1}{\sqrt{N}}\sum_{i}  (H_1' \lambda_i-\dot\lambda_i)e_{it}=O_P(\sqrt{T}C_{NT}^{-2}) .$

(ii) The proof is the same as that of part (i), by substituting in the expansion of $H_1' \lambda_i-\dot\lambda_i$, hence we ignore it for brevity. 

\end{proof}

    \begin{lem}\label{lc.2}   
    
  For any bounded determinisitic sequence $c_i,$\\
  $
  \frac{1}{|\mathcal G|_0}\sum_{i\in\mathcal G} \|\frac{1}{T}\sum_{s\notin I}  f_sw_s'e_{is} c_{i }'(\widetilde f_s-H_1^{-1}f_s )\|^2  = O_P(C_{NT}^{-4}).
  $
     \end{lem}
     
     \begin{proof}
  For $y_s= Ag_s+\diag(X_s) \Lambda f_s+u_s$,
\begin{eqnarray}\label{ec.10}
\widetilde f_s&=& \widetilde D_{fs}^{-1} \frac{1}{N}\widetilde\Lambda'\diag(X_s) M_{\widetilde\alpha}y_s\cr
&=&  H_1^{-1} f_s+ \widetilde D_{fs}^{-1} \frac{1}{N}\widetilde\Lambda'\diag(X_s) M_{\widetilde\alpha}(AH_2-\widetilde A)H_2^{-1}g_s
\cr
&& 
+ \widetilde D_{fs}^{-1} \frac{1}{N}\widetilde\Lambda'\diag(X_s) M_{\widetilde\alpha}\diag(X_s)(\Lambda H_1-\widetilde \Lambda)H_1^{-1} f_s \cr
&&+ \widetilde D_{fs}^{-1} \frac{1}{N}\widetilde\Lambda'\diag(X_s) M_{\widetilde\alpha} u_s  .
\end{eqnarray}
 Without loss of generality, we assume $\dim(g_s)=\dim(f_s)=1$, as we can always work with their elements given that the number of factors is fixed. It suffices to prove, for $h_s:=f_sw_s$, 
$$
 \frac{1}{|\mathcal G|_0}\sum_{i\in\mathcal G} \|\frac{1}{T}\sum_{s\notin I}  h_se_{is}  (\widetilde f_s-H_1^{-1}f_s )\|^2  = O_P(C_{NT}^{-4}). 
$$
We  plug in  each term in the   expansion of $\widetilde f_s$:
\begin{eqnarray*}
&& \frac{1}{|\mathcal G|_0}\sum_{i\in\mathcal G} \|\frac{1}{T}\sum_{s\notin I}  h_se_{is}  \widetilde D_{fs}^{-1} \frac{1}{N}\widetilde\Lambda'\diag(X_s) M_{\widetilde\alpha}(AH_2-\widetilde A)H_2^{-1}g_s \|^2  \cr
&&+  \frac{1}{|\mathcal G|_0}\sum_{i\in\mathcal G} \|\frac{1}{T}\sum_{s\notin I}  h_se_{is}  \widetilde D_{fs}^{-1} \frac{1}{N}\widetilde\Lambda'\diag(X_s) M_{\widetilde\alpha}\diag(X_s)(\Lambda H_1-\widetilde \Lambda)H_1^{-1} f_s   \|^2\cr
&&+  \frac{1}{|\mathcal G|_0}\sum_{i\in\mathcal G} \|\frac{1}{T}\sum_{s\notin I}  h_se_{is}   \widetilde D_{fs}^{-1} \frac{1}{N}\widetilde\Lambda'\diag(X_s) M_{\widetilde\alpha} u_s   \|^2.
\end{eqnarray*}
     
     First term:  $ \frac{1}{|\mathcal G|_0}\sum_{i\in\mathcal G} \|\frac{1}{T}\sum_{s\notin I}  h_se_{is}  \widetilde D_{fs}^{-1} \frac{1}{N}\widetilde\Lambda'\diag(X_s) M_{\widetilde\alpha}(AH_2-\widetilde A)H_2^{-1}g_s \|^2$.  By Lemma \ref{lc.1},
     \begin{eqnarray*}
&&  \frac{1}{|\mathcal G|_0}\sum_{i\in\mathcal G} \|\frac{1}{T}\sum_{s\notin I}  h_se_{is}  (\widetilde D_{fs}^{-1}- (H_1'\bar D_{fs}H_1)^{-1}) \frac{1}{N}\widetilde\Lambda'\diag(X_s) M_{\widetilde\alpha}(AH_2-\widetilde A)H_2^{-1}g_s \|^2
\cr
&\leq&O_P( C_{NT}^{-4})   \frac{1}{|\mathcal G|_0}\sum_{i\in\mathcal G} 
\frac{1}{T}\sum_{s} \|   \frac{1}{\sqrt{N}}\widetilde\Lambda'\diag(X_s)   \|^2\|g_s\|^2\|h_se_{is}\|^2\cr
&\leq&O_P( C_{NT}^{-4})   \frac{1}{|\mathcal G|_0}\sum_{i\in\mathcal G} 
\frac{1}{T}\sum_{s} \|   \frac{1}{\sqrt{N}}\widetilde\Lambda'\diag(X_s)   \|^2\|g_s\|^2\|h_se_{is}\|^2=O_P( C_{NT}^{-4}) .\cr
&&  \frac{1}{|\mathcal G|_0}\sum_{i\in\mathcal G} \|\frac{1}{T}\sum_{s\notin I}  h_se_{is}   (H_1'\bar D_{fs}H_1)^{-1} \frac{1}{N}\widetilde\Lambda'\diag(X_s) M_{\widetilde\alpha}(AH_2-\widetilde A)H_2^{-1}g_s \|^2
\cr
 &\leq&O_P(C_{NT}^{-2})\frac{1}{|\mathcal G|_0}\sum_{i\in\mathcal G} \|\frac{1}{T\sqrt{N}}\sum_{s\notin I} h_sg_se_{is}(H_1'\bar D_{fs}H_1)^{-1} \widetilde\Lambda'\diag(X_s)\| ^2\cr
    &\leq&O_P(C_{NT}^{-2})\frac{1}{|\mathcal G|_0}\sum_{i\in\mathcal G}  \frac{1}{N}\sum_j\widetilde\lambda_j^2[(\frac{1}{T}\sum_{s\notin I} \bar D_{fs}^{-1} h_sg_se_{is}e_{js})^2+(\frac{1}{T}\sum_{s\notin I} \bar D_{fs}^{-1} h_sg_se_{is}w_{s})^2].
\end{eqnarray*}
     Because $e_{is}$ is serially independent conditionally on $(W, F, G)$,  s the above is 
         \begin{eqnarray*}
    &\leq&O_P(C_{NT}^{-2})\frac{1}{|\mathcal G|_0}\sum_{i\in\mathcal G}[  \E_I   (\frac{1}{T}\sum_{s\notin I} \bar D_{fs}^{-1} h_sg_se_{is}e_{js})^2+  \E_I   (\frac{1}{T}\sum_{s\notin I} \bar D_{fs}^{-1} h_sg_se_{is}w_{s})^2]\cr
    &=&O_P(C_{NT}^{-4}).
\end{eqnarray*}
     Put together the first term is $O_P(C_{NT}^{-4}).$
     
     Second term: $  \frac{1}{|\mathcal G|_0}\sum_{i\in\mathcal G} \|\frac{1}{T}\sum_{s\notin I}  h_se_{is}  \widetilde D_{fs}^{-1} \frac{1}{N}\widetilde\Lambda'\diag(X_s) M_{\widetilde\alpha}\diag(X_s)(\Lambda H_1-\widetilde \Lambda)H_1^{-1} f_s   \|^2$.
              \begin{eqnarray*}
     && \frac{1}{|\mathcal G|_0}\sum_{i\in\mathcal G} \|\frac{1}{T}\sum_{s\notin I}  h_se_{is} ( \widetilde D_{fs}^{-1}- (H_1'\bar D_{fs}H_1)^{-1}) \frac{1}{N}\widetilde\Lambda'\diag(X_s) M_{\widetilde\alpha}\diag(X_s)(\Lambda H_1-\widetilde \Lambda)H_1^{-1} f_s   \|^2\cr
     &\leq&O_P( C_{NT}^{-2})  \frac{1}{|\mathcal G|_0}\sum_{i\in\mathcal G} \frac{1}{T}\sum_{s} \|   \frac{1}{\sqrt{N}}\widetilde\Lambda'\diag(X_s) M_{\widetilde\alpha}\diag(X_s)  \|^2\|h_s\|^2\|f_se_{is}\|^2 = O_P( C_{NT}^{-2}) .
\end{eqnarray*}
     The same proof leads to 
         \begin{eqnarray*}
     && \frac{1}{|\mathcal G|_0}\sum_{i\in\mathcal G} \|\frac{1}{T}\sum_{s\notin I}  h_se_{is}  (H_1'\bar D_{fs}H_1)^{-1} \frac{1}{N}\widetilde\Lambda'\diag(X_s) (M_{\widetilde\alpha}- M_\alpha)\diag(X_s)(\Lambda H_1-\widetilde \Lambda)H_1^{-1} f_s   \|^2\cr
     &\leq&  O_P( C_{NT}^{-2}) .
\end{eqnarray*}
 Now let  $M_{\widetilde\alpha, ij}$   be  the $(i,j)$ th component of $M_{\widetilde\alpha}$.
Let $z_{js}=  \sum_k\widetilde\lambda_k  M_{ \alpha,kj} x_{ks}$, 
       \begin{eqnarray*}
     && \frac{1}{|\mathcal G|_0}\sum_{i\in\mathcal G} \|\frac{1}{T}\sum_{s\notin I}  h_se_{is}  (H_1'\bar D_{fs}H_1)^{-1} \frac{1}{N}\widetilde\Lambda'\diag(X_s)  M_\alpha\diag(X_s)(\Lambda H_1-\widetilde \Lambda)H_1^{-1} f_s   \|^2\cr
     &\leq&  O_P( C^{-2}_{NT})  \frac{1}{|\mathcal G|_0}\sum_{i\in\mathcal G}\|
     \frac{1}{{T}}\sum_{s} f_sh_se_{is}  \bar D_{fs} ^{-1}\frac{1}{\sqrt{N}}\widetilde\Lambda'\diag(X_s) M_{ \alpha}\diag(X_s)  \|^2\cr
      &\leq&  O_P( C^{-2}_{NT})   \frac{1}{|\mathcal G|_0}\sum_{i\in\mathcal G}\frac{1}{N}  \sum_j(
\frac{1}{{T}}\sum_{s} f_sh_s \bar D_{fs} ^{-1}e_{is}  z_{js}x_{js})^2  \cr
 &\leq&  O_P( C^{-2}_{NT})   \frac{1}{|\mathcal G|_0}\sum_{i\in\mathcal G}\frac{1}{N}  \sum_j      (
\frac{1}{{T}}\sum_{s} \E_I(f_sh_s\bar D_{fs}^{-1}   e_{is}x_{js}^2)^2( \widetilde\lambda_j  M_{ \alpha,jj} )^2    \cr
 &&+  O_P( C^{-2}_{NT})   \frac{1}{|\mathcal G|_0}\sum_{i\in\mathcal G}\frac{1}{N}  \sum_j   \frac{1}{N} \sum_{k\neq j} (  
\frac{1}{{T}}\sum_{s} \E_If_sh_s\bar D_{fs}^{-1}x_{ks}e_{is}x_{js})^2\cr
 &&+
  O_P( C^{-2}_{NT})   \frac{1}{|\mathcal G|_0}\sum_{i\in\mathcal G}\frac{1}{N}  \sum_j
  \Var_I(
\frac{1}{{T}}\sum_{s} f_sh_s\bar D_{fs}^{-1}e_{is} z_{js}x_{js})  \cr
&\leq&
O_P( C^{-2}_{NT})   \frac{1}{|\mathcal G|_0}\sum_{i\in\mathcal G}\frac{1}{N}  \sum_j (
\frac{1}{{T}}\sum_{s} \E_I(f_sh_s\bar D_{fs}^{-1}   e_{is}x_{js}^2)^2( \widetilde\lambda_j  -H_1'\lambda_j )^2
\cr
&&+ O_P( C^{-2}_{NT})   \frac{1}{|\mathcal G|_0}\sum_{i\in\mathcal G}\frac{1}{N}  \sum_j (
\frac{1}{{T}}\sum_{s} \E_If_sh_s\bar D_{fs}^{-1}   e_{is}e_{js}^2)^2 \cr
&&+ O_P( C^{-2}_{NT})   \frac{1}{|\mathcal G|_0}\sum_{i\in\mathcal G}\frac{1}{N}  \sum_j (
\frac{1}{{T}}\sum_{s} \E_If_sh_s\bar D_{fs}^{-1}   e_{is}l_jw_{s}e_{js})^2 \cr
&&+O_P( C^{-2}_{NT})   \frac{1}{|\mathcal G|_0}\sum_{i\in\mathcal G}\frac{1}{N}  \sum_j  \frac{1}{N} \sum_{k\neq j} (  
\frac{1}{{T}}\sum_{s} \E_If_sh_s\bar D_{fs}^{-1}e_{ks}e_{is}e_{js})^2
\cr
&&+O_P( C^{-2}_{NT})   \frac{1}{|\mathcal G|_0}\sum_{i\in\mathcal G}\frac{1}{N}  \sum_j   \frac{1}{N} \sum_{k\neq j} (  
\frac{1}{{T}}\sum_{s} \E_If_sh_s\bar D_{fs}^{-1}e_{ks}e_{is}l_jw_{s})^2\cr
&&+O_P( C^{-2}_{NT})   \frac{1}{|\mathcal G|_0}\sum_{i\in\mathcal G}\frac{1}{N}  \sum_j  \frac{1}{N} \sum_{k\neq j} (  
\frac{1}{{T}}\sum_{s} \E_If_sh_s\bar D_{fs}^{-1}l_kw_{s}e_{is}e_{js})^2+O_P( C^{-4}_{NT}) =O_P( C^{-4}_{NT}) 
\end{eqnarray*}
 given that  $\sum_j|  \E_I(e_{is}e_{js}|f_s, w_s)|+ \frac{1}{N} \sum_{k\neq j}  |\E_I(e_{ks}e_{is}e_{js}|f_s, w_s)|<\infty$.
  Put together,  the second term is $O_P(C_{NT}^{-4}).$
     
     Third term: $ \frac{1}{|\mathcal G|_0}\sum_{i\in\mathcal G} \|\frac{1}{T}\sum_{s\notin I}  h_se_{is}   \widetilde D_{fs}^{-1} \frac{1}{N}\widetilde\Lambda'\diag(X_s) M_{\widetilde\alpha} u_s   \|^2$.
     \begin{eqnarray*}
  && \frac{1}{|\mathcal G|_0}\sum_{i\in\mathcal G} \| \frac{1}{ T}\sum_{s\notin I } h_se_{is}( \widetilde D_{fs}^{-1} -(H_1'\bar D_{fs}H_1)^{-1})\frac{1}{N}\widetilde\Lambda'\diag(X_s) M_{\widetilde\alpha} u_s\|^2\cr
  &\leq& O_P( C_{NT}^{-2})\frac{1}{|\mathcal G|_0}\sum_{i\in\mathcal G}   \frac{1}{T}\sum_{s\notin I }| h_se_{is}\frac{1}{N}\widetilde\Lambda'\diag(X_s) M_{\widetilde\alpha} u_s|^2     
\cr
  &\leq & O_P( C_{NT}^{-2})\frac{1}{|\mathcal G|_0}\sum_{i\in\mathcal G} \frac{1}{T}\sum_{s\notin I }\E_I| h_se_{is}\frac{1}{N}\widetilde\Lambda'\diag(X_s) M_{\widetilde\alpha} u_s|^2 \cr
   &=&O_P( C_{NT}^{-2})\frac{1}{|\mathcal G|_0}\sum_{i\in\mathcal G}\frac{1}{T}\sum_{s\notin I }\E_I h_s^2e_{is}^2\frac{1}{N^2}\widetilde\Lambda'\diag(X_s) M_{\widetilde\alpha} \E_I(u_su_s'| X_{s}, f_s)    M_{\widetilde\alpha} \diag(X_s)   \widetilde\Lambda \cr
   &\leq& O_P( C_{NT}^{-2}N^{-1})\frac{1}{|\mathcal G|_0}\sum_{i\in\mathcal G}  \frac{1}{T}\sum_{s\notin I }\E_I h_s^2e_{is}^2\frac{1}{N}\|\widetilde\Lambda'\diag(X_s)    \|^2 \cr
   &=&O_P( C_{NT}^{-2}N^{-1})=O_P( C_{NT}^{-4}).
                  \end{eqnarray*}

      Next, due to $\E(u_s|f_s, D_I, e_s)=0$, and $e_s$ is conditionally serially independent given $(f_s, u_s)$,
        \begin{eqnarray*}
  && \frac{1}{|\mathcal G|_0}\sum_{i\in\mathcal G} \| \frac{1}{{T}}\sum_{s\notin I } h_se_{is}(H_1'\bar D_{fs}H_1)^{-1}\frac{1}{N}\widetilde\Lambda'\diag(X_s) (M_{\widetilde\alpha} -M_\alpha)u_s\|^2\cr
        &=& \frac{1}{|\mathcal G|_0}\sum_{i\in\mathcal G}  [\tr \frac{1}{T}\sum_{s\notin I } u_sh_se_{is}(H_1'D_{x}H_1)^{-1}\frac{1}{N}\widetilde\Lambda'\diag(X_s) (M_{\widetilde\alpha}-M_\alpha) ]^2\cr
        &\leq&O_P( C_{NT}^{-2}) \frac{1}{|\mathcal G|_0}\sum_{i\in\mathcal G} \| \frac{1}{T}\sum_{s\notin I } u_sh_se_{is}(H_1'  \bar D_{fs}^{-1}  H_1)^{-1}\frac{1}{N}\widetilde\Lambda'\diag(X_s) \|_F^2  \cr
              &\leq&O_P(C_{NT}^{-2})\frac{1}{|\mathcal G|_0}\sum_{i\in\mathcal G}   \frac{1}{N}\sum_j \E_I\| \frac{1}{T}\sum_{s\notin I } \bar D_{fs}^{-1}u_sh_se_{is} \frac{1}{\sqrt{N}}x_{js} \|_F ^2\widetilde\lambda_j^2  \cr
                   &\leq&O_P( C_{NT}^{-2})\max_{ij} \E_I\E_I(\| \frac{1}{T}\sum_{s\notin I } \bar D_{fs}^{-1}u_sh_se_{is} \frac{1}{\sqrt{N}}x_{js} \|_F ^2 |X, F, W) \cr
       &\leq&O_P( C_{NT}^{-2})\max_{ijk}  \frac{1}{T}\Var_I(  \bar D_{fs}^{-1}u_{ks}h_se_{is}  x_{js} )    \leq O_P( C_{NT}^{-2}T^{-1})=O_P( C_{NT}^{-4}).
                     \end{eqnarray*}
                       Next, 
                           \begin{eqnarray*}
  &&  \frac{1}{|\mathcal G|_0}\sum_{i\in\mathcal G}  \|\frac{1}{{T}}\sum_{s\notin I } h_se_{is}(H_1'\bar D_{fs}H_1)^{-1}\frac{1}{N}(\widetilde\Lambda-\Lambda H_1)'\diag(X_s)  M_\alpha u_s\|^2 \cr
    &\leq&O_P(1)   \frac{1}{|\mathcal G|_0}\sum_{i\in\mathcal G}\|   \frac{1}{N}\sum_j(\widetilde\lambda_j-H_1'\lambda_j)\frac{1}{{T}}\sum_{s\notin I } \bar D_{fs}^{-1}h_se_{is}M_{\alpha, j} ' u_sx_{js}  \|^2\cr
        &\leq&O_P(C_{NT}^{-2})    \frac{1}{|\mathcal G|_0}\sum_{i\in\mathcal G}   \frac{1}{N}\sum_j (\frac{1}{{T}}\sum_{s\notin I } \bar D_{fs}^{-1}h_se_{is}M_{\alpha, j} ' u_sx_{js}  )^2 \cr
                &\leq&O_P(C_{NT}^{-2}T^{-1})  \frac{1}{|\mathcal G|_0}\sum_{i\in\mathcal G}   \frac{1}{N}\sum_j \frac{1}{T}\sum_{s\notin I }\E_I \bar D_{fs}^{-2}h_s^2x_{js} ^2e_{is}^2 M_{\alpha, j} ' \Var_I(u_s |e_s, w_s, f_s) M_{\alpha, j}  \cr
      &\leq &O_P(C_{NT}^{-2}T^{-1})   \frac{1}{|\mathcal G|_0}\sum_{i\in\mathcal G}   \frac{1}{N}\sum_j \frac{1}{T}\sum_{s\notin I }\E_I \bar D_{fs}^{-2}h_s^2x_{js} ^2e_{is}^2 \|M_{\alpha, j}\|^2                  \cr
      &=&O_P(C_{NT}^{-2}T^{-1})=O_P(C_{NT}^{-4} ) .
                     \end{eqnarray*}
                
      Finally, due to the conditional serial independence,
           \begin{eqnarray*}
  &&   \frac{1}{|\mathcal G|_0}\sum_{i\in\mathcal G} \|\frac{1}{{T}}\sum_{s\notin I } h_se_{is}(H_1'\bar D_{fs}H_1)^{-1}\frac{1}{N}H_1 '\Lambda'\diag(X_s)  M_\alpha u_s\|^2\cr
  &\leq&O_P(N^{-2}T^{-2}) \sum_{s\notin I }\E_I \bar D_{fs}^{-2}h_s^2e_{is}^2  \Lambda'\diag(X_s)  M_\alpha \E_I(u_s
  u_s'|E, G, F)M_\alpha \diag(X_s)\Lambda 
 \cr
   &\leq&O_P(N^{-2}T^{-1}) \E_I \bar D_{fs}^{-2}h_s^2e_{is}^2 \| \Lambda'\diag(X_s)  \| ^2 =O_P(N^{-1}T^{-1})= O_P( C_{NT}^{-4}) .
                       \end{eqnarray*}
  Put together,  the third term is $O_P(C_{NT}^{-4}).$
     This finishes the proof. 
     \end{proof}

  \subsection{Technical lemmas for $\widehat f_t$  }

\begin{lem}\label{ld.6}  Assume $\frac{1}{N}\sum_{ij} |\Cov(e_{it}^2, e_{jt}^2)|<C$. 
For each fixed $t$,\\  (i) $\widehat B_t-B=O_P(C_{NT}^{-1})$.\\
(ii) The upper  two blocks of  $\widehat B_t^{-1}\widehat S_t-B^{-1}S$ are both $O_P(C_{NT}^{-2})$.
\end{lem}

\begin{proof} Throughout the proof, we assume $\dim(\alpha_i)=\dim(\lambda_i)=1$ without loss of generality. 

(i) $\widehat B_t-B= b_1+b_2$, where 
\begin{eqnarray*}
b_1&=&\frac{1}{N}\sum_i\begin{pmatrix}
 \widetilde\lambda_i  \widetilde\lambda_i ' \widehat e_{it}^2 -H_1'\lambda_i\lambda_i'H_1e_{it}^2&  \widetilde\lambda_i \widetilde\alpha_i' \widehat e_{it} - H_1'\lambda_i\alpha_iH_2'e_{it}   \\
  \widetilde\alpha_i \widetilde\lambda_i'\widehat e_{it} - H_2'\alpha_i\lambda_iH_1'e_{it}&  \widetilde\alpha_i \widetilde\alpha_i'-H_2'\alpha_i\alpha_i'H_2
\end{pmatrix}\cr
 b_2 &=&
 \frac{1}{N}\sum_i\begin{pmatrix}
 H_1'\lambda_i\lambda_i'H_1(e_{it}^2-\E e_{it}^2)&   H_1'\lambda_i\alpha_iH_2'e_{it}\\
  H_2'\alpha_i\lambda_iH_1'e_{it}&  0\end{pmatrix}.
\end{eqnarray*}
To prove the convergence of $b_1$, first note that 
\begin{eqnarray}\label{ed.17}
&&  
  \frac{1}{N}\sum_i\widetilde\lambda_i  \widetilde\lambda_i ' (\widehat e_{it}^2-e_{it}^2)
  =  \frac{1}{N}\sum_i\widetilde\lambda_i  \widetilde\lambda_i ' (\widehat e_{it}-e_{it}) ^2
  + \frac{2}{N}\sum_i\widetilde\lambda_i  \widetilde\lambda_i ' (\widehat e_{it}-e_{it})e_{it}\cr
  &\leq& 
   \frac{1}{N}\sum_i(\widetilde\lambda_i  -H_1'\lambda_i)(\widetilde\lambda_i-H_1'\lambda_i) ' (\widehat e_{it}-e_{it}) ^2
   +  \frac{1}{N}\sum_i(\widetilde\lambda_i  -H_1'\lambda_i) \lambda_i ' H_1'(\widehat e_{it}-e_{it}) ^2  \cr
   &&
   + O_P(1)\frac{1}{N}\sum_i (\widehat e_{it}-e_{it}) ^2  
   +\frac{2}{N}\sum_i    (\widetilde\lambda_i  -H_1'\lambda_i)(\widetilde\lambda_i-H_1'\lambda_i) '    (\widehat e_{it}-e_{it})e_{it}\cr
   &&   +\frac{2}{N}\sum_i    (\widetilde\lambda_i  -H_1'\lambda_i) \lambda_i'H_1'    (\widehat e_{it}-e_{it})e_{it}
   +O_P(1)\frac{1}{N}\sum_i    \lambda_i  \lambda_i'  (\widehat e_{it}-e_{it})e_{it}\cr
   &\leq& O_P(C_{NT}^{-2}) \max_{it} |\widehat e_{it}-e_{it}|\max_{it}|e_{it}|
 + O_P(C_{NT}^{-1}) (\frac{1}{N}\sum_i (\widehat e_{it}-e_{it}) ^4)^{1/2}\cr
 &&  + O_P(1)\frac{1}{N}\sum_i (\widehat e_{it}-e_{it}) ^2  
 + O_P(C_{NT}^{-1}) (\frac{1}{N}\sum_i (\widehat e_{it}-e_{it}) ^2e_{it}^2)^{1/2}\cr
 &&  +O_P(1)\frac{1}{N}\sum_i    \lambda_i  \lambda_i'  (\widehat e_{it}-e_{it})e_{it}=O_P(C_{NT}^{-2}), 
\end{eqnarray}
by Lemma \ref{ld.1a}.
In addition, still by Lemma \ref{ld.1a},
\begin{eqnarray}
&&\frac{1}{N}\sum_i\widetilde\lambda_i \widetilde\alpha_i' (\widehat e_{it} -e_{it})
=\frac{1}{N}\sum_i(\widetilde\lambda_i-H_1'\lambda_i) (\widetilde\alpha_i-H_2'\alpha_i)' (\widehat e_{it} -e_{it})\cr
&&+\frac{1}{N}\sum_i(\widetilde\lambda_i-H_1'\lambda_i)  \alpha_i' H_2(\widehat e_{it} -e_{it})
+O_P(1)\frac{1}{N}\sum_i\lambda_i \alpha_i' (\widehat e_{it} -e_{it})\cr
&\leq& O_P(C_{NT}^{-2}) \max_{it} |\widehat e_{it}-e_{it}|
+O_P(C_{NT}^{-1}) (\frac{1}{N}\sum_i (\widehat e_{it}-e_{it}) ^2)^{1/2}\cr
&&+O_P(1)\frac{1}{N}\sum_i\lambda_i \alpha_i' (\widehat e_{it} -e_{it})=O_P(C_{NT}^{-2}).
\end{eqnarray}

 So the first term of $b_1$ is, due to the serial independence of $e_{it}^2$, 
\begin{eqnarray*}
&& \frac{1}{N}\sum_i\widetilde\lambda_i  \widetilde\lambda_i ' \widehat e_{it}^2 -H_1'\lambda_i\lambda_i'H_1e_{it}^2\cr
&\leq &
  \frac{1}{N}\sum_i\widetilde\lambda_i  \widetilde\lambda_i ' (\widehat e_{it}^2-e_{it}^2)
  + \frac{1}{N}\sum_i(\widetilde\lambda_i  \widetilde\lambda_i ' - H_1'\lambda_i\lambda_i'H_1) e_{it}^2 
\cr
&\leq& O_P(C_{NT}^{-2})+ O_P(1) \frac{1}{N}\sum_i\|\widetilde\lambda_i  \widetilde\lambda_i ' - H_1'\lambda_i\lambda_i'H_1\|_F \E_Ie_{it}^2 \cr
&\leq& O_P(C_{NT}^{-2})+ O_P(1)\frac{1}{N}\sum_i\|\widetilde\lambda_i  \widetilde\lambda_i ' - H_1'\lambda_i\lambda_i'H_1\|_F =O_P(C_{NT}^{-1}).
\end{eqnarray*}
The second term is,  
 \begin{eqnarray*}
&&   \frac{1}{N}\sum_i\widetilde\lambda_i \widetilde\alpha_i' \widehat e_{it} - H_1'\lambda_i\alpha_iH_2'e_{it} \cr
&\leq& O_P(C_{NT}^{-1})+ O_P(1)\frac{1}{N}\sum_i \|\widetilde\lambda_i \widetilde\alpha_i'   - H_1'\lambda_i\alpha_iH_2' \|_F\E_I|e_{it}|\leq O_P(C_{NT}^{-1}).
\end{eqnarray*}
The third term of $b_1$ is bounded similarly.  The last term of $b_1$ is easy to show to be $O_P(C_{NT}^{-1}).$ 
As for $b_2$, by the assumption that $\frac{1}{N}\sum_{ij} |\Cov(e_{it}^2, e_{jt}^2)|<C$, thus $b_2=O_P(N^{-1/2})$. Hence $\widehat B_t-B=O_P(C_{NT}^{-1}).$

(ii) We have \begin{eqnarray*}
  B_t^{-1}\widehat S_t-  B^{-1} S&=&(B_t^{-1}-  B^{-1}) (\widehat S_t-S)
  +(B_t^{-1}-  B^{-1})  S
  +B^{-1} (\widehat S_t-S).
  \end{eqnarray*}

 We first bound the four blocks of  $\widehat S_t-S$. 
  We have
  $  \widehat S_t-S=c_t+d_t$,
 \begin{eqnarray*}
c_t& =& \frac{1}{N}\sum_i
\begin{pmatrix}
\widetilde \lambda_i   \lambda_i'H_1e_{it}(\widehat e_{it}-e_{it})   +  \widetilde \lambda_i   \widetilde\lambda_i'(\widehat e_{it}^2-e_{it}^2) & 
  \widetilde \lambda_i  (\widehat e_{it} -e_{it})( \alpha_i'H_2-\widetilde\alpha_i')\\
 \widetilde\alpha_i \lambda_i'H_1(e_{it}-\widehat e_{it})
  +\widetilde\alpha_i (\lambda_i'H_1 -\widetilde\lambda_i')(\widehat e_{it} -e_{it})& 
0
\end{pmatrix}\cr
d_t& =& \frac{1}{N}\sum_i
\begin{pmatrix}
( \widetilde \lambda_i   \lambda_i'H_1-  \widetilde \lambda_i   \widetilde\lambda_i'  )e_{it}^2
- H_1'\lambda_i ( \lambda_i'H_1-  \widetilde\lambda_i'  )\E e_{it}^2
 & 
  \widetilde \lambda_i   e_{it} ( \alpha_i'H_2-\widetilde\alpha_i')\\
\widetilde\alpha_i (\lambda_i'H_1 -\widetilde\lambda_i')e_{it} & 
0
\end{pmatrix}.
\end{eqnarray*}
Call each block of $c_t$ to be $c_{t,1}...c_{t,4}$ in the clockwise order. Note that $c_{t,4}=0$.
 
 As for $c_{t,1}$,  it follows from Lemma \ref{ld.1a} that 
\begin{eqnarray*}
&& \frac{1}{N}\sum_i\widetilde \lambda_i   \lambda_i'H_1e_{it}(\widehat e_{it}-e_{it}) \leq(\frac{1}{N}\sum_ie_{it}^2(\widehat e_{it}-e_{it})^2)^{1/2}O_P(C_{NT}^{-1})
\cr
&&+O_P(1)\frac{1}{N}\sum_i \lambda_i   \lambda_ie_{it}(\widehat e_{it}-e_{it}) \leq  O_P(C_{NT}^{-2}).
\end{eqnarray*} We have also shown $\frac{1}{N}\sum_i\widetilde\lambda_i  \widetilde\lambda_i ' (\widehat e_{it}^2-e_{it}^2)=    O_P(C_{NT}^{-2}).    $ Thus $c_{t,1}=O_P(C_{NT}^{-2})$.

 For $c_{t,2}$,  from Lemma \ref{ld.1a},
 \begin{eqnarray*}
&&   \frac{1}{N}\sum_i \widetilde \lambda_i  (\widehat e_{it} -e_{it})( \alpha_i'H_2-\widetilde\alpha_i')\cr
&\leq& O_P(C_{NT}^{-2})\max_i|e_{it}-\widehat e_{it}|
+O_P(C_{NT}^{-1}) (\frac{1}{N}\sum_i   (\widehat e_{it}-e_{it})^2)^{1/2}=O_P(C_{NT}^{-2}).
\end{eqnarray*}
   For the third term of $c_t$,  similarly, $
  \frac{1}{N}\sum_i  \widetilde\alpha_i (\lambda_i'H_1 -\widetilde\lambda_i')(\widehat e_{it} -e_{it})= O_P(C_{NT}^{-2})$.  Also,  by Lemma \ref{ld.1a}, 
  \begin{eqnarray*}
&& \frac{1}{N}\sum_i \widetilde\alpha_i \lambda_i'H_1(e_{it}-\widehat e_{it})\cr
&\leq& O_P(C_{NT}^{-1}) (\frac{1}{N}\sum_i (\widehat e_{it}-e_{it}) ^2)^{1/2}+O_P(1)\frac{1}{N}\sum_i\lambda_i \alpha_i' (\widehat e_{it} -e_{it})\cr
&\leq& O_P(C_{NT}^{-1}) .
\end{eqnarray*}
So $c_{t,3}= O_P(C_{NT}^{-1}) .$

 As for $d_t$, we first prove that $\frac{1}{N}\sum_i \lambda_i ( \lambda_i'H_1-  \widetilde\lambda_i'  )(e_{it}^2 -\E e_{it}^2)=O_P(C_{NT}^{-1}N^{-1/2}).$ Note that $\E_I \frac{1}{N}\sum_i   \lambda_i ( \lambda_i'H_1-  \widetilde\lambda_i'  )(e_{it}^2 -\E e_{it}^2)=0$. Let $\Upsilon_t$ be an $N\times 1$ vector of $e_{it}^2$, and  $\diag(\Lambda)$ be  diagonal matrix consisting of elements of $\Lambda$.  Then 
 $$
 \|\Var(\Upsilon_t)\|\leq\max_i\sum_{j}|\Cov(e_{it}^2, e_{jt}^2)|<C,
 $$
 and 
 $\frac{1}{N}\sum_i \lambda_i ( \lambda_i'H_1-  \widetilde\lambda_i'  )e_{it}^2 = \frac{1}{N}(H_1\Lambda -\widetilde\Lambda) '\diag(\Lambda)\Upsilon_t$. So 
 \begin{eqnarray*}
&& \Var_I(\frac{1}{N}\sum_i \lambda_i ( \lambda_i'H_1-  \widetilde\lambda_i'  )e_{it}^2)=\frac{1}{N^2}
 (H_1\Lambda -\widetilde\Lambda) '\diag(\Lambda) \Var(\Upsilon_t)\diag(\Lambda) (H_1\Lambda -\widetilde\Lambda) \cr
 &\leq& C \frac{1}{N^2}\| H_1\Lambda -\widetilde\Lambda\|_F^2=O_P(C_{NT}^{-2}N^{-1}).
\end{eqnarray*}
This implies $\frac{1}{N}\sum_i \lambda_i ( \lambda_i'H_1-  \widetilde\lambda_i'  )(e_{it}^2 -\E e_{it}^2)=O_P(C_{NT}^{-1}N^{-1/2}).$

 Thus  the first term of $d_t$ is 
  \begin{eqnarray*}
  &&\frac{1}{N}\sum_i( \widetilde \lambda_i   \lambda_i'H_1-  \widetilde \lambda_i   \widetilde\lambda_i'  )e_{it}^2
- H_1'\lambda_i ( \lambda_i'H_1-  \widetilde\lambda_i'  )\E e_{it}^2\cr
&\leq&  \frac{1}{N}\sum_i (\widetilde \lambda_i -H_1'\lambda_i) ( \lambda_i'H_1-  \widetilde\lambda_i'  )e_{it}^2
+   \frac{1}{N}\sum_i  H_1'\lambda_i ( \lambda_i'H_1-  \widetilde\lambda_i'  )(e_{it}^2 -\E e_{it}^2)
\cr
& =&O_P(1)\frac{1}{N}\sum_i (\widetilde \lambda_i -H_1'\lambda_i) ( \lambda_i'H_1-  \widetilde\lambda_i'  )\E_I e_{it}^2+O_P(C_{NT}^{-1}N^{-1/2})=O_P(C_{NT}^{-2}).
\end{eqnarray*}

As for the  second term of $d_t$,   $\frac{1}{N}\sum_i   \widetilde \lambda_i   e_{it} ( \alpha_i'H_2-\widetilde\alpha_i')$, note that
 \begin{eqnarray*}
 &&\frac{1}{N}\sum_i  ( \widetilde \lambda_i  -H_1'\lambda_i) e_{it} ( \alpha_i'H_2-\widetilde\alpha_i')=  O_P(C_{NT}^{-1})(\frac{1}{N}\sum_i  e_{it}^2 ( \alpha_i'H_2-\widetilde\alpha_i')^2)^{1/2}\cr
 &\leq&O_P(C_{NT}^{-1})(\frac{1}{N}\sum_i  ( \alpha_i'H_2-\widetilde\alpha_i')^2\E_I  e_{it}^2)^{1/2}=O_P(C_{NT}^{-2}).
 \end{eqnarray*}
 And, $\E_I\frac{1}{N}\sum_i   \widetilde \lambda_i   e_{it} ( \alpha_i'H_2-\widetilde\alpha_i')=0$. 
    \begin{eqnarray*}
&& \Var_I( \frac{1}{N}\sum_i   \lambda_i   e_{it} ( \alpha_i'H_2-\widetilde\alpha_i'))
=\frac{1}{N^2}\Var_I(( A H_2-\widetilde A)'\diag( \Lambda) e_t)
\cr
&\leq& \frac{1}{N^2} ( A H_2-\widetilde A)'\diag( \Lambda) \Var(e_t) \diag( \Lambda) ( A H_2-\widetilde A)= O_P(C_{NT}^{-2}N^{-1}),
\end{eqnarray*}
implying $\frac{1}{N}\sum_i   \widetilde \lambda_i   e_{it} ( \alpha_i'H_2-\widetilde\alpha_i')=  O_P(C_{NT}^{-2} ).
$ 

Finally  the third term of $d_t$, $  \frac{1}{N}\sum_i\widetilde \lambda_i   e_{it} ( \alpha_i'H_2-\widetilde\alpha_i')$, is bounded similarly.  So $d_t=O_P(C_{NT}^{-2} ). $ 
Put together, we have: $\widehat S_t-S=c_t+d_t= O_P(C_{NT}^{-1}) .$ On the other hand, the upper two blocks of  $\widehat S_t-S$ are $O_P(C_{NT}^{-2})$, determined by  $c_{t,1}, c_{t,2}$ and the upper blocks of $d_t$.  In addition, note that  both $B, S$ are block diagonal matrices, and the diagonal  blocks of $S$  are  $O_P(C_{NT}^{-1})$. Due to
\begin{eqnarray*}
\widehat   B_t^{-1}\widehat S_t-  B^{-1} S&=&(\widehat  B_t^{-1}-  B^{-1}) (\widehat S_t-S)
  +(\widehat  B_t^{-1}-  B^{-1})  S
  +B^{-1} (\widehat S_t-S),
    \end{eqnarray*}
  hence the upper blocks of $\widehat   B_t^{-1}\widehat S_t-  B^{-1} S$ are both $O_P(C_{NT}^{-2})$.

\end{proof}

    \begin{lem}\label{lc.6}       Suppose $\max_{t\leq T}|\frac{1}{N}\sum_i\alpha_i\lambda_i e_{it}|=O_P(1)$,\\ $C_{NT}^{-1} \max_{it}|e_{it}|^2 +\max_{t}\|\frac{1}{N}\sum_i  \lambda_i   \lambda_i ' \bar e_i e_{it} \|_F=o_P(1)$. Then

    (i) $\max_{t\leq T}\|\widehat B_{t}^{-1}\|=O_P(1)$, $\max_{t\leq T}|\widehat B_t-B_t|=o_P(1)$.
    
    (ii) $\frac{1}{T}\sum_{s\notin I} \|\widehat B_{s}^{-1}-B^{-1}\|^2= O_P(C^{-2}_{NT})$.
    
(iii) Write
 $$
 \widehat S_t-  S=\begin{pmatrix}
\Delta_{t1}\\
 \Delta_{t2}
 \end{pmatrix},
 $$
 whose partition matches with that of $(f_t', g_t')'$. Then 
 $\frac{1}{T}\sum_{t\notin I} \|\Delta_{t1}\|^2=O_P(C^{-4}_{NT})$ and  $\frac{1}{T}\sum_{t\notin I} \| \Delta_{t2}\|^2=O_P(C^{-2}_{NT})$.
 
(iv) $\max_{t\leq T}\|\widehat S_t\|=O_P(1)$ and $\max_{t\leq T}\|\widehat S_t- S\|=o_P(1)$.
    \end{lem}

\begin{proof}
Define
$$
B_t=\frac{1}{N}\sum_i\begin{pmatrix}
H_1'\lambda_i\lambda_i'H_1  e_{it}^2 &0 \\
0& H_2'\alpha_i\alpha_i'H_2
\end{pmatrix}.
$$
Then $\widehat B_t-B_t=b_{1t}+b_{2t}$, 
\begin{eqnarray*}
b_{1t}&=&\frac{1}{N}\sum_i\begin{pmatrix}
 \widetilde\lambda_i  \widetilde\lambda_i ' \widehat e_{it}^2 -H_1'\lambda_i\lambda_i'H_1e_{it}^2&  \widetilde\lambda_i \widetilde\alpha_i' \widehat e_{it} - H_1'\lambda_i\alpha_iH_2'e_{it}   \\
  \widetilde\alpha_i \widetilde\lambda_i'\widehat e_{it} - H_2'\alpha_i\lambda_iH_1'e_{it}&  \widetilde\alpha_i \widetilde\alpha_i'-H_2'\alpha_i\alpha_i'H_2
\end{pmatrix}\cr
 b_{2t} &=&
 \frac{1}{N}\sum_i\begin{pmatrix}
0&   H_1'\lambda_i\alpha_iH_2'e_{it}\\
  H_2'\alpha_i\lambda_iH_1'e_{it}&  0\end{pmatrix}.
\end{eqnarray*}
(i) We now show $\max_{t\leq T}|b_{1t}|=o_P(1)$. In addition, by assumption $\max_{t\leq T}|b_{2t}|=o_P(1)$.
Thus $\max_{t\leq T}|\widehat B_t-B_t|=o_P(1)$, and thus $\max_{t\leq T}\|\widehat B_{t}^{-1}\|=o_P(1)+\max_{t\leq T}\| B_{t}^{-1}\|=O_P(1)$, by the assumption that $\|B_t^{-1}\|=O_P(1)$ uniformly in $t$. To show $\max_{t\leq T}|b_{1t}|=o_P(1)$, note that: 

First term:
 \begin{eqnarray*}
&&\max_{t}\|\frac{1}{N}\sum_i \widetilde\lambda_i  \widetilde\lambda_i ' (\widehat e_{it}^2-e_{it}^2)\|
\cr
&\leq&  
 O_P(1)(\frac{1}{N}\sum_i  \|\widetilde\lambda_i-H_1'\lambda_i\|^2)^{1/2}\max_{t}[\frac{1}{N}\sum_i    (\widehat e_{it}-e_{it} )^4+(\widehat e_{it}-e_{it} )^2e_{it}^2]^{1/2}\cr
 &&+
 \frac{1}{N}\sum_i \|\widetilde\lambda_i-H_1'\lambda_i\|^2 [2\max_{t} |(\widehat e_{it}-e_{it} )e_{it}  | 
 + \max_{t\leq T}(\widehat e_{it}-e_{it} )^2]
 \cr
&&   +O_P(1)\max_{t}\|\frac{1}{N}\sum_i  \lambda_i   \lambda_i ' (\widehat e_{it}-e_{it} )e_{it} \|_F\cr
 &&+O_P(1)\max_{t}\|\frac{1}{N}\sum_i  \lambda_i  \lambda_i ' (\widehat e_{it}-e_{it} )^2\|_F\cr
 &\leq& O_P(1+\max_{t\leq T}\| w_t\|^2+b_{NT,2}^2)C_{NT}^{-2}  + O_P(b_{NT,1}^2+b_{NT,3}^2)   \cr
 &&+ O_P(1+\max_{t\leq T}\| w_t\|+  b_{NT,2}   )\max_{it}|e_{it}|C_{NT}^{-1} +O_P(b_{NT,1}+b_{NT,3}) (\max_{t\leq T}\frac{1}{N}\sum_ie_{it}^2)^{1/2}C_{NT}^{-1} \cr
 &&+ O_P(\phi_{NT}^2C_{NT}^{-2}) +\phi_{NT} \max_{it}|e_{it}|C_{NT}^{-2} 
 + 
   \max_{t\leq T}\|\frac{1}{N}\sum_ic_i e_{it}      l_i \|_FO_P(b_{NT,1}+b_{NT,3})  \cr
   &=&o_P(1),
\end{eqnarray*}
given   assumptions
$C_{NT}^{-1}(b_{NT,2}  +  \max_{t\leq T}\|w_t\| ) \max_{it}|e_{it}|=o_P(1)$,   $ \phi_{NT}\max_{it}|e_{it}|=O_P(1)$,    $\max_{t\leq T}\|\frac{1}{N}\sum_ie_{it}\alpha_i\lambda_i'\|_F=o_P(1)$, 
and 
 $(b_{NT,1}+b_{NT,3}) [(\max_{t\leq T}\frac{1}{N}\sum_ie_{it}^2)^{1/2}C_{NT}^{-1} +1]=o_P(1)$.

In addition, $
\max_{t\leq T}\|\frac{1}{N}\sum_i( \widetilde\lambda_i  \widetilde\lambda_i '   -H_1'\lambda_i\lambda_i'H_1)e_{it}^2 \|
\leq O_P(C_{NT}^{-1})\max_{it}e_{it}^2=o_P(1).
$ So the first term of $\max_{t\leq T}|b_{1t}|$ is $o_P(1)$.

Second term,    \begin{eqnarray*}
&&\max_{t}\| \frac{1}{N}\sum_i \widetilde\lambda_i \widetilde\alpha_i' (\widehat e_{it}-e_{it})\|_F\cr
&\leq&\max_{it}|\widehat e_{it}-e_{it}| \frac{1}{N}\sum_i\|( \widetilde\lambda_i-H_1'\lambda_i) (\widetilde\alpha_i-H_2'\alpha_i)' \|_F\cr
&&
+O_P(1)(\frac{1}{N}\sum_i\|\widetilde\lambda_i-H_1'\lambda_i\|^2 
+(\frac{1}{N}\sum_i\|\widetilde\alpha_i-H_2'\alpha_i\|^2
  )^{1/2}\max_{t}( \frac{1}{N}\sum_i  (\widehat e_{it}-e_{it})^2)^{1/2}
\cr
&&
+O_P(1)\max_{t}\| \frac{1}{N}\sum_i  \lambda_i  \alpha_i' (\widehat e_{it}-e_{it})\|_F\cr
&\leq& O_P(\phi_{NT}+1+\max_{t\leq T}\| w_t\| ) C_{NT}^{-2}+  O_P(b_{NT,1}+b_{NT,3}+C_{NT}^{-1}b_{NT,2})=o_P(1).
\end{eqnarray*}

  Next,
$$
\max_{t\leq T}\|\frac{1}{N}\sum_i(\widetilde\lambda_i \widetilde\alpha_i'  - H_1'\lambda_i\alpha_iH_2')e_{it} \|_F=O_P(C_{NT}^{-1})\max_{it}|e_{it}|=o_P(1).
$$ So the second term of $\max_{t\leq T}|b_{1t}|$ is $o_P(1)$. Similarly, the third term is also $o_P(1)$. 
 Finally, the last term $\|\frac{1}{N}\sum_i\widetilde\alpha_i \widetilde\alpha_i'-H_2'\alpha_i\alpha_i'H_2\|_F=o_P(1)$.
 
  (ii) Because we have proved $\max_{t\leq T}\|\widehat B_{t}^{-1}\|=O_P(1)$, it suffices to prove \\
$\frac{1}{T}\sum_{s\notin I} \|\widehat B_{s}-B\|_F^2= O_P(C^{-2}_{NT})$, or 
$\frac{1}{T}\sum_{s\notin I} \|b_{1t}\|^2_F=O_P(C^{-2}_{NT})= \frac{1}{T}\sum_{s\notin I} \|b_{2t}\|^2_F  $.

    First term of $b_{1t}$: by  (\ref{ed.17}), the first term is bounded by,  by Lemma \ref{ld.2a}, 
        \begin{eqnarray} 
    &&       \frac{1}{T}\sum_{t\notin I} \| \frac{1}{N}\sum_i  \widetilde\lambda_i  \widetilde\lambda_i ' \widehat e_{it}^2 -H_1'\lambda_i\lambda_i'H_1e_{it}^2\|_F^2\cr
 &\leq &    \frac{1}{T}\sum_{t\notin I} \|    \frac{1}{N}\sum_i\widetilde\lambda_i  \widetilde\lambda_i ' (\widehat e_{it}^2-e_{it}^2)  \|_F^2 
 +  \frac{1}{T}\sum_{t\notin I} \| \frac{1}{N}\sum_i  (\widetilde\lambda_i  \widetilde\lambda_i '   -H_1'\lambda_i\lambda_i'H_1)e_{it}^2\|_F^2\cr
     &\leq& O_P(C_{NT}^{-4}) \max_{it} |\widehat e_{it}-e_{it}|^2\max_{it}|e_{it}|^2
 + O_P(C_{NT}^{-2})  \frac{1}{T}\sum_{t\notin I}\frac{1}{N}\sum_i (\widehat e_{it}-e_{it}) ^4\cr
 &&  + O_P(1) \frac{1}{T}\sum_{t\notin I}(\frac{1}{N}\sum_i (\widehat e_{it}-e_{it}) ^2  )^2
 + O_P(C_{NT}^{-2})  \frac{1}{T}\sum_{t\notin I}\frac{1}{N}\sum_i (\widehat e_{it}-e_{it}) ^2e_{it}^2\cr
 &&  +O_P(1) \frac{1}{T}\sum_{t\notin I}(\frac{1}{N}\sum_i    \lambda_i  \lambda_i'  (\widehat e_{it}-e_{it})e_{it})^2\cr
 &&  + O_P(1)
\tr(\frac{1}{N^2}\sum_{ij}(\widetilde\lambda_i  \widetilde\lambda_i ' - H_1'\lambda_i\lambda_i'H_1)
(\widetilde\lambda_j \widetilde\lambda_j ' - H_1'\lambda_j\lambda_i'H_1)  \frac{1}{T}\sum_{t\notin I}\E_Ie_{jt}^2  e_{it}^2)\cr
 &\leq & O_P(C_{NT}^{-2})+O_P(1)
 (\frac{1}{N}\sum_{i}\|(\widetilde\lambda_i  \widetilde\lambda_i ' - H_1'\lambda_i\lambda_i'H_1\|_F)^2
\cr
 &\leq & O_P(C_{NT}^{-4}\phi_{NT}^2)  \max_{it}|e_{it}|^2
 + O_P(C_{NT}^{-2})
=O_P(C_{NT}^{-2}).
 \end{eqnarray}

  Second term of $b_{1t}$: 
   \begin{eqnarray} 
    &&      \frac{1}{T}\sum_{t\notin I} \| \frac{1}{N}\sum_i     \widetilde\lambda_i \widetilde\alpha_i' \widehat e_{it} - H_1'\lambda_i\alpha_iH_2'e_{it} \|_F^2\cr
    &\leq&
      \frac{1}{T}\sum_{t\notin I} \|    \frac{1}{N}\sum_i\widetilde\lambda_i \widetilde\alpha_i' (\widehat e_{it} -e_{it})   \|_F^2
      +        \frac{1}{T}\sum_{t\notin I} \| \frac{1}{N}\sum_i  (   \widetilde\lambda_i \widetilde\alpha_i'    - H_1'\lambda_i\alpha_iH_2')e_{it} \|_F^2\cr
      &\leq&O_P(C_{NT}^{-4}) \max_i \frac{1}{T}\sum_t(\widehat e_{it}-e_{it})^2
          +O_P(C_{NT}^{-2})\frac{1}{NT}\sum_{it}(\widehat e_{it}-e_{it})^2
\cr
&&
      +     \frac{1}{T}\sum_{t\notin I} \|    \frac{1}{N}\sum_i \lambda_i  \alpha_i' (\widehat e_{it} -e_{it})   \|_F^2\cr
      &&+O_P(1)\tr 
      \frac{1}{N^2}\sum_{ij}  (   \widetilde\lambda_i \widetilde\alpha_i'    - H_1'\lambda_i\alpha_iH_2')
  (   \widetilde\lambda_j \widetilde\alpha_j'    - H_1'\lambda_j\alpha_jH_2') \frac{1}{T}\sum_{t\notin I} \E_I e_{it}e_{jt}\cr
 &\leq& O_P(b_{NT, 4}^2+b_{NT,5}^2 +C_{NT}^{-2})   C_{NT}^{-4}  +O_P(C_{NT}^{-2})
  =O_P(C_{NT}^{-2}).
     \end{eqnarray}

     The third term of $b_{1t}$ is bounded similarly. Finally, it is straightforward to  see that fourth term of $b_{1t}$ is $ \| \frac{1}{N}\sum_i   \widetilde\alpha_i \widetilde\alpha_i'-H_2'\alpha_i\alpha_i'H_2 \|_F^2=O_P(C_{NT}^{-2}).$ Thus $\frac{1}{T}\sum_{s\notin I} \|b_{1t}\|^2_F=O_P(C^{-2}_{NT}).$

    As for $b_{2t}$,  it suffices to prove 
   \begin{eqnarray*} 
 &&      \frac{1}{T}\sum_{t\notin I} \| \frac{1}{N}\sum_i    \alpha_i\lambda_i'e_{it} \|_F^2
 = O_P(1) \frac{1}{T}\sum_{t\notin I} \E\| \frac{1}{N}\sum_i    \alpha_i\lambda_i'e_{it} \|_F^2
 =O_P(N^{-1}).
     \end{eqnarray*}
     Thus  $\frac{1}{T}\sum_{s\notin I} \|b_{2t}\|^2_F=O_P(C^{-2}_{NT}).$

        (iii)    Note that 
$\widehat S_t-S=c_t+d_t$,  where 
 \begin{eqnarray*}
 	c_t& =& \frac{1}{N}\sum_i
 	\begin{pmatrix}
 		\widetilde \lambda_i   \lambda_i'H_1e_{it}(\widehat e_{it}-e_{it})   +  \widetilde \lambda_i   \widetilde\lambda_i'(\widehat e_{it}^2-e_{it}^2) & 
 		\widetilde \lambda_i  (\widehat e_{it} -e_{it})( \alpha_i'H_2-\widetilde\alpha_i')\\
 		\widetilde\alpha_i \lambda_i'H_1(e_{it}-\widehat e_{it})
 		+\widetilde\alpha_i (\lambda_i'H_1 -\widetilde\lambda_i')(\widehat e_{it} -e_{it})& 
 		0
 	\end{pmatrix}\cr
 	d_t& =& \frac{1}{N}\sum_i
 	\begin{pmatrix}
 		( \widetilde \lambda_i   \lambda_i'H_1-  \widetilde \lambda_i   \widetilde\lambda_i'  )e_{it}^2
 		- H_1'\lambda_i ( \lambda_i'H_1-  \widetilde\lambda_i'  )\E e_{it}^2
 		& 
 		\widetilde \lambda_i   e_{it} ( \alpha_i'H_2-\widetilde\alpha_i')\\
 		\widetilde\alpha_i (\lambda_i'H_1 -\widetilde\lambda_i')e_{it} & 
 		0
 	\end{pmatrix}.
 \end{eqnarray*}
 As for the  upper blocks of $c_t$,  first note that 
\begin{eqnarray*}
 &&\frac{1}{T}\sum_{t\notin I} (\frac{1}{N}\sum_i      \|\widetilde\lambda_i  -H_1'\lambda_i\|^2 |e_{it}|  )^2
 \cr
 &=&O_P(1) 
 \frac{1}{N^2}\sum_{ij}      \|\widetilde\lambda_i  -H_1'\lambda_i\|^2 \|\widetilde\lambda_j  -H_1'\lambda_j\|^2 \frac{1}{T}\sum_{t\notin I}\E_I |e_{jt }e_{it}| \cr
 &=&O_P(1) 
( \frac{1}{N}\sum_{i}      \|\widetilde\lambda_i  -H_1'\lambda_i\|^2)^2=O_P(C_{NT}^{-4}).
      \end{eqnarray*}

Then by Lemma \ref{ld.2a}, 
 
 \begin{eqnarray*}
 	&&\frac{1}{T}\sum_{t\notin I} \|\frac{1}{N}\sum_i\widetilde \lambda_i   \lambda_i'H_1e_{it}(\widehat e_{it}-e_{it})\|_F^2 \cr
 	&\leq&
 	O_P(1)
 	\frac{1}{N}\sum_i\|\widetilde \lambda_i  -H_1'\lambda_i\|^2\frac{1}{NT}\sum_i \sum_{t\notin I} e_{it}^2(\widehat e_{it}-e_{it})^2
 	\cr
 	&&+O_P(1)\frac{1}{T}\sum_{t\notin I} \|\frac{1}{N}\sum_i \lambda_i   \lambda_ie_{it}(\widehat e_{it}-e_{it})\|_F^2 \cr
 	&=&O_P(C_{NT}^{-4})\cr
 	&&\frac{1}{T}\sum_{t\notin I} \|\frac{1}{N}\sum_i \widetilde \lambda_i   \widetilde\lambda_i'(\widehat e_{it}^2-e_{it}^2) \|_F^2 \cr
 	&\leq& (\frac{1}{N}\sum_i\|\widetilde\lambda_i  -H_1'\lambda_i\|^2)	^2 \max_{it} [(\widehat e_{it}-e_{it}) ^4  ]\cr
 	&&+\frac{1}{T}\sum_{t\notin I} (\frac{1}{N}\sum_i      \|\widetilde\lambda_i  -H_1'\lambda_i\|^2 |e_{it}|  )^2  \max_{it}|\widehat e_{it}-e_{it}|^2 \cr
 	&&+ \frac{1}{N}\sum_i\|\widetilde\lambda_i  -H_1'\lambda_i\|^2 
 	\frac{1}{NT}\sum_{t\notin I}  \sum_i[(\widehat e_{it}-e_{it}) ^4+ (\widehat e_{it}-e_{it})^2e_{it} ^2  ]
 	\cr
 	&&+\frac{1}{T}\sum_{t\notin I} \|\frac{1}{N}\sum_i    \lambda_i  \lambda_i'  (\widehat e_{it}-e_{it})e_{it}   \|_F^2   +\frac{1}{T}\sum_{t\notin I} \| \frac{1}{N}\sum_i \lambda_i\lambda_i' (\widehat e_{it}-e_{it}) ^2      \|_F^2  \cr
 	&=&O_P(C_{NT}^{-4}).
 \end{eqnarray*}

 Also, 
 \begin{eqnarray*}
 &&\frac{1}{T}\sum_{t\notin I} \|\frac{1}{N}\sum_i    	\widetilde \lambda_i  (\widehat e_{it} -e_{it})( \alpha_i'H_2-\widetilde\alpha_i')  \|_F^2\cr
  &\leq & \frac{1}{N}\sum_i     \|\alpha_i'H_2-\widetilde\alpha_i'\|^2 \frac{1}{N}\sum_i    	\|\widetilde \lambda_i -H_1'\lambda_i\|^2 \max_{it} (\widehat e_{it} -e_{it})^2\cr
   &&+O_P(1)\frac{1}{N}\sum_i  \| \alpha_i'H_2-\widetilde\alpha_i'  \|^2
     	 \frac{1}{N}\sum_i\frac{1}{T}\sum_{t\notin I}     (\widehat e_{it} -e_{it})^2
=O_P(C_{NT}^{-4}).
      \end{eqnarray*}
 Similarly, $\frac{1}{T}\sum_{t\notin I} \|\frac{1}{N}\sum_i\widetilde\alpha_i (\lambda_i'H_1 -\widetilde\lambda_i')(\widehat e_{it} -e_{it}) \|_F^2=O_P(C_{NT}^{-4}).$
 Finally,  \begin{eqnarray*}
 \frac{1}{T}\sum_{t\notin I} \|\frac{1}{N}\sum_i   \widetilde\alpha_i \lambda_i'H_1(e_{it}-\widehat e_{it})
  \|_F^2 &=&O_P(C_{NT}^{-4})+ O_P(1)\frac{1}{T}\sum_{t\notin I} \|\frac{1}{N}\sum_i    \alpha_i \lambda_i'(e_{it}-\widehat e_{it})
  \|_F^2\cr
  &=&O_P(C_{NT}^{-4}).
       \end{eqnarray*}
         So if we let the two upper blocks of $c_t$ be $c_{t,1}, c_{t,2}$, and the third block of $c_t$ be $c_{t,3}$, then 
 $ \frac{1}{T}\sum_{t\notin I}\|c_{t,1}\|_F^2+\frac{1}{T}\sum_{t\notin I}\|c_{t,2}\|_F^2=O_P(C_{NT}^{-4})$ while $\frac{1}{T}\sum_{t\notin I}\|c_{t,3}\|_F^2=O_P(C_{NT}^{-2})$.

    As for $ d_t $, let $d_{t,1},...,d_{t,3}$  denote the nonzero blocks.  Then $d_{t,1}$ depends on 
 $
  \frac{1}{T}\sum_{t\notin I} \| \frac{1}{N}\sum_i \lambda_i ( \lambda_i'H_1-  \widetilde\lambda_i'  )(e_{it}^2 -\E e_{it}^2)
  \|_F^2  $.
              Let $\Upsilon_t$ be an $N\times 1$ vector of $e_{it}^2$, and  $\diag(\Lambda)$ be  diagonal matrix consisting of elements of $\Lambda$.   Suppose $\dim(\lambda_i)=1$ (focus on each element), then   $ \frac{1}{T}\sum_{t\notin I}\|d_{t,1}\|_F^2$ is bounded by
                \begin{eqnarray*}
  && \frac{1}{T}\sum_{t\notin I} \|\frac{1}{N}\sum_i  ( \widetilde \lambda_i   \lambda_i'H_1-  \widetilde \lambda_i   \widetilde\lambda_i'  )e_{it}^2
     		- H_1'\lambda_i ( \lambda_i'H_1-  \widetilde\lambda_i'  )\E e_{it}^2
    \|_F^2 \cr
    &\leq&
 \frac{1}{T}\sum_{t\notin I} \|\frac{1}{N}\sum_i   (\widetilde \lambda_i -H_1'\lambda_i) ( \lambda_i'H_1-  \widetilde\lambda_i'  )e_{it}^2
    \|_F^2 \cr
    &&+ \frac{1}{T}\sum_{t\notin I} \|\frac{1}{N}\sum_i   H_1'\lambda_i ( \lambda_i'H_1-  \widetilde\lambda_i'  )(e_{it}^2 -\E e_{it}^2)
    \|_F^2 \cr
    &\leq&O_P(1) 
    \frac{1}{N^2}\sum_{ij}    \|\widetilde \lambda_i -H_1'\lambda_i \|^2
 \|\widetilde \lambda_j -H_1'\lambda_j\|^2 \frac{1}{T}\sum_{t\notin I}\E_Ie_{jt}^2 e_{it}^2\cr
&&+ O_P(1) \frac{1}{T}\sum_{t\notin I}  \frac{1}{N^2}
(\widetilde\Lambda- \Lambda H_1)' \diag(\Lambda)  \Var_I
(\Upsilon_t) \diag(\Lambda)   (\widetilde\Lambda- \Lambda H_1)\cr
&=&O_P(C_{NT}^{-4}).
         \end{eqnarray*}
  In addition, using the same technique, it is easy to show 
  $$
 \frac{1}{T}\sum_{t\notin I}\|   \frac{1}{N}\sum_i \widetilde \lambda_i   e_{it} ( \alpha_i'H_2-\widetilde\alpha_i')\|_F^2
=O_P(C_{NT}^{-4})=
 \frac{1}{T}\sum_{t\notin I}\| \frac{1}{N}\sum_i  \widetilde\alpha_i (\lambda_i'H_1 -\widetilde\lambda_i')e_{it}   \|_F^2.
  $$
  Hence $\frac{1}{T}\sum_{t\notin I}\|d_{t,k}\|_F^2=O_P(C_{NT}^{-4})$ for $k=1,2,3.$ 
 Together, we have 
 $$\frac{1}{T}\sum_{t\notin I}\|\Delta_{t1}\|^2\leq   \frac{1}{T}\sum_{t\notin I}\|c_{t,1}+d_{t,1}   \|^2
 +\frac{1}{T}\sum_{t\notin I}\|c_{t,2}+d_{t,2}   \|^2
 =O_P(C_{NT}^{-4}) $$ and 
  $\frac{1}{T}\sum_{t\notin I}\|\Delta_{t2}\|^2=   \frac{1}{T}\sum_{t\notin I}\|c_{t,3}+d_{t,3}\|^2=O_P(C_{NT}^{-2}). $
  
  (iv) Recall that 
$\widehat S_t-S=c_t+d_t$.  It suffices to show $\max_{t\leq T}\|c_t\|=o_P(1)=\max_{t\leq T}\|d_t\|$.
\begin{eqnarray*}
\max_{t\leq T}\|c_{t}\|&\leq &O_P(1)\max_{it}|\widehat e_{it}-e_{it}| (|e_{it}| +1) =o_P(1).\cr
\max_{t\leq T}\|d_{t}\|&\leq & (\max_{it}|e_{it}|+\max_{it}|e_{it}|^2+1) O_P(C_{NT}^{-1})=o_P(1).
 \end{eqnarray*}

\end{proof}

  \begin{lem}\label{ld.9} For terms defined in (\ref{ed.24}), 
   and for each fixed $t\notin I$, 
   
   (i) $ \sum_{d=2}^5A_{dt}=O_P(C_{NT}^{-2})$
   
   (ii)  For the ``upper block" of $A_{6t}$, $\frac{1}{N}\sum_i 
 \lambda_i e_{it} ( l_i'w_t\lambda_i'f_t - \widehat{l_i'w_t}\dot\lambda_i'\widetilde f_t  )=O_P(C_{NT}^{-2})$.
 
 (iii) The upper block of $A_{1t}$ is $O_P(C_{NT}^{-2})$.
  \end{lem}

\begin{proof}
(i) Term $ A_{2t}$. Given $B_t^{-1}-  B^{-1}=O_P(C_{NT}^{-1})$ and the cross-sectional weak correlations in $u_{it}$, it is easy to see $A_{2t}= O_P(C_{NT}^{-2})$. 

Term $ A_{3t}$. It suffices to prove: 
\begin{eqnarray}\label{ec.25}
&&  \frac{1}{N}\sum_i (\widetilde\lambda_i\widehat e_{it}  - H_1'\lambda_i e_{it} )u_{it}=O_P(C_{NT}^{-2})\cr
&&\frac{1}{N}\sum_i (\widetilde\alpha_i -H_2'\alpha_i)u_{it}=O_P(C_{NT}^{-2}).
\end{eqnarray}
First,  let $\Upsilon_t$ be an $N\times 1$ vector of $e_{it}u_{it}$. Due to the serial independence of $(u_{it}, e_{it})$, we have 
\begin{eqnarray*}
&&\E_I \frac{1}{N}\sum_i( \widetilde\lambda_i 
 - H_1'\lambda_i )e_{it} u_{it}=0,\cr
 &&\Var_I( \frac{1}{N}\sum_i( \widetilde\lambda_i 
 - H_1'\lambda_i )e_{it} u_{it})=\frac{1}{N^2}(\widetilde\Lambda-\Lambda H_1)'\Var_I(\Upsilon_t)(\widetilde\Lambda-\Lambda H_1)\cr
 &\leq& O_P(C_{NT}^{-2}N^{-1})\max_i\sum_j|\Cov(e_{it}u_{it}, e_{jt}u_{jt})|= O_P(C_{NT}^{-2}N^{-1}).
\end{eqnarray*}
Similarly, $\E_I\frac{1}{N}\sum_i (\widetilde\alpha_i -H_2'\alpha_i)u_{it}=0$ and
$\Var_I(\frac{1}{N}\sum_i (\widetilde\alpha_i -H_2'\alpha_i)u_{it})=O_P(C_{NT}^{-2}N^{-1})$. 
\begin{eqnarray}
&& \frac{1}{N}\sum_i( \widetilde\lambda_i 
 - H_1'\lambda_i )e_{it} u_{it}= O_P(C_{NT}^{-1}N^{-1/2})\cr
 &&\frac{1}{N}\sum_i (\widetilde\alpha_i -H_2'\alpha_i)u_{it}=O_P(C_{NT}^{-1}N^{-1/2})\cr
 &&\frac{1}{N}\sum_i \|\widetilde\lambda_i-H_1'\lambda_i\|^2u_{it}^2
=O_P(1)\frac{1}{N}\sum_i \|\widetilde\lambda_i-H_1'\lambda_i\|^2\E_Iu_{it}^2
=O_P(C_{NT}^{-2}).\cr
\end{eqnarray}
 Thus, the first term of (\ref{ec.25}) is 
\begin{eqnarray*}
&&  \frac{1}{N}\sum_i (\widetilde\lambda_i\widehat e_{it}  - H_1'\lambda_i e_{it} )u_{it} 
\leq
(  \frac{1}{N}\sum_i \|\widetilde\lambda_i-H_1'\lambda_i\|^2u_{it}^2)^{1/2}
(  \frac{1}{N}\sum_i  (\widehat e_{it}  -e_{it})^2 )^{1/2}
  \cr
 &&+H_1' \frac{1}{N}\sum_i \lambda_i(\widehat e_{it}  -e_{it})u_{it}
 + \frac{1}{N}\sum_i( \widetilde\lambda_i 
 - H_1'\lambda_i )e_{it} u_{it}\cr
 &\leq& O_P( C_{NT}^{-1})(  \frac{1}{N}\sum_i  (\widehat e_{it}  -e_{it})^2 )^{1/2}+
 O_P( C_{NT}^{-2})+O_P(1)\frac{1}{N}\sum_i \lambda_i(\widehat e_{it}  -e_{it})u_{it}\cr
  &=&O_P( C_{NT}^{-2}).
\end{eqnarray*}
where the last equality follows from Lemma \ref{ld.1a}.

Term $A_{4t}$. Recall that $\mu_{it}=l_i'w_t$. 
    It suffices to prove each of the following terms is $O_P( C_{NT}^{-2})$:  
   \begin{eqnarray}\label{ec.37a}
C_{1t}&=& \frac{1}{N}\sum_i (\widetilde\lambda_i - H_1'\lambda_i) e_{it}  (\widehat e_{it}-e_{it})( \dot\lambda_i-H_1'\lambda_i)'\widetilde f_t  \cr
C_{2t}&=& \frac{1}{N}\sum_i (\widetilde\lambda_i - H_1'\lambda_i)(\widehat  e_{it} -e_{it}) ^2( \dot\lambda_i-H_1'\lambda_i)'\widetilde f_t  \cr
C_{3t}&=& \frac{1}{N}\sum_i H_1'\lambda_i(\widehat  e_{it} -e_{it})^2( \dot\lambda_i-H_1'\lambda_i)'\widetilde f_t  \cr
C_{4t}&=& \frac{1}{N}\sum_i  (\widetilde\lambda_i - H_1'\lambda_i) e_{it} (\widehat e_{it}-e_{it}) \lambda_i'H_1\widetilde f_t  \cr
C_{5t}&=& \frac{1}{N}\sum_i (\widetilde\lambda_i - H_1'\lambda_i)(\widehat  e_{it} -e_{it}) ^2 \lambda_i'H_1\widetilde f_t  \cr
C_{6t}&=& \frac{1}{N}\sum_i  H_1'\lambda_i(\widehat  e_{it} -e_{it}) ^2\lambda_i'H_1\widetilde f_t  \cr
C_{7t}&=&   \frac{1}{N}\sum_i (\widetilde\lambda_i - H_1'\lambda_i) e_{it}\mu_{it} (\dot \lambda_i-H_1'\lambda_i)' \widetilde f_t\cr
C_{8t}&=&   \frac{1}{N}\sum_i  (\widetilde\lambda_i - H_1'\lambda_i)(\widehat  e_{it} -e_{it})\mu_{it} (\dot \lambda_i-H_1'\lambda_i)' \widetilde f_t\cr
C_{9t}&=&   \frac{1}{N}\sum_i H_1'\lambda_i(\widehat  e_{it} -e_{it}) \mu_{it} (\dot \lambda_i-H_1'\lambda_i)' \widetilde f_t\cr
C_{10t}&=& \frac{1}{N}\sum_i  (\widetilde\lambda_i - H_1'\lambda_i) e_{it}\mu_{it} \lambda_i'H_1(\widetilde f_t- H_1^{-1}f_t)\cr
C_{11t}&=& \frac{1}{N}\sum_i (\widetilde\lambda_i - H_1'\lambda_i)(\widehat  e_{it} -e_{it})\mu_{it} \lambda_i'H_1(\widetilde f_t- H_1^{-1}f_t)\cr
C_{12t}&=& \frac{1}{N}\sum_i H_1'\lambda_i(\widehat  e_{it} -e_{it}) \mu_{it} \lambda_i'H_1(\widetilde f_t- H_1^{-1}f_t)\cr
C_{13t}&=& \frac{1}{N}\sum_i (\widetilde\alpha_i -H_2'\alpha_i) (\widehat e_{it}-e_{it})( \dot\lambda_i-H_1'\lambda_i)'\widetilde f_t  \cr
C_{14t}&=& \frac{1}{N}\sum_i(\widetilde\alpha_i -H_2'\alpha_i)  (\widehat e_{it}-e_{it}) \lambda_i'H_1\widetilde f_t  \cr
C_{15t}&=&   \frac{1}{N}\sum_i (\widetilde\alpha_i -H_2'\alpha_i) \mu_{it} (\dot \lambda_i-H_1'\lambda_i)' \widetilde f_t\cr
C_{16t}&=& \frac{1}{N}\sum_i (\widetilde\alpha_i -H_2'\alpha_i) \mu_{it} \lambda_i'H_1(\widetilde f_t- H_1^{-1}f_t).
\end{eqnarray}
The proof follows from repeatedly applying the Cauchy-Schwarz inequality and is straightforward. 
In addition,   it  is also  straightforward to apply   Cauchy-Schwarz to prove that
\begin{equation}\label{ec.32a}  \frac{1}{T}\sum_{t\notin I} \|C_{dt}\|^2=O_P(C_{NT}^{-4})
,\quad d=1,...,16.
\end{equation}

 Term $A_{5t}$. 
 Note that  $$
       A_{5t}= (\widehat B_t^{-1}-B^{-1} )\begin{pmatrix}H_1'\sum_{d=1}^4 B_{dt}\\
       H_2'\sum_{d=5}^8 B_{dt}
       \end{pmatrix}
       $$
 where $ B_{dt}$ are defined below. 
 Given $B_t^{-1}-  B^{-1}=O_P(C_{NT}^{-1})$, it suffices to prove 
 the following terms  are  $O_P( C_{NT}^{-2})$.
   \begin{eqnarray}\label{ed.48}
B_{1t}&=& \frac{1}{N}\sum_i\lambda_i e_{it}  (\widehat e_{it}-e_{it})( \dot\lambda_i-H_1'\lambda_i)'\widetilde f_t  \cr
B_{2t}&=& \frac{1}{N}\sum_i\lambda_i e_{it}  (\widehat e_{it}-e_{it}) \lambda_i'H_1\widetilde f_t  \cr
B_{3t}&=&   \frac{1}{N}\sum_i\lambda_i e_{it}l_i'w_t (\dot \lambda_i-H_1'\lambda_i)' \widetilde f_t\cr
B_{4t}&=& \frac{1}{N}\sum_i\lambda_i e_{it} l_i'w_t \lambda_i'H_1(\widetilde f_t- H_1^{-1}f_t)\cr
B_{5t}&=& \frac{1}{N}\sum_i\alpha_i  (\widehat e_{it}-e_{it})( \dot\lambda_i-H_1'\lambda_i)'\widetilde f_t  
\end{eqnarray}
and that the following terms  are  $O_P( C_{NT}^{-1})$:
   \begin{eqnarray}\label{ed.49}
B_{6t}&=& \frac{1}{N}\sum_i \alpha_i (\widehat e_{it}-e_{it}) \lambda_i'H_1\widetilde f_t  \cr
B_{7t}&=&   \frac{1}{N}\sum_i\alpha_il_i'w_t(\dot \lambda_i-H_1'\lambda_i)' \widetilde f_t\cr
B_{8t}&=& \frac{1}{N}\sum_i\alpha_i l_i'w_t \lambda_i'H_1(\widetilde f_t- H_1^{-1}f_t).
\end{eqnarray}

In fact, $B_{1t}, B_{5,t}\sim B_{8t}$ follow immediately from the Cauchy-Schwarz inequality.  $B_{2t}= O_P(C_{NT}^{-2})$ due to Lemma \ref{ld.1a}. 
By Lemma \ref{ld.5} that   $\frac{1}{\sqrt{N}}\sum_{i}  (H_1' \lambda_i-\dot\lambda_i)e_{it}=O_P(\sqrt{N}C_{NT}^{-2}) .$ So  $t\notin I$, 
   $B_{3t}=O_P(1)   \frac{1}{N}\sum_i\lambda_i e_{it}\mu_i (\dot \lambda_i-H_1'\lambda_i)' 
   =O_P(C_{NT}^{-2}). 
   $
In addition, for each fixed $t$, $\widetilde f_t- H_1^{-1}f_t=O_P(C_{NT}^{-1})$. Thus 
   $$
   B_{4t}=O_P(C_{NT}^{-1}) \frac{1}{N}\sum_i\lambda_i e_{it} \mu_i \lambda_i'=O_P(C_{NT}^{-2}). 
   $$
So $A_{5t}=O_P(C_{NT}^{-2}).$

(ii)   ``upper block" of $A_{6t}$.  
   Note that $l_i'w_t-\widehat{l_i'w_t}= \widehat e_{it}-e_{it}$.
       $$
       A_{6t}= B^{-1} \begin{pmatrix}H_1'\sum_{d=1}^4 B_{dt}\\
       H_2'\sum_{d=5}^8 B_{dt}
       \end{pmatrix}
       $$
           So we only need to look at $\sum_{d=1}^4 B_{dt}$.
   From the proof of (i), 
  we have $B_{dt}=  O_P( C_{NT}^{-2})$, for $d=1,...,4$.   It follows immediately that $\frac{1}{N}\sum_i 
 \lambda_i e_{it} ( l_i'w_t\lambda_i'f_t - \widehat{l_i'w_t} \dot\lambda_i'\widetilde f_t  )=O_P(C_{NT}^{-2})$.
 
 (iii) Lastly, note that the upper block of $A_{1t}$  is determined by the upper blocks of 
$\widehat B_t^{-1}\widehat S_t-B^{-1}S$, and are    both $O_P(C_{NT}^{-2})$ by Lemma \ref{ld.6}.

\end{proof}

  \subsection{Technical lemmas for $\widehat \lambda_i$}
  
    \begin{lem}\label{lc.8add} Suppose $\max_i\|\frac{1}{T}\sum_{s }   f_s  f_s'  (e_{is}^2-\E e_{is}^2)\|=o_P(1)= \max_i\|\frac{1}{T}\sum_{s }   f_s  g_s'  e_{is}\|.$\\
 (i) $\frac{1}{|\mathcal G|_0}\sum_{i\in\mathcal G}\|\widehat D_{i}-D_i\|^2=O_P(C_{NT}^{-2})$.\\
 (ii) $ \max_{i\leq N}\|\widehat D_{i}^{-1}-D_i^{-1}\|=o_P(1)$. \\
  (iii) $\frac{1}{|\mathcal G|_0}\sum_{i\in\mathcal G}\|\widehat D_{i}^{-1}-D_i^{-1}\|^2=O_P(C_{NT}^{-2})$.\\
 \end{lem}

\begin{proof} $\widehat D_{i}-D_i$ is a two-by-two block matrix. The first block is 
\begin{eqnarray*}
 d_{1i}&:=&\frac{1}{T_0}\sum_{s\notin I}\widehat f_s\widehat f_s'\widehat e_{is}^2-H_ff_s  f_s'H_f' \E e_{is}^2\cr
 &=&\frac{1}{T_0}\sum_{s\notin I}(\widehat f_s-H_ff_s)(\widehat f_s-H_ff_s)'[(\widehat e_{is}-e_{is})^2+ (\widehat e_{is}-e_{is})e_{is} + e_{is}^2]\cr
 &&
 +\frac{1}{T_0}\sum_{s\notin I}(\widehat f_s-H_ff_s)  f_s'H_f'[ (\widehat e_{is}-e_{is})^2+(\widehat e_{is}-e_{is})e_{is}  + e_{is}^2  ]   \cr
 &&
 +\frac{1}{T_0}\sum_{s\notin I}H_f f_s(\widehat f_s-H_ff_s)'[(\widehat e_{is}-e_{is})^2+(\widehat e_{is}-e_{is})e_{is}  +e_{is}^2 ]
 \cr
 &&
  +\frac{1}{T_0}\sum_{s\notin I}H_f f_s f_s'H_f'[  (\widehat e_{is}-e_{is})^2  +(\widehat e_{is}-e_{is})e_{is}]
  +\frac{1}{T_0}\sum_{s\notin I}  H_ff_s  f_s'H_f'  (e_{is}^2-\E e_{is}^2).
  \end{eqnarray*}
   It follows from Lemma \ref{ld.2a} that $\frac{1}{NT}\sum_{it}(\widehat e_{it}-e_{it})^ 2e_{it}^2=O_P(C_{NT}^{-2})     $  \\  $\frac{1}{NT}\sum_{it}(\widehat e_{it}-e_{it})^ 4= O_P(C_{NT}^{-4})$ and Lemma \ref{ld.1a} $\max_{it}|\widehat e_{it}-e_{it}|e_{it}=O_P(1)$.
   \begin{eqnarray*}
\frac{1}{|\mathcal G|_0}\sum_{i\in\mathcal G}\|   d_{1i}\|^2&\leq&
O_P(1)\frac{1}{T_0}\sum_{s\notin I}\|\widehat f_s-H_ff_s\|^2(1+\|f_t\|^2+\|f_t\|^2\frac{1}{|\mathcal G|_0}\sum_{i\in\mathcal G}e_{it}^4)
\cr
&&+O_P(C_{NT}^{-2})  =O_P(C_{NT}^{-2})  .
   \end{eqnarray*}

   The second block is 
  \begin{eqnarray*}
 d_{2i}&:=& \frac{1}{T_0}\sum_{s\notin I} \widehat f_s\widehat g_s'\widehat e_{is}
=\frac{1}{T_0}\sum_{s\notin I}( \widehat f_s-H_ff_s)(\widehat g_s-H_gg_s)'\widehat e_{is}
 \cr
 &&+\frac{1}{T_0}\sum_{s\notin I}( \widehat f_s-H_ff_s)g_s'H_g'\widehat e_{is}
+\frac{1}{T_0}\sum_{s\notin I}H_f f_s(\widehat g_s-H_gg_s)'\widehat e_{is}
 \cr
 &&+\frac{1}{T_0}\sum_{s\notin I}H_f f_s g_s'H_g'(\widehat e_{is}-e_{is})
 +\frac{1}{T_0}\sum_{s\notin I}H_f f_s g_s'H_g'e_{is}
  \end{eqnarray*}
  So $\frac{1}{|\mathcal G|_0}\sum_{i\in\mathcal G}\|   d_{2i}\|^2=O_P(C_{NT}^{-2})  .$
  The third block is similar. The fourth block is $ d_{4i}= \frac{1}{T_0}\sum_{s\notin I} \widehat g_s\widehat g_s' - H_gg_s g_s'H_g'$. So $\frac{1}{|\mathcal G|_0}\sum_{i\in\mathcal G}\|   d_{4i}\|^2=O_P(C_{NT}^{-2})  .$

 Hence $\frac{1}{|\mathcal G|_0}\sum_{i\in\mathcal G}\|\widehat D_{i} -D_i \|^2=O_P(C_{NT}^{-2})$.

 (ii)   Note that 
\begin{eqnarray*}
 \max_i\|d_{1i}\|&=&o_P(1)+O_P(1) \max_i\|\frac{1}{T}\sum_{s }   f_s  f_s'  (e_{is}^2-\E e_{is}^2)\|=o_P(1)\cr
 \max_i\|d_{2i}\|&=&o_P(1)+O_P(1) \max_i\|\frac{1}{T}\sum_{s }   f_s  g_s'  e_{is}\|=o_P(1).
 \end{eqnarray*}
 So $\max_i\| 
 \widehat D_{i} -D_i\|=o_P(1).$
 Now because $\min_i\lambda_{\min}(D_i)\geq c$,   with probability approaching one
 $\min_i\lambda_{\min}(\widehat D_i)\geq\min_i\lambda_{\min}( D_i) -\max_i\| 
 \widehat D_{i} -D_i\|>c/2$.  So $\max_i\|D_i^{-1}\|+ \max_i\|\widehat D_i^{-1}\|=O_P(1)$.  So
 $$
\max_i \|\widehat D_{i}^{-1}-D_i^{-1}\| 
\leq  \max_i\|\widehat D_i^{-1}\| \max_i\|  D_i^{-1}\|\max_i \|\widehat D_{i}-D_i\|=o_P(1). 
 $$
 
 (iii) Because $\min_i\lambda_{\min}(D_i)\geq c$, 
 \begin{eqnarray*}
 \frac{1}{|\mathcal G|_0}\sum_{i\in\mathcal G}\|\widehat D_{i}^{-1}-D_i^{-1}\|^2 
&\leq& \frac{1}{|\mathcal G|_0}\sum_{i\in\mathcal G}\|\widehat D_{i}^{-1}\|^2\|D_{i}^{-1}\|^2\|\widehat D_{i}-D_i\|^2 \cr
&\leq &\max_i [\|\widehat D_{i}^{-1}-D_i^{-1}\| +\|D_i^{-1}\|]^2\frac{C}{|\mathcal G|_0}\sum_{i\in\mathcal G} \|\widehat D_{i}-D_i\|^2\cr
&=&O_P(1)\frac{1}{|\mathcal G|_0}\sum_{i\in\mathcal G}\|\widehat D_{i} -D_i \|^2=O_P(C_{NT}^{-2}).
\end{eqnarray*}
   
 \end{proof}

\begin{lem}\label{lc.8}
   $\frac{1}{|\mathcal G|_0}\sum_{i\in\mathcal G}\| \frac{1}{T}\sum_{s\notin I} (\widehat f_s -H_ff_s) e_{is} f_s' \|^2=O_P(C_{NT}^{-4}).$
\end{lem}

This lemma is needed to prove the performance for $\widehat\lambda_i$, which controls the effect of $\widehat f_s-H_ff_s$ on the estimation of $\lambda_i$. 

\begin{proof}

  Use (\ref{ed.24}),  $\frac{1}{|\mathcal G|_0}\sum_{i\in\mathcal G} \|\frac{1}{T}\sum_{s\notin I} (\widehat f_s -H_ff_s) e_{is} f_s'\|^2$ is bounded by, up to a multiplier of order $O_P(1)$, 
\begin{eqnarray*}
	\frac{1}{|\mathcal G|_0}\sum_{i\in\mathcal G} \| \frac{1}{T}\sum_{t\notin I}  \frac{1}{N}\sum_{j=1}^N\lambda_j e_{jt}u_{jt}e_{it} f_t'\|^2+ \sum_{d=1}^6 \frac{1}{|\mathcal G|_0}\sum_{i\in\mathcal G} \|\frac{1}{T}\sum_{t\notin I}  a_{dt}    e_{it}    f_t'\|^2.
\end{eqnarray*}
where $a_{dt}$ is the upper block of $A_{dt}$.   We shall assume $\dim(f_t)=1$ without loss of generality.  The proof for $ \frac{1}{|\mathcal G|_0}\sum_{i\in\mathcal G}\| \frac{1}{T}\sum_{t\notin I}   a_{1t}    e_{it}    f_t'\|^2$ is the hardest, and we shall present it at    last. 

Step 1. Show the first term.

Let $\nu_t=(\lambda_1e_{1t},...,\lambda_Ne_{Nt})'$. Then
\begin{eqnarray*}
	&&\E     \frac{1}{|\mathcal G|_0}\sum_{i\in\mathcal G}        \| \frac{1}{T}\sum_{t\notin I}  \frac{1}{N}\sum_{j=1}^N\lambda_j e_{jt}u_{jt}e_{it} f_t'\|^2 
	=\frac{1}{N^2T^2}
\frac{1}{|\mathcal G|_0}\sum_{i\in\mathcal G}\sum_{t\notin I} \E (  \nu_{t}'u_{t}e_{it} f_t) ^2	\cr
&\leq& \frac{1}{N^2T^2}
\frac{1}{|\mathcal G|_0}\sum_{i\in\mathcal G}\sum_{t\notin I} \E  f_t^2 \nu_{t}'\E(u_{t}u_t'|E, F)\nu_te_{it}  ^2	\leq  \frac{1}{N^2T^2}
\frac{1}{|\mathcal G|_0}\sum_{i\in\mathcal G}\sum_{t\notin I}  \sum_{j=1}^N\E  f_t^2 \lambda_j^2e_{jt}^2e_{it}  ^2\cr
&=&O(N^{-1}T^{-1})= O(C_{NT}^{-4}).
\end{eqnarray*}

Step 2. Show  $\sum_{d=2}^6 \frac{1}{|\mathcal G|_0}\sum_{i\in\mathcal G} \|\frac{1}{T}\sum_{t\notin I}  a_{dt}    e_{it}    f_t'\|^2.$

Up to a multiplier of order $O_P(1)$, by Lemma \ref{lc.6}, 
\begin{eqnarray*}
 && \frac{1}{|\mathcal G|_0}\sum_{i\in\mathcal G}\|  \frac{1}{T}\sum_{t\notin I}   a_{2t}    e_{it}    f_t'\|^2\cr
 &\leq& 
   \frac{1}{|\mathcal G|_0}\sum_{i\in\mathcal G}|  \frac{1}{T}\sum_{t\notin I}     \| e_{it}    f_t'\| \|\widehat B_t^{-1}- B^{-1}\| [\|\frac{1}{N}\sum_j\lambda_je_{jt}u_{jt}\| + \| \frac{1}{N}\sum_j \alpha_ju_{jt}\|]|^2\cr
     &\leq&
    \frac{1}{T}\sum_{t\notin I} \| \widehat B_t-  B\|^2
  \frac{1}{|\mathcal G|_0}\sum_{i\in\mathcal G}    \frac{1}{T}\sum_{t\notin I}      \| e_{it}    f_t'\| ^2  [\|\frac{1}{N}\sum_j\lambda_je_{jt}u_{jt}\| + \| \frac{1}{N}\sum_j \alpha_ju_{jt}\|]    ^2   \cr
     &=&O_P(C_{NT}^{-4})\cr
   &&   \frac{1}{|\mathcal G|_0}\sum_{i\in\mathcal G}\|    \frac{1}{T}\sum_{t\notin I}   a_{3t}    e_{it}    f_t'\|^2\cr
   &\leq&
        \max_{t\leq T} \|\widehat B_t^{-1}\| ^2 \frac{1}{|\mathcal G|_0}\sum_{i\in\mathcal G} (\frac{1}{T}\sum_{t\notin I}  \|   e_{it}    f_t'\| \|\frac{1}{N}\sum_j
   (\widetilde\lambda_j\widehat e_{jt}  - H_1'\lambda_j e_{jt} )u_{jt}\|    )^2    \cr
   &&+ 
             \max_{t\leq T} \|\widehat B_t^{-1}\|^2 \frac{1}{|\mathcal G|_0}\sum_{i\in\mathcal G}(  \frac{1}{T}\sum_{t\notin I}  \|   e_{it}    f_t'\|  \|\frac{1}{N}\sum_j
   (\widetilde\alpha_j -H_2'\alpha_j)u_{jt}  \|)^2\cr
   &\leq& O_P(1)   \frac{1}{T}\sum_{t\notin I}  \|\frac{1}{N}\sum_j
   (\widetilde\lambda_j - H_1'\lambda_j ) e_{jt} u_{jt}\|^2    +O_P(1)   \frac{1}{T}\sum_{t\notin I} \|\frac{1}{N}\sum_j (\widetilde\alpha_j -H_2'\alpha_j)u_{jt} 
     \|^2   \cr
   &&+O_P(1) \frac{1}{T}\sum_{t\notin I}   \|\frac{1}{N}\sum_j
   (\widetilde\lambda_j - H_1'\lambda_j )(\widehat e_{jt} -e_{jt})u_{jt}\|  ^2\cr
   &&+ O_P(1) \frac{1}{T}\sum_{t\notin I}   \|\frac{1}{N}\sum_j
 \lambda_j (\widehat e_{jt} -e_{jt})u_{jt}\| ^2 =O_P(C_{NT}^{-4}).
     \end{eqnarray*}
     The first term is bounded by, for $\omega_{it}=e_{it}u_{it}$, $$ 
   O_P(1)    \frac{1}{T}\sum_{t\notin I}  \frac{1}{N^2} 
  (\widetilde\Lambda - \Lambda H_1 ) '   \Var_I(\omega_{t})(\widetilde\Lambda - \Lambda H_1 ) 
  \leq    O_P(1)     \frac{1}{N^2} 
  \|\widetilde\Lambda - \Lambda H_1 \|^2=O_P(C_{NT}^{-4}).
     $$
     The second term  $  \frac{1}{T}\sum_{t\notin I} \|\frac{1}{N}\sum_j (\widetilde\alpha_j -H_2'\alpha_j)u_{jt} 
     \|^2$ is bounded similarly.  The third term follows from Cauchy-Schwarz. 
     The last term follows from Lemma \ref{ld.2a}. 
     Next, 
     \begin{eqnarray*}
      \frac{1}{|\mathcal G|_0}\sum_{i\in\mathcal G}\|    \frac{1}{T}\sum_{t\notin I}   a_{4t}    e_{it}    f_t'\|^2
      \leq    \max_{t\leq T} \|\widehat B_t^{-1}\|  ^2 \sum_{d=1}^{16}\frac{1}{|\mathcal G|_0}\sum_{i\in\mathcal G}(\frac{1}{T}\sum_{t\notin I}  \|   e_{it}    f_t'\| \|  C_{dt} \|)^2=O_P(C_{NT}^{-4}),
      \end{eqnarray*}
     where $C_{dt}$'s are defined in the proof of Lemma \ref{ld.9}.
 Applying the simple  Cauchy-Schwarz  proves $ \frac{1}{|\mathcal G|_0}\sum_{i\in\mathcal G}(\frac{1}{T}\sum_{t\notin I}  \|   e_{it}    f_t'\| \|  C_{dt} \|)^2=O_P(C_{NT}^{-4}) $ for all $d\leq16.$
     
  Next,  for $B_{dt}$ defined in the proof of Lemma \ref{ld.9}, repeatedly use Cauchy-Schwarz, 
\begin{eqnarray*} 
  \frac{1}{|\mathcal G|_0}\sum_{i\in\mathcal G}\|\frac{1}{T}\sum_{t\notin I}   a_{5t}    e_{it}    f_t'\|^2&\leq&O_P(C_{NT}^{-2}) \sum_{d=1}^8   \frac{1}{|\mathcal G|_0}\sum_{i\in\mathcal G}
 \frac{1}{T}\sum_{t\notin I}   \|  B_{dt}  e_{it}    f_t' \|^2=O_P(C_{NT}^{-4}).
         \end{eqnarray*}
 Also,
 \begin{eqnarray*} 
  \frac{1}{|\mathcal G|_0}\sum_{i\in\mathcal G}\|\frac{1}{T}\sum_{t\notin I}   a_{6t}    e_{it}    f_t'\|^2&\leq& \sum_{d=1}^4   \frac{1}{|\mathcal G|_0}\sum_{i\in\mathcal G}
 \frac{1}{T}\sum_{t\notin I}   \|  B_{dt}  e_{it}    f_t' \|^2=O_P(C_{NT}^{-4}), \quad \text{where } \cr
   \frac{1}{|\mathcal G|_0}\sum_{i\in\mathcal G}
 \frac{1}{T}\sum_{t\notin I}   \|  B_{1t}  e_{it}    f_t' \|^2&=&O_P(C_{NT}^{-4})  \text{ Cauchy-Schwarz}\cr
  \frac{1}{|\mathcal G|_0}\sum_{i\in\mathcal G}
 \frac{1}{T}\sum_{t\notin I}   \|  B_{2t}  e_{it}    f_t' \|^2&\leq& O_P(1)
 \frac{1}{T}\sum_{t\notin I}    |       \frac{1}{N}\sum_j\lambda_j^2 e_{jt}  (\widehat e_{jt}-e_{jt})   |^2 =O_P(C_{NT}^{-4})\cr
 && \text{(Lemma \ref{ld.2a})} \cr
  \frac{1}{|\mathcal G|_0}\sum_{i\in\mathcal G}
 \frac{1}{T}\sum_{t\notin I}   \|  B_{3t}  e_{it}    f_t' \|^2&\leq& O_P(1)
   \frac{1}{T}\sum_{t\notin I}   \|       \frac{1}{N}\sum_jl_j \lambda_je_{jt}  (\dot \lambda_j-H_1'\lambda_j)  \|^2  =O_P(C_{NT}^{-4})\cr
 && \text{(Lemma \ref{ld.5})} \cr
   \frac{1}{|\mathcal G|_0}\sum_{i\in\mathcal G} \frac{1}{T}\sum_{t\notin I}   \|  B_{4t}  e_{it}    f_t' \|^2&\leq& O_P(C_{NT}^{-2}) \frac{1}{T}\sum_{t\notin I}   \|     \frac{1}{N}\sum_j\lambda_j ^2l_je_{jt} \|^2=O_P(C_{NT}^{-4}).
         \end{eqnarray*}

Step 3. Finally, we bound   $  \frac{1}{|\mathcal G|_0}\sum_{i\in\mathcal G} \|\frac{1}{T}\sum_{t\notin I}  a_{1t}    e_{it}    f_t'\|^2.$

  Note that 
$$
     A_{1t}=  (\widehat B_t^{-1}\widehat S_t-  B^{-1} S )\begin{pmatrix}
 H_1^{-1}f_t\\
 H_2^{-1}g_t
\end{pmatrix}
:= (A_{1t,a} +A_{1t,b}+A_{1t,c}) \begin{pmatrix}
 H_1^{-1}f_t\\
 H_2^{-1}g_t
\end{pmatrix},
$$
where
\begin{eqnarray*}
A_{1t,a}&=& (\widehat  B_t^{-1}-  B^{-1}) (\widehat S_t-S) \cr
A_{1t,b}&=& (\widehat  B_t^{-1}-  B^{-1})  S \cr
  A_{1t,c}&=&B^{-1} (\widehat S_t-S).
\end{eqnarray*}
     Let $(a_{1t,a}, a_{1t,b}, a_{1t,c})$ respectively be the upper blocks of $(A_{1t,a}, A_{1t,b},A_{1t,c})$. Then 
      \begin{eqnarray*}
    \frac{1}{|\mathcal G|_0}\sum_{i\in\mathcal G} \|	\frac{1}{T}\sum_{t\notin I}   a_{1t}    e_{it}    f_t'\|^2&\leq&  \frac{1}{|\mathcal G|_0}\sum_{i\in\mathcal G} ( \frac{1}{T}\sum_{t\notin I}  (\|a_{1t,a}\|+\|a_{1t,b}\|+\|a_{1t,c}\|)   \|e_{it}    f_t\| (\|f_t\|+\|g_t\|) )^2.
   \end{eqnarray*}
   
     By   the Cauchy-Schwarz and Lemma \ref{lc.6}, and $B$ is a block diagonal matrix, 
    \begin{eqnarray*}
 && \frac{1}{|\mathcal G|_0}\sum_{i\in\mathcal G}[ \frac{1}{T}\sum_{t\notin I} \|a_{1t,b}\|  \|e_{it}    f_t\| (\|f_t\|+\|g_t\|)]^2\leq  O_P(1)  \frac{1}{T}\sum_{t\notin I} \|A_{1t,b}\|  ^2 \cr
 &\leq& O_P(\|S\|^2)   \frac{1}{T}\sum_{t\notin I} \|\widehat B_t-B\|  ^2 = O_P(C_{NT}^{-4}). \cr
&&  \frac{1}{|\mathcal G|_0}\sum_{i\in\mathcal G}[\frac{1}{T}\sum_{t\notin I} \|a_{1t,c}\|  \|e_{it}    f_t\| (\|f_t\|+\|g_t\|) ]^2\leq  O_P(1)  \frac{1}{T}\sum_{t\notin I} \|a_{1t,c}\|  ^2 \cr
  &\leq &O_P(1)  \frac{1}{T}\sum_{t\notin I} \|\Delta_{t1}\|  ^2 = O_P(C_{NT}^{-4}),
  \end{eqnarray*}
    where $\Delta_{t1}$ is defined in Lemma \ref{lc.6}, the upper block of $\widehat S_t-S$.

        The treatment of $\frac{1}{T}\sum_{t\notin I} \|a_{1t,a}\|  \|e_{it}    f_t\| (\|f_t\|+\|g_t\|)$ is slightly different. Note that $\max_{t\leq T}\|\widehat B_t^{-1}\|+\|B^{-1}\|=O_P(1)$, shown in Lemma \ref{lc.6}.
         Partition $$
           \widehat S_t-S =\begin{pmatrix}
        \Delta_{t1}\\
         (\Delta_{t2,1}+\Delta_{t2,2},0)
         \end{pmatrix}
         $$
        where the notation $\Delta_{t1}$ is  defined in the proof of Lemma \ref{lc.6}.    The proof of Lemma \ref{lc.6} also gives 
    \begin{eqnarray*}
        \Delta_{t2,1}&=& \frac{1}{N}\sum_i	\widetilde\alpha_i (\lambda_i'H_1 -\widetilde\lambda_i')e_{it} 
 		+ \frac{1}{N}\sum_i	\widetilde\alpha_i (\lambda_i'H_1 -\widetilde\lambda_i')(\widehat e_{it} -e_{it})\cr
		&&+\frac{1}{N}\sum_i	(\widetilde\alpha_i -H_2'\alpha_i)\lambda_i'H_1(e_{it}-\widehat e_{it})
      \cr
        \Delta_{t2,2}&=&H_2' \frac{1}{N}\sum_i	\alpha_i \lambda_i'(e_{it}-\widehat e_{it})H_1.
       \end{eqnarray*}
        Therefore, 
\begin{eqnarray}\label{eqd.18add}
       \|a_{1t,a}\|&\leq &\|    (\widehat  B_t^{-1}-  B^{-1}) (\widehat S_t-S)   \| \cr
       &\leq& (\max_{t}\|\widehat B_t^{-1}\|+\|B\|^{-1})(\|\Delta_{t1}\|+ \|\Delta_{t2,1}\|)+\|\widehat B_t^{-1}-B^{-1}\|\|\Delta_{t2,2}\|.\cr
  \end{eqnarray}
       Note that the above bound treats $\Delta_{t1}$ and $\Delta_{t2}$ differently because by the proof of Lemma  \ref{lc.6},  $\frac{1}{T}\sum_{t\notin I} \|\Delta_{t1}\|^2=O_P(C^{-4}_{NT})=\frac{1}{T}\sum_{t\notin I} \|\Delta_{t2,1}\|^2$ but  the rate of convergence for $\frac{1}{T}\sum_{t\notin I} \| \Delta_{t2,2}\|^2 $  is slower ($=O_P(C^{-2}_{NT})$).

       Hence 
                    \begin{eqnarray*}
&& \frac{1}{|\mathcal G|_0}\sum_{i\in\mathcal G}[ \frac{1}{T}\sum_{t\notin I} \|a_{1t,a}\|  \|e_{it}    f_t\| (\|f_t\|+\|g_t\|)]^2\leq O_P(1)  \frac{1}{T}\sum_{t\notin I} \|\Delta_{t1}\|  ^2+\|\Delta_{t2,1}\|  ^2 \cr
 &&+  \frac{1}{T}\sum_{t\notin I} \|\widehat B_t^{-1}-B^{-1} \|  ^2 
 \frac{1}{T}\sum_{t\notin I} \|\Delta_{t2,2} \|  ^2  \frac{1}{|\mathcal G|_0}\sum_{i\in\mathcal G}\|e_{it}    f_t\|^2 (\|f_t\|+\|g_t\|)^2    \cr
  &\leq^{(a)}& O_P(C^{-4}_{NT})+O_P(C^{-2}_{NT})    \frac{1}{T}\sum_{t\notin I} \|\Delta_{t2,2} \|  ^2 \| \frac{1}{|\mathcal G|_0}\sum_{i\in\mathcal G}  \|e_{it}    f_t\|^2 (\|f_t\|+\|g_t\|)^2    \cr
  &\leq^{(b)}&O_P(C^{-4}_{NT}).
  \end{eqnarray*}
  where $(a)$ follows from the proof of Lemma \ref{lc.6},  while $(b)$ follows from Lemma \ref{ld.1a}.  Thus $$  \frac{1}{|\mathcal G|_0}\sum_{i\in\mathcal G} \|\frac{1}{T}\sum_{t\notin I}  a_{1t}    e_{it}    f_t'\|^2=O_P(C^{-4}_{NT}).$$
  
Also note that in the above proof, we have also proved (to be used later)
\begin{eqnarray}\label{eqd.18}
\frac{1}{T}\sum_{t\notin I} [\|a_{1t,b}\|^2+ \|a_{1t,c}\|^2 +\|\Delta_{t1}\|^2+\|\Delta_{t2,1}\|^2 ]
\leq O_P(C_{NT}^{-4})  
\end{eqnarray}
       
\end{proof}

    \begin{lem}\label{ld.13} 
    for $\mu_{it}=l_i'w_t$, $\widehat\mu_{it}= \widehat{l_i'w_t}$, 
   	$$
   	\frac{1}{|\mathcal G|_0}\sum_{i\in\mathcal G}\|\frac{1}{T}\sum_{s\notin I}\widehat f_s\widehat e_{is} ( \mu_{it}\lambda_i'f_s-   \widehat\mu_{it}\dot\lambda_i'\widetilde f_s  )\|^2 =O_P(C_{NT}^{-4})$$
	$$ 	\frac{1}{|\mathcal G|_0}\sum_{i\in\mathcal G}\|\frac{1}{T}\sum_{s\notin I}\widehat g_s (\mu_{it}   \lambda_i'f_s- \widehat{\mu}_{it}\dot\lambda_i'\widetilde f_s  )\|^2 =O_P(C_{NT}^{-2}).
   	$$
   	 
   \end{lem}

   \begin{proof}
    It  suffices to prove that the following statements:
    $$
    \frac{1}{|\mathcal G|_0}\sum_{i\in\mathcal G}\|R_{3i,d}\|^2= O_P(C_{NT}^{-4}), d=1\sim 6, \quad   \frac{1}{|\mathcal G|_0}\sum_{i\in\mathcal G}\|R_{3i,7}\|^2= O_P(C_{NT}^{-2}),
    $$
    where
    \begin{eqnarray}\label{ec.39}
   R_{3i,1} &:=&\frac{1}{T}\sum_{s\notin I} (\widehat f_s -H_ff_s)(\widehat e_{is}-e_{is})\mu_{it}\lambda_i'f_s  +\frac{1}{T}\sum_{s\notin I} (\widehat f_s\widehat e_{is}-H_ff_se_{is})(\mu_{it}\lambda_i'f_s -\widehat \mu_{it}\dot\lambda_i'\widetilde f_s)  \cr
   R_{3i,2} &:=&\frac{1}{T}\sum_{s\notin I}  f_se_{is}(\widehat e_{is}-e_{is})(\dot\lambda_i'\widetilde f_s - \lambda_i' f_s) +\frac{1}{T}\sum_{s\notin I}  f_se_{is}\mu_{is}(\dot\lambda_i-H_1'\lambda_i)'(\widetilde f_s -H_1^{-1}f_s) \cr
  R_{3i,3} &:=&\frac{1}{T}\sum_{s\notin I}  f_se_{is}\mu_{is}f_s  ' H_1^{-1'}(\dot\lambda_i-H_1'\lambda_i)  \cr
           R_{3i,4} &:=&\frac{1}{T}\sum_{s\notin I}  f_se_{is}(\widehat e_{is}-e_{is}) \lambda_i'  f_s   \cr
   R_{3i,5} &:=&\frac{1}{T}\sum_{s\notin I}  f_se_{is} \mu_{is}\lambda_i'(\widetilde f_s-H_1^{-1}f_s )  \cr
   R_{3i,6} &:=&\frac{1}{T}\sum_{s\notin I} (\widehat f_s -H_ff_s) e_{is} \mu_{is}\lambda_i'f_s   \cr
   R_{3i,7} &:=&\frac{1}{T}\sum_{s\notin I} \widehat g_s(\mu_{is}\lambda_i'f_s -\widehat{\mu}_{is}\dot\lambda_i'\widetilde f_s)  .
   \end{eqnarray}
 So
     \begin{eqnarray*}
   \frac{1}{|\mathcal G|_0}\sum_{i\in\mathcal G}\|R_{3i,d}\|^2&=& O_P(C_{NT}^{-4}),\quad d=1,2,3, 
       \text{by Cauchy-Schwarz}\cr
      \frac{1}{|\mathcal G|_0}\sum_{i\in\mathcal G}\|R_{3i,4}\|^2&=& O_P(C_{NT}^{-4}),
       \text{ Lemma \ref{ld.2a} }\cr
           \frac{1}{|\mathcal G|_0}\sum_{i\in\mathcal G}\|R_{3i,5}\|^2&=& O_P(C_{NT}^{-4}),
       \text{ Lemma  \ref{lc.2} }\cr
                \frac{1}{|\mathcal G|_0}\sum_{i\in\mathcal G}\|R_{3i,6}\|^2&=& O_P(C_{NT}^{-4}),
       \text{ Lemma  \ref{lc.8} }, \cr
          \frac{1}{|\mathcal G|_0}\sum_{i\in\mathcal G}\|R_{3i,7}\|^2&=& O_P(C_{NT}^{-2}),
       \text{ by Cauchy-Schwarz}. 
     \end{eqnarray*}

      \end{proof}
      
      \begin{lem}\label{ld.14}
      Let  $r_{5i}$   be the upper block of $R_{5i}$, where  
	$$ R_{5i}=\widehat D_i^{-1}\frac{1}{T_0}\sum_{s\notin I}
\begin{pmatrix}
\widehat e_{is}\I&0  \\
0&\I
  \end{pmatrix}     \begin{pmatrix}
 \widehat f_s  -H_ff_s  \\
\widehat g_s -H_gg_s
  \end{pmatrix}    u_{is}. $$
      \end{lem}
      Then $\frac{1}{|\mathcal G|_0}\sum_{i\in\mathcal G}\|r_{5i}\|=O_P(C_{NT}^{-2}).$
      
      \begin{proof}
      For any matrix $A$ of the same size $R_{5i}$, we write 
      $
\mathcal U({A})
      $
      to denote its upper block (the first $\dim(f_t)$-subvector of $A$). Then in this notation $r_{5i}=  \mathcal U({R_{5i}}). $
            
      First,  by Cauchy-Schwarz 
\begin{eqnarray*}
 &&\left[\frac{1}{|\mathcal G|_0}\sum_{i\in\mathcal G}\|  \widehat D_i^{-1}-D_i^{-1}  \|
\left(\|\frac{1}{T}\sum_{s\notin I} \widehat e_{is}(\widehat f_s-H_ff_s)u_{is}\|
+\|\frac{1}{T}\sum_{s\notin I}  (\widehat g_s-H_gg_s)u_{is}\|\right)\right] ^2\cr
&\leq& O_P(C_{NT}^{-2})
\frac{1}{|\mathcal G|_0}\sum_{i\in\mathcal G}\left(\frac{1}{T}\sum_{s\notin I} (\widehat f_s-H_ff_s)^2\frac{1}{T}\sum_{s\notin I}  u_{is}^2\widehat e_{is}^2
+\frac{1}{T}\sum_{s\notin I}  (\widehat g_s-H_gg_s)^2  \frac{1}{T}\sum_{s\notin I} u_{is}^2\right)\cr
&\leq &O_P(C_{NT}^{-4}).
\end{eqnarray*}
Now  let $\bar R_{5i}$ be  defined as $R_{5i} $ but with $\widehat D_i^{-1}$ replaced with $D_i^{-1}$. Substitute in the expansion (\ref{ed.24}), then 
$$
\bar R_{5i}=   D_i^{-1}\frac{1}{T_0}\sum_{t\notin I}
\begin{pmatrix}
\widehat e_{it}\I&0  \\
0&\I
  \end{pmatrix}      
B^{-1}    \frac{1}{N}\sum_j
\begin{pmatrix}
 H_1'\lambda_je_{jt}   \\
 H_2'\alpha_j
\end{pmatrix}
 u_{jt}    u_{it}
  +  \sum_{d=1}^6 D_i^{-1}\frac{1}{T_0}\sum_{t\notin I}
\begin{pmatrix}
\widehat e_{it}\I&0  \\
0&\I
  \end{pmatrix}   A_{dt}   u_{it}.
$$ 
Both $B^{-1}$ and $D_i^{-1}$ are block diagonal; let $b, d_i$  respectively denote their first diagonal block. Then  
$$
\mathcal U({\bar R_{5i}}) =\frac{1}{T_0}\sum_{t\neq I} d_i\widehat e_{it} bH_1'    \frac{1}{N}\sum_j  \lambda_je_{jt}u_{jt}u_{it}
+\sum_{d=1}^6d_i\frac{1}{T_0}\sum_{t\neq I}\widehat e_{it} u_{it}\mathcal U(A_{dt}).
$$
The goal is now to prove $\frac{1}{|\mathcal G|_0}\sum_{i\in\mathcal G}\|\mathcal U({\bar R_{5i}}) \|=O_P(C_{NT}^{-2}).$

First,  let $\nu_t=(\lambda_{j}u_{jt}: j\leq N)$, a column vector.
\begin{eqnarray*}
&&\left[\frac{1}{|\mathcal G|_0}\sum_{i\in\mathcal G}\|\frac{1}{T_0}\sum_{t\neq I} d_i(\widehat e_{it}-e_{it}) bH_1'    \frac{1}{N}\sum_j  \lambda_je_{jt}u_{jt}u_{it}\|
\right]^2\cr
&\leq&  O_P(C_{NT}^{-2})    \frac{1}{T_0}\sum_{t\neq I} \frac{1}{N^2} \E\Var( \nu_t'e_{t} u_{it}| U)
\leq  O_P(C_{NT}^{-2})    \frac{1}{T_0}\sum_{t\neq I} \frac{1}{N^2} \E u_{it}^2  \nu_t'\Var(e_{t}| U) \nu_t\cr
&\leq&  O_P(C_{NT}^{-2}N^{-1})    \frac{1}{T_0}\sum_{t\neq I}\E u_{it}^2    \frac{1}{N} \sum_{j=1}^N\lambda_j^2u_{jt}^2=O_P(C_{NT}^{-4}). 
\end{eqnarray*}
Next,  for $\omega_{it}=e_{it}u_{it}$,
\begin{eqnarray*}
&&\left[\frac{1}{|\mathcal G|_0}\sum_{i\in\mathcal G}\|\frac{1}{T_0}\sum_{t\neq I} d_i e_{it} bH_1'    \frac{1}{N}\sum_j  \lambda_je_{jt}u_{jt}u_{it}\|
\right]^2\cr
&\leq& O_P(1)\frac{1}{|\mathcal G|_0}\sum_{i\in\mathcal G}\E |\frac{1}{NT}\sum_{t  }\sum_j  e_{it}       \lambda_je_{jt}u_{jt}u_{it}|
 ^2\cr
 &\leq& O_P(1)\frac{1}{N^2}   (\max_i   \sum_j           |\E \omega_{jt}\omega_{it}|) ^2+O_P(1)\frac{1}{N^2T^2}\frac{1}{|\mathcal G|_0}\sum_{i\in\mathcal G}\sum_{t  }\Var (\sum_j     \lambda_j \omega_{jt}\omega_{it})\cr
 &\leq& O_P(N^{-2}) +O_P(1)\frac{1}{N^2T^2}\frac{1}{|\mathcal G|_0}\sum_{i\in\mathcal G}\sum_{t  } \sum_{jk}  |  \Cov(   \omega_{jt}\omega_{it}, 
\omega_{kt}\omega_{it})|\cr
&=&O_P(N^{-2}+N^{-1}T^{-1})=O_P(C_{NT}^{-4}). 
\end{eqnarray*}
We now  show $\sum_{d=1}^6\frac{1}{|\mathcal G|_0}\sum_{i\in\mathcal G}\| d_i\frac{1}{T_0}\sum_{t\neq I}\widehat e_{it} u_{it}\mathcal U(A_{dt})\|= O_P(C_{NT}^{-2})$. For each $d\leq 6$, 
 \begin{eqnarray*}
 &&[\frac{1}{|\mathcal G|_0}\sum_{i\in\mathcal G}\| d_i\frac{1}{T_0}\sum_{t\neq I}(\widehat e_{it}-e_{it}) u_{it}\mathcal U(A_{dt})\|]^2
\cr
& \leq& \frac{1}{|\mathcal G|_0}\sum_{i\in\mathcal G}  \frac{1}{T_0}\sum_{t\neq I}(\widehat e_{it}-e_{it}) ^2u_{it}^2d_i^2\frac{1}{T_0}\sum_{t\neq I}\|A_{dt}\|^2
=O_P(C_{NT}^{-4}),
 \end{eqnarray*}
where it  follows from applying Cauchz-Schwarz that $\frac{1}{T_0}\sum_{t\neq I}\|A_{dt}\|^2=O_P(C_{NT}^{-2})$.  
  It remains to prove $\sum_{d=1}^6\frac{1}{|\mathcal G|_0}\sum_{i\in\mathcal G}\| d_i\frac{1}{T_0}\sum_{t\neq I} \omega_{it}  \mathcal U(A_{dt})\|= O_P(C_{NT}^{-2})$.
  
\textbf{term $\mathcal U(A_{1t})$.}
 Using the notation of the proof of Lemma \ref{lc.8}, we have 
$$
   \mathcal U(A_{1t})
:= (a_{1t,a} +a_{1t,b}+a_{1t,c}) \begin{pmatrix}
 H_1^{-1}f_t\\
 H_2^{-1}g_t
\end{pmatrix},
$$
Also, (\ref{eqd.18add}) and (\ref{eqd.18}) yield, for 
$
m_{it}:= \|d_i\omega_{it}\|(\|f_t\|+\|g_t\|),
$
\begin{eqnarray*}
&&\frac{1}{|\mathcal G|_0}\sum_{i\in\mathcal G}\| d_i\frac{1}{T_0}\sum_{t\neq I} \omega_{it}\mathcal U(A_{dt})\|\cr
&\leq& [\frac{1}{T_0}\sum_{t\neq I}  \|a_{1t,b}+a_{1t,c}\|^2]^{1/2}[ \frac{1}{|\mathcal G|_0}\sum_{i\in\mathcal G}\frac{1}{T_0}\sum_{t\neq I} m_{it}^2 ]^{1/2}+ \frac{1}{|\mathcal G|_0}\sum_{i\in\mathcal G} \frac{1}{T_0}\sum_{t\neq I} \| m_{it}  a_{1t,a}\| \cr
&\leq& O_P(C_{NT}^{-2})
+ [\frac{1}{T_0}\sum_{t\neq I}  \|\Delta_{t1}+\Delta_{t2,1}\|^2]^{1/2} [ \frac{1}{|\mathcal G|_0}\sum_{i\in\mathcal G}\frac{1}{T_0}\sum_{t\neq I} m_{it}^2 ]^{1/2}\cr
&&+ \frac{1}{|\mathcal G|_0}\sum_{i\in\mathcal G} \frac{1}{T_0}\sum_{t\neq I}  m_{it}   \|\widehat B_t^{-1}-B^{-1}\|\|\Delta_{t2,2}\|\cr
&\leq & O_P(C_{NT}^{-2})+ O_P(C_{NT}^{-1}) [ \frac{1}{T}\sum_{t}\|\Delta_{t2,2}\|^4]^{1/4}\cr
&\leq & O_P(C_{NT}^{-2})+ O_P(C_{NT}^{-1}) [ \frac{1}{T}\sum_{t} \frac{1}{N}\sum_i	 (e_{it}-\widehat e_{it})^4 ]^{1/4}=O_P(C_{NT}^{-2}).
\end{eqnarray*}

      \textbf{terms $\mathcal U(A_{2t}),\mathcal U(A_{4t}),\mathcal U(A_{5t})$.} $\frac{1}{|\mathcal G|_0}\sum_{i\in\mathcal G}\| d_i\frac{1}{T_0}\sum_{t\neq I} \omega_{it}\mathcal U(A_{dt})\|= O_P(C_{NT}^{-2})$, $d=2,4,5$ follow from the simple Cauchy-Schwarz. 
      
            \textbf{term $\mathcal U(A_{3t})$.}   Because $\max_{t\leq T}\|\widehat B_t^{-1}\|=O_P(1)$,
            \begin{eqnarray*}
&&\frac{1}{|\mathcal G|_0}\sum_{i\in\mathcal G}\| d_i\frac{1}{T_0}\sum_{t\neq I} \omega_{it}\mathcal U(A_{3t})\|\leq \frac{1}{|\mathcal G|_0}\sum_{i\in\mathcal G}\| d_i\frac{1}{T_0}\sum_{t\neq I} \omega_{it} A_{3t}\|\cr
&\leq&   [ \frac{1}{T_0}\sum_{t\neq I}  \|\frac{1}{N}\sum_j\lambda_i  u_{jt} (\widehat e_{jt}  -  e_{jt}  ) \|^2]^{1/2} 
+ [ \frac{1}{T_0}\sum_{t\neq I}  \|\frac{1}{N}\sum_ju_{jt}( \widetilde\lambda_j - H_1'\lambda_j  )(\widehat e_{jt} -e_{jt}) \|^2]^{1/2} 
\cr
&&+
[ \frac{1}{T_0}\sum_{t\neq I}  \|\frac{1}{N}\sum_j\omega_{jt}( \widetilde\lambda_j - H_1'\lambda_j  ) \|^2]^{1/2} + [ \frac{1}{T_0}\sum_{t\neq I}  \|\frac{1}{N}\sum_ju_{jt}(\widetilde\alpha_j- H_2'\alpha_j) \|^2]^{1/2} \cr
&=&O_P(C_{NT}^{-2}).
\end{eqnarray*}
            The first term follows from Lemma \ref{ld.2a}. The second term follows from Cauchy-Schwarz. The third and fourth are due to, for instance, \begin{eqnarray*}
&&           \frac{1}{T_0}\sum_{t\neq I}\E_I   \|\frac{1}{N}\sum_j\omega_{jt}( \widetilde\lambda_j - H_1'\lambda_j  ) \|^2\cr
&\leq& \frac{1}{T_0}\sum_{t\neq I}\frac{1}{N^2}\Var_I(\omega_t'(\widetilde\Lambda-\Lambda H_1))
\leq  \frac{1}{T_0}\sum_{t\neq I}\frac{1}{N^2}\|\widetilde\Lambda-\Lambda H_1\|^2\|\Var_I(\omega_t)\|\cr
&\leq& O(C_{NT}^{-4}).
            \end{eqnarray*}

              \textbf{term $\mathcal U(A_{6t})$.} Note that $B^{-1}$ is block diagonal.  Let $b$ be the first block of $B^{-1}$. Recall that Lemma \ref{ld.9} shows 
                                  $$
     \mathcal U(  A_{6t})=b H_1'\sum_{d=1}^4 B_{dt}.
       $$
       So         $\frac{1}{|\mathcal G|_0}\sum_{i\in\mathcal G}\| d_i\frac{1}{T_0}\sum_{t\neq I} \omega_{it}\mathcal U(A_{6t})\|\leq  O_P(1)\sum_{d=1}^4
\frac{1}{|\mathcal G|_0}\sum_{i\in\mathcal G}\| d_i\frac{1}{T_0}\sum_{t\neq I} \omega_{it} B_{dt}\|.$

For $d=1,4$,  we apply Cauchy-Schwarz,
  \begin{eqnarray*}
&& 
 \frac{1}{|\mathcal G|_0}\sum_{i\in\mathcal G}\| d_i\frac{1}{T_0}\sum_{t\neq I} \omega_{it} B_{dt}\|
\leq O_P(1) [\frac{1}{T}\sum_t\|B_{dt}\|^2]^{1/2}=O_P(C_{NT}^{-2}).
\end{eqnarray*}
For $d=2$, we still apply Cauchy-Schwarz and Lemma \ref{ld.2a} that  
 \begin{eqnarray*}
&& 
 \frac{1}{|\mathcal G|_0}\sum_{i\in\mathcal G}\| d_i\frac{1}{T_0}\sum_{t\neq I} \omega_{it} B_{2t}\|
\leq O_P(C_{NT}^{-2})+O_P(1) [\frac{1}{T}\sum_t (\frac{1}{N}\sum_i\lambda_i ^2e_{it}  (\widehat e_{it}-e_{it}) )^2 ]^{1/2}\cr
&=&O_P(C_{NT}^{-2}).
\end{eqnarray*}
              Finally, we bound for $d=3$. 
     \begin{eqnarray*}
&& 
 [\frac{1}{|\mathcal G|_0}\sum_{i\in\mathcal G}\| d_i\frac{1}{T_0}\sum_{t\neq I} \omega_{it} B_{3t}\|]^2
\leq  O_P(1)  \frac{1}{|\mathcal G|_0}\sum_{i\in\mathcal G}\|  \frac{1}{T_0}\sum_{t\neq I} \omega_{it} B_{3t}\|^2\cr
&\leq&O_P(1)  \frac{1}{|\mathcal G|_0}\sum_{j\in\mathcal G}\|  \frac{1}{T_0}\sum_{t\neq I} \omega_{jt}   \frac{1}{N}\sum_il_i\lambda_i e_{it} w_t (\dot \lambda_i-H_1'\lambda_i)' \widetilde f_t\|^2\cr
&\leq& O_P(C_{NT}^{-2})  \frac{1}{|\mathcal G|_0}\sum_{j\in\mathcal G} \frac{1}{N}\sum_i \|  \frac{1}{T_0}\sum_{t\neq I} \omega_{jt}   e_{it} w_t ( \widetilde f_t-H_1^{-1}f_t)\|^2\cr
&&+  O_P(C_{NT}^{-2})  \frac{1}{|\mathcal G|_0}\sum_{j\in\mathcal G} \frac{1}{N}\sum_i \|  \frac{1}{T_0}\sum_{t\neq I} \omega_{jt}   e_{it} w_t  f_t\|^2\cr
&=&O_P(C_{NT}^{-4}) +O_P(C_{NT}^{-2})  \frac{1}{|\mathcal G|_0}\sum_{j\in\mathcal G} \frac{1}{N}\sum_i \|  \frac{1}{T_0}\sum_{t\neq I} e_{jt} u_{jt}  e_{it} w_t  f_t\|^2\cr
&=&O_P(C_{NT}^{-4}).
\end{eqnarray*}
The last equality is due to the conditional serial  independence of $(e_t, u_t)$ and that $\E(u_t|E, F, W)=0.$
          
      \end{proof}

      \begin{lem}\label{ld.15}
      Let  $r_{6i}$   be the upper block of $R_{6i}$, where  
	$$ R_{6i}=\widehat D_i^{-1}\frac{1}{T_0}\sum_{s\notin I}
   \begin{pmatrix}
\widehat f_s \widehat e_{is} \\
\widehat g_s 
  \end{pmatrix}( \lambda_i'H_f^{-1}, \alpha_i'H_g^{-1})\begin{pmatrix}
\widehat e_{is}\I&0  \\
0&\I
  \end{pmatrix}     \begin{pmatrix}
 \widehat f_s  -H_ff_s  \\
\widehat g_s -H_gg_s
  \end{pmatrix} . $$
      \end{lem}
      Then $\frac{1}{|\mathcal G|_0}\sum_{i\in\mathcal G}\|r_{6i}\|=O_P(C_{NT}^{-2}).$
          \begin{proof}  
    Write $B^{-1}=\diag(b_1, b_2)$,
\begin{eqnarray*}
\Gamma^j&:=& \Gamma_0^j+\Gamma_1^j+...+\Gamma_6^j,\cr
  \Gamma_0^j&:=&\frac{1}{T_0}\sum_{t\notin I}
   \begin{pmatrix}
\widehat f_t \widehat e_{jt} \\
\widehat g_t 
  \end{pmatrix}( \lambda_j'H_f^{-1}\widehat e_{jt}b_1, \alpha_j'H_g^{-1}b_2)       \frac{1}{N}\sum_i 
\begin{pmatrix}
 H_1'\lambda_i e_{it}   \\
 H_2'\alpha_i
\end{pmatrix}
 u_{it} 
\cr
  \Gamma_d^j&:=& \frac{1}{T_0}\sum_{t\notin I}
   \begin{pmatrix}
\widehat f_t \widehat e_{jt} \\
\widehat g_t
  \end{pmatrix}( \lambda_j'H_f^{-1}\widehat e_{jt}, \alpha_j'H_g^{-1}) A_{dt}  ,\quad d=1,...,6.
  \end{eqnarray*}
Then $ R_{6j}=\widehat D_j^{-1}\Gamma^j$. Let $d_j$ be the first diagonal block of $D_j^{-1}$ (which is a block diagonal matrix),
      We have 
\begin{eqnarray*}
    [  \frac{1}{|\mathcal G|_0}\sum_{j\in\mathcal G}\|r_{6j}\|]^2
     & \leq&  \frac{1}{|\mathcal G|_0}\sum_{j\in\mathcal G}\|\widehat D_j^{-1}-D_j^{-1}\|^2 \frac{1}{|\mathcal G|_0}\sum_{j\in\mathcal G}\|\Gamma^j\|^2
      + [  \frac{1}{|\mathcal G|_0}\sum_{j\in\mathcal G}\|d_j\mathcal U(\Gamma^j)\|]^2\cr
      &\leq& O_P(C_{NT}^{-2}) \frac{1}{|\mathcal G|_0}\sum_{j\in\mathcal G}\|\Gamma^j\|^2
+ \frac{1}{|\mathcal G|_0}\sum_{j\in\mathcal G}\|\mathcal U(\Gamma^j)\|^2\cr
&\leq &   O_P(1) \sum_{d=2}^5\frac{1}{|\mathcal G|_0}\sum_{j\in\mathcal G}\| \Gamma_d^j\|^2+O_P(1)  \frac{1}{|\mathcal G|_0}\sum_{j\in\mathcal G}\|\Gamma^j_0\|^2\cr
&&+O_P(C_{NT}^{-2})  \frac{1}{|\mathcal G|_0}\sum_{j\in\mathcal G}\|\Gamma^j_1\|^2
+  \frac{1}{|\mathcal G|_0}\sum_{j\in\mathcal G}\|\mathcal U(\Gamma_1^j)\|^2\cr
&&+O_P(C_{NT}^{-2})  \frac{1}{|\mathcal G|_0}\sum_{j\in\mathcal G}\|\Gamma^j_6\|^2
+  \frac{1}{|\mathcal G|_0}\sum_{j\in\mathcal G}\|\mathcal U(\Gamma_6^j)\|^2.
  \end{eqnarray*}
      In the above inequalities, we treat individual $\Gamma_d^j$ differently. This is because we aim to show the following:
\begin{eqnarray*}
      \frac{1}{|\mathcal G|_0}\sum_{j\in\mathcal G}\| \Gamma_d^j\|^2&=&O_P(C_{NT}^{-4}),\quad d=0, 2\sim 5
,\cr
      \frac{1}{|\mathcal G|_0}\sum_{j\in\mathcal G}\| \Gamma_d^j\|^2&=&O_P(C_{NT}^{-2}),\quad d= 1, 6
,\cr
      \frac{1}{|\mathcal G|_0}\sum_{j\in\mathcal G}\| \mathcal U(\Gamma_d^j)\|^2&=&O_P(C_{NT}^{-4}),\quad d= 1, 6.
  \end{eqnarray*}
      That is, while  it is not likely to prove $ \frac{1}{|\mathcal G|_0}\sum_{j\in\mathcal G}\| \Gamma_d^j\|^2$ are  also $ O_P(C_{NT}^{-4})$ for $d=1,6$, it can be proved that their  upper bounds are.
      
      Step 1: $  \frac{1}{|\mathcal G|_0}\sum_{j\in\mathcal G}\| \Gamma_d^j\|^2$ for $d=0,2\sim 5$.   Let $$m_t= \max_{jt}|  \widehat e_{jt} ^4-  e_{jt} ^4 |+   \frac{1}{|\mathcal G|_0}\sum_{j\in\mathcal G}e_{jt}^4+f_t^2+g_t^2+\max_{jt}|  \widehat e_{jt} ^2-  e_{jt} ^2 |+   \frac{1}{|\mathcal G|_0}\sum_{j\in\mathcal G}  e_{jt}^2+1.$$
      \begin{eqnarray*}
       \frac{1}{|\mathcal G|_0}\sum_{j\in\mathcal G}\| \Gamma_0^j\|^2
      & \leq&        \frac{1}{|\mathcal G|_0}\sum_{j\in\mathcal G}\|  \frac{1}{T}\sum_{t\notin I}  \widehat f_t \frac{1}{N}\sum_i
    \widehat e_{jt} ^2u_{it} \lambda_j\lambda_ie_{it}\|^2
     +    \frac{1}{|\mathcal G|_0}\sum_{j\in\mathcal G}\|  \frac{1}{T}\sum_{t\notin I}  \widehat f_t\frac{1}{N}\sum_i
    \widehat e_{jt} u_{it}\alpha_j\alpha_i\|^2\cr
       &&+      \frac{1}{|\mathcal G|_0}\sum_{j\in\mathcal G}\|  \frac{1}{T}\sum_{t\notin I}  \widehat g_t\frac{1}{N}\sum_i
    u_{it} \lambda_j\lambda_i\widehat e_{jt}e_{it}\|^2
      +    \frac{1}{|\mathcal G|_0}\sum_{j\in\mathcal G}\|  \frac{1}{T}\sum_{t\notin I}      \widehat g_t\frac{1}{N}\sum_i
u_{it}\alpha_j\alpha_i\|^2\cr
    & \leq&  
   O_P(C_{NT}^{-2})\frac{1}{T}\sum_{t\notin I} m_t \| \frac{1}{N}\sum_i
  u_{it} \lambda_ie_{it}\|^2
+    O_P(C_{NT}^{-2})\frac{1}{T}\sum_{t\notin I}     m_t\|\frac{1}{N}\sum_i
   u_{it} \alpha_i\|^2\cr
         &&+      \frac{1}{|\mathcal G|_0}\sum_{j\in\mathcal G}\|  \frac{1}{T}\sum_{t\notin I}    g_t\frac{1}{N}\sum_i
    u_{it} \lambda_j\lambda_i e_{jt}e_{it}\|^2
      +    \frac{1}{|\mathcal G|_0}\sum_{j\in\mathcal G}\|  \frac{1}{T}\sum_{t\notin I}        g_t\frac{1}{N}\sum_i
u_{it}\alpha_j\alpha_i\|^2\cr
&&+ \frac{1}{|\mathcal G|_0}\sum_{j\in\mathcal G}\|  \frac{1}{T}\sum_{t\notin I}        e_{jt} ^2 f_t \frac{1}{N}\sum_i
u_{it} \lambda_j\lambda_ie_{it}\|^2
     +    \frac{1}{|\mathcal G|_0}\sum_{j\in\mathcal G}\|  \frac{1}{T}\sum_{t\notin I}        e_{jt}f_t\frac{1}{N}\sum_i
  u_{it}\alpha_j\alpha_i\|^2\cr
 &\leq&  O_P(C_{NT}^{-4}).
      \end{eqnarray*}   
      Now let 
  $$
  \mathcal D= \frac{1}{T_0}\sum_{t\in I^c}
 \|\widehat f_t-H_ff_t \|^2+\| \widehat g_t -H_gg_t  \|^2
 =O_P(C_{NT}^{-2}).
  $$
      \begin{eqnarray*}
 &&    \frac{1}{|\mathcal G|_0}\sum_{j\in\mathcal G}\| \Gamma_2^j\|^2 \cr
     &=&
         \frac{1}{|\mathcal G|_0}\sum_{j\in\mathcal G}\| \frac{1}{T_0}\sum_{t\in I^c}
  \begin{pmatrix}
\widehat f_t \widehat e_{jt} \\
\widehat g_t 
  \end{pmatrix} ( \lambda_j'H_f^{-1}\widehat e_{jt}, \alpha_j'H_g^{-1})   (\widehat B_t^{-1}- B^{-1})  \frac{1}{N}\sum_i 
\begin{pmatrix}
 H_1'\lambda_i e_{it}   \\
 H_2'\alpha_i
\end{pmatrix}u_{it}   \|^2\cr
&\leq& \frac{1}{|\mathcal G|_0}\sum_{j\in\mathcal G}\mathcal D   \frac{1}{T_0}\sum_{t\in I^c} (|\widehat e_{jt}|^2+1)^2\left\|  \frac{1}{N}\sum_i 
\begin{pmatrix}
 \lambda_i e_{it}   \\
 \alpha_i
\end{pmatrix}u_{it} \right\|^2\max_{t\leq T}\|\widehat B_t^{-1}-B^{-1}\|\cr
&&+  \frac{1}{T_0}\sum_{t\in I^c}
 \|\widehat B_t^{-1}-B^{-1} \|^2   \frac{1}{|\mathcal G|_0}\sum_{j\in\mathcal G}  \frac{1}{T_0}\sum_{t\in I^c}\left\|
  \begin{pmatrix}
 f_t \widehat e_{jt} \\
 g_t 
  \end{pmatrix} \right\|^2  (|\widehat e_{jt}|+1)^2\left\| \frac{1}{N}\sum_i 
\begin{pmatrix}
 \lambda_i e_{it}   \\
 \alpha_i
\end{pmatrix}u_{it}  \right\|^2\cr
&\leq& O_P(C_{NT}^{-2}) \frac{1}{T_0}\sum_{t\in I^c}n_t\left\| \frac{1}{N}\sum_i 
\begin{pmatrix}
 \lambda_i e_{it}   \\
 \alpha_i
\end{pmatrix}u_{it}  \right\|^2=O_P(C_{NT}^{-4}),
     \end{eqnarray*}   
     where 
     $$
     n_t=\max_{t\leq T}\|\widehat B_t^{-1}-B^{-1}\|+\frac{1}{|\mathcal G|_0}\sum_{j\in\mathcal G}  \frac{1}{T_0}\sum_{t\in I^c}   (| e_{jt}|+1+\max_{jt}|e_{jt}-\widehat e_{jt}|)^2[\|f_t\|^2+\|g_t\|^2+\|f_te_{jt}\|^2].
     $$
       
   Next, let  $$
   \mathcal A_t=\frac{1}{N}\sum_i 
\begin{pmatrix}
\widetilde\lambda_i\widehat e_{it}  - H_1'\lambda_i e_{it}   \\
\widetilde\alpha_i -H_2'\alpha_i
\end{pmatrix}
 u_{it} .
   $$
    \begin{eqnarray*}
  \frac{1}{|\mathcal G|_0}\sum_{j\in\mathcal G}\| \Gamma_3^j\|^2 &=& \frac{1}{|\mathcal G|_0}\sum_{j\in\mathcal G}\| \frac{1}{T_0}\sum_{t\in I^c}
  \begin{pmatrix}
\widehat f_t \widehat e_{jt} \\
\widehat g_t 
  \end{pmatrix} ( \lambda_j'H_f^{-1}\widehat e_{jt}, \alpha_j'H_g^{-1})  \widehat B_t^{-1}   \mathcal A_t \|^2\cr
  &\leq&
 \frac{1}{T_0}\sum_{t\in I^c}  \mathcal D \frac{1}{|\mathcal G|_0}\sum_{j\in\mathcal G}  ( e_{jt}^2+1)^2 \|   \mathcal A_t\|^2 + O_P(1) \frac{1}{T_0}\sum_{t\in I^c} \| \mathcal A_t\|^2\cr
 &\leq&O_P(C_{NT}^{-4})    + \frac{1}{T_0}\sum_{t\in I^c}  \|\frac{1}{N}\sum_i (\widetilde\alpha_i -H_2'\alpha_i)u_{it} \|^2\cr
 &&+   \frac{1}{T_0}\sum_{t\in I^c} \|\frac{1}{N}\sum_i ( \widetilde\lambda_i- H_1'\lambda_i ) e_{it}  u_{it} \|^2 +   \frac{1}{T_0}\sum_{t\in I^c}  \|\frac{1}{N}\sum_i  \lambda_i (\widehat e_{it}-e_{it}  )u_{it} \|^2\cr
  &\leq&O_P(C_{NT}^{-4})  ,
           \end{eqnarray*}   
  where the last equality is due to Lemma \ref{ld.2a} and :
  \begin{eqnarray*}
&& \frac{1}{T_0}\sum_{t\in I^c}   \E[| \frac{1}{N}\sum_i (\widetilde\lambda_i - H_1'\lambda_i )e_{it} u_{it}|^2| E, I]\cr
&=&  \frac{1}{T_0}\sum_{t\in I^c}  \frac{1}{N^2}(\widetilde\Lambda-\Lambda H_1)'\diag(e_t)\Var( u_t |E, I)  \diag(e_t)(\widetilde\Lambda-\Lambda H_1)\cr
&\leq& O_P(1)  \frac{1}{N^2} \sum_i(\widetilde\lambda_i - H_1'\lambda_i )^2\frac{1}{T_0}\sum_{t\in I^c} \E_I e_{it}^2 =O_P(C_{NT}^{-4}).
     \end{eqnarray*}

Next,  for $C_{dt} $   defined in the proof of Lemma \ref{ld.9},
  \begin{eqnarray*}
  \frac{1}{|\mathcal G|_0}\sum_{j\in\mathcal G}\| \Gamma_4^j\|^2 &=& \sum_{d=1}^{16} 
    \frac{1}{|\mathcal G|_0}\sum_{j\in\mathcal G}\| \frac{1}{T_0}\sum_{t\in I^c}
  \begin{pmatrix}
\widehat f_t \widehat e_{jt} \\
\widehat g_t 
  \end{pmatrix} ( \lambda_j'H_f^{-1}\widehat e_{jt}, \alpha_j'H_g^{-1})  \widehat B_t^{-1}C_{dt} \|^2  \cr
  &\leq& \sum_{d=1}^{16} \mathcal D 
    \frac{1}{|\mathcal G|_0}\sum_{j\in\mathcal G}  \frac{1}{T_0}\sum_{t\in I^c}|e_{jt}|^4 \|C_{dt}\|^2 +O_P(1) \sum_{d=1}^{16}  \frac{1}{T_0}\sum_{t\in I^c}  \|C_{dt} \|^2\cr
   & =&O_P(C_{NT}^{-4}).
           \end{eqnarray*}   

For $B_{dt} $   defined in the proof of Lemma \ref{ld.9},  and
$$n_t=\|
  f_t  \|^2\frac{1}{|\mathcal G|_0}\sum_{j\in\mathcal G} (e_{jt} ^2 +  |e_{jt}| )^2+ \|
 g_t  \|^2 \frac{1}{|\mathcal G|_0}\sum_{j\in\mathcal G} (1 +  |e_{jt}| )^2 + \frac{1}{|\mathcal G|_0}\sum_{j\in\mathcal G}(|  e_{jt}|^2+1)^2,$$
  \begin{eqnarray*}
  \frac{1}{|\mathcal G|_0}\sum_{j\in\mathcal G}\| \Gamma_5^j\|^2 &=&   
    \frac{1}{|\mathcal G|_0}\sum_{j\in\mathcal G}\|  \sum_{d=1}^{8} \frac{1}{T_0}\sum_{t\in I^c}
  \begin{pmatrix}
\widehat f_t \widehat e_{jt} \\
\widehat g_t 
  \end{pmatrix} ( \lambda_j'H_f^{-1}\widehat e_{jt}, \alpha_j'H_g^{-1})   (\widehat B_t^{-1} -B^{-1})  B_{dt}\|^2     \cr
  &\leq &[ \max_{t\leq T}\|\widehat B_t^{-1}-B^{-1}\|\mathcal D +  O_P(C_{NT}^{-2})  ]\sum_{d=1}^{8}   \frac{1}{T_0}\sum_{t\in I^c}n_t \|B_{dt}\|^2\cr
  &\leq&   O_P(C_{NT}^{-2})  \sum_{d=1}^{8}   \frac{1}{T_0}\sum_{t\in I^c}n_t \|B_{dt}\|^2= O_P(C_{NT}^{-4}),
      \end{eqnarray*}   
           where the last equality follows from simply applying Cauchy-Schwarz to show $   \frac{1}{T_0}\sum_{t\in I^c}n_t \|B_{dt}\|^2= O_P(C_{NT}^{-2})$ for $d=1\sim 8.$
           
  Step 2: $  \frac{1}{|\mathcal G|_0}\sum_{j\in\mathcal G}\| \Gamma_d^j\|^2$ for $d=1,6$. 
  
By Lemma \ref{lc.6}, $\max_{t\leq T}[\|\widehat B_t^{-1}+\|\widehat S_t\|]=O_P(1),$ $\frac{1}{T}\sum_t\|\widehat B_t^{-1}\widehat S_t-  B^{-1} S \|^2 =O_P(C_{NT}^{-2})$.
\begin{eqnarray*}   
 &&     \frac{1}{|\mathcal G|_0}\sum_{j\in\mathcal G}\| \Gamma_1^j\|^2\cr
 & =&      \frac{1}{|\mathcal G|_0}\sum_{j\in\mathcal G}\|  \frac{1}{T_0}\sum_{t\in I^c}
  \begin{pmatrix}
\widehat f_t \widehat e_{jt} \\
\widehat g_t 
  \end{pmatrix} ( \lambda_j'H_f^{-1}\widehat e_{jt}, \alpha_j'H_g^{-1})  (\widehat B_t^{-1}\widehat S_t-  B^{-1} S )\begin{pmatrix}
 H_1^{-1}f_t\\
 H_2^{-1}g_t
\end{pmatrix}  \|^2\cr
&\leq&  O_P(C_{NT}^{-2})  \frac{1}{|\mathcal G|_0}\sum_{j\in\mathcal G}  \frac{1}{T_0}\sum_{t\in I^c}
 |\widehat e_{jt}| ^2 (|\widehat e_{jt}|^2+1)\|\begin{pmatrix}
 H_1^{-1}f_t\\
 H_2^{-1}g_t
\end{pmatrix}  \|^2\cr
&&+  \frac{1}{T}\sum_t\|(\widehat B_t^{-1}\widehat S_t-  B^{-1} S )\|^2   \frac{1}{|\mathcal G|_0}\sum_{j\in\mathcal G}  \frac{1}{T_0}\sum_{t\in I^c}\|
  \begin{pmatrix}
  f_t \widehat e_{jt} \\
  g_t 
  \end{pmatrix} ( \lambda_j'H_f^{-1}\widehat e_{jt}, \alpha_j'H_g^{-1})  \|^2\|\begin{pmatrix}
 H_1^{-1}f_t\\
 H_2^{-1}g_t
\end{pmatrix}  \|^2\cr
&\leq&  O_P(C_{NT}^{-2})  .\cr
  && \frac{1}{|\mathcal G|_0}\sum_{j\in\mathcal G}\| \Gamma_6^j\|^2\cr
   &=& \frac{1}{|\mathcal G|_0}\sum_{j\in\mathcal G}\|\sum_{d=1}^6\frac{1}{T_0}\sum_{t\notin I}
   \begin{pmatrix}
\widehat f_t \widehat e_{jt} \\
\widehat g_t
  \end{pmatrix}( \lambda_j'H_f^{-1}\widehat e_{jt}, \alpha_j'H_g^{-1}) B^{-1} \begin{pmatrix}H_1'\sum_{d=1}^4 B_{dt}\\
       H_2'\sum_{d=5}^8 B_{dt}
       \end{pmatrix}  \|^2\cr
       &\leq& O_P(C_{NT}^{-2}) \sum_{d=1}^8    \frac{1}{T_0}\sum_{t\in I^c}
 n_t \|B_{dt} \|^2\cr
  &&+   \sum_{d=1}^8  \frac{1}{T}\sum_t\| B_{dt}\|^2   \frac{1}{|\mathcal G|_0}\sum_{j\in\mathcal G}  \frac{1}{T_0}\sum_{t\in I^c}\|
  \begin{pmatrix}
  f_t \widehat e_{jt} \\
  g_t 
  \end{pmatrix} ( \lambda_j'H_f^{-1}\widehat e_{jt}, \alpha_j'H_g^{-1})  \|^2 =O_P(C_{NT}^{-2}) .
          \end{eqnarray*}   
          
      Step 3: $  \frac{1}{|\mathcal G|_0}\sum_{j\in\mathcal G}\| \mathcal U(\Gamma_d^j)\|^2$ for $d=1,6$. 
            Let $$q_t= \frac{1}{|\mathcal G|_0}\sum_{j\in\mathcal G} \sqrt{ e_{jt}^4+e_{jt}^2}(\|f_t\|^2+\|g_t\|^2)(\|f_t\|^2+1)
            $$
\begin{eqnarray*}   
      \frac{1}{|\mathcal G|_0}\sum_{j\in\mathcal G}\| \mathcal U(\Gamma_1^j)\|^2&=&
   \frac{1}{|\mathcal G|_0}\sum_{j\in\mathcal G}\|   \frac{1}{T_0}\sum_{t\in I^c}
  \widehat f_t \widehat e_{jt} ( \lambda_j'H_f^{-1}\widehat e_{jt}, \alpha_j'H_g^{-1})  ( A_{1t,a}+ A_{1t,b}+A_{1t,c} )\begin{pmatrix}
 H_1^{-1}f_t\\
 H_2^{-1}g_t
\end{pmatrix} \|^2    \cr
&\leq&      O_P(C_{NT}^{-4})+O_P(C_{NT}^{-2}) \frac{1}{T_0}\sum_{t\in I^c}
 q_t\|  \widehat S_t-S\|^2\cr
&&+ O_P(1)    \frac{1}{T_0}\sum_{t\in I^c} [\|\Delta_{t1,1} \|^2 +  \| \Delta_{t1,2}   \|^2]+O_P(1)   \frac{1}{|\mathcal G|_0}\sum_{j\in\mathcal G}\|   \frac{1}{T_0}\sum_{t\in I^c}
 f_t ^2  e_{jt}   \Delta_{t2,1}     \|^2\cr
               \end{eqnarray*}   
        where $A_{1t,a}= (\widehat  B_t^{-1}-  B^{-1}) (\widehat S_t-S)$, $A_{1t,b}= (\widehat  B_t^{-1}-  B^{-1})  S$, $  A_{1t,c}=B^{-1} (\widehat S_t-S),$
\begin{eqnarray*}
  \widehat S_t-S&=&\begin{pmatrix}
  \Delta_{t1,1} &   \Delta_{t1,2}\\
    \Delta_{t2,1}   &  0
  \end{pmatrix}
\end{eqnarray*}       
               
\begin{eqnarray*}
\Delta_{t1,1}&=&  \frac{1}{N}\sum_i\widetilde \lambda_i   \lambda_i'H_1e_{it}(\widehat e_{it}-e_{it})   +  \widetilde \lambda_i   \widetilde\lambda_i'(\widehat e_{it}^2-e_{it}^2) +( \widetilde \lambda_i   \lambda_i'H_1-  \widetilde \lambda_i   \widetilde\lambda_i'  )e_{it}^2
 		- H_1'\lambda_i ( \lambda_i'H_1-  \widetilde\lambda_i'  )\E e_{it}^2\cr
\Delta_{t1,2}&=&  \frac{1}{N}\sum_i\widetilde \lambda_i  (\widehat e_{it} -e_{it})( \alpha_i'H_2-\widetilde\alpha_i')+\widetilde \lambda_i   e_{it} ( \alpha_i'H_2-\widetilde\alpha_i')	\cr
\Delta_{t2,1}&=&  \frac{1}{N}\sum_i \widetilde\alpha_i \lambda_i'H_1(e_{it}-\widehat e_{it})
 		+\widetilde\alpha_i (\lambda_i'H_1 -\widetilde\lambda_i')(\widehat e_{it} -e_{it})+\widetilde\alpha_i (\lambda_i'H_1 -\widetilde\lambda_i')e_{it} 
\end{eqnarray*}       

By Lemma \ref{lc.6},  $  \frac{1}{T_0}\sum_{t\in I^c} [\|\Delta_{t1,1} \|^2 +  \| \Delta_{t1,2}   \|^2]=  O_P(C_{NT}^{-4})$.
It is straightforward to use Cauchy-Schwarz to show  $\frac{1}{T_0}\sum_{t\in I^c}
 q_t\|  \widehat S_t-S\|^2= O_P(C_{NT}^{-2}).$ In addition, 
\begin{eqnarray*}      
 &&
  \frac{1}{|\mathcal G|_0}\sum_{j\in\mathcal G}\|   \frac{1}{T_0}\sum_{t\in I^c}
 f_t ^2  e_{jt}   \Delta_{t2,1}     \|^2\leq \frac{1}{|\mathcal G|_0}\sum_{j\in\mathcal G}\|   \frac{1}{T_0}\sum_{t\in I^c}
    \frac{1}{N}\sum_i  \alpha_i \lambda_i f_t ^2  e_{jt}  (e_{it}-\widehat e_{it})    \|^2\cr
 &&+  \frac{1}{|\mathcal G|_0}\sum_{j\in\mathcal G}\|  \frac{1}{N}\sum_i  \frac{1}{T_0}\sum_{t\in I^c}
 f_t ^2  e_{jt}  \alpha_i (\lambda_i'H_1 -\widetilde\lambda_i')(\widehat e_{it} -e_{it})  \|^2\cr
 &&+  \frac{1}{|\mathcal G|_0}\sum_{j\in\mathcal G}\|  \frac{1}{N}\sum_i  \frac{1}{T_0}\sum_{t\in I^c}
 f_t ^2    \alpha_i (\lambda_i'H_1 -\widetilde\lambda_i')(e_{it} e_{jt} -\E e_{it} e_{jt} )
   \|^2 \cr
   && +  \frac{1}{|\mathcal G|_0}\sum_{j\in\mathcal G}\|  \frac{1}{N}\sum_i  \frac{1}{T_0}\sum_{t\in I^c}
 f_t ^2    \alpha_i (\lambda_i'H_1 -\widetilde\lambda_i')\E e_{it} e_{jt} 
   \|^2+O_P(C_{NT}^{-4}).
\end{eqnarray*}      
By Lemma \ref{ld.2a}, $ \frac{1}{|\mathcal G|_0}\sum_{j\in\mathcal G}\|   \frac{1}{T_0}\sum_{t\in I^c}
    \frac{1}{N}\sum_i  \alpha_i \lambda_i f_t ^2  e_{jt}  (e_{it}-\widehat e_{it})    \|^2=  O_P(C_{NT}^{-4}).$ The middle two terms follow from Cauchy-Schwarz. Also,  by Cauchy-Schwarz
\begin{eqnarray*}    
  &&  \frac{1}{|\mathcal G|_0}\sum_{j\in\mathcal G}\|  \frac{1}{N}\sum_i   \alpha_i (\lambda_i'H_1 -\widetilde\lambda_i') \frac{1}{T_0}\sum_{t\in I^c}
 f_t ^2  \E e_{it} e_{jt} 
   \|^2\cr
   &=&O_P(C_{NT}^{-4})  \max_j \sum_{i=1}^N    | \E e_{it} e_{jt} | =O_P(C_{NT}^{-4}).
   \end{eqnarray*}    
   Hence  $ \frac{1}{|\mathcal G|_0}\sum_{j\in\mathcal G}\|   \frac{1}{T_0}\sum_{t\in I^c}
 f_t ^2  e_{jt}   \Delta_{t2,1}     \|^2= O_P(C_{NT}^{-4}).$ Thus $   \frac{1}{|\mathcal G|_0}\sum_{j\in\mathcal G}\| \mathcal U(\Gamma_1^j)\|^2=O_P(C_{NT}^{-4}). $
 
 Finally,
 \begin{eqnarray*}   
      \frac{1}{|\mathcal G|_0}\sum_{j\in\mathcal G}\| \mathcal U(\Gamma_6^j)\|^2
   &   =&      \frac{1}{|\mathcal G|_0}\sum_{j\in\mathcal G}\| \frac{1}{T_0}\sum_{t\in I^c}
   \widehat f_t \widehat e_{jt}  ( \widehat e_{jt}\lambda_j'H_f^{-1}, \alpha_j'H_g^{-1})  B^{-1}  
   \begin{pmatrix}
   	H_1'&0  \\
   	0& H_2 
   \end{pmatrix}
   \begin{pmatrix}
   	\sum_{d=1}^4B_{dt}  \\
   	\sum_{d=5}^8B_{dt} 
   \end{pmatrix} \|^2
        \end{eqnarray*}     
      It suffices to prove
 \begin{eqnarray*} 
        \frac{1}{|\mathcal G|_0}\sum_{j\in\mathcal G}\|   \frac{1}{T_0}\sum_{t\in I^c}
        \widehat f_t \widehat e_{jt}^2 B_{dt}    \|^2& =&O_P(C_{NT}^{-4}),\quad d=1\sim 4\cr
         \frac{1}{|\mathcal G|_0}\sum_{j\in\mathcal G}\|   \frac{1}{T_0}\sum_{t\in I^c}
        \widehat f_t \widehat e_{jt}  B_{dt}    \|^2& =&O_P(C_{NT}^{-4}),\quad d=5\sim 8.
         \end{eqnarray*}   
     This is proved by Lemma \ref{ld.16}.
      \end{proof}

\begin{lem}\label{ld.16} For $B_{dt}$ defined in (\ref{ed.48}) and (\ref{ed.49}), 
	 \begin{eqnarray*} 
		\frac{1}{|\mathcal G|_0}\sum_{j\in\mathcal G}\|   \frac{1}{T_0}\sum_{t\in I^c}
		\widehat f_t \widehat e_{jt}^2 B_{dt}    \|^2& =&O_P(C_{NT}^{-4}),\quad d=1\sim 4\cr
		\frac{1}{|\mathcal G|_0}\sum_{j\in\mathcal G}\|   \frac{1}{T_0}\sum_{t\in I^c}
		\widehat f_t \widehat e_{jt}  B_{dt}    \|^2& =&O_P(C_{NT}^{-4}),\quad d=5\sim 8.
	\end{eqnarray*}  
\end{lem}

\begin{proof}
	For $r=1,2$ and all $d=1,..,8$,
	\begin{eqnarray*} 
	&&	\frac{1}{|\mathcal G|_0}\sum_{j\in\mathcal G}\|   \frac{1}{T_0}\sum_{t\in I^c}
		\widehat f_t \widehat e_{jt}^r B_{dt}    \|^2\leq\max_{jt}|	\widehat e_{jt}^r-e_{jt}^r|\frac{1}{T_0}\sum_{t\in I^c}   \|	\widehat f_t-H_ff_t \|^2       \frac{1}{T_0}\sum_{t\in I^c}
 \|B_{dt}    \|^2\cr
		&&+ \frac{1}{|\mathcal G|_0}\sum_{j\in\mathcal G}    \frac{1}{T_0}\sum_{t\in I^c}
	\|\widehat f_t-H_ff_t\|^2   e_{jt}^{2r} \frac{1}{T_0}\sum_{t\in I^c}\| B_{dt}    \|^2
		+\frac{1}{|\mathcal G|_0}\sum_{j\in\mathcal G}\|   \frac{1}{T_0}\sum_{t\in I^c}
	  f_t \widehat e_{jt}^r B_{dt}    \|^2\cr
	 &\leq& O_P(C_{NT}^{-4}) + \frac{1}{|\mathcal G|_0}\sum_{j\in\mathcal G}\|   \frac{1}{T_0}\sum_{t\in I^c}
	 f_t e_{jt}^r B_{dt}    \|^2+\frac{1}{|\mathcal G|_0}\sum_{j\in\mathcal G}\|   \frac{1}{T_0}\sum_{t\in I^c}
	 f_t (\widehat e_{jt}^r-  e_{jt}^r ) B_{dt}    \|^2.
	\end{eqnarray*}  

It remains to show the following three steps.

Step 1. Show $ \frac{1}{|\mathcal G|_0}\sum_{j\in\mathcal G}\|   \frac{1}{T_0}\sum_{t\in I^c}
f_t e_{jt} ^2B_{dt}    \|^2= O_P(C_{NT}^{-4})$ for $d=1\sim 4$.

When $d=1, 4$, it follows from Cauchy-Schwarz. 
 \begin{eqnarray*}
 	\frac{1}{|\mathcal G|_0}\sum_{j\in\mathcal G}\|   \frac{1}{T_0}\sum_{t\in I^c}
 	f_t e_{jt} ^2B_{2t}    \|^2 &=&O_P(C_{NT}^{-4})
 	+	\frac{1}{|\mathcal G|_0}\sum_{j\in\mathcal G}\|   \frac{1}{T_0}\sum_{t\in I^c}
 	f_t^2 e_{jt} ^2\frac{1}{N}\sum_i\lambda_i^2 e_{it}  (\widehat e_{it}-e_{it})   \|^2,
\end{eqnarray*} 
which is $O_P(C_{NT}^{-4})$ due to Lemma \ref{ld.2a}.

 Next, by Lemma \ref{ld.5}, for any fixed sequence $c_i$, 
$\frac{1}{T}\sum_{t\notin I}   \|       \frac{1}{N}\sum_i c_ie_{it}  (\dot \lambda_i-H_1'\lambda_i)  \|^2 =O_P( C_{NT}^{-4})$. So 
  \begin{eqnarray*}
 	\frac{1}{|\mathcal G|_0}\sum_{j\in\mathcal G}\|   \frac{1}{T_0}\sum_{t\in I^c}
 	f_t e_{jt} ^2B_{3t}    \|^2 &=&O_P(C_{NT}^{-4})
 	+\frac{1}{|\mathcal G|_0}\sum_{j\in\mathcal G}\|   \frac{1}{T_0}\sum_{t\in I^c}
 	f_t^2 e_{jt} ^2\frac{1}{N}\sum_i\lambda_i e_{it}l_i'w_t (\dot \lambda_i-H_1'\lambda_i)'    \|^2\cr
 	&\leq &O_P(C_{NT}^{-4}).
 \end{eqnarray*}

Step 2. Show $ \frac{1}{|\mathcal G|_0}\sum_{j\in\mathcal G}\|   \frac{1}{T_0}\sum_{t\in I^c}
f_t e_{jt} B_{dt}    \|^2= O_P(C_{NT}^{-4})$ for $d=5\sim 8$.

When $d=5$, it follows from Cauchy-Schwarz. 

Next  by Lemma \ref{ld.2a}, 
  \begin{eqnarray*}
	\frac{1}{|\mathcal G|_0}\sum_{j\in\mathcal G}\|   \frac{1}{T_0}\sum_{t\in I^c}
	f_t e_{jt} B_{6t}    \|^2 &=& \frac{1}{|\mathcal G|_0}\sum_{j\in\mathcal G}\|   \frac{1}{T_0}\sum_{t\in I^c}
	f_t ^2e_{jt} \frac{1}{N}\sum_i \alpha_i (\widehat e_{it}-e_{it}) \lambda_i   \|^2+O_P(C_{NT}^{-4})\cr
	&\leq &O_P(C_{NT}^{-4}).\cr
		\frac{1}{|\mathcal G|_0}\sum_{j\in\mathcal G}\|   \frac{1}{T_0}\sum_{t\in I^c}
	f_t e_{jt} B_{7t}    \|^2 &= &	\frac{1}{|\mathcal G|_0}\sum_{j\in\mathcal G}\|   \frac{1}{T_0}\sum_{t\in I^c}
	f_t ^2 e_{jt} w_t \|^2\|\frac{1}{N}\sum_i\alpha_il_i (\dot \lambda_i-H_1'\lambda_i)   \|^2+O_P(C_{NT}^{-4})\cr
	&\leq&O_P(C_{NT}^{-4}). 
\end{eqnarray*} 

Next, by Lemma \ref{lc.2},    
  \begin{eqnarray*}
	\frac{1}{|\mathcal G|_0}\sum_{j\in\mathcal G}\|   \frac{1}{T_0}\sum_{t\in I^c}
f_t e_{jt} B_{8t}    \|^2 &=&	\frac{1}{|\mathcal G|_0}\sum_{j\in\mathcal G}\|   \frac{1}{T_0}\sum_{t\in I^c}
f_t e_{jt} (\widetilde f_t- H_1^{-1}f_t)  w_t\|^2\|\frac{1}{N}\sum_i\alpha_i l_i  \lambda_i   \|^2\cr
&\leq& O_P(C_{NT}^{-4}). 
\end{eqnarray*} 

Step 3. Show $ \frac{1}{|\mathcal G|_0}\sum_{j\in\mathcal G}\|   \frac{1}{T_0}\sum_{t\in I^c}
f_t (\widehat e_{jt} ^r-  e_{jt}^r  ) B_{dt}    \|^2= O_P(C_{NT}^{-4})$ for $d=1\sim 8$.

It is bounded by Cauchy-Schwarz. 

\end{proof}

  \subsection{Technical lemmas for covariance estimations}

 \begin{lem}\label{ld.15add}
Suppose uniformly over $i\leq N$,  the following terms are $o_P(1)$: $q \|\frac{1}{T}\sum_t f_te_{it}u_{it}\|,q \|\frac{1}{T}\sum_t g_tu_{it}\|,q \|\frac{1}{T}\sum_{s }  f_sw_s'e_{is}\|$,  $q|\frac{1}{T}\sum_{tj}c_j(m_{jt}m_{it} -  \E m_{jt}m_{it} ) |$ for $m_{it}\in \{e_{it}u_{it} ,  u_{it}\}$, $|\frac{1}{T}\sum_t  e_{it}^2 -\E e_{it}^2|$ .
Suppose  
$ q\max_{t\leq T}\|\frac{1}{N}\sum_i\lambda_i\omega_{it}+\alpha_iu_{it}\|=o_P(1) $.

Also, $q  C_{NT}^{-1} (\max_{it}|e_{it}u_{it}|)=o_P(1)$, where
  $$q:=\max_{it}(\|f_tx_{it}\|+1+|x_{it}|).$$

Then (i) $ \max_i[\|\widehat\lambda_i-H_f^{'-1}\lambda_i\|+ \|\widehat\alpha_i-H_g^{'-1}\alpha_i\|]q =o_P(1)$.

(ii) $\max_{t\leq T}[\|\widehat f_t-H_ff_t\|+\|\widehat g_t-H_gg_t\|] q=o_P(1)$.

(iii) $\max_{is}|\widehat u_{is}-u_{is}|=o_P(1)$, 

(iv) $\max_i|\widehat \sigma_i^2-\sigma_i^2|=o_P(1)$, where $\sigma_i^2=\E e_{it}^2$ and $\widehat\sigma_i^2=\frac{1}{T}\sum_t\widehat e_{it}^2$.

 \end{lem}
 
 \begin{proof}
 (i) In view of (\ref{ec.32}),  it  suffices to show $\max_i\|R_{di}\|q=o_P(1)$ for all $d=1\sim 6$.
 First, by our assumption,  
 $q\max_i\|R_{1i}\|=o_P(1)$.  Also by Lemma \ref{lc.8add}, $\max_i\|\widehat D_i^{-1}-D_i^{-1}\|
 =o_P(1)$.  Also,  $q\max_i\|R_{2i}\|=o_P(1)$ follows from   $q\max_{it}|\widehat e_{it}-e_{it}||e_{it}+u_{it}|=o_P(1)$ .

 \textbf{Term $\max_i[\|R_{3i}\|+\|R_{4i}\| ]$. }  It suffices to prove $q\max_i\|R_{3i,d}\|=o_P(1)$ for $R_{3i,d}$ defined in  Lemma \ref{ld.13}.
 \begin{eqnarray*}
 \max_i\|R_{3i,1}\|q&\leq& o_P(1)
 +q\max_i\| \frac{1}{T}\sum_{s } (\widehat f_s-H_ff_s)(\widehat e_{is} -e_{is})(\mu_{it}\lambda_i'f_s -\widehat \mu_{it}\dot\lambda_i'\widetilde f_s) \|\cr
&&  +q\max_i\| \frac{1}{T}\sum_{s } (\widehat f_s-H_ff_s) e_{is} (\mu_{it}\lambda_i'f_s -\widehat \mu_{it}\dot\lambda_i'\widetilde f_s) \|\cr
 &&+  qO_P(1)\max_i\| \frac{1}{T}\sum_{s } f_s(\widehat e_{is}-e_{is})(\mu_{it}\lambda_i'f_s -\widehat \mu_{it}\dot\lambda_i'\widetilde f_s) \|\cr
 &=&o_P(1)
 \end{eqnarray*}
 following from the Cauchy-Schwarz, $q\max_{it}|\widehat e_{it}-e_{it}||e_{it}|=o_P(1)$, $qC_{NT}^{-1}\max_{it}e_{it}^2=o_P(1)$ and  Lemma \ref{ld.17add}.
 Similarly, $ \max_i\|R_{3i,d}\|q=o_P(1)$ for all $d=2\sim 7.$
 
 \textbf{Term $\max_i\|R_{5i}\|$. } 
By the assumption:
 $$
 q\max_i\|\frac{1}{T}\sum_{tj}\lambda_j(\omega_{jt}\omega_{it} -  \E \omega_{jt}\omega_{it} )\|+q\max_i\|\frac{1}{T}\sum_{tj}\alpha_j(u_{jt}u_{it} -  \E u_{jt}u_{it} )\|=o_P(1).
 $$
 It suffices to prove for $A_{dt}$ defined in (\ref{ed.24}),
 $$
q\sum_{d=1}^6\max_i\|    \frac{1}{T}\sum_{t }
\begin{pmatrix}
\widehat e_{it}\I&0  \\
0&\I
  \end{pmatrix}   A_{dt}   u_{it}\|=o_P(1).
$$ 
First, for $A_{1t,a}= (\widehat  B_t^{-1}-  B^{-1}) (\widehat S_t-S)$, $A_{1t,b}= (\widehat  B_t^{-1}-  B^{-1})  S$, $  A_{1t,c}=B^{-1} (\widehat S_t-S),$
\begin{eqnarray*}
&&q \max_i\|    \frac{1}{T}\sum_{t }
\begin{pmatrix}
\widehat e_{it}\I&0  \\
0&\I
  \end{pmatrix}   A_{1t}   u_{it}\|
  \leq q\max_i\|    \frac{1}{T}\sum_{t }
\begin{pmatrix}
e_{it}\I&0  \\
0&\I
  \end{pmatrix}     A_{1t,a}\begin{pmatrix}
 H_1^{-1}f_t\\
 H_2^{-1}g_t
\end{pmatrix}   u_{it}\|\cr
&&+q \max_i\|    \frac{1}{T}\sum_{t }
\begin{pmatrix}
e_{it}\I&0  \\
0&\I
  \end{pmatrix}     A_{1t,b}\begin{pmatrix}
 H_1^{-1}f_t\\
 H_2^{-1}g_t
\end{pmatrix}   u_{it}\|\cr
&&+ q\max_i\|    \frac{1}{T}\sum_{t }
\begin{pmatrix}
e_{it}\I&0  \\
0&\I
  \end{pmatrix}     A_{1t,c}\begin{pmatrix}
 H_1^{-1}f_t\\
 H_2^{-1}g_t
\end{pmatrix}   u_{it}\| +o_P(1)\cr
&=&qO_P(C_{NT}^{-1})(\max_{it}|e_{it}u_{it}|)+o_P(1)=o_P(1).
\end{eqnarray*}

Similarly, $d=2\sim 6$,  Cauchy-Schwarz and $q O_P(C_{NT}^{-1})(\max_{it}|e_{it}u_{it}|)=o_P(1)$ imply 
\begin{eqnarray*}
&& q\max_i\|    \frac{1}{T}\sum_{t }
\begin{pmatrix}
\widehat e_{it}\I&0  \\
0&\I
  \end{pmatrix}   A_{dt}   u_{it}\|
  \leq o_P(1).
  \end{eqnarray*}

 \textbf{Term $\max_i\|R_{6i}\|$. } Lemma \ref{lc.8add} shows $\max_i\|\widehat D_i^{-1}\|=O_P(1)$.  It suffices to prove  $q\max_i\|\Gamma_d^i\|=o_P(1)$, where $\Gamma_d^i$ is as defined in 
 the proof of Lemma \ref{ld.15}.
\begin{eqnarray*}
 q\max_j\|\Gamma_0^j\|&=&q  \max_j\| \frac{1}{T_0}\sum_{t\notin I}
   \begin{pmatrix}
\widehat f_t \widehat e_{jt} \\
\widehat g_t 
  \end{pmatrix}( \lambda_j'H_f^{-1}\widehat e_{jt}b_1, \alpha_j'H_g^{-1}b_2)       \frac{1}{N}\sum_i 
\begin{pmatrix}
 H_1'\lambda_i e_{it}   \\
 H_2'\alpha_i
\end{pmatrix}
 u_{it} 
\|\cr
&=&qO_P(C_{NT}^{-1}) \max_j|\widehat e_{jt}^2+\widehat e_{jt}| =o_P(1)\cr
 \max_j\|\Gamma_d^j\|&=& q\max_j\|  \frac{1}{T_0}\sum_{t\notin I}
   \begin{pmatrix}
\widehat f_t \widehat e_{jt} \\
\widehat g_t
  \end{pmatrix}( \lambda_j'H_f^{-1}\widehat e_{jt}, \alpha_j'H_g^{-1}) A_{dt}\|=o_P(1)
   \end{eqnarray*}
   due to $O_P(C_{NT}^{-1}) \max_j|\widehat e_{jt}^2+\widehat e_{jt}| =o_P(1)$ and $\max_i\| \dot\lambda_i-H_1'\lambda_i\|=o_P(1).  $
   
   (ii)  In view of (\ref{ed.24}), we have 
   $\max_{t\leq T}\|\frac{1}{N}\sum_i\lambda_i\omega_{it}+\alpha_iu_{it}\|=o_P(1)$.  So it suffices to prove $\max_{t\leq T}\|A_{dt}\|q=o_P(1)$ for $d=1\sim 8. $
   
   The conclusion of Lemma \ref{lc.6} can be strengthened to ensure that 
   
   $q\max_{t\leq T} (\|f_t\|+\|g_t\|)\max_{t\leq T}[\|\widehat S_t-S\|+\|\widehat B_t^{-1}-B^{-1}\|]=o_P(1)$.  So this is true for $d=1,2.$ Next, $C_{NT}^{-1}q=o_P(1)$ and Lemma \ref{ld.17add} shows $q\max_{t}\| \widetilde f_t-H_1^{-1}f_t\| =o_P(1).  $ So it is easy to verify using the Cauchy-Schwarz that it also holds true for $d=3\sim 8$.

   (iii) Uniformly in $i,s$,  by part (ii), 
   \begin{eqnarray*}
   	|\widehat u_{is}-u_{is}|&\leq& |x_{it}| \|\widehat\lambda_i-\lambda_i\|\|\widehat f_t\|
   	+C|x_{it}| \|\widehat f_t-f_t\| 
   	\cr
   	&&+\|\widehat\alpha_i-\alpha_i\|\|\widehat g_t\|  + C  \|\widehat g_t-g_t\| =o_P(1).
   \end{eqnarray*}

 (iv) Uniformly in $i\leq N$, 
  \begin{eqnarray*}
 	 \frac{1}{T}\sum_t\widehat e_{it}^2-\E e_{it}^2&\leq& | \frac{1}{T}\sum_t\widehat e_{it}^2-e_{it}^2|+  \frac{1}{T}\sum_t  e_{it}^2 -\E e_{it}^2=o_P(1). 
 \end{eqnarray*}
   
    \end{proof}
   
\begin{lem}\label{ld.17add} Suppose the following are $o_P(1)$:  $\max_{it}|\widehat e_{it}-e_{it}||e_{it}+u_{it}| $   ,  $ C_{NT}^{-1}  \max_{it}| x_{it}|^3 \|f_t\|$, $\max_i|\frac{1}{T}\sum_tu_{it}^2-\E u_{it}^2| $, 
$ q\max_i\|\frac{1}{T}\sum_t f_t  e_{it}u_{it}\|,  $
$ q\max_i\|\frac{1}{T}\sum_t f_t  w_tu_{it}\|, 
$  $q\max_{it}|x_{it} | \max_i\|\frac{1}{T}\sum_tu_{it}g_t\|$,\\
$  \max_i \|\frac{1}{T}\sum_t \omega_{it}^2-\E \omega_{it}^2 \|
+   \max_i \|\frac{1}{T}\sum_t u_{it}^2 w_t^2-\E u_{it}^2 w_t^2 \|$,

$
\max_{t\leq T}|\frac{1}{N}\sum_iu_{it}^2-\E u_{it}^2| , 
$
$q\max_{t\leq T}[\|\frac{1}{N}\sum_i\lambda_ie_{it}u_{it}\|+\| \frac{1}{N}\sum_i\lambda_iw_{t}u_{it}  \|]$, $q\max_{it}\max_{t\leq T}\|\frac{1}{N}\sum_iu_{it}\alpha_i\|$,\\
$\max_{t\leq T}|\frac{1}{N}\sum_i\omega_{it}^2-\E \omega_{it}^2|$, $\max_i\|\frac{1}{N}\sum_iu_{it}^2w_tw_t'-\E u_{it}^2w_tw_t'\|$,

 where
$$q:=\max_{it}(\|f_tx_{it}\|+1+|x_{it}|).$$

 Then  
 (i) $q\max_{it}\| \dot\lambda_i-H_1'\lambda_i\| =o_P(1).  $
(ii) $q\max_{t}\| \widetilde f_t-H_1^{-1}f_t\| =o_P(1).  $

\end{lem}

\begin{proof}
  Recall         \begin{eqnarray*}
\dot\lambda_i&=& \widetilde D_{\lambda i}^{-1} \frac{1}{T}\widetilde F'\diag(X_i) M_{\widetilde g}y_i\cr
&=&  H_1' \lambda_i+ \widetilde D_{\lambda i}^{-1} \frac{1}{T}\widetilde F'\diag(X_i) M_{\widetilde g}(GH_2^{-1'}-\widetilde G)H_2'\alpha_i
\cr
&& 
+ \widetilde D_{\lambda i}^{-1} \frac{1}{T}\widetilde F'\diag(X_i) M_{\widetilde g}\diag(X_i)(FH_1^{-1'}-\widetilde  F)H_1 '\lambda_i + \widetilde D_{\lambda i}^{-1} \frac{1}{T}\widetilde F'\diag(X_i) M_{\widetilde g} u_i.
\end{eqnarray*}
(i) By Lemma \ref{lc.3},  $\max_i\|\widetilde D_{\lambda i}^{-1}\|=O_P(1)$. 

 First we show $q\max_i\|\frac{1}{T}\widetilde F'\diag(X_i) M_{\widetilde g}(GH_2^{-1'}-\widetilde G) \|=o_P(1)$.  It is in fact bounded by
$$
  O_P(C_{NT}^{-1}) \max_{it}|x_{it}|q=o_P(1).
$$
Next, $q\max_i\|\frac{1}{T}\widetilde F'\diag(X_i) M_{\widetilde g}\diag(X_i)(FH_1^{-1'}-\widetilde  F)H_1 '\lambda_i\|$ is bounded by 
$$
O_P(C_{NT}^{-1})q  \max_{it}|x_{it} |^2=o_P(1)
$$
Finally, note that 
$
\max_i\frac{1}{T}\|u_i\|^2=\max_i|\frac{1}{T}\sum_tu_{it}^2-\E u_{it}^2|+O_P(1)=O_P(1).
$ So
\begin{eqnarray*}
&&q\max_i\|\frac{1}{T}\widetilde F'\diag(X_i) M_{\widetilde g} u_i\|\cr
&\leq& q O_P(C_{NT}^{-1})\max_{it}|x_{it}|  \max_i \frac{1}{\sqrt{T}}\| u_i\|\cr
&&+ q\max_i\|\frac{1}{T}\sum_t\widetilde f_t x_{it} u_{it}\|\cr
&&+O_P(1)q
\max_i\|\frac{1}{T}\widetilde F'\diag(X_i) G\| \frac{1}{T} \max_i\| G' u_i\|\cr
&\leq& o_P(1) + O_P(C_{NT}^{-1})q \max_i(\frac{1}{T}\sum_t u_{it}^2x_{it}^2)^{1/2}
\cr
&&+O_P(1)q\max_i[\|\frac{1}{T}\sum_t f_t e_{it} u_{it}\| +\|\frac{1}{T}\sum_t f_t  w_tu_{it}\|]\cr
&&+ O_P(1)q\max_{it}|x_{it} | \max_i\|\frac{1}{T}\sum_tu_{it}g_t\|\cr
&\leq&  o_P(1) + O_P(C_{NT}^{-1})q \max_i(\frac{1}{T}\sum_t \omega_{it}^2-\E \omega_{it}^2)^{1/2}
\cr
&&+O_P(C_{NT}^{-1})q \max_i(\frac{1}{T}\sum_t u_{it}^2 w_t^2-\E u_{it}^2 w_t^2)^{1/2} =o_P(1). 
\end{eqnarray*}

(ii) Recall that 
\begin{eqnarray*}
\widetilde f_s&=& \widetilde D_{fs}^{-1} \frac{1}{N}\widetilde\Lambda'\diag(X_s) M_{\widetilde\alpha}y_s\cr
&=&  H_1^{-1} f_s+ \widetilde D_{fs}^{-1} \frac{1}{N}\widetilde\Lambda'\diag(X_s) M_{\widetilde\alpha}(AH_2-\widetilde A)H_2^{-1}g_s
\cr
&& 
+ \widetilde D_{fs}^{-1} \frac{1}{N}\widetilde\Lambda'\diag(X_s) M_{\widetilde\alpha}\diag(X_s)(\Lambda H_1-\widetilde \Lambda)H_1^{-1} f_s \cr
&&+ \widetilde D_{fs}^{-1} \frac{1}{N}\widetilde\Lambda'\diag(X_s) M_{\widetilde\alpha} u_s  .
\end{eqnarray*}

By Lemma \ref{lc.1},  $\max_s\|\widetilde D_{fs}^{-1}\|=O_P(1)$. 
The rest of the proof is very similar to part (i), so we omit details to avoid repetitions. The only extra assumption, parallel to those of part (i), is that the following terms are $o_P(1)$:
$
\max_{t\leq T}|\frac{1}{N}\sum_iu_{it}^2-\E u_{it}^2| , 
$
$q\max_{t\leq T}[\|\frac{1}{N}\sum_i\lambda_ie_{it}u_{it}\|+\| \frac{1}{N}\sum_i\lambda_iw_{t}u_{it}  \|]$, $q\max_{it}\max_{t\leq T}\|\frac{1}{N}\sum_iu_{it}\alpha_i\|$,\\
$\max_{t\leq T}|\frac{1}{N}\sum_i\omega_{it}^2-\E \omega_{it}^2|$, $\max_i\|\frac{1}{N}\sum_iu_{it}^2w_tw_t'-\E u_{it}^2w_tw_t'\|$.

\end{proof}

\end{document}